\documentclass[10pt]{article}
\usepackage{lmodern}
\usepackage{amsmath}
\usepackage[T1]{fontenc}
\usepackage[utf8]{inputenc}
\usepackage{authblk}
\usepackage{amsfonts}
\usepackage{graphicx}
\usepackage{rotating}
\usepackage{amssymb}
\usepackage[english]{babel}
\usepackage{color}
\usepackage{amsthm}
\usepackage{graphicx}
\usepackage{tkz-euclide}
\tikzset{elegant/.style={smooth,thick,samples=50,cyan}}
\usepackage{color,tikz}
\usetikzlibrary{patterns}
\usetikzlibrary{decorations.pathreplacing,calc}
\usepackage{rotating}
\usepackage{mathrsfs}
\usepackage{makecell}
\usepackage{microtype}
\usepackage{enumerate}
\usepackage{mathscinet}
\usepackage{array}
\usepackage{multirow}
\usepackage{booktabs}%table
\usepackage{enumerate}
\usepackage[cal=boondoxo,bb=ams]{mathalfa}
\usepackage{hyperref}
\hypersetup{hidelinks}
\usepackage{titlesec}
\usepackage{color,tikz}
\usetikzlibrary{decorations.pathreplacing,calc}
\usepackage{rotating}
\tikzset{liltext/.style={font=\tiny}}
\usepackage{float}
\numberwithin{equation}{section}
\newtheorem{theorem}{Theorem}[section]
\newtheorem{proposition}{Proposition}[section]
\newtheorem{lemma}{Lemma}[section]
\newtheorem{corollary}{Corollary}[section]
\newtheorem{remark}{Remark}[section]

\newtheorem{definition}{Definition}[section]

\newcommand\Zhyp{Z_{{\text{hyp}}}}
\newcommand\Zpd{Z_{{\text{pd}}}}
\newcommand\Zred{Z_{{\text{red}}}}
\newcommand\Zell{Z_{{\text{ell}}}}
\newcommand\Zdiss{Z_{{\text{diss}}}}

\DeclareMathOperator{\diag}{diag}
\DeclareMathOperator{\antidiag}{antidiag}

\title{Evolution models with time-dependent coefficients in friction and viscoelastic damping terms}
\author[1]{Halit Sevki Aslan\thanks{Halit Sevki Aslan (halitsevkiaslan@gmail.com)}}
\affil[1]{Department of Computer Science and Mathematics, University of S\~ao Paulo (USP),
\newline 14040-901 Ribeir\~ao Preto, SP, Brazil}
\author[2]{Michael Reissig \thanks{Michael Reissig (reissig@math.tu-freiberg.de)}}
\affil[2]{Faculty for Mathematics and Computer Science, TU Bergakademie Freiberg,
\newline Pr\"{u}ferstr. 9, 09596, Freiberg, Germany}
\date{}
\begin{document}

\maketitle

\begin{abstract}
We study the following Cauchy problem for the linear wave equation with both time-dependent friction and time-dependent viscoelastic damping:
\begin{equation} \label{EqAbstract}\tag{$\ast$}
\begin{cases}
u_{tt}- \Delta u + b(t)u_t - g(t)\Delta u_t=0, &(t,x) \in (0,\infty) \times \mathbb{R}^n, \\
u(0,x)= u_0(x),\quad u_t(0,x)= u_1(x), &x \in \mathbb{R}^n.
\end{cases}
\end{equation}
Our aim is to derive decay estimates for higher order energy norms of solutions to this problem. We focus on the interplay between the time-dependent coefficients in both damping terms and their influence on the qualitative behavior of solutions. The analysis is based on a classification of the damping mechanisms, frictional damping $b(t)u_t$ and viscoelastic damping $-g(t)\Delta u_t$ as well, and employs the WKB-method in the extended phase space.
\end{abstract}
	
	
\noindent\textbf{AMS Classification (2020)}  35L05, 35L15, 35B40

\fontsize{12}{15}
\selectfont
%\tableofcontents

\tableofcontents
\section{Introduction}\label{Sec_Intro}
%\subsection{Background for evolution models with time-dependent coefficients in friction or visco-elastic damping} \label{Section1.1}
%\subsection{for the friction and visco-elastic damped wave equations} \label{Section1.1}
\subsection{Our model and goal in this paper} \label{Section_OurGoal}
We consider the following Cauchy problem for the linear wave model with time-dependent friction and time-dependent viscoelastic damping terms:
\begin{equation} \label{MainEquation}
\begin{cases}
u_{tt}- \Delta u + b(t)u_t - g(t)\Delta u_t=0, &(t,x) \in (0,\infty) \times \mathbb{R}^n, \\
u(0,x)= u_0(x),\quad u_t(0,x)= u_1(x), &x \in \mathbb{R}^n.
\end{cases}
\end{equation}
Our goal in this paper is to study decay behaviors of the solution to the Cauchy problem \eqref{MainEquation}.
%WKB-analysis is one of the suitable method in order to solve models with time-dependent coefficients.

Very recently, the authors of this paper studied in \cite{AslanReissig2023} the following Cauchy problem for the linear viscoelastic damped wave model with general time-dependent coefficient $g=g(t)$:
\begin{equation} \label{ViscoelasticEquation}
\begin{cases}
u_{tt}- \Delta u -g(t)\Delta u_t=0, &(t,x) \in (0,\infty) \times \mathbb{R}^n, \\
u(0,x)= u_0(x),\quad u_t(0,x)= u_1(x), &x \in \mathbb{R}^n.
\end{cases}
\end{equation}
The authors distinguished the following separate cases for the time-dependent coefficient $g=g(t)$:
\begin{enumerate}
\item models with increasing time-dependent coefficient $g=g(t)$,
\item models with integrable and decaying time-dependent coefficient $g=g(t)$,
\item models with non-integrable and decreasing time-dependent coefficient $g=g(t)$,
\item models with non-integrable and slowly increasing time-dependent coefficient $g=g(t)$.
\end{enumerate}
They studied decay rates of energies of higher order for solutions to the Cauchy problem \eqref{ViscoelasticEquation} based on the classification of time-dependent coefficient $g=g(t)$ in the viscoelastic damping.

On the other hand, in the PhD thesis \cite{WirthThesis} and in the papers \cite{Wirth-Noneffective=2006,Wirth-Effective=2007} the author proposed a classification of time-dependent friction for the following model:
\begin{equation} \label{WirthModel}
\begin{cases}
u_{tt}- \Delta u + b(t)u_t=0, &(t,x) \in (0,\infty) \times \mathbb{R}^n, \\
u(0,x)= u_0(x),\quad u_t(0,x)= u_1(x), &x \in \mathbb{R}^n.
\end{cases}
\end{equation}
The classification of the term $b(t)u_t$ was proposed as follows:
\begin{enumerate}
\item scattering producing to free wave equation,
\item non-effective dissipation,
\item effective dissipation,
\item over-damping producing.
\end{enumerate}

After these classifications of the models \eqref{ViscoelasticEquation} and \eqref{WirthModel}, the following challenging question naturally comes into mind:\medskip

\textbf{Question:} \textit{How does the classification of the damping term $b(t)u_t$ and the classification of the viscoelastic damping term $-g(t)\Delta u_t$ affect the decay rates of energies of higher order for solutions to the Cauchy problem \eqref{MainEquation}? In other words, how does the relationship between both damping terms $b(t)u_t$ and $-g(t)\Delta u_t$ affect the decay rates of the solutions to the Cauchy problem \eqref{MainEquation}?}
\medskip

We are going to divide our considerations into following cases depending on the classification of the friction:
\begin{itemize}
\item the friction term is scattering producing in Section \ref{Section_Scattering},
\item the friction term is non-effective in Section \ref{Section_Noneffective},
\item the friction term is effective in Section \ref{Section_Effective},
\item the friction term is over-damping producing in Section \ref{Section_Overdamping}.
\end{itemize}
Then, we will employ in our considerations the classification of the viscoelastic damping, as well.

\subsection{Background for some evolution models with time-dependent coefficients} \label{Section1.1}
Let us consider some known results for the linear Cauchy problem to the wave equation with time-dependent dissipation in \eqref{WirthModel}. The term $b(t)u_t$ is called the damping term, which prevents the motion of the wave and reduces its energy. The coefficient $b=b(t)$ represents the strength of the damping, and the asymptotic behavior of solutions, as well as the decay of wave energy, depends crucially on this time-dependent coefficient.\\
According to Wirth's classification,
\begin{itemize}
\item if $b\in L^1([0,\infty))$, the damping term does not significantly influence the long-time behavior of solutions. In this case, the solution scatters to that of the free wave equation as $t\to\infty$, and the damping is said to be \textit{scattering producing};
\item if the $L^p-L^q$ estimates for the solution to the Cauchy problem \eqref{WirthModel} are closely related to those of the solutions to the free wave equation, then the damping term is called \textit{non-effective};
\item if the solution decays like that of the corresponding parabolic Cauchy problem
\begin{align*}
\begin{cases}
v_t=\frac{1}{b(t)}\Delta v, & (t,x)\in[0,\infty)\times \mathbb{R}^n, \\
v(0,x)=v_0(x), & x\in\mathbb{R}^n,
\end{cases}
\end{align*}
for some suitable data $v_0$, depending on $u_0$, $u_1$ and $b=b(t)$,  the damping term is called \textit{effective}, as it has a stronger influence on the solution's behavior;
\item if $\frac{1}{b(t)}\in L^1([0,\infty))$, the damping is so strong that no decay estimate for the energy is possible. This situation is referred to as \textit{over-damping producing}.
\end{itemize}
We note that in both the scattering and over-damping cases, no energy decay occurs in general. Detailed definitions and a more in-depth discussion of this classification will be provided in the corresponding sections.

Regarding recent contributions to the study of the following Cauchy problem for \emph{structurally damped $\sigma$-evolution equation with time-dependent dissipation}:
\begin{equation*}\label{Equation_sigma-delta}
\begin{cases}
u_{tt}+ (-\Delta)^\sigma u+ g(t) (-\Delta)^\delta u_t = 0, &(t,x) \in (0,\infty) \times \mathbb{R}^n, \\
u(0,x)= u_0(x),\quad u_t(0,x)= u_1(x), &x \in \mathbb{R}^n,
\end{cases}
\end{equation*}
with $\sigma\geq 1$ and $\delta\in(0,\sigma]$, we refer to \cite{{DAbbiccoEbert2016}, {EbertReissigBook}, {KainaneReissig2015-1}, {KainaneReissig2015-2}, {ReissigLu2009}, {Reissig=2011}, {JuniordaLuzNoneffective}, {JuniordaLuzEffective}, {WirthThesis}, {Wirth-Noneffective=2006}, {Wirth-Effective=2007}}, and especially to the recent paper  \cite{AslanReissig2023} and the references therein.

Finally, in \cite{AbdelatifReissig=2023}, the authors studied the model \eqref{MainEquation} with scale-invariant friction and viscoelastic damping, specifically choosing $b(t)=\dfrac{\mu_1}{1+t}$ and $g(t)=\mu_2(1+t)$ in \eqref{MainEquation}, with the nonlinearity $|u|^p$ on the right-hand side, where $\mu_1$ and $\mu_2$ are positive constants.
%By deriving estimates for the associated linear Cauchy problem, they established global (in time) existence results for small data Sobolev solutions to the semilinear problem, under suitable conditions on the exponent $p$. Additionally, they proved blow-up results identifying the critical exponent in relation to certain $L^m$ integrability condition on the second initial datum, with $m\in(1,2)$.
Let us note that the treatment of the linear problem in this case differs significantly due to the scale-invariance of the coefficients, which allows for an explicit representation of solutions using special functions and facilitates a more precise analysis.
\subsection{Two changes of variables} \label{Section_OurApproach}
To analyze the Cauchy problem \eqref{MainEquation}, we apply the partial Fourier transformation with respect to spatial variables $\hat{u}(t,\xi)=\mathcal{F}_{x\rightarrow\xi}\left( u(t,x) \right)$, and obtain
\begin{equation} \label{MainEquationFourier}
\begin{cases}
\hat{u}_{tt} + |\xi|^2\hat{u} + \big( b(t) + g(t)|\xi|^2 \big)\hat{u}_t = 0, &(t,\xi) \in [0,\infty) \times \mathbb{R}^n, \\
\hat{u}(0,\xi) = \hat{u}_0(\xi),\quad \hat{u}_t(0,\xi)= \hat{u}_1(\xi), &\xi \in \mathbb{R}^n.
\end{cases}
\end{equation}
Depending on whether or not we utilize derivatives of $b$ in the transformed model, we will consider two different changes of variables in the forthcoming analysis.
\medskip

\noindent \textbf{First change of variables:} If we apply the change of variables
\[ \hat{u}(t,\xi)=\exp\Big(-\frac{1}{2} \int_0^t g(\tau)|\xi|^2 d\tau\Big) v(t,\xi), \]
then we get
\begin{equation} \label{AuxiliaryEquation1}
\begin{cases}
v_{tt} + |\xi|^2\Big(1-\dfrac{b(t)g(t)}{2}-\dfrac{g^2(t)|\xi|^2}{4}-\dfrac{g'(t)}{2}\Big)v+b(t)v_t=0, &(t,\xi) \in [0,\infty) \times \mathbb{R}^n, \\
v(0,\xi)= v_0(\xi),\quad v_t(0,\xi)= v_1(\xi), &\xi \in \mathbb{R}^n,
\end{cases}
\end{equation}
where
\[ v_0(\xi) = \hat{u}_0(\xi) \quad \text{and} \quad v_1(\xi) = \frac{g(0)}{2}|\xi|^2\hat{u}_0(\xi) + \hat{u}_1(\xi). \]
\textbf{Second change of variables:} Applying the change of variables
\[ \hat{u}(t,\xi)=\exp\Big(-\frac{1}{2} \int_0^t \big( b(\tau)+g(\tau)|\xi|^2 \big)d\tau\Big) w(t,\xi) \]
we arrive at
\begin{equation} \label{AuxiliaryEquation3}
\begin{cases}
w_{tt} + \Big[ |\xi|^2\Big( 1-\dfrac{b(t)g(t)}{2}-\dfrac{g^2(t)|\xi|^2}{4}-\dfrac{g'(t)}{2} \Big)-\dfrac{b^2(t)}{4}-\dfrac{b'(t)}{2} \Big]w=0, &(t,\xi) \in [0,\infty) \times \mathbb{R}^n, \\
w(0,\xi)= w_0(\xi),\quad w_t(0,\xi)= w_1(\xi), &\xi \in \mathbb{R}^n,
\end{cases}
\end{equation}
where
\[ w_0(\xi) = \hat{u}_0(\xi) \quad \text{and} \quad w_1(\xi) = \Big( \frac{b(0)}{2} + \frac{g(0)}{2}|\xi|^2 \Big)\hat{u}_0(\xi) + \hat{u}_1(\xi). \]
\textbf{Notations}
\begin{itemize}
\item We write $f\lesssim g$ when there exists a constant $C>0$ such that $f\leq Cg$, and $f \approx g$ when $g\lesssim f\lesssim g$.
\item $f\sim g$ denotes $\lim_{t\rightarrow\infty}\dfrac{f(t)}{g(t)}=1$, that is, $f$ and $g$ have the same asymptotic behavior.
\item $D_t$ stands for $-i\partial_t$.
\item As usual, the spaces $H^a$ and $\dot{H}^a$ with $a \geq 0$ stand for Bessel and Riesz potential spaces based on the $L^2$ space. Here $\langle D\rangle^a$ and $|D|^a$ denote the pseudo-differential operator with symbol $\langle\xi\rangle^a$ and the fractional Laplace operator with symbol $|\xi|^a$, respectively.
\item $(|A|)$ denotes the matrix of absolute values of its entries $a_{jk}$ and $(|A|)\leq (|B|)$ denotes $a_{jk}\leq b_{jk}$ for $j,k=1, \cdots,n$. Moreover, the identity matrix is denoted by $I$.
\end{itemize}

\section{The friction term is scattering producing} \label{Section_Scattering}
In this section, we would like to study the Cauchy problem \eqref{MainEquation} with a scattering producing friction $b(t)u_t$ and with the given classification of the viscoelastic damping $-g(t)\Delta u_t$ in \cite{AslanReissig2023}, separately in the following subsections.
\subsection{Model with increasing time-dependent coefficient $g=g(t)$} \label{Subsect_Scattering_Increasing}
As we studied in Section 2 of the paper \cite{AslanReissig2023}, we assume the following properties for the coefficient $g=g(t)$ for all $t \in [0,\infty)$:
\begin{enumerate}
\item[\textbf{(A1)}] $g(t)>0$ and $g'(t)>0$,
\item[\textbf{(A2)}] $\dfrac{1}{g} \in L^1([0,\infty))$,
\item[\textbf{(A3)}] $|d_t^kg(t)|\leq C_kg(t)\Big( \dfrac{g(t)}{G(t)} \Big)^k$ for $k=1,2$, where $G(t):=\dfrac{1}{2}\displaystyle\int_0^t g(\tau)d\tau$ and $C_1$, $C_2$ are positive constants.
\end{enumerate}
On the other hand, we pose the following conditions to the coefficient $b=b(t)>0$ for all $t \in [0,\infty)$:
\begin{enumerate}
\item[\textbf{(A'1)}] $b'(t)<0$ and $b\in L^1([0,\infty))$,
\item[\textbf{(A'2)}] $b(t)\leq \widetilde{C}\dfrac{g(t)}{G(t)}$, where $\widetilde{C}$ is positive constant,
\item[\textbf{(A'3)}] $|b'(t)| \leq c \dfrac{b(t)}{1+t}$, where $c$ is a positive constant.
\end{enumerate}
We are not able to consider a general scattering producing friction $b(t)u_t$ with $b \in L^1([0,\infty))$, but we have to suppose more structural properties \textbf{(A'2)} and \textbf{(A'3)}.
\begin{theorem} \label{Theorem_Sect-Scattering_Increasing}
Let us consider the Cauchy problem
\begin{equation*}
\begin{cases}
u_{tt}- \Delta u + b(t)u_t -g(t)\Delta u_t=0, &(t,x) \in [0,\infty) \times \mathbb{R}^n, \\
u(0,x)= u_0(x),\quad u_t(0,x)= u_1(x), &x \in \mathbb{R}^n.
\end{cases}
\end{equation*}
We assume that the coefficients $g=g(t)$ and $b=b(t)$ satisfy the conditions \textbf{(A1)} to \textbf{(A3)} and \textbf{(A'1)} to \textbf{(A'3)}, respectively. Moreover, we suppose that $(u_0,u_1)\in \dot{H}^{|\beta|} \times \dot{H}^{|\beta|-2}$ with $|\beta|\geq 2$. Then, we have the following estimates for Sobolev solutions:
\begin{align*}
\|\,|D|^{|\beta|} u(t,\cdot)\|_{L^2} & \lesssim  \|u_0\|_{\dot{H}^{|\beta|}} + \|u_1\|_{\dot{H}^{|\beta|-2}},\\
\|\,|D|^{|\beta|-2} u_t(t,\cdot)\|_{L^2} & \lesssim g(t)\big( \|u_0\|_{\dot{H}^{|\beta|}} + \|u_1\|_{\dot{H}^{|\beta|-2}} \big).
\end{align*}
\end{theorem}
\begin{remark} \label{Remark2.1}
We may see that Theorem \ref{Theorem_Sect-Scattering_Increasing} implies that Theorem 2.1 from \cite{AslanReissig2023} remains valid with a special class of additional scattering producing friction terms satisfying conditions \textbf{(A'1)} to \textbf{(A'3)}. This means that a scattering producing friction term from this class has no essential influence on the estimates of the solution to the Cauchy problem \eqref{MainEquation}.
\end{remark}
\begin{proof}[Proof of Theorem \ref{Theorem_Sect-Scattering_Increasing}]
We divide the extended phase space $[0,\infty)\times \mathbb{R}^n$ into zones as follows:
\begin{itemize}
\item pseudo-differential zone:
\begin{align*} \label{zonesellipticcase}
\Zpd(N)=\left\{ (t,\xi)\in [0,\infty)\times\mathbb{R}^n: G(t)|\xi|^2\leq N \right\},
\end{align*}
\item elliptic zone:
\begin{align*}
\Zell(N)=\left\{ (t,\xi)\in [0,\infty)\times \mathbb{R}^n: G(t)|\xi|^2\geq N \right\},
\end{align*}
\end{itemize}
where $N>0$ is sufficiently large. The separating line $t_\xi=t(|\xi|)$ is defined by
\[ t_\xi=\left\{ (t,\xi) \in [0,\infty) \times \mathbb{R}^n: G(t)|\xi|^2=N \right\}. \]
\subsubsection{Considerations in the elliptic zone $\Zell(N)$} \label{Sect_Scattering_Increasing-Zell}
Let us write the transformed equation \eqref{AuxiliaryEquation1} in the following form:
\begin{equation} \label{Eq:ScatteringPseudoForm}
D_t^2v + \dfrac{g^2(t)}{4}|\xi|^4v + \Big( \dfrac{g'(t)}{2} -1 \Big)|\xi|^2v + \dfrac{b(t)g(t)}{2}|\xi|^2v - ib(t)D_tv=0.
\end{equation}
We introduce the micro-energy $V=V(t,\xi):=\big( \dfrac{g(t)}{2}|\xi|^2v,D_tv \big)^{\text{T}}$. Then, by \eqref{Eq:ScatteringPseudoForm} we obtain that $V=V(t,\xi)$ satisfies the following system of first order:
\begin{equation*} \label{ScatteringSystemWithbANDg}
D_tV=\left( \begin{array}{cc}
0 & \dfrac{g(t)}{2}|\xi|^2 \\
-\dfrac{g(t)}{2}|\xi|^2 & 0
\end{array} \right)V + \left( \begin{array}{cc}
\dfrac{D_tg(t)}{g(t)} & 0 \\
-\dfrac{g'(t)-2}{g(t)}-b(t) & ib(t)
\end{array} \right)V.
\end{equation*}
We carry out the first step of diagonalization procedure. For this reason, we set
\[ M := \left( \begin{array}{cc}
1 & -1 \\
i & i
\end{array} \right), \qquad M^{-1}=\frac{1}{2}\left( \begin{array}{cc}
1 & -i \\
-1 & -i
\end{array} \right). \]
We define $V^{(0)}:=M^{-1}V$ and get the system
\begin{equation*}
D_tV^{(0)}=\big( \mathcal{D}(t,\xi)+\mathcal{R}_b(t)+\mathcal{R}_g(t) \big)V^{(0)},
\end{equation*}
where
\begin{align*}
\mathcal{D}(t,\xi) &= \left( \begin{array}{cc}
i\dfrac{g(t)}{2}|\xi|^2 & 0 \\
0 & -i\dfrac{g(t)}{2}|\xi|^2
\end{array} \right), \qquad \mathcal{R}_b(t) = \left( \begin{array}{cc}
ib(t) & 0 \\
ib(t) & 0
\end{array} \right), \\
\mathcal{R}_g(t) &= \frac{1}{2} \left( \begin{array}{cc}
\dfrac{D_tg(t)}{g(t)}+i\dfrac{g'(t)-2}{g(t)} & -\dfrac{D_tg(t)}{g(t)}-i\dfrac{g'(t)-2}{g(t)} \\
-\dfrac{D_tg(t)}{g(t)}+i\dfrac{g'(t)-2}{g(t)} & \dfrac{D_tg(t)}{g(t)}-i\dfrac{g'(t)-2}{g(t)}
\end{array} \right).
\end{align*}
Here we take account of
\begin{align*}
 M^{-1}\left( \begin{array}{cc}
0 & 0 \\
-b(t) & 0
\end{array} \right)M = \frac{1}{2}\left( \begin{array}{cc}
ib(t) & -ib(t) \\
ib(t) & -ib(t)
\end{array} \right) \qquad \text{and} \qquad M^{-1}\left( \begin{array}{cc}
0 & 0 \\
0 & ib(t)
\end{array} \right)M = \frac{1}{2}\left( \begin{array}{cc}
ib(t) & ib(t) \\
ib(t) & ib(t)
\end{array} \right).
\end{align*}
We introduce $F_0(t)=\diag( \mathcal{R}_b(t)+\mathcal{R}_g(t))$ and carry out the next step of diagonalization procedure. The difference of the diagonal entries of the matrix $\mathcal{D}(t,\xi)+F_0(t)$ is
\begin{equation*}
i\delta(t,\xi):=g(t)|\xi|^2 + \frac{g'(t)-2}{g(t)} + b(t)\sim g(t)|\xi|^2
\end{equation*}
for $t\geq t_\xi$ if we choose the zone constant $N$ sufficiently large and apply conditions \textbf{(A3)} and \textbf{(A'1)}. We choose a matrix $N^{(1)}=N^{(1)}(t,\xi)$ such that
\[ N^{(1)}(t,\xi)=\left( \begin{array}{cc}
0 & \dfrac{\mathcal{R}_{12}}{\delta(t,\xi)} \\
-\dfrac{\mathcal{R}_{21}}{\delta(t,\xi)} & 0
\end{array} \right)\sim \left( \begin{array}{cc}
0 & -i\dfrac{D_tg(t)}{2g^2(t)|\xi|^2}+\dfrac{g'(t)-2}{2g^2(t)|\xi|^2} \\
i\dfrac{D_tg(t)}{2g^2(t)|\xi|^2}+\dfrac{g'(t)-2}{2g^2(t)|\xi|^2}+\dfrac{b(t)}{g(t)|\xi|^2} & 0
\end{array} \right). \]
For a sufficiently large zone constant $N$ and all $t\geq t_\xi$ the matrix $N_1=N_1(t,\xi)$, where $N_{1}(t,\xi)=I+N^{(1)}(t,\xi)$, is invertible with uniformly bounded inverse $N_1^{-1}=N_1^{-1}(t,\xi)$. Indeed, in the elliptic zone $\Zell(N)$ it holds
\begin{align*}
\Big|\dfrac{D_tg(t)}{2g^2(t)|\xi|^2}\Big| \leq \frac{C}{G(t)|\xi|^2}\leq \frac{C}{N} \qquad \text{and} \qquad \Big| \dfrac{b(t)}{g(t)|\xi|^2} \Big| \leq \frac{\widetilde{C}}{G(t)|\xi|^2}\leq \frac{\widetilde{C}}{N},
\end{align*}
where we used conditions \textbf{(A3)} and \textbf{(A'2)}. Let
\begin{align*}
B^{(1)}(t,\xi) &= D_tN^{(1)}(t,\xi)-\big( \mathcal{R}_b(t)+\mathcal{R}_g(t)-F_0(t,\xi) \big)N^{(1)}(t,\xi), \\
\mathcal{R}_1(t,\xi) &= -N_1^{-1}(t,\xi)B^{(1)}(t,\xi).
\end{align*}
Then, we have the following operator identity:
\begin{equation*}
\big( D_t-\mathcal{D}(t,\xi)-\mathcal{R}_b(t)-\mathcal{R}_g(t) \big)N_1(t,\xi)=N_1(t,\xi)\big( D_t-\mathcal{D}(t,\xi)-F_0(t)-\mathcal{R}_1(t,\xi) \big).
\end{equation*}
\begin{proposition} \label{Prop_Scattering_EllZone}
Let $E_{\text{ell}}^V=E_{\text{ell}}^V(t,s,\xi)$ be the fundamental solution of the operator
\[ D_t-\mathcal{D}(t,\xi)-F_0(t)-\mathcal{R}_1(t,\xi). \]
Then, $E_{\text{ell}}^V=E_{\text{ell}}^V(t,s,\xi)$ holds the following estimate:
\begin{align*} \label{correctformula}
(|E_{\text{ell}}^V(t,s,\xi)|) \lesssim \frac{g(t)}{g(s)} \exp \Big(\frac{|\xi|^2}{2} \int_s^t g(\tau) d\tau \Big) \left(\begin{array}{cc}
1 & 1 \\
1 & 1
\end{array}\right)
\end{align*}
with $(t,\xi),(s,\xi)\in \Zell(N)$ and $t_\xi\leq s\leq t$.
\end{proposition}
\begin{proof}
We transform the system for $E_{\text{ell}}^V=E_{\text{ell}}^V(t,s,\xi)$ to an integral equation for a new matrix-valued function $\mathcal{Q}_{\text{ell}}=\mathcal{Q}_{\text{ell}}(t,s,\xi)$. If we differentiate the term
\[ \exp \bigg\{ -i\int_{s}^{t}\big( \mathcal{D}(\tau,\xi)+F_0(\tau) \big)d\tau \bigg\}E_{\text{ell}}^V(t,s,\xi) \]
and, then integrate on the interval $[s,t]$, we find that $E_{\text{ell}}^V=E_{\text{ell}}^V(t,s,\xi)$ satisfies the following integral equation:
\begin{align*}
E_{\text{ell}}^V(t,s,\xi) & = \exp\bigg\{ i\int_{s}^{t}\big( \mathcal{D}(\tau,\xi)+F_0(\tau) \big)d\tau \bigg\}E_{\text{ell}}^V(s,s,\xi)\\
& \quad + i\int_{s}^{t} \exp \bigg\{ i\int_{\theta}^{t}\big( \mathcal{D}(\tau,\xi)+F_0(\tau) \big)d\tau \bigg\}\mathcal{R}_1(\theta,\xi)E_{\text{ell}}^V(\theta,s,\xi)\,d\theta.
\end{align*}
Let us define
\[ \mathcal{Q}_{\text{ell}}^V(t,s,\xi)=\exp\bigg\{ -\int_{s}^{t}\beta(\tau,\xi)d\tau \bigg\} E_{\text{ell}}^V(t,s,\xi), \]
with a suitable $\beta=\beta(t,\xi)$ which will be fixed later. Then, it satisfies the new integral equation
\begin{align*}
\mathcal{Q}_{\text{ell}}(t,s,\xi)=&\exp \bigg\{ \int_{s}^{t}\big( i\mathcal{D}(\tau,\xi)+iF_0(\tau)-\beta(\tau,\xi)I \big)d\tau \bigg\}\mathcal{Q}_{\text{ell}}(s,s,\xi)\\
& \quad +\int_{s}^{t} \exp \bigg\{ \int_{\theta}^{t}\big( i\mathcal{D}(\tau,\xi)+iF_0(\tau)-\beta(\tau,\xi)I \big)d\tau \bigg\}\mathcal{R}_1(\theta,\xi)\mathcal{Q}_{\text{ell}}(\theta,s,\xi)\,d\theta.
\end{align*}
The function $\mathcal{R}_1=\mathcal{R}_1(\theta,\xi)$ is uniformly integrable over the elliptic zone (see Proposition 2.4 in \cite{AslanReissig2023}).

The main entries of the diagonal matrix $i\mathcal{D}(t,\xi)+iF_0(t)$ are given by
\begin{align*}
(I)  =& \frac{g(t)}{2}|\xi|^2+\frac{g'(t)}{2g(t)}+\frac{g'(t)-2}{2g(t)},\\
(II) =& -\frac{g(t)}{2}|\xi|^2+\frac{g'(t)}{2g(t)}-\frac{g'(t)-2}{2g(t)}-b(t).
\end{align*}
It follows that the term $(I)$ is dominant. Therefore, we choose the weight
\begin{align*}
\beta=\beta(t,\xi)=(I)=\frac{g(t)}{2}|\xi|^2+\frac{g'(t)}{2g(t)}+\frac{g'(t)-2}{2g(t)}.
\end{align*}
By this choice, we get
\[ i\mathcal{D}(\tau,\xi)+iF_0(\tau)-\beta(\tau,\xi)I = \left( \begin{array}{cc}
-g(\tau)|\xi|^2-\dfrac{g'(\tau)-2}{g(\tau)}-b(t) & 0 \\
0 & 0
\end{array} \right). \]
This implies
\begin{align*}
H(t,s,\xi) & =\exp \bigg\{ \int_{s}^{t}\big( i\mathcal{D}(\tau,\xi)+iF_0(\tau)-\beta(\tau,\xi)I \big)d\tau \bigg\}\\
& = \diag \bigg( \exp \bigg\{ \int_{s}^{t}\Big( -g(\tau)|\xi|^2-\dfrac{g'(\tau)-2}{g(\tau)}\Big)d\tau \bigg\},1 \bigg)\rightarrow \left( \begin{array}{cc}
0 & 0 \\
0 & 1
\end{array} \right)
\end{align*}
as $t\rightarrow \infty$ for any fixed $s\geq t_\xi$. Hence, the matrix $H=H(s,t,\xi)$ is uniformly bounded for $(s,\xi),(t,\xi)\in \Zell(N)$. So, the representation of $\mathcal{Q}_{\text{ell}}=\mathcal{Q}_{\text{ell}}(t,s,\xi)$ by a Neumann series gives
\begin{align*}
\mathcal{Q}_{\text{ell}}(t,s,\xi)=H(t,s,\xi)+\sum_{k=1}^{\infty}i^k\int_{s}^{t}H(t,t_1,\xi)&\mathcal{R}_1(t_1,\xi)\int_{s}^{t_1}H(t_1,t_2,\xi)\mathcal{R}_1(t_2,\xi) \\
& \cdots \int_{s}^{t_{k-1}}H(t_{k-1},t_k,\xi)\mathcal{R}_1(t_k,\xi)dt_k\cdots dt_2dt_1.
\end{align*}
Then, convergence of this series is obtained from the symbol estimates, since $\mathcal{R}_1=\mathcal{R}_1(t,\xi)$ is uniformly integrable over $\Zell(N)$. Hence, from the last considerations we may conclude
\begin{align*}
E_{\text{ell}}^V(t,s,\xi)&=\exp \bigg\{ \int_{s}^{t}\beta(\tau,\xi)d\tau \bigg\}\mathcal{Q}_{\text{ell}}(t,s,\xi) \\
& = \exp \bigg\{ \int_{s}^{t}\bigg( \frac{g(\tau)}{2}|\xi|^2+\frac{g'(\tau)}{2g(\tau)}+\frac{g'(\tau)-2}{2g(\tau)} \bigg)d\tau \bigg\}\mathcal{Q}_{\text{ell}}(t,s,\xi),
\end{align*}
where $\mathcal{Q}_{\text{ell}}=\mathcal{Q}_{\text{ell}}(t,s,\xi)$ is a uniformly bounded matrix. Then, it follows
\begin{align*}
(|E_{\text{ell}}^V(t,s,\xi)|) & \lesssim \exp \bigg\{ \int_{s}^{t}\bigg( \frac{g(\tau)}{2}|\xi|^2+\frac{g'(\tau)}{2g(\tau)}+\frac{g'(\tau)-2}{2g(\tau)} \bigg)d\tau \bigg\} \left( \begin{array}{cc}
1 & 1 \\
1 & 1
\end{array} \right) \\
& \lesssim \frac{g(t)}{g(s)} \exp\bigg( |\xi|^2\int_{s}^{t} \frac{g(\tau)}{2}d\tau \bigg)\left( \begin{array}{cc}
1 & 1 \\
1 & 1
\end{array} \right).
\end{align*}
This completes the proof.	
\end{proof}
\subsubsection{Considerations in the pseudo-differential zone $\Zpd(N)$} \label{Sect_Scattering_Increasing-Zpd}
We define the micro-energy $U=(\gamma(t,\xi)\hat{u},D_t\hat{u})^\text{T}$ with $\gamma(t,\xi):=\dfrac{g(t)}{2}|\xi|^2$. Then, the Cauchy problem \eqref{MainEquationFourier} leads to the system of first order
\begin{equation} \label{Eq:Scattering_Increasing_SystemZpd}
D_tU=\underbrace{\left( \begin{array}{cc}
\dfrac{D_t\gamma(t,\xi)}{\gamma(t,\xi)} & \gamma(t,\xi) \\
\dfrac{|\xi|^2}{\gamma(t,\xi)} & i\big( b(t)+g(t)|\xi|^2 \big)
\end{array} \right)}_{A(t,\xi)}U.
\end{equation}
We are interested in the fundamental solution $E_{\text{pd}}=E_{\text{pd}}(t,s,\xi)$ to the system \eqref{Eq:Scattering_Increasing_SystemZpd} as follows:
\[ D_tE_{\text{pd}}(t,s,\xi)=A(t,\xi)E_{\text{pd}}(t,s,\xi), \quad E_{\text{pd}}(s,s,\xi)=I, \]
for all $0\leq s \leq t$ and $(t,\xi), (s,\xi) \in \Zpd(N)$. Let us introduce the auxiliary function
\[ \delta=\delta(t,\xi):=\exp\bigg( \int_{0}^{t}\big( b(\tau)+g(\tau)|\xi|^2 \big)d\tau\bigg) = \exp\Big( 2|\xi|^2G(t)+\int_0^tb(\tau)d\tau \Big) \lesssim 1, \]
where we use definition of $\Zpd(N)$ and $b\in L^1([0,\infty))$.
\medskip

The entries $E_{\text{pd}}^{(k\ell)}(t,s,\xi)$, $k,\ell=1,2,$ of the fundamental solution $E_{\text{pd}}(t,s,\xi)$ satisfy the following system for $\ell=1,2$:
\begin{align*}
D_tE_{\text{pd}}^{(1\ell)}(t,s,\xi) &= \frac{D_t\gamma(t,\xi)}{\gamma(t,\xi)}E_{\text{pd}}^{(1\ell)}(t,s,\xi)+\gamma(t,\xi)E_{\text{pd}}^{(2\ell)}(t,s,\xi), \\
D_tE_{\text{pd}}^{(2\ell)}(t,s,\xi) &= \frac{|\xi|^2}{\gamma(t,\xi)}E_{\text{pd}}^{(1\ell)}(t,s,\xi)+i\big( b(t)+g(t)|\xi|^2 \big)E_{\text{pd}}^{(2\ell)}(t,s,\xi).
\end{align*}
Then, by straight-forward calculations (with $\delta_{k \ell}=1$ if $k=\ell$ and $\delta_{k \ell}=0$ otherwise), we get
\begin{align*}
E_{\text{pd}}^{(1\ell)}(t,s,\xi) & = \frac{\gamma(t,\xi)}{\gamma(s,\xi)}\delta_{1\ell}+i\gamma(t,\xi)\int_{s}^{t}E_{\text{pd}}^{(2\ell)}(\tau,s,\xi)d\tau, \\
% E_{\text{pd}}^{(21)}(t,s,\xi) & = \frac{i|\xi|^2}{\delta(t)}\int_{s}^{t}\frac{1}{\gamma(\tau,\xi)}\delta(\tau)E_{\text{pd}}^{(11)}(\tau,s,\xi)d\tau, \\
% E_{\text{pd}}^{(12)}(t,s,\xi) & = i\gamma(t,\xi)\int_{s}^{t}E_{\text{pd}}^{(22)}(\tau,s,\xi)d\tau, \\
E_{\text{pd}}^{(2\ell)}(t,s,\xi) & = \frac{\delta(s,\xi)}{\delta(t,\xi)}\delta_{2\ell}+\frac{i|\xi|^2}{\delta(t,\xi)}\int_{s}^{t}\frac{1}{\gamma(\tau,\xi)}\delta(\tau,\xi)E_{\text{pd}}^{(1\ell)}(\tau,s,\xi)d\tau.
\end{align*}
\begin{proposition} \label{Prop_Scattering_Incresing_Zpd}
We have the following estimates in the pseudo-differential zone $\Zpd(N)$ to the Cauchy problem \eqref{MainEquation}:
\begin{equation*}
(|E_{\text{pd}}(t,s,\xi)|) \lesssim \frac{g(t)}{g(s)}
\left( \begin{array}{cc}
1 & 1 \\
1 & 1
\end{array} \right)
\end{equation*}
with $(s,\xi),(t,\xi)\in\Zpd(N)$ and $0\leq s\leq t\leq t_\xi$.
\end{proposition}
\begin{proof}
First let us consider the first column. Plugging the representation for $E_{\text{pd}}^{(21)}=E_{\text{pd}}^{(21)}(t,s,\xi)$ into the integral equation for $E_{\text{pd}}^{(11)}=E_{\text{pd}}^{(11)}(t,s,\xi)$ gives
\begin{align*}
E_{\text{pd}}^{(11)}(t,s,\xi) = \frac{\gamma(t,\xi)}{\gamma(s,\xi)}-|\xi|^2\gamma(t,\xi)\int_{s}^{t}\int_{s}^{\tau}\frac{\delta(\theta,\xi)}{\delta(\tau,\xi)}\frac{1}{\gamma(\theta,\xi)}E_{\text{pd}}^{(11)}(\theta,s,\xi)d\theta d\tau.
\end{align*}
By setting $y(t,s,\xi):=\dfrac{\gamma(s,\xi)}{\gamma(t,\xi)}E_{\text{pd}}^{(11)}(t,s,\xi)$ we obtain
\begin{align*}
y(t,s,\xi) &= 1-|\xi|^2\int_{s}^{t}\int_{\theta}^{t}\frac{\delta(\theta,\xi)}{\delta(\tau,\xi)}y(\theta,s,\xi)d\tau d\theta \\
& = 1+|\xi|^2\int_{s}^{t}\bigg( \int_{\theta}^{t}\frac{1}{b(\tau)+g(\tau)|\xi|^2}\partial_\tau \bigg( \frac{\delta(\theta,\xi)}{\delta(\tau,\xi)} \bigg)d\tau \bigg) y(\theta,s,\xi)d\theta \\
& = 1+|\xi|^2\int_{s}^{t}\bigg( \frac{1}{b(\tau)+g(\tau)|\xi|^2}\frac{\delta(\theta,\xi)}{\delta(\tau,\xi)}\Big|_{\theta}^{t}+\int_{\theta}^{t}\frac{b'(\tau)+g'(\tau)|\xi|^2}{( b(\tau)+g(\tau)|\xi|^2)^2} \frac{\delta(\theta,\xi)}{\delta(\tau,\xi)} d\tau \bigg) y(\theta,s,\xi)d\theta.
\end{align*}
Thus, since $b'(t)<0$ it follows
\begin{align} \label{Eq:Scattering-Incr-Zpd-Est-y1}
|y(t,s,\xi)| &\leq 1+|\xi|^2\int_{s}^{t}\bigg( \frac{1}{b(t)+g(t)|\xi|^2}\underbrace{\frac{\delta(\theta,\xi)}{\delta(t,\xi)}}_{\leq1}+\int_{\theta}^{t}\frac{g'(\tau)|\xi|^2-b'(\tau)}{( b(\tau)+g(\tau)|\xi|^2)^2} \underbrace{\frac{\delta(\theta,\xi)}{\delta(\tau,\xi)}}_{\leq1}d\tau \bigg)|y(\theta,s,\xi)|d\theta \nonumber \\
&\leq 1+|\xi|^2\int_{s}^{t}\bigg( \frac{1}{b(t)+g(t)|\xi|^2} - \int_{\theta}^{t}\frac{b'(\tau)+g'(\tau)|\xi|^2}{(b(\tau)+g(\tau)|\xi|^2)^2} + 2\int_{\theta}^{t}\frac{g'(\tau)|\xi|^2}{(b(\tau)+g(\tau)|\xi|^2)^2} d\tau \bigg)|y(\theta,s,\xi)|d\theta \nonumber \\
&\leq 1+|\xi|^2\int_{s}^{t}\bigg( \frac{1}{b(t)+g(t)|\xi|^2} - \int_{\theta}^{t}\frac{b'(\tau)+g'(\tau)|\xi|^2}{(b(\tau)+g(\tau)|\xi|^2)^2} + 2\int_{\theta}^{t}\frac{g'(\tau)|\xi|^2}{(g(\tau)|\xi|^2)^2} d\tau \bigg)|y(\theta,s,\xi)|d\theta \nonumber \\
&= 1+|\xi|^2\int_{s}^{t}\bigg( \frac{1}{b(t)+g(t)|\xi|^2} + \frac{1}{b(\tau)+g(\tau)|\xi|^2}\Big|_\theta^t - \frac{2}{g(\tau)|\xi|^2}\Big|_\theta^t \bigg)|y(\theta,s,\xi)|d\theta \nonumber \\
& = 1+|\xi|^2\int_{s}^{t}\bigg( \frac{2}{b(t)+g(t)|\xi|^2} - \frac{1}{b(\theta)+g(\theta)|\xi|^2} - \frac{2}{g(t)|\xi|^2} + \frac{2}{g(\theta)|\xi|^2}\bigg)|y(\theta,s,\xi)|d\theta \nonumber \\
& \leq 1+|\xi|^2\int_{s}^{t}\bigg( \frac{2}{g(t)|\xi|^2} - \frac{1}{b(\theta)+g(\theta)|\xi|^2} - \frac{2}{g(t)|\xi|^2} + \frac{2}{g(\theta)|\xi|^2}\bigg)|y(\theta,s,\xi)|d\theta \nonumber \\
& \leq 1+|\xi|^2\int_{s}^{t}\bigg( \frac{2}{g(\theta)|\xi|^2} \bigg)|y(\theta,s,\xi)|d\theta.
\end{align}
Applying Gronwall's inequality and employing \textbf{(A2)}, we get the estimate
\begin{align*}
|y(t,s,\xi)| \leq \exp \bigg( \int_{s}^{t}\frac{1}{g(\theta)}d\theta \bigg) \lesssim 1.
\end{align*}
Thus, we may conclude that
\[ |E_{\text{pd}}^{(11)}(t,s,\xi)| \lesssim \frac{\gamma(t,\xi)}{\gamma(s,\xi)}=\frac{g(t)}{g(s)}. \]
Now we consider $E_{\text{pd}}^{(21)}(t,s,\xi)$. By using the estimate for $|E_{\text{pd}}^{(11)}(t,s,\xi)|$ we obtain
\begin{align*}
\frac{\gamma(s,\xi)}{\gamma(t,\xi)}|E_{\text{pd}}^{(21)}(t,s,\xi)| & \lesssim |\xi|^2\int_{s}^{t}\frac{1}{\gamma(\tau,\xi)}\frac{\delta(\tau,\xi)}{\delta(t,\xi)}\frac{\gamma(s,\xi)}{\gamma(t,\xi)}|E_{\text{pd}}^{(11)}(\tau,s,\xi)|d\tau \\
& \lesssim \int_{s}^{t}\frac{1}{g(\tau)}\underbrace{\frac{\delta(\tau,\xi)\gamma(\tau,\xi)}{\delta (t,\xi)\gamma(t,\xi)}}_{\leq1}d\tau \lesssim \int_{s}^{t}\frac{1}{g(\tau)}d\tau \lesssim 1,
\end{align*}
where we have used
\begin{align} \label{Eq:Scattering-Incr-Zpd-Inq1}
\frac{\delta(\tau,\xi)\gamma(\tau,\xi)}{\delta(t,\xi)\gamma(t,\xi)} \leq \frac{\delta(t,\xi)\gamma(t,\xi)}{\delta(t,\xi)\gamma(t,\xi)} = 1,
\end{align}
because $f(t,\xi) := \delta(t,\xi)\gamma(t,\xi)$ is increasing in $t$. Namely, we have
\begin{align*}
f_t(t,\xi) = \exp\Big( \int_{0}^{t}\big( b(\tau)+g(\tau)|\xi|^2 \big)d\tau \Big)\frac{1}{2}\big( ( g(t)|\xi|^2)^2 + g(t)b(t)|\xi|^2 + g'(t)|\xi|^2 \big),
\end{align*}
and we may conclude that $f_t(t,\xi)>0$, namely $f=f(t,\xi)$ is increasing in $t$. Thus, we arrive at
\[ |E_{\text{pd}}^{(21)}(t,s,\xi)| \lesssim \frac{\gamma(t,\xi)}{\gamma(s,\xi)}=\frac{g(t)}{g(s)}. \]
Next, we consider the entries of the second column. Plugging the representation for $E_{\text{pd}}^{(22)}=E_{\text{pd}}^{(22)}(t,s,\xi)$ into the integral equation for $E_{\text{pd}}^{(12)}=E_{\text{pd}}^{(12)}(t,s,\xi)$ gives
\begin{align*}
E_{\text{pd}}^{(12)}(t,s,\xi) = i\gamma(t,\xi)\int_{s}^{t}\frac{\delta(s,\xi)}{\delta(\tau,\xi)}d\tau-
|\xi|^2\gamma(t,\xi)\int_{s}^{t}\int_{s}^{\tau}\frac{\delta(\theta,\xi)}{\delta(\tau,\xi)}\frac{1}{\gamma(\theta,\xi)}E_{\text{pd}}^{(12)}(\theta,s,\xi)d\theta d\tau.
\end{align*}
After setting $y(t,s,\xi):=\dfrac{\gamma(s,\xi)}{\gamma(t,\xi)}E_{\text{pd}}^{(12)}(t,s,\xi)$, it follows
\begin{align*}
y(t,s,\xi) &= -i\gamma(s,\xi)\int_{s}^{t}\frac{1}{b(\tau)+g(\tau)|\xi|^2}\partial_\tau \bigg( \frac{\delta(s,\xi)}{\delta(\tau,\xi)} \bigg)d\tau \\
& \qquad + |\xi|^2\int_{s}^{t}\bigg( \int_{\theta}^{t}\frac{1}{b(\tau)+g(\tau)|\xi|^2}\partial_\tau \bigg( \frac{\delta(\theta,\xi)}{\delta(\tau,\xi)} \bigg)d\tau \bigg)y(\theta,s,\xi)d\theta d\tau \\
&= i\gamma(s,\xi) \bigg( -\frac{1}{b(\tau)+g(\tau)|\xi|^2}\frac{\delta(s,\xi)}{\delta(\tau,\xi)}\Big|_s^t - \int_{s}^{t}\frac{b'(\tau)+g'(\tau)|\xi|^2}{( b(\tau)+g(\tau)|\xi|^2)^2} \frac{\delta(\theta,\xi)}{\delta(\tau,\xi)}d\tau \bigg) \\
& \qquad + |\xi|^2\int_{s}^{t}\bigg( \frac{1}{b(t)+g(t)|\xi|^2}\frac{\delta(\theta,\xi)}{\delta(t,\xi)}+\int_{\theta}^{t}\frac{b'(\tau)+g'(\tau)|\xi|^2}{( b(\tau)+g(\tau)|\xi|^2)^2} \frac{\delta(\theta,\xi)}{\delta(\tau,\xi)}d\tau \bigg)y(\theta,s,\xi)d\theta.
\end{align*}
Since $b'(\tau)<0$, we have
\begin{align} \label{Eq:Scattering-Incr-Zpd-Est-y2}
|y(t,s,\xi)| &\leq \gamma(s,\xi) \bigg( \frac{1}{b(t)+g(t)|\xi|^2}\underbrace{\frac{\delta(s,\xi)}{\delta(t,\xi)}}_{\leq1} +\frac{1}{b(s)+g(s)|\xi|^2} + \int_{s}^{t}\frac{g'(\tau)|\xi|^2-b'(\tau)}{( b(\tau)+g(\tau)|\xi|^2)^2}\underbrace{\frac{\delta(\theta,\xi)}{\delta(\tau,\xi)}}_{\leq1}d\tau \bigg) \nonumber \\
& \qquad + |\xi|^2\int_{s}^{t}\bigg( \frac{1}{b(t)+g(t)|\xi|^2}\underbrace{\frac{\delta(\theta,\xi)}{\delta(t,\xi)}}_{\leq1}+\int_{\theta}^{t}\frac{g'(\tau)|\xi|^2-b'(\tau)}{( b(\tau)+g(\tau)|\xi|^2)^2}\underbrace{\frac{\delta(\theta,\xi)}{\delta(\tau,\xi)}}_{\leq1}d\tau \bigg)|y(\theta,s,\xi)|d\theta,
\end{align}
where for the first summand in \eqref{Eq:Scattering-Incr-Zpd-Est-y2} we use the following relations:
\[ \gamma(s,\xi) \frac{1}{b(t)+g(t)|\xi|^2}\frac{\delta(s,\xi)}{\delta(t,\xi)} \leq \frac{1}{2} \qquad \text{and} \qquad \gamma(s,\xi) \frac{1}{b(s)+g(s)|\xi|^2}\leq \frac{1}{2}, \]
and
\begin{align*}
& \gamma(s,\xi) \int_{s}^{t}\frac{g'(\tau)|\xi|^2-b'(\tau)}{( b(\tau)+g(\tau)|\xi|^2)^2}\underbrace{\frac{\delta(\theta,\xi)}{\delta(\tau,\xi)}}_{\leq1}d\tau \leq  \gamma(s,\xi) \int_{s}^{t}\frac{g'(\tau)|\xi|^2-b'(\tau)}{( b(\tau)+g(\tau)|\xi|^2)^2} d\tau \\
& \qquad = -\gamma(s,\xi) \int_{s}^{t}\frac{b'(\tau)+g'(\tau)|\xi|^2}{( b(\tau)+g(\tau)|\xi|^2)^2} d\tau + 2\gamma(s,\xi) \int_{s}^{t}\frac{g'(\tau)|\xi|^2}{( b(\tau)+g(\tau)|\xi|^2)^2} d\tau \\
& \qquad = \frac{\gamma(s,\xi)}{b(t)+g(t)|\xi|^2}-\frac{\gamma(s,\xi)}{b(s)+g(s)|\xi|^2} + 2\gamma(s,\xi) \int_{s}^{t}\frac{g'(\tau)|\xi|^2}{( b(\tau)+g(\tau)|\xi|^2)^2} d\tau \\
& \qquad \leq \frac{\gamma(s,\xi)}{g(t)|\xi|^2} + 2\gamma(s,\xi)\int_{s}^{t}\frac{g'(\tau)|\xi|^2}{( g(\tau)|\xi|^2)^2} d\tau \\
& \qquad = \frac{\gamma(s,\xi)}{g(t)|\xi|^2} + \big( g(s)|\xi|^2 \big)\Big[ -\frac{1}{g(\tau)||\xi|^2} \Big]_s^t \\
& \qquad = \frac{1}{2}\frac{g(s)|\xi|^2}{g(t)|\xi|^2} - \frac{g(s)|\xi|^2}{g(t)|\xi|^2} + \frac{g(s)|\xi|^2}{g(s)|\xi|^2} \leq1.
\end{align*}
For the second summand in \eqref{Eq:Scattering-Incr-Zpd-Est-y2} following the same approach to \eqref{Eq:Scattering-Incr-Zpd-Est-y1}, we find
\[ |y(t,s,\xi)| \leq 1+|\xi|^2\int_{s}^{t}\bigg( \frac{2}{g(\theta)|\xi|^2} \bigg)|y(\theta,s,\xi)|d\theta. \]
Applying Gronwall's inequality and using condition \textbf{(A2)} we have the estimate
\begin{align*}
|y(t,s,\xi)| \lesssim \exp\bigg( \int_{s}^{t}\frac{2}{g(\theta)}d\theta \bigg) \lesssim 1.
\end{align*}
This implies that
\begin{equation} \label{Eq:Scattering-Incr-Zpd-Est-E12}
|E_{\text{pd}}^{(12)}(t,s,\xi)| \lesssim \frac{\gamma(t,\xi)}{\gamma(s,\xi)} = \frac{g(t)}{g(s)}.
\end{equation}
Finally, let us estimate $|E_{\text{pd}}^{(22)}(t,s,\xi)|$ by using the estimate of $|E_{\text{pd}}^{(12)}(t,s,\xi)|$ from \eqref{Eq:Scattering-Incr-Zpd-Est-E12}. It holds
\begin{align*}
|E_{\text{pd}}^{(22)}(t,s,\xi)| & \lesssim \frac{\delta(s,\xi)}{\delta(t,\xi)}+|\xi|^2\int_{s}^{t}\frac{\delta(\tau,\xi)}{\delta(t,\xi)}\frac{1}{\gamma(\tau,\xi)}|E_{\text{pd}}^{(12)}(\tau,s,\xi)|d\tau \\
& \lesssim \frac{\delta(s,\xi)}{\delta(t,\xi)}+|\xi|^2\int_{s}^{t}\frac{\delta(\tau,\xi)}{\delta(t,\xi)}\frac{1}{\gamma(\tau,\xi)}\frac{\gamma(\tau,\xi)}{\gamma(s,\xi)}d\tau.
\end{align*}
Then, we have
\begin{align*}
\frac{\gamma(s,\xi)}{\gamma(t,\xi)}|E_{\text{pd}}^{(22)}(t,s,\xi)| & \lesssim \underbrace{\frac{\delta(s,\xi)\gamma(s,\xi)}{\delta(t,\xi)\gamma(t,\xi)}}_{\leq1}+|\xi|^2\int_{s}^{t}
\underbrace{\frac{\delta(\tau,\xi)\gamma(\tau,\xi)}{\delta(t,\xi)\gamma(t,\xi)}}_{\leq1}\frac{1}{\gamma(\tau,\xi)}d\tau \\
& \lesssim 1 + |\xi|^2\int_{s}^{t}\frac{1}{g(\tau)|\xi|^2}d\tau \lesssim 1,
\end{align*}
where we used \eqref{Eq:Scattering-Incr-Zpd-Inq1}. This shows that
\[ |E_{\text{pd}}^{(22)}(t,s,\xi)| \lesssim \frac{\gamma(t,\xi)}{\gamma(s,\xi)}=\frac{g(t)}{g(s)}. \]
This completes the proof.
\end{proof}
Following the same strategy as in Section 2 of the paper \cite{AslanReissig2023}, estimates in $\Zell(N)$ and $\Zpd(N)$ may conclude the proof of Theorem \ref{Theorem_Sect-Scattering_Increasing}.
\end{proof}

\subsection{Model with integrable and decaying time-dependent coefficient $g=g(t)$} \label{Subsect_Scattering}
In this section, we consider scattering producing damping terms $b(t)u_t$ and viscoelastic damping terms $-g(t)\Delta u_t$ with integrable and decaying time-dependent coefficient $g=g(t)$ to our equation \eqref{MainEquation}.
\medskip

We pose the following condition to the coefficient $b=b(t)$ for all $t \in [0,\infty)$:
\begin{enumerate}
\item[\textbf{(A'1)}] $b(t)>0$ and $b\in L^1([0,\infty))$.
%\item[\textbf{(A'2)}] $b\in\mathcal{C}^1([0,\infty))$.
\end{enumerate}

We assume that the coefficient $g=g(t)$ satisfies the following conditions for all $t \in [0,\infty)$:
\begin{enumerate}
\item[\textbf{(G1)}] $g(t)>0$ and $g'(t)<0$,
\item[\textbf{(G2)}] $g \in L^1([0,\infty))$,
\item[\textbf{(G3)}] $\dfrac{g''}{g'} \sim \dfrac{g'}{g}$,
\item[\textbf{(G4)}] $b(t) \leq -a\dfrac{g'(t)}{g(t)}$ with a suitable constant $a\in(0,1)$.
\end{enumerate}
Let us remember our equation in \eqref{MainEquationFourier}, that is,
\begin{equation*} \label{Eq:Scattering_PseudoForm}
\hat{u}_{tt} + |\xi|^2\hat{u} + b(t)\hat{u}_t + g(t)|\xi|^2\hat{u}_t=0.
\end{equation*}
Then, we divide the extended phase space $[0,\infty)\times \mathbb{R}^n$ into the following zones:
\begin{itemize}
\item elliptic zone:
\[ \Zell(N_1,N_2) = \Big\{ (t,\xi)\in[0,\infty)\times\mathbb{R}^n : |\xi|\geq N_1b(t) \Big\}\cap \Big\{ (t,\xi)\in[0,\infty)\times\mathbb{R}^n : |\xi|\geq \frac{N_2}{g(t)} \Big\}, \]
\item reduced zone:
\[ \Zred(N_1,N_2,\varepsilon) = \Big\{ (t,\xi)\in[0,\infty)\times\mathbb{R}^n : |\xi|\geq N_1b(t) \Big\}\cap\Big\{ (t,\xi)\in[0,\infty)\times\mathbb{R}^n : \frac{\varepsilon}{g(t)}\leq |\xi|\leq \frac{N_2}{g(t)}  \Big\}, \]
\item hyperbolic zone:
\[ \Zhyp(N_1,\varepsilon) = \Big\{ (t,\xi)\in[0,\infty)\times\mathbb{R}^n : |\xi|\geq N_1b(t) \Big\} \cap \Big\{ (t,\xi)\in[0,\infty)\times\mathbb{R}^n : |\xi|\leq \frac{\varepsilon}{g(t)} \Big\}, \]
\item dissipative zone:
\[ \Zdiss(N_1,\varepsilon) = \Big\{ (t,\xi)\in[0,\infty)\times\mathbb{R}^n : |\xi|\leq N_1b(t) \Big\}\cap\Big\{ (t,\xi)\in[0,\infty)\times\mathbb{R}^n : |\xi|\leq \frac{\varepsilon}{g(t)} \Big\}, \]
\end{itemize}
where $N_1>0$ and $N_2>0$ are sufficiently large and $\varepsilon>0$ is sufficiently small constants. Moreover, the separating lines between these zones may be defined by using the following functions:
\begin{align*}
t_{\xi,1} &= \big\{ (t,\xi) \in [0,\infty) \times \mathbb{R}^n: |\xi| = N_1b(t) \big\}, \quad (\text{between} \quad \Zdiss(N_1,\varepsilon) \,\,\, \text{and} \,\,\, \Zhyp(N_1,\varepsilon) \,), \\
t_{\xi,2} &= \big\{ (t,\xi) \in [0,\infty) \times \mathbb{R}^n: g(t)|\xi| = \varepsilon \big\}, \quad\,\,\,\, (\text{between} \quad \Zhyp(N_1,\varepsilon) \,\,\, \text{and} \,\,\,\Zred(N_2,N_1,\varepsilon) \,),\\
t_{\xi,3} &= \big\{ (t,\xi) \in [0,\infty) \times \mathbb{R}^n: g(t)|\xi| = N_2 \big\},  \quad (\text{between} \quad \Zred(N_1,N_2,\varepsilon) \,\,\, \text{and} \,\,\, \Zell(N_1,N_2)\,).
\end{align*}
For small frequencies we consider $ \Zdiss(N_1,\varepsilon)$ and $\Zhyp(N_1,\varepsilon)$, and for large frequencies we consider $\Zell(N_1,N_2)$, $\Zred(N_1,N_2,\varepsilon)$ and $\Zhyp(N_1,\varepsilon)$ (see Figure \ref{fig.zone} in Section \ref{Subsect_Noneffective_Integrable}).
\begin{theorem} \label{Theorem_Scattering_Integrable-Decaying}
Let us consider the Cauchy problem
\begin{equation*}
\begin{cases}
u_{tt}- \Delta u + b(t)u_t -g(t)\Delta u_t=0, &(t,x) \in (0,\infty) \times \mathbb{R}^n, \\
u(0,x)= u_0(x),\quad u_t(0,x)= u_1(x), &x \in \mathbb{R}^n.
\end{cases}
\end{equation*}
We assume that the coefficients $b=b(t)$ and $g=g(t)$ satisfy the conditions \textbf{(A'1)} and \textbf{(G1)} to \textbf{(G4)}, respectively. Let $(u_0,u_1)\in \dot{H}^{|\beta|+\kappa+\frac{a}{2}} \times \dot{H}^{|\beta|+\kappa+\frac{a}{2}-2}$ with $|\beta|\geq2$, $a\in(0,1)$ and arbitrarily small $\kappa>0$. Then, we have the following estimates for Sobolev solutions:
\begin{align*}
\|\,|D|^{|\beta|}u(t,\cdot)\|_{L^2} & \lesssim \|u_0\|_{\dot H^{|\beta|+\kappa+\frac{a}{2}}} + \|u_1\|_{\dot H^{|\beta|+\kappa+\frac{a}{2}-2}}, \\
\|\,|D|^{|\beta|-1}u_t(t,\cdot)\|_{L^2} & \lesssim \|u_0\|_{\dot H^{|\beta|+\kappa+\frac{a}{2}}} + \|u_1\|_{\dot H^{|\beta|+\kappa+\frac{a}{2}-2}}.
\end{align*}
\end{theorem}
\begin{proof}
To prove this theorem, we will use the WKB analysis and derive estimates accordingly in each zone. We follow the strategy which is explained in Theorem \ref{Theorem_Noneffective_Integrable-Decaying} in Section \ref{Subsect_Noneffective_Integrable}. The main difference, however, is that now we assume $b\in L^1(\mathbb{R}^+)$ due to condition \textbf{(A'1)}. As a consequence, we obtain no longer exponential decay, in contrast to the results are given in Section \ref{Subsect_Noneffective_Integrable}.
\end{proof}

\section{The friction term is non-effective} \label{Section_Noneffective}
\begin{definition}[Non-effective dissipation, \cite{Wirth-Noneffective=2006}] \label{Definition_Non-effective}
If the nonnegative function $b=b(t)$ fulfilling the relation $\limsup_{t\to\infty}tb(t)<1$ satisfies
\begin{itemize}
\item[\textbf{(B'1)}] $b\in\mathcal{C}^1([0,\infty))$,
\item[\textbf{(B'2)}] $b'(t)<0$,
\item[\textbf{(B'3)}] $b^2(t) \lesssim -b'(t)$,
\end{itemize}
then, the damping term $b(t)u_t$ is called non-effective.
\end{definition}
\subsection{Model with increasing time-dependent coefficient $g=g(t)$} \label{Subsect_Noneffective_Increasing}
We assume the following conditions for the coefficient $g=g(t)$ for all $t \in [0,\infty)$:
\begin{enumerate}
\item[\textbf{(A1)}] $g(t)>0$ and $g'(t)>0$,
\item[\textbf{(A2)}] $\dfrac{1}{g} \in L^1([0,\infty))$,
\item[\textbf{(A3)}] $|d_t^kg(t)|\leq C_kg(t)\Big( \dfrac{g(t)}{G(t)} \Big)^k$ for $k=1,2$, where $G(t):=\dfrac{1}{2}\displaystyle\int_0^t g(\tau)d\tau$ and $C_1$, $C_2$ are positive constants.
\end{enumerate}
On the other hand, we pose the following condition for all $t \in [0,\infty)$:
\begin{itemize}
\item[\textbf{(N-EF)}] $b(t)\leq \widetilde{C}\dfrac{g(t)}{G(t)}$, where $\widetilde{C}$ is positive constant.
%\item[\textbf{(N-EF2)}] $|b'(t)| \leq c \dfrac{b(t)}{1+t}$, where $c$ is a positive constant.
\end{itemize}
\begin{theorem} \label{Theorem_Sect-Noneffective_Increasing}
Let us consider the Cauchy problem
\begin{equation*}
\begin{cases}
u_{tt}- \Delta u + b(t)u_t -g(t)\Delta u_t=0, &(t,x) \in (0,\infty) \times \mathbb{R}^n, \\
u(0,x)= u_0(x),\quad u_t(0,x)= u_1(x), &x \in \mathbb{R}^n.
\end{cases}
\end{equation*}
We assume that the coefficient $g=g(t)$ satisfy the conditions \textbf{(A1)} to \textbf{(A3)} and the coefficient $b=b(t)$ satisfy the conditions \textbf{(B'1)} to \textbf{(B'3)} and \textbf{(N-EF)}. Moreover, we suppose that $(u_0,u_1)\in \dot{H}^{|\beta|} \times \dot{H}^{|\beta|-2}$ with $|\beta|\geq 2$. Then, we have the following estimates for Sobolev solutions:
\begin{align*}
\|\,|D|^{|\beta|} u(t,\cdot)\|_{L^2} & \lesssim  \|u_0\|_{\dot{H}^{|\beta|}} + \|u_1\|_{\dot{H}^{|\beta|-2}},\\
\|\,|D|^{|\beta|-2} u_t(t,\cdot)\|_{L^2} & \lesssim g(t)\big( \|u_0\|_{\dot{H}^{|\beta|}} + \|u_1\|_{\dot{H}^{|\beta|-2}} \big).
\end{align*}
\end{theorem}
\begin{proof}[Proof of Theorem \ref{Theorem_Sect-Noneffective_Increasing}]
We divide the extended phase space $[0,\infty)\times \mathbb{R}^n$ into zones as follows:
\begin{itemize}
\item pseudo-differential zone:
\begin{align*}
\Zpd(N)=\left\{ (t,\xi)\in [0,\infty)\times\mathbb{R}^n: G(t)|\xi|^2 \leq N \right\},
\end{align*}
\item elliptic zone:
\begin{align*}
\Zell(N)=\left\{ (t,\xi)\in [0,\infty)\times \mathbb{R}^n: G(t)|\xi|^2 \geq N \right\},
\end{align*}
\end{itemize}
where $G=G(t)$ is defined in condition \textbf{(A3)} and $N>0$ is sufficiently large. Moreover, the separating line $t_\xi=t(|\xi|)$ is given by
\[ t_\xi=\left\{ (t,\xi) \in [0,\infty) \times \mathbb{R}^n: G(t)|\xi|^2=N \right\}. \]
\subsubsection{Considerations in the elliptic zone $\Zell(N)$} \label{Sect_Noneffective_Increasing-Zell}
Let us write the transformed equation \eqref{AuxiliaryEquation1} in the following form:
\begin{equation*} \label{Eq:Noneffective_PseudoForm}
D_t^2v + \dfrac{g^2(t)}{4}|\xi|^4v + \Big( \dfrac{g'(t)}{2} -1 \Big)|\xi|^2v + \dfrac{b(t)g(t)}{2}|\xi|^2v - ib(t)D_tv=0.
\end{equation*}
We introduce the micro-energy $V=V(t,\xi):=\big( \dfrac{g(t)}{2}|\xi|^2v,D_tv \big)^{\text{T}}$. Then, by \eqref{Eq:ScatteringPseudoForm} we obtain that $V=V(t,\xi)$ satisfies the following system of first order:
\begin{equation*} \label{Eq:Noneffective_SystemWithbANDg}
D_tV=\left( \begin{array}{cc}
0 & \dfrac{g(t)}{2}|\xi|^2 \\
-\dfrac{g(t)}{2}|\xi|^2 & 0
\end{array} \right)V + \left( \begin{array}{cc}
\dfrac{D_tg(t)}{g(t)} & 0 \\
-\dfrac{g'(t)-2}{g(t)}-b(t) & ib(t)
\end{array} \right)V.
\end{equation*}
We carry out the first step of diagonalization procedure. For this reason, we set
\[ M := \left( \begin{array}{cc}
1 & -1 \\
i & i
\end{array} \right), \qquad M^{-1}=\frac{1}{2}\left( \begin{array}{cc}
1 & -i \\
-1 & -i
\end{array} \right). \]
We define $V^{(0)}:=M^{-1}V$ and get the system
\begin{equation*}
D_tV^{(0)}=\big( \mathcal{D}(t,\xi)+\mathcal{R}_b(t)+\mathcal{R}_g(t) \big)V^{(0)},
\end{equation*}
where
\begin{align*}
\mathcal{D}(t,\xi) &= \left( \begin{array}{cc}
i\dfrac{g(t)}{2}|\xi|^2 & 0 \\
0 & -i\dfrac{g(t)}{2}|\xi|^2
\end{array} \right), \qquad \mathcal{R}_b(t) = \left( \begin{array}{cc}
ib(t) & 0 \\
ib(t) & 0
\end{array} \right), \\
\mathcal{R}_g(t) &= \frac{1}{2} \left( \begin{array}{cc}
\dfrac{D_tg(t)}{g(t)}+i\dfrac{g'(t)-2}{g(t)} & -\dfrac{D_tg(t)}{g(t)}-i\dfrac{g'(t)-2}{g(t)} \\
-\dfrac{D_tg(t)}{g(t)}+i\dfrac{g'(t)-2}{g(t)} & \dfrac{D_tg(t)}{g(t)}-i\dfrac{g'(t)-2}{g(t)}
\end{array} \right).
\end{align*}
Here we take account of
\begin{align*}
 M^{-1}\left( \begin{array}{cc}
0 & 0 \\
-b(t) & 0
\end{array} \right)M = \frac{1}{2}\left( \begin{array}{cc}
ib(t) & -ib(t) \\
ib(t) & -ib(t)
\end{array} \right) \qquad \text{and} \qquad M^{-1}\left( \begin{array}{cc}
0 & 0 \\
0 & ib(t)
\end{array} \right)M = \frac{1}{2}\left( \begin{array}{cc}
ib(t) & ib(t) \\
ib(t) & ib(t)
\end{array} \right).
\end{align*}
We introduce $F_0(t)=\diag( \mathcal{R}_b(t)+\mathcal{R}_g(t))$ and carry out the next step of diagonalization procedure. The difference of the diagonal entries of the matrix $\mathcal{D}(t,\xi)+F_0(t)$ is
\begin{equation*}
i\delta(t,\xi):=g(t)|\xi|^2 + \frac{g'(t)-2}{g(t)} + b(t)\sim g(t)|\xi|^2
\end{equation*}
for $t\geq t_\xi$ if we choose the zone constant $N$ sufficiently large and apply conditions \textbf{(A'1)} and \textbf{(N-EF)}. We introduce a matrix $N^{(1)}=N^{(1)}(t,\xi)$ and $\mathcal{R}:=\mathcal{R}_b+\mathcal{R}_g$ such that
\[ N^{(1)}(t,\xi)=\left( \begin{array}{cc}
0 & \dfrac{\mathcal{R}_{12}}{\delta(t,\xi)} \\
-\dfrac{\mathcal{R}_{21}}{\delta(t,\xi)} & 0
\end{array} \right)\sim \left( \begin{array}{cc}
0 & -i\dfrac{D_tg(t)}{2g^2(t)|\xi|^2}+\dfrac{g'(t)-2}{2g^2(t)|\xi|^2} \\
i\dfrac{D_tg(t)}{2g^2(t)|\xi|^2}+\dfrac{g'(t)-2}{2g^2(t)|\xi|^2}+\dfrac{b(t)}{g(t)|\xi|^2} & 0
\end{array} \right). \]
For a sufficiently large zone constant $N$ and all $t\geq t_\xi$ the matrix $N_1=N_1(t,\xi)$, where $N_{1}(t,\xi)=I+N^{(1)}(t,\xi)$, is invertible with uniformly bounded inverse $N_1^{-1}=N_1^{-1}(t,\xi)$. Indeed, in the elliptic zone $\Zell(N)$ it holds
\begin{align*}
\Big|\dfrac{D_tg(t)}{2g^2(t)|\xi|^2}\Big| \leq \frac{C}{G(t)|\xi|^2}\leq \frac{C}{N} \qquad \text{and} \qquad \Big| \dfrac{b(t)}{g(t)|\xi|^2} \Big| \leq \frac{\widetilde{C}}{G(t)|\xi|^2}\leq \frac{\widetilde{C}}{N},
\end{align*}
where we used conditions \textbf{(A3)} and \textbf{(N-EF)}. Let
\begin{align*}
B^{(1)}(t,\xi) &= D_tN^{(1)}(t,\xi)-( \mathcal{R}_b(t)+\mathcal{R}_g(t)-F_0(t,\xi))N^{(1)}(t,\xi), \\
\mathcal{R}_1(t,\xi) &= -N_1^{-1}(t,\xi)B^{(1)}(t,\xi).
\end{align*}
Then, we have the following operator identity:
\begin{equation*}
\big( D_t-\mathcal{D}(t,\xi)-\mathcal{R}_b(t)-\mathcal{R}_g(t) \big)N_1(t,\xi)=N_1(t,\xi)\big( D_t-\mathcal{D}(t,\xi)-F_0(t)-\mathcal{R}_1(t,\xi) \big).
\end{equation*}
\begin{proposition} \label{Prop_Noneffective_EllZone}
Let $E_{\text{ell}}^V=E_{\text{ell}}^V(t,s,\xi)$ be the fundamental solution of the operator
\[ D_t-\mathcal{D}(t,\xi)-F_0(t)-\mathcal{R}_1(t,\xi). \]
Then, $E_{\text{ell}}^V=E_{\text{ell}}^V(t,s,\xi)$ holds the following estimate:
\begin{align*} \label{Eq:noneffective_correctformula}
(|E_{\text{ell}}^V(t,s,\xi)|) \lesssim \frac{g(t)}{g(s)} \exp \Big(\frac{|\xi|^2}{2} \int_s^t g(\tau) d\tau \Big) \left(\begin{array}{cc}
1 & 1 \\
1 & 1
\end{array}\right)
\end{align*}
with $(t,\xi),(s,\xi)\in \Zell(N)$ and $t_\xi\leq s\leq t$.
\end{proposition}
\begin{proof}
We transform the system for $E_{\text{ell}}^V=E_{\text{ell}}^V(s,t,\xi)$ to an integral equation for a new matrix-valued function $\mathcal{Q}_{\text{ell}}=\mathcal{Q}_{\text{ell}}(s,t,\xi)$. If we differentiate the term
\[ \exp \bigg\{ -i\int_{s}^{t}\big( \mathcal{D}(\tau,\xi)+F_0(\tau) \big)d\tau \bigg\}E_{\text{ell}}^V(t,s,\xi) \]
and, then integrate on the interval $[s,t]$, we find that $E_{\text{ell}}^V=E_{\text{ell}}^V(t,s,\xi)$ satisfies the following integral equation:
\begin{align*}
E_{\text{ell}}^V(t,s,\xi) & = \exp\bigg\{ i\int_{s}^{t}\big( \mathcal{D}(\tau,\xi)+F_0(\tau) \big)d\tau \bigg\}E_{\text{ell}}^V(s,s,\xi)\\
& \quad + i\int_{s}^{t} \exp \bigg\{ i\int_{\theta}^{t}\big( \mathcal{D}(\tau,\xi)+F_0(\tau) \big)d\tau \bigg\}\mathcal{R}_1(\theta,\xi)E_{\text{ell}}^V(\theta,s,\xi)\,d\theta.
\end{align*}
Let us define
\[ \mathcal{Q}_{\text{ell}}(t,s,\xi)=\exp\bigg\{ -\int_{s}^{t}\beta(\tau,\xi)d\tau \bigg\} E_{\text{ell}}^V(t,s,\xi), \]
with a suitable $\beta=\beta(t,\xi)$ which will be fixed later. Then, it satisfies the new integral equation
\begin{align*}
\mathcal{Q}_{\text{ell}}(t,s,\xi)=&\exp \bigg\{ \int_{s}^{t}\big( i\mathcal{D}(\tau,\xi)+iF_0(\tau)-\beta(\tau,\xi)I \big)d\tau \bigg\}\mathcal{Q}_{\text{ell}}(s,s,\xi)\\
& \quad +\int_{s}^{t} \exp \bigg\{ \int_{\theta}^{t}\big( i\mathcal{D}(\tau,\xi)+iF_0(\tau)-\beta(\tau,\xi)I \big)d\tau \bigg\}\mathcal{R}_1(\theta,\xi)\mathcal{Q}_{\text{ell}}(\theta,s,\xi)\,d\theta.
\end{align*}
The function $\mathcal{R}_1=\mathcal{R}_1(\theta,\xi)$ is uniformly integrable over the elliptic zone (see Proposition 2.4 in \cite{AslanReissig2023}).

The main entries of the diagonal matrix $i\mathcal{D}(t,\xi)+iF_0(t)$ are given by
\begin{align*}
(I)  =& \frac{g(t)}{2}|\xi|^2+\frac{g'(t)}{2g(t)}+\frac{g'(t)-2}{2g(t)},\\
(II) =& -\frac{g(t)}{2}|\xi|^2+\frac{g'(t)}{2g(t)}-\frac{g'(t)-2}{2g(t)}-b(t).
\end{align*}
It follows that the term $(I)$ is dominant. Therefore, we choose the weight
\begin{align*}
\beta=\beta(t,\xi)=(I)=\frac{g(t)}{2}|\xi|^2+\frac{g'(t)}{2g(t)}+\frac{g'(t)-2}{2g(t)}.
\end{align*}
By this choice, we get
\[ i\mathcal{D}(\tau,\xi)+iF_0(\tau)-\beta(\tau,\xi)I = \left( \begin{array}{cc}
-g(\tau)|\xi|^2-\dfrac{g'(\tau)-2}{g(\tau)}-b(t) & 0 \\
0 & 0
\end{array} \right). \]
This implies
\begin{align*}
H(t,s,\xi) & =\exp \bigg\{ \int_{s}^{t}\big( i\mathcal{D}(\tau,\xi)+iF_0(\tau)-\beta(\tau,\xi)I \big)d\tau \bigg\}\\
& = \diag \bigg( \exp \bigg\{ \int_{s}^{t}\Big( -g(\tau)|\xi|^2-\dfrac{g'(\tau)-2}{g(\tau)}\Big)d\tau \bigg\},1 \bigg)\rightarrow \left( \begin{array}{cc}
0 & 0 \\
0 & 1
\end{array} \right)
\end{align*}
as $t\rightarrow \infty$ for any fixed $s\geq t_\xi$. Hence, the matrix $H=H(s,t,\xi)$ is uniformly bounded for $(s,\xi),(t,\xi)\in \Zell(N)$. So, the representation of $\mathcal{Q}_{\text{ell}}=\mathcal{Q}_{\text{ell}}(t,s,\xi)$ by a Neumann series gives
\begin{align*}
\mathcal{Q}_{\text{ell}}(t,s,\xi)=H(t,s,\xi)+\sum_{k=1}^{\infty}i^k\int_{s}^{t}H(t,t_1,\xi)&\mathcal{R}_1(t_1,\xi)\int_{s}^{t_1}H(t_1,t_2,\xi)\mathcal{R}_1(t_2,\xi) \\
& \cdots \int_{s}^{t_{k-1}}H(t_{k-1},t_k,\xi)\mathcal{R}_1(t_k,\xi)dt_k\cdots dt_2dt_1.
\end{align*}
Then, convergence of this series is obtained from the symbol estimates, since $\mathcal{R}_1=\mathcal{R}_1(t,\xi)$ is uniformly integrable over $\Zell(N)$. Hence, from the last considerations we may conclude
\begin{align*}
E_{\text{ell}}^V(t,s,\xi)&=\exp \bigg\{ \int_{s}^{t}\beta(\tau,\xi)d\tau \bigg\}\mathcal{Q}_{\text{ell}}(t,s,\xi) \\
& = \exp \bigg\{ \int_{s}^{t}\bigg( \frac{g(\tau)}{2}|\xi|^2+\frac{g'(\tau)}{2g(\tau)}+\frac{g'(\tau)-2}{2g(\tau)} \bigg)d\tau \bigg\}\mathcal{Q}_g(t,s,\xi),
\end{align*}
where $\mathcal{Q}_{\text{ell}}=\mathcal{Q}_{\text{ell}}(t,s,\xi)$ is a uniformly bounded matrix. Then, it follows
\begin{align*}
(|E_{\text{ell}}^V(t,s,\xi)|) & \lesssim \exp \bigg\{ \int_{s}^{t}\bigg( \frac{g(\tau)}{2}|\xi|^2+\frac{g'(\tau)}{2g(\tau)}+\frac{g'(\tau)-2}{2g(\tau)} \bigg)d\tau \bigg\} \left( \begin{array}{cc}
1 & 1 \\
1 & 1
\end{array} \right) \\
& \lesssim \frac{g(t)}{g(s)} \exp\bigg( |\xi|^2\int_{s}^{t} \frac{g(\tau)}{2}d\tau \bigg)\left( \begin{array}{cc}
1 & 1 \\
1 & 1
\end{array} \right).
\end{align*}
This completes the proof.	
\end{proof}
\subsubsection{Considerations in the pseudo-differential zone $\Zpd(N)$} \label{Sect_Noneffective_Increasing-Zpd}
We define the micro-energy $U=\big( \gamma(t,\xi)\hat{u},D_t\hat{u} \big)^\text{T}$ with $\gamma(t,\xi):=\dfrac{g(t)}{2}|\xi|^2$. Then, the Cauchy problem \eqref{MainEquationFourier} leads to the system of first order
\begin{equation} \label{Eq:Noneffective_Increasing_SystemZpd}
D_tU=\underbrace{\left( \begin{array}{cc}
\dfrac{D_t\gamma(t,\xi)}{\gamma(t,\xi)} & \gamma(t,\xi) \\
\dfrac{|\xi|^2}{\gamma(t,\xi)} & ig(t)|\xi|^2+ib(t)
\end{array} \right)}_{A(t,\xi)}U.
\end{equation}
We are interested in the fundamental solution $E_{\text{pd}}=E_{\text{pd}}(t,s,\xi)$ to the system \eqref{Eq:Noneffective_Increasing_SystemZpd} as follows:
\[ D_tE_{\text{pd}}(t,s,\xi)=A(t,\xi)E_{\text{pd}}(t,s,\xi), \quad E_{\text{pd}}(s,s,\xi)=I, \]
for all $0\leq s \leq t$ and $(t,\xi), (s,\xi) \in \Zpd(N)$. Let us introduce the auxiliary function
\[ \delta=\delta(t,\xi):=\exp\bigg( \int_{0}^{t}\big( b(\tau)+g(\tau)|\xi|^2 \big)d\tau\bigg). \]
The entries $E_{\text{pd}}^{(k\ell)}(t,s,\xi)$, $k,\ell=1,2,$ of the fundamental solution $E_{\text{pd}}(t,s,\xi)$ satisfy the following system for $\ell=1,2$:
\begin{align*}
D_tE_{\text{pd}}^{(1\ell)}(t,s,\xi) &= \frac{D_t\gamma(t,\xi)}{\gamma(t,\xi)}E_{\text{pd}}^{(1\ell)}(t,s,\xi)+\gamma(t,\xi)E_{\text{pd}}^{(2\ell)}(t,s,\xi), \\
D_tE_{\text{pd}}^{(2\ell)}(t,s,\xi) &= \frac{|\xi|^2}{\gamma(t,\xi)}E_{\text{pd}}^{(1\ell)}(t,s,\xi)+i\big( b(t)+g(t)|\xi|^2+ \big)E_{\text{pd}}^{(2\ell)}(t,s,\xi).
\end{align*}
Then, by straight-forward calculations (with $\delta_{k \ell}=1$ if $k=\ell$ and $\delta_{k \ell}=0$ otherwise), we get
\begin{align*}
E_{\text{pd}}^{(1\ell)}(t,s,\xi) & = \frac{\gamma(t,\xi)}{\gamma(s,\xi)}\delta_{1\ell}+i\gamma(t,\xi)\int_{s}^{t}E_{\text{pd}}^{(2\ell)}(\tau,s,\xi)d\tau, \\
% E_{\text{pd}}^{(21)}(t,s,\xi) & = \frac{i|\xi|^2}{\delta(t)}\int_{s}^{t}\frac{1}{\gamma(\tau,\xi)}\delta(\tau)E_{\text{pd}}^{(11)}(\tau,s,\xi)d\tau, \\
% E_{\text{pd}}^{(12)}(t,s,\xi) & = i\gamma(t,\xi)\int_{s}^{t}E_{\text{pd}}^{(22)}(\tau,s,\xi)d\tau, \\
E_{\text{pd}}^{(2\ell)}(t,s,\xi) & = \frac{\delta(s,\xi)}{\delta(t,\xi)}\delta_{2\ell}+\frac{i|\xi|^2}{\delta(t,\xi)}\int_{s}^{t}\frac{1}{\gamma(\tau,\xi)}\delta(\tau,\xi)E_{\text{pd}}^{(1\ell)}(\tau,s,\xi)d\tau.
\end{align*}
\begin{proposition} \label{Prop_Noneffective_Incresing_Zpd}
We have the following estimates in the pseudo-differential zone $\Zpd(N)$ to the Cauchy problem \eqref{MainEquation}:
\begin{equation*}
(|E_{\text{pd}}(t,s,\xi)|) \lesssim \frac{g(t)}{g(s)}
\left( \begin{array}{cc}
1 & 1 \\
1 & 1
\end{array} \right)
\end{equation*}
with $(s,\xi),(t,\xi)\in\Zpd(N)$ and $0\leq s\leq t\leq t_\xi$.
\end{proposition}
\begin{proof}
First let us consider the first column. Plugging the representation for $E_{\text{pd}}^{(21)}=E_{\text{pd}}^{(21)}(t,s,\xi)$ into the integral equation for $E_{\text{pd}}^{(11)}=E_{\text{pd}}^{(11)}(t,s,\xi)$ gives
\begin{align*}
E_{\text{pd}}^{(11)}(t,s,\xi) = \frac{\gamma(t,\xi)}{\gamma(s,\xi)}-|\xi|^2\gamma(t,\xi)\int_{s}^{t}\int_{s}^{\tau}\frac{\delta(\theta,\xi)}{\delta(\tau,\xi)}\frac{1}{\gamma(\theta,\xi)}E_{\text{pd}}^{(11)}(\theta,s,\xi)d\theta d\tau.
\end{align*}
By setting $y(t,s,\xi):=\dfrac{\gamma(s,\xi)}{\gamma(t,\xi)}E_{\text{pd}}^{(11)}(t,s,\xi)$ we obtain
\begin{align*}
y(t,s,\xi) &= 1-|\xi|^2\int_{s}^{t}\int_{\theta}^{t}\frac{\delta(\theta,\xi)}{\delta(\tau,\xi)}y(\theta,s,\xi)d\tau d\theta \\
& = 1+|\xi|^2\int_{s}^{t}\bigg( \int_{\theta}^{t}\frac{1}{b(\tau)+g(\tau)|\xi|^2}\partial_\tau \bigg( \frac{\delta(\theta,\xi)}{\delta(\tau,\xi)} \bigg)d\tau \bigg) y(\theta,s,\xi)d\theta \\
& = 1+|\xi|^2\int_{s}^{t}\bigg( \frac{1}{b(\tau)+g(\tau)|\xi|^2}\frac{\delta(\theta,\xi)}{\delta(\tau,\xi)}\Big|_{\theta}^{t}+\int_{\theta}^{t}\frac{b'(\tau)+g'(\tau)|\xi|^2}{( b(\tau)+g(\tau)|\xi|^2)^2} \frac{\delta(\theta,\xi)}{\delta(\tau,\xi)} d\tau \bigg) y(\theta,s,\xi)d\theta.
\end{align*}
Since $b'(t)<0$ from condition \textbf{(B'2)}, it follows
\begin{align} \label{Eq:Noneffective-Incr-Zpd-Est-y1}
|y(t,s,\xi)| &\leq 1+|\xi|^2\int_{s}^{t}\bigg( \frac{1}{b(t)+g(t)|\xi|^2}\underbrace{\frac{\delta(\theta,\xi)}{\delta(t,\xi)}}_{\leq1}+\int_{\theta}^{t}\frac{g'(\tau)|\xi|^2-b'(\tau)}{( b(\tau)+g(\tau)|\xi|^2)^2} \underbrace{\frac{\delta(\theta,\xi)}{\delta(\tau,\xi)}}_{\leq1}d\tau \bigg)|y(\theta,s,\xi)|d\theta \nonumber \\
&\leq 1+|\xi|^2\int_{s}^{t}\bigg( \frac{1}{b(t)+g(t)|\xi|^2} - \int_{\theta}^{t}\frac{b'(\tau)+g'(\tau)|\xi|^2}{(b(\tau)+g(\tau)|\xi|^2)^2} + 2\int_{\theta}^{t}\frac{g'(\tau)|\xi|^2}{(b(\tau)+g(\tau)|\xi|^2)^2} d\tau \bigg)|y(\theta,s,\xi)|d\theta \nonumber \\
&\leq 1+|\xi|^2\int_{s}^{t}\bigg( \frac{1}{b(t)+g(t)|\xi|^2} - \int_{\theta}^{t}\frac{b'(\tau)+g'(\tau)|\xi|^2}{(b(\tau)+g(\tau)|\xi|^2)^2} + 2\int_{\theta}^{t}\frac{g'(\tau)|\xi|^2}{(g(\tau)|\xi|^2)^2} d\tau \bigg)|y(\theta,s,\xi)|d\theta \nonumber \\
&= 1+|\xi|^2\int_{s}^{t}\bigg( \frac{1}{b(t)+g(t)|\xi|^2} + \frac{1}{b(\tau)+g(\tau)|\xi|^2}\Big|_\theta^t - \frac{2}{g(\tau)|\xi|^2}\Big|_\theta^t \bigg)|y(\theta,s,\xi)|d\theta \nonumber \\
& = 1+|\xi|^2\int_{s}^{t}\bigg( \frac{2}{b(t)+g(t)|\xi|^2} - \frac{1}{b(\theta)+g(\theta)|\xi|^2} - \frac{2}{g(t)|\xi|^2} + \frac{2}{g(\theta)|\xi|^2}\bigg)|y(\theta,s,\xi)|d\theta \nonumber \\
& \leq 1+|\xi|^2\int_{s}^{t}\bigg( \frac{2}{g(t)|\xi|^2} - \frac{1}{b(\theta)+g(\theta)|\xi|^2} - \frac{2}{g(t)|\xi|^2} + \frac{2}{g(\theta)|\xi|^2}\bigg)|y(\theta,s,\xi)|d\theta \nonumber \\
& \leq 1+|\xi|^2\int_{s}^{t}\bigg( \frac{2}{g(\theta)|\xi|^2} \bigg)|y(\theta,s,\xi)|d\theta.
\end{align}
Applying Gronwall's inequality and employing \textbf{(A2)}, we get the estimate
\begin{align*}
|y(t,s,\xi)| \leq \exp \bigg( \int_{s}^{t}\frac{1}{g(\theta)}d\theta \bigg) \lesssim 1.
\end{align*}
Thus, we may conclude that
\[ |E_{\text{pd}}^{(11)}(t,s,\xi)| \lesssim \frac{\gamma(t,\xi)}{\gamma(s,\xi)}=\frac{g(t)}{g(s)}. \]
Now we consider $E_{\text{pd}}^{(21)}(t,s,\xi)$. By using the estimate for $|E_{\text{pd}}^{(11)}(t,s,\xi)|$ we obtain
\begin{align*}
\frac{\gamma(s,\xi)}{\gamma(t,\xi)}|E_{\text{pd}}^{(21)}(t,s,\xi)| & \lesssim |\xi|^2\int_{s}^{t}\frac{1}{\gamma(\tau,\xi)}\frac{\delta(\tau,\xi)}{\delta(t,\xi)}\frac{\gamma(s,\xi)}{\gamma(t,\xi)}|E_{\text{pd}}^{(11)}(\tau,s,\xi)|d\tau \\
& \lesssim \int_{s}^{t}\frac{1}{g(\tau)}\underbrace{\frac{\delta(\tau,\xi)\gamma(\tau,\xi)}{\delta (t,\xi)\gamma(t,\xi)}}_{\leq1}d\tau \lesssim \int_{s}^{t}\frac{1}{g(\tau)}d\tau \lesssim 1,
\end{align*}
where we have used
\begin{align} \label{Eq:Noneffective-Incr-Zpd-Inq1}
\frac{\delta(\tau,\xi)\gamma(\tau,\xi)}{\delta(t,\xi)\gamma(t,\xi)} \leq \frac{\delta(t,\xi)\gamma(t,\xi)}{\delta(t,\xi)\gamma(t,\xi)} = 1,
\end{align}
because $f(t,\xi) := \delta(t,\xi)\gamma(t,\xi)$ is increasing in $t$. Namely, we have
\begin{align*}
f_t(t,\xi) = \exp\Big( \int_{0}^{t}\big( b(\tau)+g(\tau)|\xi|^2 \big)d\tau \Big)\frac{1}{2}\big( ( g(t)|\xi|^2)^2 + g(t)b(t)|\xi|^2 + g'(t)|\xi|^2 \big),
\end{align*}
and we may conclude that $f_t(t,\xi)>0$, namely $f=f(t,\xi)$ is increasing in $t$. Thus, we arrive at
\[ |E_{\text{pd}}^{(21)}(t,s,\xi)| \lesssim \frac{\gamma(t,\xi)}{\gamma(s,\xi)}=\frac{g(t)}{g(s)}. \]
Next, we consider the entries of the second column. Plugging the representation for $E_{\text{pd}}^{(22)}=E_{\text{pd}}^{(22)}(t,s,\xi)$ into the integral equation for $E_{\text{pd}}^{(12)}=E_{\text{pd}}^{(12)}(t,s,\xi)$ gives
\begin{align*}
E_{\text{pd}}^{(12)}(t,s,\xi) = i\gamma(t,\xi)\int_{s}^{t}\frac{\delta(s,\xi)}{\delta(\tau,\xi)}d\tau-
|\xi|^2\gamma(t,\xi)\int_{s}^{t}\int_{s}^{\tau}\frac{\delta(\theta,\xi)}{\delta(\tau,\xi)}\frac{1}{\gamma(\theta,\xi)}E_{\text{pd}}^{(12)}(\theta,s,\xi)d\theta d\tau.
\end{align*}
After setting $y(t,s,\xi):=\dfrac{\gamma(s,\xi)}{\gamma(t,\xi)}E_{\text{pd}}^{(12)}(t,s,\xi)$, it follows
\begin{align*}
y(t,s,\xi) &= -i\gamma(s,\xi)\int_{s}^{t}\frac{1}{b(\tau)+g(\tau)|\xi|^2}\partial_\tau \bigg( \frac{\delta(s,\xi)}{\delta(\tau,\xi)} \bigg)d\tau \\
& \qquad + |\xi|^2\int_{s}^{t}\bigg( \int_{\theta}^{t}\frac{1}{b(\tau)+g(\tau)|\xi|^2}\partial_\tau \bigg( \frac{\delta(\theta,\xi)}{\delta(\tau,\xi)} \bigg)d\tau \bigg)y(\theta,s,\xi)d\theta d\tau \\
&= i\gamma(s,\xi) \bigg( -\frac{1}{b(\tau)+g(\tau)|\xi|^2}\frac{\delta(s,\xi)}{\delta(\tau,\xi)}\Big|_s^t - \int_{s}^{t}\frac{b'(\tau)+g'(\tau)|\xi|^2}{( b(\tau)+g(\tau)|\xi|^2)^2} \frac{\delta(\theta,\xi)}{\delta(\tau,\xi)}d\tau \bigg) \\
& \qquad + |\xi|^2\int_{s}^{t}\bigg( \frac{1}{b(t)+g(t)|\xi|^2}\frac{\delta(\theta,\xi)}{\delta(t,\xi)}+\int_{\theta}^{t}\frac{b'(\tau)+g'(\tau)|\xi|^2}{( b(\tau)+g(\tau)|\xi|^2)^2} \frac{\delta(\theta,\xi)}{\delta(\tau,\xi)}d\tau \bigg)y(\theta,s,\xi)d\theta.
\end{align*}
Since $b'(\tau)<0$, we have
\begin{align} \label{Eq:Noneffective-Incr-Zpd-Est-y2}
|y(t,s,\xi)| &\leq \gamma(s,\xi) \bigg( \frac{1}{b(t)+g(t)|\xi|^2}\underbrace{\frac{\delta(s,\xi)}{\delta(t,\xi)}}_{\leq1} +\frac{1}{b(s)+g(s)|\xi|^2} + \int_{s}^{t}\frac{g'(\tau)|\xi|^2-b'(\tau)}{( b(\tau)+g(\tau)|\xi|^2)^2}\underbrace{\frac{\delta(\theta,\xi)}{\delta(\tau,\xi)}}_{\leq1}d\tau \bigg) \nonumber \\
& \qquad + |\xi|^2\int_{s}^{t}\bigg( \frac{1}{b(t)+g(t)|\xi|^2}\underbrace{\frac{\delta(\theta,\xi)}{\delta(t,\xi)}}_{\leq1}+\int_{\theta}^{t}\frac{g'(\tau)|\xi|^2-b'(\tau)}{( b(\tau)+g(\tau)|\xi|^2)^2}\underbrace{\frac{\delta(\theta,\xi)}{\delta(\tau,\xi)}}_{\leq1}d\tau \bigg)|y(\theta,s,\xi)|d\theta,
\end{align}
where for the first summand in \eqref{Eq:Noneffective-Incr-Zpd-Est-y2} we use the following relations:
\[ \gamma(s,\xi) \frac{1}{b(t)+g(t)|\xi|^2}\frac{\delta(s,\xi)}{\delta(t,\xi)} \leq \frac{1}{2} \qquad \text{and} \qquad \gamma(s,\xi) \frac{1}{b(s)+g(s)|\xi|^2}\leq \frac{1}{2}, \]
and
\begin{align*}
& \gamma(s,\xi) \int_{s}^{t}\frac{g'(\tau)|\xi|^2-b'(\tau)}{( b(\tau)+g(\tau)|\xi|^2)^2}\underbrace{\frac{\delta(\theta,\xi)}{\delta(\tau,\xi)}}_{\leq1}d\tau \leq  \gamma(s,\xi) \int_{s}^{t}\frac{g'(\tau)|\xi|^2-b'(\tau)}{( b(\tau)+g(\tau)|\xi|^2)^2} d\tau \\
& \qquad = -\gamma(s,\xi) \int_{s}^{t}\frac{b'(\tau)+g'(\tau)|\xi|^2}{( b(\tau)+g(\tau)|\xi|^2)^2} d\tau + 2\gamma(s,\xi) \int_{s}^{t}\frac{g'(\tau)|\xi|^2}{( b(\tau)+g(\tau)|\xi|^2)^2} d\tau \\
& \qquad = \frac{\gamma(s,\xi)}{b(t)+g(t)|\xi|^2}-\frac{\gamma(s,\xi)}{b(s)+g(s)|\xi|^2} + 2\gamma(s,\xi) \int_{s}^{t}\frac{g'(\tau)|\xi|^2}{( b(\tau)+g(\tau)|\xi|^2)^2} d\tau \\
& \qquad \leq \frac{\gamma(s,\xi)}{g(t)|\xi|^2} + 2\gamma(s,\xi)\int_{s}^{t}\frac{g'(\tau)|\xi|^2}{( g(\tau)|\xi|^2)^2} d\tau \\
& \qquad = \frac{\gamma(s,\xi)}{g(t)|\xi|^2} + \big( g(s)|\xi|^2 \big)\Big[ -\frac{1}{g(\tau)||\xi|^2} \Big]_s^t \\
& \qquad = \frac{1}{2}\frac{g(s)|\xi|^2}{g(t)|\xi|^2} - \frac{g(s)|\xi|^2}{g(t)|\xi|^2} + \frac{g(s)|\xi|^2}{g(s)|\xi|^2} \leq1.
\end{align*}
For the second summand in \eqref{Eq:Noneffective-Incr-Zpd-Est-y2} following the same approach to \eqref{Eq:Noneffective-Incr-Zpd-Est-y1}, we find
\[ |y(t,s,\xi)| \leq 1+|\xi|^2\int_{s}^{t}\bigg( \frac{2}{g(\theta)|\xi|^2} \bigg)|y(\theta,s,\xi)|d\theta. \]
Applying Gronwall's inequality and using condition \textbf{(A2)} we have the estimate
\begin{align*}
|y(t,s,\xi)| \lesssim \exp\bigg( \int_{s}^{t}\frac{2}{g(\theta)}d\theta \bigg) \lesssim 1.
\end{align*}
This implies that
\begin{equation} \label{Eq:Noneffective-Incr-Zpd-Est-E12}
|E_{\text{pd}}^{(12)}(t,s,\xi)| \lesssim \frac{\gamma(t,\xi)}{\gamma(s,\xi)} = \frac{g(t)}{g(s)}.
\end{equation}
Finally, let us estimate $|E_{\text{pd}}^{(22)}(t,s,\xi)|$ by using the estimate of $|E_{\text{pd}}^{(12)}(t,s,\xi)|$ from \eqref{Eq:Noneffective-Incr-Zpd-Est-E12}. It holds
\begin{align*}
|E_{\text{pd}}^{(22)}(t,s,\xi)| & \lesssim \frac{\delta(s,\xi)}{\delta(t,\xi)}+|\xi|^2\int_{s}^{t}\frac{\delta(\tau,\xi)}{\delta(t,\xi)}\frac{1}{\gamma(\tau,\xi)}|E_{\text{pd}}^{(12)}(\tau,s,\xi)|d\tau \\
& \lesssim \frac{\delta(s,\xi)}{\delta(t,\xi)}+|\xi|^2\int_{s}^{t}\frac{\delta(\tau,\xi)}{\delta(t,\xi)}\frac{1}{\gamma(\tau,\xi)}\frac{\gamma(\tau,\xi)}{\gamma(s,\xi)}d\tau.
\end{align*}
Then, we have
\begin{align*}
\frac{\gamma(s,\xi)}{\gamma(t,\xi)}|E_{\text{pd}}^{(22)}(t,s,\xi)| & \lesssim \underbrace{\frac{\delta(s,\xi)\gamma(s,\xi)}{\delta(t,\xi)\gamma(t,\xi)}}_{\leq1}+|\xi|^2\int_{s}^{t}
\underbrace{\frac{\delta(\tau,\xi)\gamma(\tau,\xi)}{\delta(t,\xi)\gamma(t,\xi)}}_{\leq1}\frac{1}{\gamma(\tau,\xi)}d\tau \\
& \lesssim 1 + |\xi|^2\int_{s}^{t}\frac{1}{g(\tau)|\xi|^2}d\tau \lesssim 1,
\end{align*}
where we used \eqref{Eq:Noneffective-Incr-Zpd-Inq1}. This shows that
\[ |E_{\text{pd}}^{(22)}(t,s,\xi)| \lesssim \frac{\gamma(t,\xi)}{\gamma(s,\xi)}=\frac{g(t)}{g(s)}. \]
This completes the proof.
\end{proof}
Following the same strategy as in Section 2 of the paper \cite{AslanReissig2023}, estimates in $\Zell(N)$ and $\Zpd(N)$ may conclude the proof of Theorem \ref{Theorem_Sect-Noneffective_Increasing}.
\end{proof}
\subsection{Model with integrable and decaying time-dependent coefficient $g=g(t)$} \label{Subsect_Noneffective_Integrable}
In this section, we consider noneffective damping terms $b(t)u_t$ and viscoelastic damping terms $-g(t)\Delta u_t$ with integrable and a decaying time-dependent coefficient $g=g(t)$ to our equation \eqref{MainEquation}.
\medskip

Let us remember our equation in \eqref{MainEquationFourier}
\begin{equation*} \label{Eq:Noneffective_Integrable_PseudoForm}
\hat{u}_{tt} + |\xi|^2\hat{u} + b(t)\hat{u}_t + g(t)|\xi|^2\hat{u}_t=0.
\end{equation*}
This equation implies that we divide the extended phase space $[0,\infty)\times \mathbb{R}^n$ into the following zones:
\begin{itemize}
\item elliptic zone:
\[ \Zell(N_1,N_2) = \Big\{ (t,\xi)\in[0,\infty)\times\mathbb{R}^n : |\xi|\geq N_1b(t) \Big\}\cap \Big\{ (t,\xi)\in[0,\infty)\times\mathbb{R}^n : |\xi|\geq \frac{N_2}{g(t)} \Big\}, \]
\item reduced zone:
\[ \Zred(N_1,N_2,\varepsilon) = \Big\{ (t,\xi)\in[0,\infty)\times\mathbb{R}^n : |\xi|\geq N_1b(t) \Big\}\cap\Big\{ (t,\xi)\in[0,\infty)\times\mathbb{R}^n : \frac{\varepsilon}{g(t)}\leq |\xi|\leq \frac{N_2}{g(t)}  \Big\}, \]
\item hyperbolic zone:
\[ \Zhyp(N_1,\varepsilon) = \Big\{ (t,\xi)\in[0,\infty)\times\mathbb{R}^n : |\xi|\geq N_1b(t) \Big\} \cap \Big\{ (t,\xi)\in[0,\infty)\times\mathbb{R}^n : |\xi|\leq \frac{\varepsilon}{g(t)} \Big\}, \]
\item dissipative zone:
\[ \Zdiss(N_1,\varepsilon) = \Big\{ (t,\xi)\in[0,\infty)\times\mathbb{R}^n : |\xi|\leq N_1b(t) \Big\}\cap\Big\{ (t,\xi)\in[0,\infty)\times\mathbb{R}^n : |\xi|\leq \frac{\varepsilon}{g(t)} \Big\}, \]
\end{itemize}
where $N_1>0$ and $N_2>0$ are sufficiently large and $\varepsilon>0$ is sufficiently small constants. Moreover, the separating lines between these zones may be defined by using the following functions:
\begin{align*}
t_{\xi,1} &= \big\{ (t,\xi) \in [0,\infty) \times \mathbb{R}^n: |\xi| = N_1b(t) \big\}, \qquad (\text{between} \,\,\, \Zdiss(N_1,\varepsilon) \,\,\, \text{and} \,\,\, \Zhyp(N_1,\varepsilon) \,), \\
t_{\xi,2} &= \big\{ (t,\xi) \in [0,\infty) \times \mathbb{R}^n: g(t)|\xi| = \varepsilon \big\}, \qquad\,\,\,\, (\text{between} \,\,\, \Zhyp(N_1,\varepsilon) \,\,\, \text{and} \,\,\,\Zred(N_2,N_1,\varepsilon) \,),\\
t_{\xi,3} &= \big\{ (t,\xi) \in [0,\infty) \times \mathbb{R}^n: g(t)|\xi| = N_2 \big\},  \qquad (\text{between} \,\,\, \Zred(N_2,N_1,\varepsilon) \,\,\, \text{and} \,\,\, \Zell(N_1,N_2)\,).
\end{align*}
For small frequencies we consider $ \Zdiss(N_1,\varepsilon)$ and $\Zhyp(N_1,\varepsilon)$, and for large frequencies we consider $\Zell(N_1,N_2)$, $\Zred(N_1,N_2,\varepsilon)$ and $\Zhyp(N_1,\varepsilon)$.
\medskip

%%%%%%%%%%%%%%%%%%%%%%%    BEGIN FIGURE    %%%%%%%%%%%%%%%%%%%%%%%%%%%%%%%%%%%%%%%%%%%%%%%%%%%%%%%%
{\color{black}
\begin{figure}[H]
\begin{center}
\begin{tikzpicture}[>=latex,xscale=1.1]
	%%%%%%%%%%%%%%%%%%%%%%%%%%%%%%%%%%%%%%%%%%%%%%%%%%%%%%%%%%%%%%%%%%%%%%%%%%%%%%%%%%%%%%%%%%%%%%%%%%%%%%%
	\draw[->] (0,0) -- (6,0)node[below]{$|\xi|$};
	\draw[->] (0,0) -- (0,5)node[left]{$t$};
    \node[below left] at(0,0){$0$};
    %%%%%%%%%%%%%%%%%%%%%%%%%%%%%%%%%%%%%%%%%%%%%%%%%%%%%%%%
    \node[right]  at (0.4,4.6) {$\textcolor{green}{t_{\xi_1}}$};
    \node[right]  at (3.4,4.6) {$\textcolor{cyan}{t_{\xi_2}}$};
     \node[right]  at (5.4,4.4) {$\textcolor{red}{t_{\xi_3}}$};
	%%%%%%%%%%%%%%%%%%%%%%%%Draw Z_osc%%%%%%%%%%%%%%%%%%%%%%%%%%%%%%%%%%%%%%%%%%%%%%%%%%%%%%%%%%%%%%%%%%%%%%%%%%
	\node[color=black] at (1.9, 1.6){{\footnotesize $Z_{\text{red}}$}};
    \node[color=black] at (3.4, 1.4){{\footnotesize $Z_{\text{ell}}$}};
	%%%%%%%%%%%%%%%%%%%%%Draw Z red%%%%%%%%%%%%%%%%%%%%%%%%%%%%%%%%%%%%%%%%%%%%%%%%%%%%%%%%%%%%%%%%%
	\draw[domain=0:4.8,color=green,variable=\t] plot ({0.3+3*exp(-\t/1.3)},\t);
	%%%%%%%%%%%%%%%%%%%%%%	%Draw Z_osc%%%%%%%%%%%%%%%%%%%%%%%%%%%%%%%%%%%%%%%%%%%%%%%%%%%%%%%%%%%%%%%%%%%%%%%%%%%
	\draw[domain=0:4.6,color=red,variable=\t] plot ({2 + 0.09*pow(\t,2.4)},\t);
	%%%%%%%%%%%%%%%%%%%%%%%%%%%%%%%%%%%%%%%%%%%%%%%%%%%%%%%%%%%%%%%%5
	\draw[domain=0:4.5,color=cyan,variable=\t] plot ({1.1 + 0.09*pow(\t,2.4)},\t);
	%%%%%%%%%%%%%%%%%%%%%%%%%%%%%%%%%%%%%%%%%%%%%%%
	\node[color=black] at (.6,.8){{\footnotesize $Z_{\text{diss}}$}};
	%%%%%%%%%%%%%%%%%%%%%%%%%%%%%%%%%%%%%%%5
	\node[color=black] at (1.3,3.4){{\footnotesize $Z_{\text{hyp}}$}};
\end{tikzpicture}
\caption{Division of extended phase space into zones}
\label{fig.zone}
\end{center}
\end{figure}
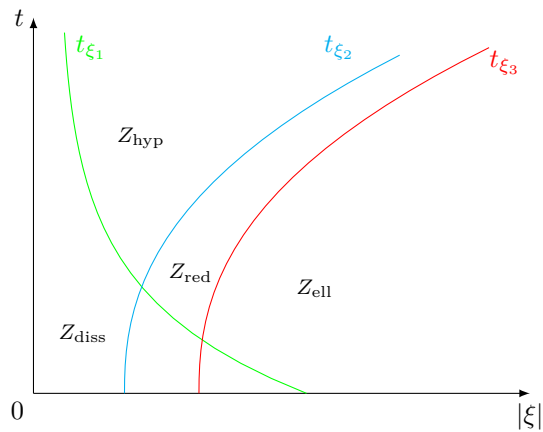}
%%%%%%%%%%%%%%%%%%%%%%%%%%%%%%%%%%%%%%   END FIGURE   %%%%%%%%%%%%%%%%%%%%%%%%%%%%%%%%%%%%%%%%%

We assume that the coefficient $g=g(t)$ satisfies the following conditions for all $t \in [0,\infty)$:
\begin{enumerate}
\item[\textbf{(G1)}] $g(t)>0$ and $g'(t)<0$,
\item[\textbf{(G2)}] $g \in L^1([0,\infty))$,
\item[\textbf{(G3)}] $\dfrac{g''}{g'} \sim \dfrac{g'}{g}$,
\item[\textbf{(G4)}] $b(t) \leq -a\dfrac{g'(t)}{g(t)}$ with a suitable constant $a\in(0,1)$.
\end{enumerate}
\begin{theorem} \label{Theorem_Noneffective_Integrable-Decaying}
Let us consider the Cauchy problem
\begin{equation*}
\begin{cases}
u_{tt}- \Delta u + b(t)u_t -g(t)\Delta u_t=0, &(t,x) \in (0,\infty) \times \mathbb{R}^n, \\
u(0,x)= u_0(x),\quad u_t(0,x)= u_1(x), &x \in \mathbb{R}^n.
\end{cases}
\end{equation*}
We assume that the coefficients $b=b(t)$ and $g=g(t)$ satisfy the conditions \textbf{(B'1)} to \textbf{(B'3)} and \textbf{(G1)} to \textbf{(G4)}, respectively. Let $(u_0,u_1)\in \dot{H}^{|\beta|+\kappa+\frac{a}{2}} \times \dot{H}^{|\beta|+\kappa+\frac{a}{2}-2}$ with $|\beta|\geq2$, $a\in(0,1)$ and arbitrarily small $\kappa>0$. Then, we have the following estimates for Sobolev solutions:
\begin{align*}
\|\,|D|^{|\beta|} u(t,\cdot)\|_{L^2} & \lesssim \exp\Big( -\frac{1}{2}\int_0^t b(\tau)d\tau \Big)\big( \|u_0\|_{\dot H^{|\beta|+\kappa+\frac{a}{2}}} + \|u_1\|_{\dot H^{|\beta|+\kappa+\frac{a}{2}-2}} \big), \\
\|\,|D|^{|\beta|-1} u_t(t,\cdot)\|_{L^2} & \lesssim \exp\Big( -\frac{1}{2}\int_0^t b(\tau)d\tau \Big)\big(\|u_0\|_{\dot H^{|\beta|+\kappa+\frac{a}{2}}} + \|u_1\|_{\dot H^{|\beta|+\kappa+\frac{a}{2}-2}} \big).
\end{align*}
\end{theorem}
\begin{proof}
To prove this theorem, we will use the WKB analysis and derive estimates accordingly in each zone.
\subsubsection{Considerations in the elliptic zone $\Zell(N_1,N_2)$} \label{Sect-Noneffective-Integrable-Zell}
Let us write the equation \eqref{AuxiliaryEquation1} in the following form:
\begin{equation} \label{Eq:Noneffective-Int-Super-Dform-Zell}
D_t^2 v + \Big( \underbrace{\dfrac{g^2(t)}{4}|\xi|^4-|\xi|^2}_{=:d^2(t,\xi)} \Big)v + \Big( \underbrace{\frac{g'(t)}{2}|\xi|^2+\frac{b(t)g(t)}{2}|\xi|^2}_{=:m(t,\xi)} \Big)v - ib(t)D_tv = 0.
\end{equation}
\begin{remark} \label{Rem:Noneffective-Integrable-Zell}
We have the following inequalities with sufficiently large $N_2$:
\begin{align} \label{Eq:Noneffective-Estimates-d-Zell}
d^2(t,\xi) \leq \frac{1}{4}g^2(t)|\xi|^4 \qquad \text{and} \qquad d^2(t,\xi)\geq \Big( \frac{1}{4}-\frac{1}{N_2^2} \Big)g^2(t)|\xi|^4.
\end{align}
Therefore, we get $d(t,\xi)\approx g(t)|\xi|^2$. Furthermore, it holds
\begin{align} \label{Eq:Noneffective-Estimates-dt-Zell}
|d_t(t,\xi)| =  \bigg| \frac{1}{4}\frac{g'(t)g(t)|\xi|^4}{\sqrt{\frac{g^2(t)}{4}|\xi|^4-|\xi|^2}} \bigg| \leq  -\frac{1}{2\sqrt{1-\frac{4}{N_2^2}}}g'(t)|\xi|^2.
\end{align}
On the other hand, we have
\begin{align*}
d_t^2(t,\xi) &=  \frac{1}{4}\frac{\Big( g''(t)g(t)|\xi|^4+(g'(t))^2|\xi|^4 \Big)\sqrt{\frac{g^2(t)}{4}|\xi|^4-|\xi|^2}}{\frac{g^2(t)}{4}|\xi|^4-|\xi|^2} - \frac{1}{16}\frac{\big( g'(t)g(t)|\xi|^4 \big)^2}{\Big( \frac{g^2(t)}{4}|\xi|^4-|\xi|^2 \Big)\sqrt{\frac{g^2(t)}{4}|\xi|^4-|\xi|^2}}.
\end{align*}
Using condition \textbf{(G3)} and estimates in \eqref{Eq:Noneffective-Estimates-d-Zell}, we get
\begin{align*}
|d_t^2(t,\xi)| &\leq  \frac{1}{4}\frac{|g''(t)|g(t)|\xi|^4+|g'(t)|^2|\xi|^4}{d(t,\xi)} + \frac{1}{16}\frac{\big( |g'(t)|g(t)|\xi|^4 \big)^2}{d^3(t,\xi)} \\
& = \frac{1}{4}\frac{|g''(t)|}{|g'(t)|}\frac{|g'(t)|g(t)|\xi|^4}{d(t,\xi)} + \frac{1}{4}\frac{|g'(t)|^2|\xi|^4}{d(t,\xi)} + \frac{1}{16}\frac{\big( |g'(t)|g(t)|\xi|^4 \big)^2}{d^3(t,\xi)} \\
& \leq \frac{1}{4}\frac{|g'(t)|}{g(t)}\frac{|g'(t)|g(t)|\xi|^4}{\big( \frac{1}{4}-\frac{1}{N_2^2} \big)^{\frac{1}{2}}g(t)|\xi|^2} + \frac{1}{4}\frac{\big( g'(t)|\xi|^2 \big)^2}{\big( \frac{1}{4}-\frac{1}{N_2^2} \big)^{\frac{1}{2}}g(t)|\xi|^2} + \frac{1}{16}\frac{\big( |g'(t)|g(t)|\xi|^4 \big)^2}{\big( \frac{1}{4}-\frac{1}{N_2^2} \big)^{\frac{3}{2}}g^3(t)|\xi|^6} \\
& \leq \bigg( \frac{1}{N_2^2\big( 1-\frac{4}{N_2^2} \big)^{\frac{1}{2}}} + \frac{1}{2N_2^2\big( 1-\frac{4}{N_2^2} \big)^{\frac{3}{2}}} \bigg)\big( g'(t)|\xi|^2 \big)^2g(t).
\end{align*}
Finally, let us note that condition \textbf{(G4)}, that is, $b(t)\leq -a \dfrac{g'(t)}{g(t)}$, $a \in (0,1)$, implies that
\[ \frac{m(t,\xi)}{d(t,\xi)} = \frac{g'(t)|\xi|^2}{2d(t,\xi)} + \frac{b(t)g(t)|\xi|^2}{2d(t,\xi)}\leq0. \]
\end{remark}
\begin{definition} \label{Def:Noneffective-integrable-symbol}
A function $f=f(t,\xi)$ belongs to the elliptic symbol class $S_{\text{ell}}^\ell\{m_1,m_2\}$, if it holds
\begin{equation*}
|D_t^kf(t,\xi)|\leq C_{k}\big( g(t)|\xi|^2 \big)^{m_1}\Big( -\frac{g'(t)}{g(t)} \Big)^{m_2+k}
\end{equation*}
for all $(t,\xi)\in \Zell(N_1,N_2)$ and all $k\leq \ell$.
\end{definition}
Some useful rules of the symbolic calculus are collected in the following proposition.
\begin{proposition} \label{Prop:Noneffective-integrable-symbol}
The following statements are true:
\begin{itemize}
\item $S_{\text{ell}}^\ell\{m_1,m_2\}$ is a vector space for all nonnegative integers $\ell$;
\item $S_{\text{ell}}^\ell\{m_1,m_2\}\cdot S_{\text{ell}}^{\ell}\{m_1',m_2'\}\hookrightarrow S^{\ell}_{\text{ell}}\{m_1+m_1',m_2+m_2'\}$;
\item $D_t^kS_{\text{ell}}^\ell\{m_1,m_2\}\hookrightarrow S_{\text{ell}}^{\ell-k}\{m_1,m_2+k\}$
for all nonnegative integers $\ell$ with $k\leq \ell$;
\item $S_{\text{ell}}^{0}\{-1,2\}\hookrightarrow L_{\xi}^{\infty}L_t^1\big( \Zell(N_1,N_2) \big)$.
\end{itemize}
\end{proposition}
\begin{proof}
Let us verify the last statement. Indeed, if $f=f(t,\xi)\in S_{\text{ell}}^{0}\{-1,2\}$, then we get
\begin{align*}
\int_{0}^{t_{\xi,3}}|f(\tau,\xi)|d\tau &\lesssim \int_{0}^{t_{\xi,3}}\frac{1}{g(\tau)|\xi|^2}\Big( -\frac{g'(\tau)}{g(\tau)} \Big)^2 d\tau \\
&= \int_{0}^{t_{\xi,3}}\frac{1}{g^2(\tau)|\xi|^2}\frac{\big( g'(\tau) \big)^2}{g(\tau)} d\tau \lesssim \int_{0}^{t_{\xi,3}}\frac{\big( g'(\tau) \big)^2}{g(\tau)} d\tau \lesssim \int_{0}^{t_{\xi,3}}g''(\tau)d\tau  \lesssim 1,
\end{align*}
where for the last two estimates we used the definition of $\Zell(N_1,N_2)$ and condition \textbf{(G3)}, respectively.
\end{proof}
\medskip

We introduce the micro-energy
\[ V=V(t,\xi):=\big( d(t,\xi)v,D_t v \big)^{\text{T}} \qquad \mbox{with} \qquad d(t,\xi) := \sqrt{\frac{g^2(t)}{4}|\xi|^4-|\xi|^2}. \]
Thus, we have to apply tools from elliptic WKB-analysis with two steps of diagonalization procedure. Transforming \eqref{Eq:Noneffective-Int-Super-Dform-Zell} to a system of first order for $V=V(t,\xi)$ gives
\begin{equation*}
D_tV=\left( \begin{array}{cc}
0 & d(t,\xi) \\
-d(t,\xi) & 0
\end{array} \right)V + \left( \begin{array}{cc}
\dfrac{D_td(t,\xi)}{d(t,\xi)} & 0 \\
-\dfrac{m(t,\xi)}{d(t,\xi)} & ib(t)
\end{array} \right)V.
\end{equation*}
Using $V=MV^{(0)}$ with $M=\begin{pmatrix} i & -i \\ 1 & 1\end{pmatrix}$, then after the first step of diagonalization we obtain
\[ D_tV^{(0)} = \big( \mathcal{D}(t,\xi) + \mathcal{R}(t,\xi) \big)V^{(0)}, \]
where
\begin{align*}
\mathcal{D}(t,\xi) &= \left( \begin{array}{cc}
-id(t,\xi) & 0 \\
0 & id(t,\xi)
\end{array} \right)\in S_{\text{ell}}^{2}\{1,0\}, \\
\mathcal{R}(t,\xi) &= \frac{1}{2} \left( \begin{array}{cc}
\dfrac{D_td(t,\xi)}{d(t,\xi)}-i\dfrac{m(t,\xi)}{d(t,\xi)}+ib(t) & -\dfrac{D_td(t,\xi)}{d(t,\xi)}+i\dfrac{m(t,\xi)}{d(t,\xi)}+ib(t) \\
-\dfrac{D_td(t,\xi)}{d(t,\xi)}-i\dfrac{m(t,\xi)}{d(t,\xi)}+ib(t) & \dfrac{D_td(t,\xi)}{d(t,\xi)}+i\dfrac{m(t,\xi)}{d(t,\xi)}+ib(t)
\end{array} \right)\in S_{\text{ell}}^{1}\{0,1\}.
\end{align*}
Here we have
\begin{align*}
\Big|\frac{D_td(t,\xi)}{d(t,\xi)}\Big| &\leq -\frac{g'(t)|\xi|^2}{\sqrt{1-\frac{4}{N_2^2}}}\frac{1}{\sqrt{1-\frac{4}{N_2^2}}g(t)|\xi|^2} \leq -\frac{1}{\big( 1-\frac{4}{N_2^2} \big)}\frac{g'(t)}{g(t)}, \\
\Big|\frac{m(t,\xi)}{d(t,\xi)}\Big| &= -\frac{m(t,\xi)}{d(t,\xi)} = -\frac{g'(t)|\xi|^2}{2d(t,\xi)} - \frac{b(t)g(t)|\xi|^2}{2d(t,\xi)} \leq -\frac{1}{\sqrt{1-\frac{4}{N_2^2}}}\frac{g'(t)}{g(t)}.
\end{align*}
Let us introduce $F_0(t,\xi):=\diag\mathcal{R}(t,\xi)$ and $\mathcal{R}_1(t,\xi):=\antidiag\mathcal{R}(t,\xi)$. To apply the second step of diagonalization procedure, we define $\delta=\delta(t,\xi)$ as the difference of the diagonal entries of the matrix $\mathcal{D}(t,\xi)+F_0(t,\xi)$ and it holds
\begin{equation*}
i\delta(t,\xi):=2d(t,\xi) + \frac{m(t,\xi)}{d(t,\xi)} = 2d(t,\xi) + \frac{g'(t)|\xi|^2}{2d(t,\xi)} + \frac{b(t)g(t)|\xi|^2}{2d(t,\xi)} \sim d(t,\xi)
\end{equation*}
for $t\leq t_{\xi_3}$ if we choose the zone constant $N_2$ sufficiently large and use condition \textbf{(G4)}.

Now we choose a matrix $N^{(1)}=N^{(1)}(t,\xi)$ such that
\begin{align*}
N^{(1)}(t,\xi) &= \left( \begin{array}{cc}
0 & -\dfrac{\mathcal{R}_{12}}{\delta(t,\xi)} \\
\dfrac{\mathcal{R}_{21}}{\delta(t,\xi)} & 0
\end{array} \right) \\
& \sim \left( \begin{array}{cc}
0 & i\dfrac{D_td(t,\xi)}{2d^2(t,\xi)}+\dfrac{m(t,\xi)}{2d^2(t,\xi)}+\dfrac{b(t)}{2d(t,\xi)} \\
-i\dfrac{D_td(t,\xi)}{2d^2(t,\xi)}+\dfrac{m(t,\xi)}{2d^2(t,\xi)}-\dfrac{b(t)}{2d(t,\xi)} & 0
\end{array} \right)\in S_{\text{ell}}^{1}\{-1,1\}.
\end{align*}
We put $N_1=N_1(t,\xi):=I+N^{(1)}(t,\xi)$. For a sufficiently large zone constant $N_2$ and all $t\leq t_{\xi,3}$ the matrix $N_1=N_1(t,\xi)$ is invertible with uniformly bounded inverse $N_1^{-1}=N_1^{-1}(t,\xi)$. Indeed, in the elliptic zone $\Zell(N_1,N_2)$, for large $N_2$ it holds
\begin{align*}
|N^{(1)}(t,\xi)| &\leq \frac{|d_t(t,\xi)|}{2d^2(t,\xi)}+\frac{|m(t,\xi)|}{2d^2(t,\xi)}+\dfrac{b(t)}{2d(t,\xi)} \\
& \leq \frac{|d_t(t,\xi)|}{2d^2(t,\xi)}+\bigg( -\frac{g'(t)|\xi|^2}{2} + \frac{b(t)g(t)|\xi|^2}{2} \bigg)\frac{1}{2d^2(t,\xi)}+\dfrac{b(t)}{2d(t,\xi)} \\
& \leq \bigg( -\frac{1}{N_2^2\Big( 1-\frac{4}{N_2^2}\Big)^{\frac{3}{2}}} - \frac{1}{N_2^2\Big( 1-\frac{4}{N_2^2}\Big)} \bigg)g'(t) + \bigg( \frac{1}{N_2^2\Big( 1-\frac{4}{N_2^2}\Big)} + \frac{1}{N_2^2\Big( 1-\frac{4}{N_2^2}\Big)^{\frac{1}{2}}} \bigg)b(t)g(t) \\
& \leq \bigg( \frac{1}{N_2^2\Big( 1-\frac{4}{N_2^2}\Big)^{\frac{3}{2}}} + \frac{1}{N_2^2\Big( 1-\frac{4}{N_2^2}\Big)} + \frac{a}{N_2^2\Big( 1-\frac{4}{N_2^2}\Big)^{\frac{1}{2}}} + \frac{a}{N_2^2\Big( 1-\frac{4}{N_2^2}\Big)} \bigg)\big(-g'(t)\big)<1,
\end{align*}
where we used estimates \eqref{Eq:Noneffective-Estimates-d-Zell} and \eqref{Eq:Noneffective-Estimates-dt-Zell} from Remark \ref{Rem:Noneffective-Integrable-Zell}, and condition \textbf{(G4)}, namely, $b(t)\leq-a\dfrac{g'(t)}{g(t)}$.\\

Let us define
\begin{align*}
B^{(1)}(t,\xi) &:= D_tN^{(1)}(t,\xi)-( \mathcal{R}(t,\xi)-F_0(t,\xi))N^{(1)}(t,\xi)\in S_{\text{ell}}^{0}\{-1,2\}, \\
\mathcal{R}_2(t,\xi) &:= -N_1^{-1}(t,\xi)B^{(1)}(t,\xi)\in S_{\text{ell}}^{0}\{-1,2\}.
\end{align*}
Consequently, we have the following operator identity:
\begin{equation*}
\big( D_t-\mathcal{D}(t,\xi)-\mathcal{R}(t,\xi) \big)N_1(t,\xi)=N_1(t,\xi)\big( D_t-\mathcal{D}(t,\xi)-F_0(t,\xi)-\mathcal{R}_2(t,\xi) \big).
\end{equation*}
\begin{proposition} \label{Prop:Noneffective-Integrable-Estimates-Zell}
The fundamental solution $E_{\text{ell}}^{V}=E_{\text{ell}}^{V}(t,s,\xi)$ to the transformed operator
\[ D_t-\mathcal{D}(t,\xi)-F_0(t,\xi)-\mathcal{R}_2(t,\xi) \]
can be estimated by
\begin{equation*}
(|E_{\text{ell}}^{V}(t,s,\xi)|) \lesssim \Big( \frac{g(s)}{g(t)} \Big)^{\frac{a}{2\sqrt{1-\frac{4}{N_2^2}}}-1}\exp \bigg( \frac{1}{2}\int_{s}^{t} g(\tau)|\xi|^2 d\tau - \frac{1}{2}\int_s^tb(\tau) d\tau \bigg) \left( \begin{array}{cc}
1 & 1 \\
1 & 1
\end{array} \right),
\end{equation*}
with $(t,\xi),(s,\xi)\in \Zell(N_1,N_2)$, $0\leq s\leq t\leq t_{\xi,3}$ and $a\in(0,1)$.
\end{proposition}
\begin{proof}
We transform the system for $E_{\text{ell}}^{V}=E_{\text{ell}}^{V}(t,s,\xi)$ to an integral equation for a new matrix-valued function $\mathcal{Q}_{\text{ell}}=\mathcal{Q}_{\text{ell}}(t,s,\xi)$ as in the proof of Proposition \ref{Prop_Scattering_EllZone}. We define
\[ \mathcal{Q}_{\text{ell}}(t,s,\xi):=\exp\bigg\{ -\int_{s}^{t}\beta(\tau,\xi)d\tau \bigg\} E_{\text{ell}}^{V}(t,s,\xi), \]
where $\beta=\beta(t,\xi)$ is chosen from the main entries of the diagonal matrix $i\mathcal{D}(t,\xi)+iF_0(t,\xi)$ as follows:
\[ \beta(t,\xi)=d(t,\xi)+\frac{d_t(t,\xi)}{2d(t,\xi)}+\frac{m(t,\xi)}{2d(t,\xi)}-\frac{b(t)}{2}. \]
The following new integral equation holds:
\begin{align*}
\mathcal{Q}_{\text{ell}}(t,s,\xi)=&\exp \bigg\{ \int_{s}^{t}\big( i\mathcal{D}(\tau,\xi)+iF_0(\tau,\xi)-\beta(\tau,\xi)I \big)d\tau \bigg\}\\
& \quad + \int_{s}^{t} \exp \bigg\{ \int_{\theta}^{t}\big( i\mathcal{D}(\tau,\xi)+iF_0(\tau,\xi)-\beta(\tau,\xi)I \big)d\tau \bigg\}\mathcal{R}_2(\theta,\xi)\mathcal{Q}_{\text{ell}}(\theta,s,\xi)\,d\theta.
\end{align*}
If we define
\begin{align*}
H(t,s,\xi) & =\exp \bigg\{ \int_{s}^{t}\big( i\mathcal{D}(\tau,\xi)+iF_0(\tau,\xi)-\beta(\tau,\xi)I \big)d\tau \bigg\}\\
& = \diag \bigg( 1, \exp \bigg\{ \int_{s}^{t}\bigg( -2d(\tau,\xi)-\frac{m(\tau,\xi)}{d(\tau,\xi)} \bigg)d\tau \bigg\} \bigg)\rightarrow \left( \begin{array}{cc}
1 & 0 \\
0 & 0
\end{array} \right)
\end{align*}
as $t\rightarrow \infty$, then we may conclude that the matrix $H=H(t,s,\xi)$ is uniformly bounded for $(s,\xi),(t,\xi)\in \Zell(N_1,N_2)$. So, the representation of $\mathcal{Q}_{\text{ell}}=\mathcal{Q}_{\text{ell}}(t,s,\xi)$ by a Neumann series gives
\begin{align*}
\mathcal{Q}_{\text{ell}}(t,s,\xi)=H(t,s,\xi)+\sum_{k=1}^{\infty}i^k\int_{s}^{t}H(t,t_1,\xi)\mathcal{R}_2(t_1,\xi)&\int_{s}^{t_1}H(t_1,t_2,\xi)\mathcal{R}_2(t_2,\xi)\\
& \cdots \int_{s}^{t_{k-1}}H(t_{k-1},t_k,\xi)\mathcal{R}_2(t_k,\xi)dt_k\cdots dt_2dt_1.
\end{align*}
Then, convergence of this series is obtained from the symbol estimates, since $\mathcal{R}_2=\mathcal{R}_2(t,\xi)$ is uniformly integrable over $\Zell(N_1,N_2)$ due to last item of Proposition \ref{Prop:Noneffective-integrable-symbol}. Thus, we may conclude
\begin{align*} \label{Eq:Noneffective-Estimate-EellV-Zell}
E_{\text{ell}}^{V}(t,s,\xi)&=\exp \bigg\{ \int_{s}^{t}\beta(\tau,\xi)d\tau \bigg\}\mathcal{Q}_{\text{ell}}(t,s,\xi) \nonumber \\
& = \exp \bigg\{ \int_{s}^{t}\bigg( d(\tau,\xi)+\frac{\partial_\tau d(\tau,\xi)}{2d(\tau,\xi)}+\frac{m(\tau,\xi)}{2d(\tau,\xi)}-\frac{b(\tau)}{2} \bigg)d\tau \bigg\}\mathcal{Q}_{\text{ell}}(t,s,\xi) \nonumber \\
& \leq  \exp \bigg\{ \int_{s}^{t}\bigg( d(\tau,\xi)+\frac{\partial_\tau d(\tau,\xi)}{2d(\tau,\xi)} + \frac{g'(\tau)}{2g(\tau)} - \frac{a}{2\sqrt{1-\frac{4}{N^2}}}\frac{g'(\tau)}{g(\tau)} - \frac{b(\tau)}{2} \bigg)d\tau \bigg\}\mathcal{Q}_{\text{ell}}(t,s,\xi) \nonumber \\
& \leq \sqrt{\frac{d(t,\xi)g(t)}{d(s,\xi)g(s)}}\Big( \frac{g(s)}{g(t)} \Big)^{\frac{a}{2\sqrt{1-\frac{4}{N_2^2}}}} \exp \bigg\{ \int_{s}^{t}d(\tau,\xi)d\tau - \frac{1}{2}\int_s^t b(\tau)d\tau \bigg\}\mathcal{Q}_{\text{ell}}(t,s,\xi),
\end{align*}
where taking into account that $g'(t)<0$ due to condition \textbf{(G1)} and using the estimates in \eqref{Eq:Noneffective-Estimates-d-Zell}, we used
\[ \frac{m(t,\xi)}{2d(t,\xi)} = \frac{g'(t)|\xi|^2}{4d(t,\xi)} + \frac{b(t)g(t)|\xi|^2}{4d(t,\xi)}  \leq \frac{g'(t)}{2g(t)} + \frac{1}{2\sqrt{1-\frac{4}{N_2^2}}}b(t) \leq \frac{g'(t)}{2g(t)} - \frac{a}{2\sqrt{1-\frac{4}{N_2^2}}}\frac{g'(t)}{g(t)}. \]
Then, using the fact that $|\mathcal{Q}_{\text{ell}}(t,s,\xi)|$ is uniformly bounded in $\Zell(N_1,N_2)$, it follows
\begin{align*}
(|E_{\text{ell}}^{V}(t,s,\xi)|) & \lesssim \frac{g(t)}{g(s)}\Big( \frac{g(s)}{g(t)} \Big)^{\frac{a}{2\sqrt{1-\frac{4}{N_2^2}}}}\exp \bigg\{ \frac{1}{2}\int_{s}^{t} g(\tau)|\xi|^2d\tau - \frac{1}{2}\int_s^tb(\tau)d\tau \bigg\} \left( \begin{array}{cc}
1 & 1 \\
1 & 1 \end{array} \right)|\mathcal{Q}_{\text{ell}}(t,s,\xi)| \\
& \lesssim \Big( \frac{g(s)}{g(t)} \Big)^{\frac{a}{2\sqrt{1-\frac{4}{N_2^2}}}-1}\exp \bigg\{ \frac{1}{2}\int_{s}^{t} g(\tau)|\xi|^2 d\tau - \frac{1}{2}\int_s^tb(\tau) d\tau \bigg\} \left( \begin{array}{cc}
1 & 1 \\
1 & 1
\end{array} \right).
\end{align*}
This completes the proof.
\end{proof}
Since we used the first change of variable from Section \ref{Section_OurApproach}, we use now the backward transformation
\[ v(t,\xi) = \exp\Big( \frac{1}{2}\int_0^t g(\tau)|\xi|^2 d\tau\Big) \hat{u}(t,\xi), \]
and arrive at the following result.
\begin{corollary} \label{Cor:Noneffective-Integrable-Zell}
We have the following estimates for $0\leq s\leq t\leq t_{\xi,3}$ and $(s,\xi), (t,\xi) \in \Zell(N_1,N_2)$:
\begin{align*}
\frac{g(t)}{2}|\xi|^{|\beta|}|\hat{u}(t,\xi)| & \lesssim \Big( \frac{g(s)}{g(t)} \Big)^{\kappa+\frac{a}{2}-1}\frac{\lambda(s)}{\lambda(t)}\Big( g(s)|\xi|^{|\beta|}|\hat{u}(s,\xi)| + |\xi|^{|\beta|-2}|\hat{u}_t(s,\xi)| \Big) \quad \mbox{for} \quad |\beta| \geq 2, \\
|\xi|^{|\beta|}|\hat{u}_t(t,\xi)| & \lesssim \Big( \frac{g(s)}{g(t)} \Big)^{\kappa+\frac{a}{2}-1}\frac{\lambda(s)}{\lambda(t)}\Big( g(s)|\xi|^{|\beta|+2}|\hat{u}(s,\xi)| + |\xi|^{|\beta|}|\hat{u}_t(s,\xi)| \Big) \quad \mbox{for} \quad |\beta| \geq 0,
\end{align*}
where $\lambda=\lambda(t) = \exp\big( \frac{1}{2}\int_0^tb(\tau)d\tau \big)$ and $\kappa=\frac{a}{2\sqrt{1-\frac{4}{N_2^2}}}-\frac{a}{2}$ is an arbitrarily small exponent for arbitrarily large $N_2=N_2(\kappa)$ for all $\kappa>0$.
\end{corollary}
\subsubsection{Considerations in the reduced zone $\Zred(N_1,N_2,\varepsilon)$} \label{Sect-Noneffective-Integrable-Zred}
We write the equation \eqref{MainEquationFourier} in the following form:
\begin{equation} \label{Eq:Noneffective-Integrable-Dform-Zred}
D_t^2\hat{u} - |\xi|^2\hat{u} - ib(t)D_t\hat{u} - ig(t)|\xi|^2D_t\hat{u} = 0.
\end{equation}
We define the micro-energy $U=(|\xi|\hat{u},D_t\hat{u})^\text{T}$. Then, Eq. \eqref{Eq:Noneffective-Integrable-Dform-Zred} leads to the system
\begin{equation*} \label{Eq:Noneffective-Integrable-System-Zred}
D_tU=\underbrace{\left( \begin{array}{cc}
0 & |\xi| \\
|\xi| & i\big( b(t) + g(t)|\xi|^2 \big)
\end{array} \right)}_{A(t,\xi)}U.
\end{equation*}
We can define for all $(t,\xi) \in \Zred(N_2,\varepsilon)$ the following energy of the solutions to \eqref{Eq:Noneffective-Integrable-Dform-Zred}:
\begin{equation*}
\mathcal{E}(t,\xi) := \frac{1}{2}\big( |\xi|^2|\hat{u}(t,\xi)|^2 + |\hat{u}_t(t,\xi)|^2 \big).
\end{equation*}
If we differentiate the energy $\mathcal{E}=\mathcal{E}(t,\xi)$ with respect to $t$ and use our equation \eqref{Eq:Noneffective-Integrable-Dform-Zred}, it follows
\[ \frac{d}{dt}\mathcal{E}(t,\xi) = -\big( b(t)+g(t)|\xi|^2 \big)|\hat{u}_t(t,\xi)|^2 \leq 0. \]
This means that $\mathcal{E}=\mathcal{E}(t,\xi)$ is monotonically decreasing in $t$. Therefore, we have $\mathcal{E}(t,\xi)\leq \mathcal{E}(s,\xi)$ for $t_{\xi,3}\leq s\leq t\leq t_{\xi,2}$ and $(s,\xi), (t,\xi) \in \Zred(N_1,N_2,\varepsilon)$. Thus, we may write
\begin{align*}
|\xi||\hat{u}(t,\xi)| &\leq \sqrt{\mathcal{E}(s,\xi)} \leq |\xi||\hat{u}(s,\xi)| + |\hat{u}_t(s,\xi)|, \\
|\hat{u}_t(t,\xi)| &\leq \sqrt{\mathcal{E}(s,\xi)} \leq |\xi||\hat{u}(s,\xi)| + |\hat{u}_t(s,\xi)|.
\end{align*}
\begin{corollary} \label{Cor:Noneffective-Integrable-Zred}
The following estimates hold with $t_{\xi,3}\leq s\leq t\leq t_{\xi,2}$ and $(s,\xi), (t,\xi) \in \Zred(N_1,N_2,\varepsilon)$:
\begin{align*}
|\xi|^{|\beta|}|\hat{u}(t,\xi)| &\lesssim |\xi|^{|\beta|}|\hat{u}(s,\xi)| + |\xi|^{|\beta|-1}|\hat{u}_t(s,\xi)| \quad \mbox{for} \quad |\beta| \geq 1, \\
|\xi|^{|\beta|}|\hat{u}_t(t,\xi)| &\lesssim |\xi|^{|\beta|+1}|\hat{u}(s,\xi)| + |\xi|^{|\beta|}|\hat{u}_t(s,\xi)| \quad \mbox{for} \quad |\beta| \geq 0.
\end{align*}
\end{corollary}
\subsubsection{Considerations in the hyperbolic zone $\Zhyp(N_1,\varepsilon)$} \label{Sect-Noneffective-Integrable-Zhyp}
Let us introduce the micro-energy $U=(|\xi|\hat{u},D_t\hat{u})^\text{T}$. Then, we obtain from \eqref{MainEquationFourier} the system of first order
\begin{equation} \label{Eq:Noneffective-Integrable-Hyp-System1}
D_tU = \underbrace{\left( \begin{array}{cc}
0 & |\xi| \\
|\xi| & i\big( b(t)+g(t)|\xi|^2 \big)
\end{array} \right)}_{A(t,\xi)}U.
\end{equation}
We want to construct the corresponding fundamental solution to \eqref{Eq:Noneffective-Integrable-Hyp-System1} as follows:
\begin{equation*}
D_tE(t,s,\xi) = A(t,\xi)E(t,s,\xi), \quad E(s,s,\xi)=I, \quad 0\leq t_{\xi_1}\leq s\leq t\leq t_{\xi_2}.
\end{equation*}
\begin{proposition} \label{Prop:Noneffective-Integrable-Hyp}
The fundamental solution $E=E(t,s,\xi)$ to the system \eqref{Eq:Noneffective-Integrable-Hyp-System1} satisfies the following estimate in $\Zhyp(N_1,\varepsilon)$:
\begin{equation*}
(|E(t,s,\xi)|) \lesssim \exp\Big( -\frac{2-\varepsilon}{4}\int_s^t\big( b(\tau)+g(\tau)|\xi|^2 \big)d\tau \Big)\left( \begin{array}{cc}
1 & 1 \\
1 & 1
\end{array} \right),
\end{equation*}
with $(t,\xi),(s,\xi)\in \Zhyp(N_1,\varepsilon)$, $t_{\xi,1}\leq s\leq t$ and $t_{\xi,2}\leq s\leq t$.
\end{proposition}
\begin{proof}
Let us start to diagonalize the matrix $A=A(t,\xi)$ in \eqref{Eq:Noneffective-Integrable-Hyp-System1}. For this reason we set
\[ M = \left( \begin{array}{cc}
1 & -1 \\
1 & 1
\end{array} \right) \qquad \text{and} \qquad  M^{-1} = \frac{1}{2}\left( \begin{array}{cc}
1 & 1 \\
-1 & 1
\end{array} \right). \]
If we define $U^{(0)}:=M^{-1}U$, then we arrive at the system
\begin{equation*}
D_tU^{(0)} = \big( \mathcal{D}(\xi) + \mathcal{R}(t,\xi) \big)U^{(0)},
\end{equation*}
where
\begin{align*}
\mathcal{D}(\xi) = \left( \begin{array}{cc}
\tau_1 & 0 \\
0 & \tau_2
\end{array} \right) =\left( \begin{array}{cc}
-|\xi| & 0 \\
0 & |\xi|
\end{array} \right) \quad \mbox{and} \quad \mathcal{R}(t,\xi) = \frac{1}{2}\left( \begin{array}{cc}
i\big( b(t) + g(t)|\xi|^2 \big) & -i\big( b(t) + g(t)|\xi|^2 \big) \\
-i\big( b(t) + g(t)|\xi|^2 \big) & i\big( b(t) + g(t)|\xi|^2 \big)
\end{array} \right).
\end{align*}
Let $F_0=F_0(t,\xi)$ be the diagonal part of $\mathcal{R}=\mathcal{R}(t,\xi)$. To carry out the second step of diagonalization procedure, we introduce the matrices
\begin{align*}
N^{(1)}(t,\xi) := \left( \begin{array}{cc}
0 & \dfrac{\mathcal{R}_{12}}{\tau_1-\tau_2} \\
\dfrac{\mathcal{R}_{21}}{\tau_2-\tau_1} & 0
\end{array} \right) = \frac{1}{4}\left( \begin{array}{cc}
0 & i\Big( g(t)|\xi| + \dfrac{b(t)}{|\xi|} \Big) \\
-i\Big( g(t)|\xi| + \dfrac{b(t)}{|\xi|} \Big) & 0
\end{array} \right),
\end{align*}
and $N_1(t,\xi) := I+N^{(1)}(t,\xi)$. For all $(t,\xi) \in \Zhyp(N_1,\varepsilon)$ the matrix $N_1=N_1(t,\xi)$ is invertible with uniformly bounded inverse $N_1^{-1}=N_1^{-1}(t,\xi)$. Indeed, in $\Zhyp(N_1,\varepsilon)$ it holds
\[ |N^{(1)}(t,\xi)| \leq \frac{g(t)|\xi|}{4} + \frac{b(t)}{4|\xi|} \leq \frac{\varepsilon}{4} + \frac{1}{4N_1}. \]
We set
\begin{align*}
B^{(1)}(t,\xi) &:= D_tN^{(1)}(t,\xi)-\big( \mathcal{R}(t,\xi)-F_0(t,\xi) \big)N^{(1)}(t,\xi) \\
& = \frac{1}{8}\left( \begin{array}{cc}
g^2(t)|\xi|^3 +2b(t)g(t)|\xi| + \dfrac{b^2(t)}{|\xi|} & 2\Big( \dfrac{b'(t)}{|\xi|} + g'(t)|\xi| \Big) \\
-2\Big( \dfrac{b'(t)}{|\xi|} + g'(t)|\xi| \Big) & -g^2(t)|\xi|^3 - 2b(t)g(t)|\xi| - \dfrac{b^2(t)}{|\xi|}
\end{array} \right),
\end{align*}
and, then introduce
\[ \mathcal{R}_1(t,\xi) := -N_1^{-1}(t,\xi)B^{(1)}(t,\xi). \]
Thus, we may conclude
\begin{equation*}
\big( D_t-\mathcal{D}(\xi)-\mathcal{R}(t,\xi) \big)N_1(t,\xi)=N_1(t,\xi)\big( D_t-\mathcal{D}(\xi)-F_0(t,\xi)-\mathcal{R}_1(t,\xi) \big).
\end{equation*}
We turn to  $U^{(1)}=U^{(1)}(t,\xi)$ as the solution to the system
\[ \big( D_t-\mathcal{D}(\xi)-F_0(t,\xi)-\mathcal{R}_1(t,\xi) \big)U^{(1)}(t,\xi) = 0. \]
We can write $U^{(1)}(t,\xi) = E_{U,1}(t,s,\xi)U^{(1)}(s,\xi)$. Here $E_{U,1}=E_{U,1}(t,s,\xi)$ is the fundamental solution to the following system:
\begin{align*}
\big( D_t-\mathcal{D}(\xi)-F_0(t,\xi)-\mathcal{R}_1(t,\xi) \big)E_{U,1}(t,s,\xi) = 0, \quad E_{U,1}(s,s,\xi) = I.
\end{align*}
The solution $E_0 = E_0(t,s,\xi)$ of the ``principal diagonal part'' satisfies
\begin{align*}
D_tE_0(t,s,\xi) = \big( \mathcal{D}(\xi)+F_0(t,\xi) \big)E_0(t,s,\xi), \quad E_{0}(s,s,\xi) = I,
\end{align*}
with $t\geq s$ and $(t,\xi), (s,\xi) \in \Zhyp(N_1,\varepsilon)$. Consequently, we have
\[ E_0(t,s,\xi) = \exp \bigg( i\int_s^t\big( \mathcal{D}(\xi)+F_0(\tau,\xi) \big)d\tau \bigg), \]
and, we can estimate it in the following way:
\[ |E_0(t,s,\xi)| \lesssim \exp \bigg( -\frac{1}{2}\int_s^t\big( b(\tau) + g(\tau)|\xi|^2 \big)d\tau \bigg). \]
Let us set
\begin{align*}
\mathcal{R}_2(t,s,\xi) &= E_0^{-1}(t,s,\xi)\mathcal{R}_1(t,\xi)E_0(t,s,\xi), \\
Q(t,s,\xi) &= I + \sum_{k=1}^\infty i^k\int_s^t\mathcal{R}_2(t_1,s,\xi)\int_s^{t_1}\mathcal{R}_2(t_2,s,\xi)\cdots\int_s^{t_{k-1}}\mathcal{R}_2(t_k,s,\xi)dt_k\cdots dt_2dt_1.
\end{align*}
Then, $Q=Q(t,s,\xi)$ solves the Cauchy problem
\begin{align*}
D_tQ(t,s,\xi) = \mathcal{R}_2(t,s,\xi)Q(t,s,\xi), \quad Q(s,s,\xi) = I.
\end{align*}
The fundamental solution $E_{U,1}=E_{U,1}(t,s,\xi)$ may be represented as $E_{U,1}(t,s,\xi) = E_0(t,s,\xi)Q(t,s,\xi)$, where the estimate for $Q=Q(t,s,\xi)$ is given as follows:
\begin{align*}
|Q(t,s,\xi)| &\leq \exp \bigg( \int_s^t|\mathcal{R}_1(\tau,\xi)|d\tau \bigg) \\
& \leq \exp \bigg( \frac{1}{8}\int_s^t\big( g^2(\tau)|\xi|^3 + 2b(\tau)g(\tau)|\xi| + \frac{b^2(\tau)}{|\xi|} - 2g'(\tau)|\xi| - \frac{2b'(\tau)}{|\xi|} \big)d\tau \bigg) \\
& \leq \exp \bigg( \frac{\varepsilon}{8}\int_s^t g(\tau)|\xi|^2d\tau + \frac{\varepsilon}{4}\int_s^t b(\tau)d\tau + \frac{1}{N_1}\int_s^t b(\tau)d\tau + \frac{1}{4}g(s)|\xi| + \frac{1}{4}\frac{b(s)}{|\xi|} \bigg) \\
& \leq \exp \Big( \frac{\varepsilon}{4}\int_s^t \big( b(\tau)+g(\tau)|\xi|^2 \big) d\tau \Big),
\end{align*}
where from the definition of $\Zhyp(N_1,\varepsilon)$ we used $g(t)|\xi|\leq \varepsilon$ and $|\xi|\geq N_1b(t)$. Consequently, we get
\begin{align*}
|E_{U,1}(t,s,\xi)| &\leq |E_0(t,s,\xi)|\,|Q(t,s,\xi)| \\
& \leq \exp \bigg( -\frac{1}{2}\int_s^t\big( b(\tau)+g(\tau)|\xi|^2 \big)d\tau+\frac{\varepsilon}{4}\int_s^t \big( b(\tau) + g(\tau)|\xi|^2 \big) d\tau \bigg) \\
& \leq \exp\bigg( -\frac{2-\varepsilon}{4}\int_s^t\big( b(\tau)+g(\tau)|\xi|^2 \big)d\tau \bigg).
\end{align*}
This completes the proof.
\end{proof}
\begin{corollary} \label{Cor:Noneffective-Integrable-Zhyp}
We have the following estimates for $(s,\xi), (t,\xi) \in \Zhyp(N_1,\varepsilon)$ with $t_{\xi,1}\leq s\leq t$ and $t_{\xi,2}\leq s\leq t$:
\begin{align*}
|\xi|^{|\beta|}|\hat{u}(t,\xi)| & \lesssim \exp\bigg( -\frac{2-\varepsilon}{4}\int_s^t\big( b(\tau)+g(\tau)|\xi|^2 \big)d\tau \bigg)\big( |\xi|^{|\beta|}|\hat{u}(s,\xi)| + |\xi|^{|\beta|-1}|\hat{u}_t(s,\xi)| \big) \quad \mbox{for} \quad |\beta|\geq 1, \\
|\xi|^{|\beta|}|\hat{u}_t(t,\xi)| & \lesssim \exp\bigg( -\frac{2-\varepsilon}{4}\int_s^t\big( b(\tau)+g(\tau)|\xi|^2 \big)d\tau \bigg)\big( |\xi|^{|\beta|+1}|\hat{u}(s,\xi)| + |\xi|^{|\beta|}|\hat{u}_t(s,\xi)| \big) \quad \mbox{for} \quad |\beta|\geq 0.
\end{align*}
\end{corollary}
\subsubsection{Considerations in the dissipative zone $\Zdiss(N_1,\varepsilon)$} \label{Sect-Noneffective-Integrable-Zdiss}
We write the equation \eqref{MainEquationFourier} in the following form:
\begin{equation} \label{Eq:Noneffective-Integrable-Dform-Zdiss}
D_t^2\hat{u} - |\xi|^2\hat{u} - ib(t)D_t\hat{u} - ig(t)|\xi|^2D_t\hat{u} = 0.
\end{equation}
We define the micro-energy $U=\big( \frac{N_1}{1+t}\hat{u}, D_t\hat{u} \big)^\text{T}$. Then, Eq. \eqref{Eq:Noneffective-Integrable-Dform-Zdiss} leads to the system of first order
\begin{equation} \label{Eq:Noneffective-Decaying-System-Zdiss}
D_tU=\underbrace{\left( \begin{array}{cc}
\dfrac{i}{1+t} & \dfrac{N_1}{1+t} \\
\dfrac{(1+t)|\xi|^{2}}{N_1} & i\big( b(t)+g(t)|\xi|^2 \big)
\end{array} \right)}_{A(t,\xi)}U.
\end{equation}
We are interested in the fundamental solution $E_{\text{diss}}=E_{\text{diss}}(t,s,\xi)$ to the system \eqref{Eq:Noneffective-Decaying-System-Zdiss}, namely
\[ D_tE_{\text{diss}}(t,s,\xi) = A(t,\xi)E_{\text{diss}}(t,s,\xi), \qquad E_{\text{diss}}(s,s,\xi) = I, \]
for all $0\leq s\leq t\leq t_{\xi,1}$ and $(t,\xi), (s,\xi) \in \Zdiss(N_1,\varepsilon)$. Thus, the solution $U=U(t,\xi)$ may be represented as
\[ U(t,\xi)=E_{\text{diss}}(t,s,\xi)U(s,\xi). \]
The entries $E_{\text{diss}}^{(k\ell)}(t,s,\xi)$, $k,\ell=1,2,$ of the fundamental solution $E_{\text{diss}}(t,s,\xi)$ satisfy the following system for $\ell=1,2$:
\begin{align*}
D_tE_{\text{diss}}^{(1\ell)}(t,s,\xi) &= \frac{i}{1+t}E_{\text{diss}}^{(1\ell)}(t,s,\xi) + \frac{N_1}{1+t}E_{\text{diss}}^{(2\ell)}(t,s,\xi), \\
D_tE_{\text{diss}}^{(2\ell)}(t,s,\xi) &= \frac{(1+t)|\xi|^2}{N_1}E_{\text{diss}}^{(1\ell)}(t,s,\xi) + i\big( b(t)+g(t)|\xi|^2 \big)E_{\text{diss}}^{(2\ell)}(t,s,\xi).
\end{align*}
Then, by straight-forward calculations (with $\delta_{k \ell}=1$ if $k=\ell$ and $\delta_{k \ell}=0$ otherwise), we get
\begin{align*}
E_{\text{diss}}^{(1\ell)}(t,s,\xi) & = \frac{1+s}{1+t}\delta_{1\ell} + i\frac{N_1}{1+t}\int_{s}^{t}E_{\text{diss}}^{(2\ell)}(\tau,s,\xi)d\tau, \\
E_{\text{diss}}^{(2\ell)}(t,s,\xi) & = \frac{\delta(s,\xi)}{\delta(t,\xi)}\delta_{2\ell} + i\frac{|\xi|^2}{N_1\delta(t,\xi)}\int_{s}^{t}(1+\tau)\delta(\tau,\xi)E_{\text{diss}}^{(1\ell)}(\tau,s,\xi)d\tau,
\end{align*}
where $\delta=\delta(t,\xi)$ is defined as follows:
\begin{equation} \label{Eq:Noneffective-Integrable-Def-delta-Zdiss}
\delta=\delta(t,\xi)=\exp\bigg( \int_{0}^{t}\big( b(\tau)+g(\tau)|\xi|^2 \big)d\tau \bigg).
\end{equation}
\begin{proposition} \label{Prop:Noneffective-Estimates-in-Zdiss}
We have the following estimates in the dissipative zone:
\[ (|E_{\text{diss}}(t,s,\xi)|) \lesssim \frac{\delta(s,\xi)}{\delta(t,\xi)}\left( \begin{array}{cc}
1 & 1 \\
1 & 1
\end{array} \right) \]
for all $0\leq s\leq t\leq t_{\xi,1}$ and $(t,\xi), (s,\xi)\in\Zdiss(N_1,\varepsilon)$.
\end{proposition}
%
%\begin{itemize}
%\item[\textbf{(C1)}] $\limsup_{t\to\infty}tb(t)<1$.
%\end{itemize}
\begin{lemma}
Under the condition $\limsup_{t\to\infty}tb(t)<1$, it holds
\[ \int_0^t\frac{d\tau}{\delta(\tau,\xi)} \sim \frac{t}{\delta(t,\xi)} \]
and $\dfrac{t}{\delta(t,\xi)}$ is monotonously increasing for large time $t$.
\end{lemma}
\begin{proof}
The proof of this lemma is similar to proof of Proposition 7 in \cite{Wirth-Noneffective=2006}. Nevertheless, to make the paper self-contained, we will present the proof.
\medskip

Applying partial integration, we find
\[ \int_0^t\frac{d\tau}{\delta(\tau,\xi)} = \frac{t}{\delta(t,\xi)} + \int_0^t\frac{\tau\big( b(\tau)+g(\tau)|\xi|^2 \big)}{\delta(\tau,\xi)}d\tau. \]
On the other hand, due to the condition $\limsup_{t\to\infty} tb(t)<1$, we use $tb(t)\leq c<1$ for $t\geq t_0$ such that
\begin{align*}
\int_0^t\frac{\tau\big( b(\tau)+g(\tau)|\xi|^2 \big)}{\delta(\tau,\xi)}d\tau &\leq (1+\varepsilon N_1)\int_0^t\frac{\tau b(\tau)}{\delta(\tau,\xi)}d\tau \\
& \leq (1+\varepsilon N_1)\int_0^{t_0}\frac{\tau b(\tau)}{\delta(\tau,\xi)}d\tau + c(1+\varepsilon N_1)\int_{t_0}^t\frac{d\tau}{\delta(\tau,\xi)} \leq C_1 + C_2\int_0^t\frac{d\tau}{\delta(\tau,\xi)},
\end{align*}
where due to the definition of $\Zdiss(N_1,\varepsilon)$, we used $g(t)|\xi|^2 \leq \varepsilon|\xi| \leq \varepsilon N_1b(t)$. Then, we arrive at
\[ \int_0^t\frac{d\tau}{\delta(\tau,\xi)} \leq \frac{1}{1-C_2}\Big( C_1+\frac{t}{\delta(t,\xi)} \Big) \lesssim \frac{t}{\delta(t,\xi)}. \]
The monotonicity is a consequence of $tb(t)<1$ for large time $t$ with
\[ \frac{d}{dt}\frac{t}{\delta(t,\xi)} = \frac{1-t\big( b(t)+g(t)|\xi|^2 \big)}{\delta(t,\xi)} \geq \frac{1-(1-\varepsilon N_1)tb(t)}{\delta(t,\xi)}, \]
where we used again $g(t)|\xi|^2 \leq \varepsilon|\xi| \leq \varepsilon N_1b(t)$.
\end{proof}
%\begin{proof} [Proof of Proposition \ref{Prop:Noneffective-Estimates-in-Zdiss}]
%Let us multiply both equations by $\dfrac{\delta(t,\xi)}{\delta(s,\xi)}$ and then, we plug the first representation into the second one. This implies
%\begin{align*}
%& \frac{\delta(t,\xi)}{\delta(s,\xi)}E_{\text{diss}}^{(2\ell)}(t,s,\xi) \\
%& \quad = \delta_{2\ell} + i\frac{|\xi|^2}{N_1}\int_{s}^{t}(1+s)\frac{\delta(\tau,\xi)}{\delta(s,\xi)}\delta_{1\ell}\,d\tau - |\xi|^2\int_s^t\delta(\tau,\xi)\int_s^\tau\frac{1}{\delta(\theta,\xi)}\Big( \frac{\delta(\theta,\xi)}{\delta(s,\xi)}E_{\text{diss}}^{(21)}(\theta,s,\xi) \Big)d\theta d\tau.
%\end{align*}
%\end{proof}
Taking into account the micro-energy $U=\big( \frac{N_1}{1+t}\hat{u}, D_t\hat{u} \big)^\text{T}$ and from definition of $\Zdiss(N_1,\varepsilon)$ $|\xi|\leq N_1b(t)$, by using
\[ \frac{N_1}{1+t} = \frac{N_1|\xi|}{(1+t)|\xi|} \geq\frac{N_1|\xi|}{(1+t)N_1b(t)} \geq |\xi| \]
for large time $t$ (here $(1+t)b(t)<1$), we may arrive at the following estimates.
\begin{corollary} \label{Cor:Noneffective-Integrable-Zdiss}
We have the following estimates for $0\leq s\leq t \leq t_{\xi,1}$ and $(s,\xi), (t,\xi) \in \Zdiss(N_1,\varepsilon)$:
\begin{align*}
|\xi|^{|\beta|}|\hat{u}(t,\xi)| & \lesssim \frac{\delta(s,\xi)}{\delta(t,\xi)}\big( |\xi|^{|\beta|-1}|\hat{u}(s,\xi)| + |\xi|^{|\beta|-1}|\hat{u}_t(s,\xi)| \big) \quad \mbox{for} \quad |\beta|\geq 1, \\
|\xi|^{|\beta|}|\hat{u}_t(t,\xi)| & \lesssim \frac{\delta(s,\xi)}{\delta(t,\xi)}\big( |\xi|^{|\beta|}|\hat{u}(s,\xi)| + |\xi|^{|\beta|}|\hat{u}_t(s,\xi)| \big) \quad \mbox{for} \quad |\beta|\geq 0,
\end{align*}
where $\delta=\delta(t,\xi)$ is given in \eqref{Eq:Noneffective-Integrable-Def-delta-Zdiss}.
\end{corollary}
\subsubsection{Conclusion} \label{Sec:Noneffective-Integrable-Conclusions}
From the statements of Corollaries \ref{Cor:Noneffective-Integrable-Zell}, \ref{Cor:Noneffective-Integrable-Zred}, \ref{Cor:Noneffective-Integrable-Zhyp} and \ref{Cor:Noneffective-Integrable-Zdiss} we derive our desired statements.
\medskip

\noindent\textbf{Small frequencies:}\\
\noindent \textit{Case 1:} $t\leq t_{\xi,1}.$ Due to Corollary \ref{Cor:Noneffective-Integrable-Zdiss}, we have
\begin{align*}
|\xi|^{|\beta|}|\hat{u}(t,\xi)| & \lesssim \exp\Big( -\int_{0}^{t}\big( b(\tau)+g(\tau)|\xi|^2 \big)d\tau \Big)\big( |\xi|^{|\beta|-1}|\hat{u}_0(\xi)| + |\xi|^{|\beta|-1}|\hat{u}_1(\xi)| \big) \quad \mbox{for} \quad |\beta|\geq 1, \\
|\xi|^{|\beta|}|\hat{u}_t(t,\xi)| & \lesssim \exp\Big( -\int_{0}^{t}\big( b(\tau)+g(\tau)|\xi|^2 \big)d\tau \Big)\big( |\xi|^{|\beta|}|\hat{u}_0(\xi)| + |\xi|^{|\beta|}|\hat{u}_1(\xi)| \big) \quad \mbox{for} \quad |\beta|\geq 0.
\end{align*}
\noindent \textit{Case 2:} $t_{\xi,1}\leq t$. In this case we apply Corollary \ref{Cor:Noneffective-Integrable-Zhyp} and then, use the estimates from \textit{Case 1}. Then, we get
\begin{align*}
|\xi|^{|\beta|}|\hat{u}(t,\xi)| & \leq \exp\Big( -\frac{2-\varepsilon}{4}\int_{t_{\xi,1}}^t\big( b(\tau)+g(\tau)|\xi|^2 \big)d\tau \Big)\big( |\xi|^{|\beta|}|\hat{u}(t_{\xi,1},\xi)| + |\xi|^{|\beta|-1}|\hat{u}_t(t_{\xi,1},\xi)| \big) \\
& \lesssim \exp\Big( -\frac{2-\varepsilon}{4}\int_{t_{\xi,1}}^t\big( b(\tau)+g(\tau)|\xi|^2 \big)d\tau \Big)\exp\Big( -\int_{0}^{t_{\xi,1}}\big( b(\tau)+g(\tau)|\xi|^2 \big)d\tau \Big) \\
& \qquad \times \big( |\xi|^{|\beta|-1}|\hat{u}_0(\xi)| + |\xi|^{|\beta|-1}|\hat{u}_1(\xi)| \big) \\
& \lesssim \exp\Big( - \frac{2-\varepsilon}{4}\int_{0}^t\big( b(\tau)+g(\tau)|\xi|^2 \big)d\tau \Big)\big( |\xi|^{|\beta|-1}|\hat{u}_0(\xi)| + |\xi|^{|\beta|-1}|\hat{u}_1(\xi)| \big), \\
|\xi|^{|\beta|}|\hat{u}_t(t,\xi)| & \leq \exp\Big( -\frac{2-\varepsilon}{4}\int_{t_{\xi,1}}^t\big( b(\tau)+g(\tau)|\xi|^2 \big)d\tau \Big)\big( |\xi|^{|\beta|+1}|\hat{u}(t_{\xi,1},\xi)| + |\xi|^{|\beta|}|\hat{u}_t(t_{\xi,1},\xi)| \\
& \lesssim \exp\Big( -\frac{2-\varepsilon}{4}\int_{t_{\xi,1}}^t\big( b(\tau)+g(\tau)|\xi|^2 \big)d\tau \Big)\exp\Big( -\int_{0}^{t_{\xi,1}}\big( b(\tau)+g(\tau)|\xi|^2 \big)d\tau \Big) \\
& \qquad \times \big( |\xi|^{|\beta|}|\hat{u}_0(\xi)| + |\xi|^{|\beta|}|\hat{u}_1(\xi)| \big) \\
& \lesssim \exp\Big( -\frac{2-\varepsilon}{4}\int_{0}^t\big( b(\tau)+g(\tau)|\xi|^2 \big)d\tau \Big)\big( |\xi|^{|\beta|}|\hat{u}_0(\xi)| + |\xi|^{|\beta|}|\hat{u}_1(\xi)| \big).
\end{align*}
\noindent\textbf{Large frequencies:}\\
\noindent \textit{Case 1:} $t\leq t_{\xi,3}$. In this case we apply Corollary \ref{Cor:Noneffective-Integrable-Zell}. It holds
\begin{align*}
|\xi|^{|\beta|}|\hat{u}(t,\xi)| & \lesssim \big( g(t) \big)^{-\kappa-\frac{a}{2}}\exp\Big( -\frac{1}{2}\int_0^tb(\tau)d\tau \Big)\big( |\xi|^{|\beta|}|\hat{u}_0(\xi)| + |\xi|^{|\beta|-2}|\hat{u}_1(\xi)| \big) \quad \mbox{for} \quad |\beta| \geq 2, \\
|\xi|^{|\beta|}|\hat{u}_t(t,\xi)| & \lesssim \big( g(t) \big)^{1-\kappa-\frac{a}{2}}\exp\Big( -\frac{1}{2}\int_0^tb(\tau)d\tau \Big)\big( |\xi|^{|\beta|+2}|\hat{u}_0(\xi)| + |\xi|^{|\beta|}|\hat{u}_1(\xi)| \big) \quad \mbox{for} \quad |\beta| \geq 0.
\end{align*}
\noindent \textit{Case 2:} $t_{\xi,3}\leq t\leq t_{\xi,2}$. In this case, we start with Corollary \ref{Cor:Noneffective-Integrable-Zred} and use the estimates from \textit{Case 1} with $g(t_{\xi,3})|\xi| = N_2$. Thus, we have
\begin{align*}
|\xi|^{|\beta|}|\hat{u}(t,\xi)| & \lesssim |\xi|^{|\beta|}|\hat{u}(t_{\xi,3},\xi)| + |\xi|^{|\beta|-1}|\hat{u}_t(t_{\xi,3},\xi)| \\
& \lesssim \big( g(t_{\xi,3}) \big)^{-\kappa-\frac{a}{2}}\exp\Big( -\frac{1}{2}\int_0^{t}b(\tau)d\tau \Big)\exp\Big(\frac{1}{2}\int^t_{t_{\xi,3}}b(\tau)d\tau \Big)\big( |\xi|^{|\beta|}|\hat{u}_0(\xi)| + |\xi|^{|\beta|-2}|\hat{u}_1(\xi)| \big) \\
& \qquad + \big( g(t_{\xi,3}) \big)^{-\kappa-\frac{a}{2}+1}\exp\Big( -\frac{1}{2}\int_0^{t}b(\tau)d\tau \Big)\exp\Big(\frac{1}{2}\int^t_{t_{\xi,3}}b(\tau)d\tau \Big)\big( |\xi|^{|\beta|+1}|\hat{u}_0(\xi)| + |\xi|^{|\beta|-1}|\hat{u}_1(\xi)| \big) \\
& \lesssim \exp\Big( -\frac{1}{2}\int_0^{t}b(\tau)d\tau \Big)\big( |\xi|^{|\beta|+\kappa+\frac{a}{2}}|\hat{u}_0(\xi)| + |\xi|^{|\beta|+\kappa+\frac{a}{2}-2}|\hat{u}_1(\xi)| \big).
\end{align*}
Analogously, \begin{align*}
|\xi|^{|\beta|}|\hat{u}_t(t,\xi)| &\lesssim |\xi|^{|\beta|+1}|\hat{u}(t_{\xi,3},\xi)| + |\xi|^{|\beta|}|\hat{u}_t(t_{\xi,3},\xi)| \\
& \lesssim \big( g(t_{\xi,3}) \big)^{-\kappa-\frac{a}{2}}\exp\Big( -\frac{1}{2}\int_0^{t}b(\tau)d\tau \Big)\exp\Big(\frac{1}{2}\int^t_{t_{\xi,3}}b(\tau)d\tau \Big)\big( |\xi|^{|\beta|+1}|\hat{u}_0(\xi)| + |\xi|^{|\beta|-1}|\hat{u}_1(\xi)| \big) \\
& \qquad + \big( g(t_{\xi,3}) \big)^{1-\kappa-\frac{a}{2}}\exp\Big( -\frac{1}{2}\int_0^{t}b(\tau)d\tau \Big)\exp\Big(\frac{1}{2}\int^t_{t_{\xi,3}}b(\tau)d\tau \Big)\big( |\xi|^{|\beta|+2}|\hat{u}_0(\xi)| + |\xi|^{|\beta|}|\hat{u}_1(\xi)| \big) \\
& \lesssim \exp\Big( -\frac{1}{2}\int_0^{t}b(\tau)d\tau \Big)\big( |\xi|^{|\beta|+\kappa+\frac{a}{2}+1}|\hat{u}_0(\xi)| + |\xi|^{|\beta|+\kappa+\frac{a}{2}-1}|\hat{u}_1(\xi)| \big).
\end{align*}
Here we used \textbf{(G4)} to get for all $t \in [t_{\xi,3},t_{\xi,2}]$ the uniform estimate
\begin{align*}
\exp\Big(\frac{1}{2}\int^t_{t_{\xi,3}}b(\tau)d\tau \Big)\leq \exp\Big(-\frac{a}{2}\int^{t_{\xi,2}}_{t_{\xi,3}}\frac{g'(\tau)}{g(\tau)}d\tau \Big)
=\Big(\frac{g(t_{\xi,3})}{g(t_{\xi,2})}\Big)^{\frac{a}{2}}=\Big(\frac{g(t_{\xi,3})|\xi|}{g(t_{\xi,2})|\xi|}\Big)^{\frac{a}{2}}=
\Big(\frac{N_2}{\varepsilon}\Big)^{\frac{a}{2}}.
\end{align*}
\noindent \textit{Case 3:} $t_{\xi,2}\leq t$. In this case, we start with Corollary \ref{Cor:Noneffective-Integrable-Zhyp} and use the estimates from \textit{Case 2} as follows:
\begin{align*}
|\xi|^{|\beta|}|\hat{u}(t,\xi)| & \leq \exp\Big( -\frac{2-\varepsilon}{4}\int_{t_{\xi,2}}^t\big( b(\tau)+g(\tau)|\xi|^2 \big)d\tau \Big)\big( |\xi|^{|\beta|}|\hat{u}(t_{\xi,2},\xi)| + |\xi|^{|\beta|-1}|\hat{u}_t(t_{\xi,2},\xi)| \big) \\
& \leq \exp\Big( -\frac{2-\varepsilon}{4}\int_{t_{\xi,2}}^t\big( b(\tau)+g(\tau)|\xi|^2 \big)d\tau \Big)\exp\Big( -\frac{1}{2}\int_0^{t_{\xi,2}}b(\tau)d\tau \Big)\\
& \qquad \times \big( |\xi|^{|\beta|+\kappa+\frac{a}{2}}|\hat{u}_0(\xi)| + |\xi|^{|\beta|+\kappa+\frac{a}{2}-2}|\hat{u}_1(\xi)| \big), \\
|\xi|^{|\beta|}|\hat{u}_t(t,\xi)| & \leq \exp\Big( -\frac{2-\varepsilon}{4}\int_{t_{\xi,2}}^t\big( b(\tau) + g(\tau)|\xi|^2\big)d\tau \Big)\big( |\xi|^{|\beta|+1}|\hat{u}(t_{\xi,2},\xi)| + |\xi|^{|\beta|}|\hat{u}_t(t_{\xi,2},\xi)| \big) \\
& \leq \exp\Big( -\frac{2-\varepsilon}{4}\int_{t_{\xi,2}}^t\big( b(\tau) + g(\tau)|\xi|^2\big)d\tau \Big)\exp\Big( -\frac{1}{2}\int_0^{t_{\xi,2}}b(\tau)d\tau \Big)\\
& \qquad \times \big( |\xi|^{|\beta|+\kappa+\frac{a}{2}+1}|\hat{u}_0(\xi)| + |\xi|^{|\beta|+\kappa+\frac{a}{2}-1}|\hat{u}_1(\xi)| \big).
\end{align*}
Thus, the proof of Theorem \ref{Theorem_Noneffective_Integrable-Decaying} is completed.
\end{proof}

\section{The friction term is effective} \label{Section_Effective}
In this section, we study our Cauchy problem \eqref{MainEquation} with effective damping $b(t)u_t$. More specifically, we assume that the damping term $b(t)u_t$ satisfies the so-called \textit{effective} assumptions in the following definition according to the Wirth's classification given in \cite {WirthThesis, Wirth-Effective=2007}.
\begin{definition}[Effective dissipation] \label{DefinitionEffective}
If the strictly positive function $b=b(t)$ satisfies
\begin{itemize}
\item[\textbf{(B1)}] $b\in\mathcal{C}^2([0,\infty))$,
\item[\textbf{(B2)}] $b'(t)$ does not change its sign and $tb(t)\to\infty$ as $t\to \infty$,
\item[\textbf{(B3)}] $\dfrac{|b^{(k)}(t)|}{b(t)} \lesssim \dfrac{1}{(1+t)^{k}}$ for $k=1,2$,
\end{itemize}
then the damping term $b(t)u_t$ is called effective.
%\begin{align*}
%&\textbf{(B1)}\quad b\in\mathcal{C}^3([0,\infty)), &\textbf{(B4)}\quad &\dfrac{1}{b(t)} \notin L^{1}([0,\infty)), \\
%&\textbf{(B2)}\quad b'(t) \text{ does not change its sign and }t b(t)\rightarrow \infty \text{ as }t\to \infty, &\textbf{(B5)}\quad &\big((1+t)^{2}b(t)\big)^{-1} \in L^{1}([0,\infty)). \\
%&\textbf{(B3)}\quad \dfrac{|b^{(k)}(t)|}{b(t)} \lesssim \dfrac{1}{(1+t)^{k}}\,\,\,\,\mbox{for}\,\,\,\,k=1,2,3,
%\end{align*}
%then the damping term $b(t)u_t$ is called effective.
\end{definition}
\noindent Moreover, we assume the following condition for $b=b(t)$:
\begin{enumerate}
\item[\textbf{(EF)}] $|b'(t)| \leq ab^2(t)$, where $a<1$.
\end{enumerate}
We will consider our equation \eqref{AuxiliaryEquation3} (namely, the equation after the second change of variables) in the following form:
\begin{equation} \label{Eq:EffectiveCase}
D_t^2w + \Big( \frac{b(t)}{2}+\frac{g(t)|\xi|^2}{2} \Big)^2w - |\xi|^2w + \frac{g'(t)}{2}|\xi|^2w + \frac{b'(t)}{2}w = 0.
\end{equation}
\subsection{Model with increasing time-dependent coefficient $g=g(t)$} \label{Subsect:Effective-Increasing}
We assume the following conditions for the coefficient $g=g(t)$ for all $t \in [0,\infty)$:
\begin{enumerate}
\item[\textbf{(A1)}] $g(t)>0$ and $g'(t)>0$,
\item[\textbf{(A2)}] $\dfrac{1}{g} \in L^1([0,\infty))$,
\item[\textbf{(A3)}] $|d_t^kg(t)|\leq C_kg(t)\Big( \dfrac{g(t)}{G(t)} \Big)^k$ for $k=1,2$, where $G(t):=\dfrac{1}{2}\displaystyle\int_0^t g(\tau)d\tau$ and $C_1$, $C_2$ are positive constants.
\end{enumerate}
\begin{theorem} \label{Theorem:Effective_Increasing}
We consider the Cauchy problem
\begin{equation*}
\begin{cases}
u_{tt}- \Delta u + b(t)u_t -g(t)\Delta u_t=0, &(t,x) \in [0,\infty) \times \mathbb{R}^n, \\
u(0,x)= u_0(x),\quad u_t(0,x)= u_1(x), &x \in \mathbb{R}^n.
\end{cases}
\end{equation*}
Let us assume that the coefficient $g=g(t)$ satisfies the conditions \textbf{(A1)} to \textbf{(A3)} and the coefficient $b=b(t)$ satisfies \textbf{(B1)} to \textbf{(B3)} and \textbf{(EF)}. Then, we have the following estimates for Sobolev solutions with $|\beta|\geq 0$:
\begin{align*}
\|\,|D|^{|\beta|} u(t,\cdot)\|_{L^2} & \lesssim  \|u_0\|_{\dot{H}^{|\beta|}} + \|\langle D \rangle^{-2} u_1\|_{\dot{H}^{|\beta|}}.
\end{align*}
If $b'\geq0$, then we have
\begin{align*}
\|\,|D|^{|\beta|} u_t(t,\cdot)\|_{L^2} & \lesssim g(t)\big( \|u_0\|_{\dot{H}^{|\beta|+2}} + \|\langle D \rangle^{-2}u_1\|_{\dot{H}^{|\beta|+2}} \big) + b(t)\big( \|u_0\|_{\dot{H}^{|\beta|}} + \|\langle D \rangle^{-2}u_1\|_{\dot{H}^{|\beta|}} \big).
\end{align*}
If $b'<0$, then we have
\begin{align*}
\|\,|D|^{|\beta|} u_t(t,\cdot)\|_{L^2} & \lesssim g(t)\big( \|u_0\|_{\dot{H}^{|\beta|+2}} + \|u_1\|_{\dot{H}^{|\beta|}} \big) + b(t)\big( \|u_0\|_{\dot{H}^{|\beta|}} + \||D|^{-2}u_1\|_{\dot{H}^{|\beta|}} \big).
\end{align*}
\end{theorem}
\begin{proof}
We start by introducing the elliptic zone $\Zell(N)$ and the pseudo-differential zone $\Zpd(N)$ as follows:
\begin{align*}
\Zell(N) &= \left\{ (t,\xi)\in[0,\infty)\times\mathbb{R}^n : G(t)|\xi|^2 \geq N \right\}, \\
\Zpd(N) &= \left\{ (t,\xi)\in[0,\infty)\times\mathbb{R}^n : G(t)|\xi|^2 \leq N \right\},
\end{align*}
where $G=G(t)$ is defined in condition \textbf{(A3)}. Moreover, the separating line between these two zones is given by
\[ t_\xi=\left\{ (t,\xi) \in [0,\infty) \times \mathbb{R}^n: G(t)|\xi|^2 = N \right\}. \]
\subsubsection{Considerations in the elliptic zone $\Zell(N)$} \label{Subsection_Effective_Increasing_Zell}
We introduce the following family of symbol classes in the elliptic zone $\Zell(N)$.
\begin{definition} \label{Def:Effective-increasing-symbol}
A function $f=f(t,\xi)$ belongs to the elliptic symbol class $S_{\text{ell}}^\ell\{m_1,m_2\}$ if it holds
\begin{equation*}
|D_t^kf(t,\xi)|\leq C_{k}\big( b(t)+g(t)|\xi|^2 \big)^{m_1}\Big( \frac{b'(t)+g'(t)|\xi|^2}{b(t)+g(t)|\xi|^2} \Big)^{m_2+k}
\end{equation*}
for all $(t,\xi)\in \Zell(N)$ and all $k\leq \ell$.
\end{definition}
Thus, we may conclude the following rules from the definition of the symbol classes.
\begin{proposition} \label{Prop:Effective-symbol}
The following statements are true:
\begin{itemize}
\item $S_{\text{ell}}^\ell\{m_1,m_2\}$ is a vector space for all nonnegative integers $\ell$;
\item $S_{\text{ell}}^\ell\{m_1,m_2\}\cdot S_{\text{ell}}^{\ell}\{m_1',m_2'\}\hookrightarrow S^{\ell}_{\text{ell}}\{m_1+m_1',m_2+m_2'\}$;
\item $D_t^kS_{\text{ell}}^\ell\{m_1,m_2\}\hookrightarrow S_{\text{ell}}^{\ell-k}\{m_1,m_2+k\}$
for all nonnegative integers $\ell$ with $k\leq \ell$;
\item $S_{\text{ell}}^{0}\{-1,2\}\hookrightarrow L_{\xi}^{\infty}L_t^1\big( \Zell(N) \big)$.
\end{itemize}
\end{proposition}
\begin{proof}
Let us verify the last statement. Indeed, if $f=f(t,\xi)\in S_{\text{ell}}^{0}\{-1,2\}$, then using the estimates
\begin{align*}
\Big| \frac{( b'(t)+g'(t)|\xi|^2)^2}{(b(t)+g(t)|\xi|^2)^3} \Big| &\lesssim \frac{b'(t)^2+|b'(t)|g'(t)|\xi|^2+g'(t)^2|\xi|^4}{(b(t)+g(t)|\xi|^2)^3}  \\
& \lesssim \frac{G(t)^2 b'(t)^2+G(t)|b'(t)|\, g(t)^2|\xi|^2+g(t)^4|\xi|^4}{G(t)^2(b(t)+g(t)|\xi|^2)^3} \qquad \small{(\text{we used condition \textbf{(A3)}})} \\
& \lesssim \frac{(b'(t))^2}{b(t)^3} + \frac{|b'(t)|g(t)^2|\xi|^2}{G(t)(g(t)|\xi|^2)^2 b(t)} + \frac{g(t)^4|\xi|^4}{G(t)^2(g(t)|\xi|^2)^3} \\
& \lesssim \frac{b^2(t)}{(1+t)^2}\frac{1}{b^3(t)} + \frac{b(t)}{1+t}\frac{1}{G(t)|\xi|^2b(t)} + \frac{g(t)}{G(t)^2|\xi|^2} \qquad \small{(\text{we used condition \textbf{(B3)}})} \\
& = \frac{1}{b(t)(1+t)^2} + \frac{1}{(1+t)G(t)|\xi|^2} + \frac{g(t)}{G(t)^2|\xi|^2} \\
& \lesssim \frac{g(t)}{G(t)^2|\xi|^2},  \qquad \small{(\text{we used condition \textbf{(B2)} and the definition of}\,\, \Zell(N))},
\end{align*}
we get
\begin{align*}
\int_{t_\xi}^{\infty}|f(\tau,\xi)|d\tau  & \lesssim \frac{1}{|\xi|^2}\int_{t_\xi}^{\infty}\frac{g(\tau)}{G(\tau)^2} d\tau = -\frac{1}{|\xi|^2}\frac{1}{G(\tau)}\Big|_{t_\xi}^{\infty} \leq \frac{C}{G(t_{\xi})|\xi|^2} = \frac{C}{N},
\end{align*}
where we used the definition of the separating line $t_\xi$.
\end{proof}
Taking account of \eqref{Eq:EffectiveCase},
we choose the following micro-energy:
\[ W=W(t,\xi) := \Big[ \big( \frac{b(t)}{2}+\frac{g(t)}{2}|\xi|^2 \big)w,D_tw \Big]^{\text{T}}. \]
Then, by \eqref{Eq:EffectiveCase} we obtain that $W=W(t,\xi)$ satisfies the following system of first order:
\begin{equation*} \label{System_Effective_Increasing}
D_tW=\underbrace{\left[ \left( \begin{array}{cc}
0 & \dfrac{b(t)}{2}+\dfrac{g(t)}{2}|\xi|^2 \\
-\dfrac{b(t)}{2}-\dfrac{g(t)}{2}|\xi|^2 & 0
\end{array} \right) + \left( \begin{array}{cc}
\dfrac{D_t(b(t)+g(t)|\xi|^2)}{b(t)+g(t)|\xi|^2} & 0 \\
-\dfrac{b'(t)}{b(t)+g(t)|\xi|^2}-\dfrac{(g'(t)-2)|\xi|^2}{b(t)+g(t)|\xi|^2} & 0
\end{array} \right)\right]}_{A_W}W.
\end{equation*}
We want to estimate the fundamental solution $E_W=E_W(t,s,\xi)$ to the above system, namely, the solution to
\begin{equation*}
D_tE_W(t,s,\xi)=A_W(t,\xi)E_W(t,s,\xi), \quad E_W(s,s,\xi)=I \quad \mbox{for any} \quad t\geq s\geq t_\xi.
\end{equation*}
We denote by $M$ the matrix consisting of eigenvectors of the first matrix on the right-hand side and its inverse matrix
\[ M = \left( \begin{array}{cc}
i & -i \\
1 & 1
\end{array} \right), \qquad M^{-1}=\frac{1}{2}\left( \begin{array}{cc}
-i & 1 \\
i & 1
\end{array} \right). \]
Then, defining $W^{(0)}:=M^{-1}W$ we get the system
\begin{equation*}
D_tW^{(0)}=\big( \mathcal{D}(t,\xi)+\mathcal{R}(t,\xi) \big)W^{(0)},
\end{equation*}
where
\begin{align*}
\mathcal{D}(t,\xi) &= \left( \begin{array}{cc}
-i\Big( \dfrac{b(t)}{2}+\dfrac{g(t)}{2}|\xi|^2 \Big) & 0 \\
0 & i\Big( \dfrac{b(t)}{2}+\dfrac{g(t)}{2}|\xi|^2 \Big)
\end{array} \right), \\
\mathcal{R}_1(t,\xi) &= \frac{1}{2} \left( \begin{array}{cc}
\dfrac{D_t(b(t)+g(t)|\xi|^2)}{b(t)+g(t)|\xi|^2}-i\dfrac{(g'(t)-2)|\xi|^2}{b(t)+g(t)|\xi|^2} & -\dfrac{D_t(b(t)+g(t)|\xi|^2)}{b(t)+g(t)|\xi|^2}+i\dfrac{(g'(t)-2)|\xi|^2}{b(t)+g(t)|\xi|^2} \\
-\dfrac{D_t(b(t)+g(t)|\xi|^2)}{b(t)+g(t)|\xi|^2}-i\dfrac{(g'(t)-2)|\xi|^2}{b(t)+g(t)|\xi|^2} & \dfrac{D_t(b(t)+g(t)|\xi|^2)}{b(t)+g(t)|\xi|^2}+i\dfrac{(g'(t)-2)|\xi|^2}{b(t)+g(t)|\xi|^2}
\end{array} \right), \\
\mathcal{R}_2(t,\xi) &= \frac{1}{2} \left( \begin{array}{cc}
-i\dfrac{b'(t)}{b(t)+g(t)|\xi|^2} & i\dfrac{b'(t)}{b(t)+g(t)|\xi|^2} \\
-i\dfrac{b'(t)}{b(t)+g(t)|\xi|^2} & i\dfrac{b'(t)}{b(t)+g(t)|\xi|^2}
\end{array} \right),
\end{align*}
where $\mathcal{D}(t,\xi)\in S_{\text{ell}}^2\{1,0\}$ and with $\mathcal{R}:=\mathcal{R}_1+\mathcal{R}_2$, $\mathcal{R}(t,\xi)\in S_{\text{ell}}^1\{0,1\}$.
\medskip

We perform one more step of diagonalization procedure. We define $F_0(t,\xi):=\diag\mathcal{R}(t,\xi)$. The difference of the diagonal entries of the matrix $\mathcal{D}(t,\xi)+F_0(t,\xi)$ is
\begin{align*}
b(t)&+g(t)|\xi|^2 + \frac{b'(t)+(g'(t)-2)|\xi|^2}{b(t)+g(t)|\xi|^2} \\
& \leq \frac{b(t)^2+2b(t)g(t)|\xi|^2+g(t)^2|\xi|^4+b'(t)+g'(t)|\xi|^2}{b(t)+g(t)|\xi|^2} \\
& \leq \frac{(b(t)+g(t)|\xi|^2)^2+b(t)^2+\frac{g(t)^2|\xi|^4}{G(t)|\xi|^2}}{b(t)+g(t)|\xi|^2} \qquad \small{(\text{we used condition \textbf{(A3)}})} \\
& \lesssim \frac{(b(t)+g(t)|\xi|^2)^2+b(t)^2+g(t)^2|\xi|^4}{b(t)+g(t)|\xi|^2} \leq b(t)+g(t)|\xi|^2=:i\delta(t,\xi),
\end{align*}
where we used the fact that $b'(t)=o(b^2(t))$ and for $t \to \infty$. \\

Now we introduce the matrix $N^{(1)}=N^{(1)}(t,\xi)$ defined by
\begin{align*}
&N^{(1)}(t,\xi) := \left( \begin{array}{cc}
0 & -\dfrac{\mathcal{R}_{12}}{\delta(t,\xi)} \\
\dfrac{\mathcal{R}_{21}}{\delta(t,\xi)} & 0
\end{array} \right) \\
& = \left( \begin{array}{cc}
0 & i\dfrac{D_t(b(t)+g(t)|\xi|^2)}{2(b(t)+g(t)|\xi|^2)^2}-\dfrac{b'(t)+(g'(t)-2)|\xi|^2}{2(b(t)+g(t)|\xi|^2)^2} \\
i\dfrac{D_t(b(t)+g(t)|\xi|^2)}{2(b(t)+g(t)|\xi|^2)^2}+\dfrac{b'(t)+(g'(t)-2)|\xi|^2}{2(b(t)+g(t)|\xi|^2)^2} & 0
\end{array} \right),
\end{align*}
where $N^{(1)}(t,\xi)\in S_{\text{ell}}^1\{-1,1\}$. For a sufficiently large zone constant $N$ and all $t\geq t_\xi$ the matrix $N_1(t,\xi) := I + N^{(1)}(t,\xi)$ belongs to $S_{\text{ell}}^1\{0,0\}$ and is invertible with uniformly bounded inverse matrix $N_1^{-1}=N_1^{-1}(t,\xi)$. Namely, by using condition \textbf{(A3)} for $g=g(t)$ and condition \textbf{(EF)} for $b=b(t)$, we have the following estimates:
\begin{align*}
\Big|\dfrac{D_t(b(t)+g(t)|\xi|^2)}{(b(t)+g(t)|\xi|^2)^2}\Big| \leq \dfrac{|b'(t)|+C_1\frac{g(t)^2}{G(t)}|\xi|^2}{(b(t)+g(t)|\xi|^2)^2} &\leq \dfrac{G(t)|b'(t)|+C_1g(t)^2|\xi|^2}{G(t)(b(t)+g(t)|\xi|^2)^2} \\
& \leq \frac{|b'(t)|}{b(t)^2}+\frac{C_1}{G(t)|\xi|^2} \leq a+\frac{C}{N}<1,
\end{align*}
with sufficiently large zone constant $N$.

Let
\begin{align*}
B^{(1)}(t,\xi) &= D_tN^{(1)}(t,\xi)-( \mathcal{R}(t,\xi)-F_0(t,\xi))N^{(1)}(t,\xi), \\
\mathcal{R}_3(t,\xi) &= -N_1^{-1}(t,\xi)B^{(1)}(t,\xi)\in S_{\text{ell}}^0\{-1,2\}.
\end{align*}
Then, we have the following operator identity:
\begin{equation*}
( D_t-\mathcal{D}(t,\xi)-\mathcal{R}(t,\xi))N_1(t,\xi)=N_1(t,\xi)( D_t-\mathcal{D}(t,\xi)-F_0(t,\xi)-\mathcal{R}_3(t,\xi)).
\end{equation*}
\begin{proposition} \label{Prop_Effective_Increasing_Est_Zell}
The fundamental solution $E_{\text{ell}}^{W}=E_{\text{ell}}^{W}(t,s,\xi)$ to the transformed operator
\[ D_t-\mathcal{D}(t,\xi)-F_0(t,\xi)-\mathcal{R}_3(t,\xi) \]
can be estimated by
\begin{equation*}
\big( |E_{\text{ell}}^{W}(t,s,\xi)| \big) \lesssim \frac{b(t)+g(t)|\xi|^2}{b(s)+g(s)|\xi|^2}\exp\bigg( \frac{1}{2}\int_{s}^{t}\big( b(\tau)+g(\tau)|\xi|^2 \big)d\tau \bigg)
\left( \begin{array}{cc}
1 & 1 \\
1 & 1
\end{array} \right),
\end{equation*}
with $(t,\xi),(s,\xi)\in \Zell(N)$, $t_\xi\leq s\leq t$.
\end{proposition}
\begin{proof}
We transform the system for $E_{\text{ell}}^{W}=E_{\text{ell}}^{W}(t,s,\xi)$ to an integral equation for a new matrix-valued function $\mathcal{Q}_{\text{ell}}=\mathcal{Q}_{\text{ell}}(t,s,\xi)$. If we differentiate the term
\[ \exp \bigg\{ -i\int_{s}^{t}\big( \mathcal{D}(\tau,\xi)+F_0(\tau,\xi) \big)d\tau \bigg\}E_{\text{ell}}^{W}(t,s,\xi), \]
and then, integrate on $[s,t]$, we obtain that $E_{\text{ell}}^{W}=E_{\text{ell}}^{W}(t,s,\xi)$ satisfies the following integral equation:
\begin{align*}
E_{\text{ell}}^{W}(t,s,\xi) & = \exp\bigg\{ i\int_{s}^{t}\big( \mathcal{D}(\tau,\xi)+F_0(\tau,\xi) \big)d\tau \bigg\}E_{\text{ell}}^{W}(s,s,\xi)\\
& \quad + i\int_{s}^{t} \exp \bigg\{ i\int_{\theta}^{t}\big( \mathcal{D}(\tau,\xi)+F_0(\tau,\xi) \big)d\tau \bigg\}\mathcal{R}_3(\theta,\xi)E_{\text{ell}}^{W}(\theta,s,\xi)\,d\theta.
\end{align*}
Let us define
\[ \mathcal{Q}_{\text{ell}}(t,s,\xi)=\exp\bigg\{ -\int_{s}^{t}\beta(\tau,\xi)d\tau \bigg\} E_{\text{ell}}^{W}(t,s,\xi), \]
with a suitable $\beta=\beta(t,\xi)$ which will be fixed later. It satisfies the new integral equation
\begin{align*}
\mathcal{Q}_{\text{ell}}(t,s,\xi)=&\exp \bigg\{ \int_{s}^{t}\big( i\mathcal{D}(\tau,\xi)+iF_0(\tau,\xi)-\beta(\tau,\xi)I \big)d\tau \bigg\}\\
& \quad + \int_{s}^{t} \exp \bigg\{ \int_{\theta}^{t}\big( i\mathcal{D}(\tau,\xi)+iF_0(\tau,\xi)-\beta(\tau,\xi)I \big)d\tau \bigg\}\mathcal{R}_3(\theta,\xi)\mathcal{Q}_{\text{ell}}(\theta,s,\xi)\,d\theta.
\end{align*}
The function $\mathcal{R}_3=\mathcal{R}_3(\theta,\xi)\in S_{\text{ell}}^0\{-1,2\}$ is uniformly integrable over the elliptic zone because of the last property of Proposition \ref{Prop:Effective-symbol}. Hence, if the exponential term is bounded, then the solution $\mathcal{Q}_{\text{ell}}=\mathcal{Q}_{\text{ell}}(t,s,\xi)$ of the integral equation is uniformly bounded over the elliptic zone for a suitable weight $\beta=\beta(t,\xi)$.\\
The main entries of the diagonal matrix $i\mathcal{D}(t,\xi)+iF_0(t,\xi)$ are given by
\begin{align*}
(I) &= \dfrac{b(t)}{2} + \dfrac{g(t)|\xi|^2}{2} + \dfrac{b'(t)+g'(t)|\xi|^2}{2(b(t)+g(t)|\xi|^2)}+\dfrac{b'(t)+(g'(t)-2)|\xi|^2}{2(b(t)+g(t)|\xi|^2)},\\
(II) &= -\dfrac{b(t)}{2} - \dfrac{g(t)|\xi|^2}{2} + \dfrac{b'(t)+g'(t)|\xi|^2}{2(b(t)+g(t)|\xi|^2)}-\dfrac{b'(t)+(g'(t)-2)|\xi|^2}{2(b(t)+g(t)|\xi|^2)}.
\end{align*}
From the difference of $(II)-(I)$, we may see that the term $(I)$ is dominant in $\Zell(N)$ with respect to (II) for $t\geq t_{\xi}$, we need this dominance for $t \geq t_\xi$. Therefore, we choose the weight $\beta=\beta(t,\xi)=(I)$. By this choice, we get
\[ i\mathcal{D}(\tau,\xi)+iF_0(\tau,\xi)-\beta(\tau,\xi)I = \left( \begin{array}{cc}
0 & 0 \\
0 & -\big( b(t)+g(t)|\xi|^2 \big)-\dfrac{b'(t)+(g'(t)-2)|\xi|^2}{b(t)+g(t)|\xi|^2}
\end{array} \right). \]
It follows
\begin{align*}
H(t,s,\xi) & =\exp \bigg\{ \int_{s}^{t}\big( i\mathcal{D}(\tau,\xi)+iF_0(\tau,\xi)-\beta(\tau,\xi)I \big)d\tau \bigg\}\\
& = \diag \bigg( 1, \exp \bigg\{ \int_{s}^{t}\Big( -\big( b(t)+g(t)|\xi|^2 \big)-\dfrac{b'(t)+(g'(t)-2)|\xi|^2}{b(t)+g(t)|\xi|^2} \Big)d\tau \bigg\} \bigg)\rightarrow \left( \begin{array}{cc}
1 & 0 \\
0 & 0
\end{array} \right)
\end{align*}
as $t\rightarrow \infty$ for any fixed $s\geq t_\xi$. Hence, the matrix $H=H(t,s,\xi)$ is uniformly bounded for $(s,\xi),(t,\xi)\in \Zell(N)$. So, the representation of $\mathcal{Q}_{\text{ell}}=\mathcal{Q}_{\text{ell}}(t,s,\xi)$ by a Neumann series gives
\begin{align*}
\mathcal{Q}_{\text{ell}}(t,s,\xi)=H(t,s,\xi)+\sum_{k=1}^{\infty}i^k\int_{s}^{t}H(t,t_1,\xi)\mathcal{R}_3(t_1,\xi)&\int_{s}^{t_1}H(t_1,t_2,\xi)\mathcal{R}_3(t_2,\xi) \\
& \cdots \int_{s}^{t_{k-1}}H(t_{k-1},t_k,\xi)\mathcal{R}_3(t_k,\xi)dt_k\cdots dt_2dt_1.
\end{align*}
Then, this series is convergent, since $\mathcal{R}_3=\mathcal{R}_3(t,\xi)$ is uniformly integrable over $\Zell(N)$. Hence, from the last considerations we may conclude
\begin{align*}
E_{\text{ell}}^{W}(t,s,\xi)&=\exp \bigg\{ \int_{s}^{t}\beta(\tau,\xi)d\tau \bigg\}\mathcal{Q}_{\text{ell}}(t,s,\xi)\\
& = \exp \bigg\{ \int_{s}^{t}\bigg( \frac{b(\tau)}{2}+\frac{g(\tau)|\xi|^2}{2}+\frac{b'(\tau)+g'(\tau)|\xi|^2}{2( b(\tau)+g(\tau)|\xi|^2)}+\frac{b'(\tau)+( g'(\tau)-2 )|\xi|^2}{2( b(\tau)+g(\tau)|\xi|^2)} \bigg)d\tau \bigg\}\mathcal{Q}_{\text{ell}}(t,s,\xi),
\end{align*}
where $\mathcal{Q}_{\text{ell}}=\mathcal{Q}_{\text{ell}}(t,s,\xi)$ is a uniformly bounded matrix. Then, it follows
\begin{align*}
(|E_{\text{ell}}^{W}(t,s,\xi)|) & \lesssim \exp \bigg\{ \int_{s}^{t}\bigg( \frac{b(\tau)}{2}+\frac{g(\tau)|\xi|^2}{2}+\frac{b'(\tau)+g'(\tau)|\xi|^2}{ b(\tau)+g(\tau)|\xi|^2}-\frac{|\xi|^2}{b(\tau)+g(\tau)|\xi|^2} \bigg\} \left( \begin{array}{cc}
1 & 1 \\
1 & 1
\end{array} \right) \\
& \lesssim \frac{b(t)+g(t)|\xi|^2}{b(s)+g(s)|\xi|^2} \exp \bigg( \frac{1}{2}\int_{s}^{t} ( b(\tau)+g(\tau)|\xi|^2)d\tau \bigg)\left( \begin{array}{cc}
1 & 1 \\
1 & 1
\end{array} \right).
\end{align*}
This completes the proof.
\end{proof}
Now let us come back to
\begin{equation} \label{Eq:Effective_Increasing_Zell_Back}
W(t,\xi) = E_W(t,s,\xi)W(s,\xi) \qquad \text{for all} \qquad t_\xi\leq s\leq t,
\end{equation}
that is,
\begin{align*}
\left( \begin{array}{cc}
\gamma(t,\xi)w(t,\xi) \\
D_t w(t,\xi)
\end{array} \right) = E_W(t,s,\xi)\left( \begin{array}{cc}
\gamma(s,\xi)w(s,\xi) \\
D_tw(s,\xi)
\end{array} \right) \qquad \text{for} \qquad t_\xi\leq s\leq t,
\end{align*}
where $\gamma=\gamma(t,\xi):=\frac{g(t)}{2}|\xi|^2+\frac{b(t)}{2}$. Therefore, from Proposition \ref{Prop_Effective_Increasing_Est_Zell} and \eqref{Eq:Effective_Increasing_Zell_Back} we may conclude the following estimates for $t_\xi \leq s \leq t$:
\begin{align*}
\gamma(t,\xi)|w(t,\xi)| & \lesssim \frac{b(t)+g(t)|\xi|^2}{b(s)+g(s)|\xi|^2}\exp\Big( \frac{1}{2}\int_{s}^{t}\big( b(\tau)+g(\tau)|\xi|^2 \big)d\tau \Big)\big( \gamma(s,\xi)|w(s,\xi)| + |w_t(s,\xi)| \big), \\
|w_t(t,\xi)| & \lesssim \frac{b(t)+g(t)|\xi|^2}{b(s)+g(s)|\xi|^2}\exp\Big( \frac{1}{2}\int_{s}^{t}\big( b(\tau)+g(\tau)|\xi|^2 \big)d\tau \Big)\big( \gamma(s,\xi)|w(s,\xi)| + |w_t(s,\xi)| \big).
\end{align*}
Using the backward transformation
\[ w(t,\xi)=\exp\Big( \frac{1}{2} \int_0^t \big( b(\tau)+g(\tau)|\xi|^2 \big)d\tau \Big)\hat{u}(t,\xi),  \]
we arrive immediately at the following result.
\begin{corollary} \label{Cor:Effective_Increasing_Zell}
We have the following estimates in the elliptic zone $\Zell(N)$ for $t_\xi \leq s \leq t$:
\begin{align*}
|\xi|^{|\beta|}|\hat{u}(t,\xi)| & \lesssim |\xi|^{|\beta|}|\hat{u}(s,\xi)| + \frac{|\xi|^{|\beta|}}{b(s)+g(s)|\xi|^2}|\hat{u}_t(s,\xi)| \quad \mbox{for} \quad |\beta|\geq 0, \\
|\xi|^{|\beta|}|\hat{u}_t(t,\xi)| & \lesssim \big( b(t)+g(t)|\xi|^2 \big)|\xi|^{|\beta|}|\hat{u}(s,\xi)| + \frac{b(t)+g(t)|\xi|^2}{b(s)+g(s)|\xi|^2}|\xi|^{|\beta|}|\hat{u}_t(s,\xi)| \quad \mbox{for} \quad |\beta|\geq 0.
\end{align*}
\end{corollary}
\subsubsection{Considerations in the pseudo-differential zone $\Zpd(N)$} \label{Subsection_Effective-Incr_Zpd}
Let us introduce the micro-energy $U=\big( \gamma(t,\xi)\hat{u},D_t\hat{u} \big)^\text{T}$ with $\gamma(t,\xi):=\dfrac{b(t)+g(t)|\xi|^2}{2}$. Then, by \eqref{MainEquationFourier}, we find the system
\begin{equation} \label{Eq:Effective-Incr_Zpd_System}
D_tU=\underbrace{\left( \begin{array}{cc}
\dfrac{D_t\gamma(t,\xi)}{\gamma(t,\xi)} & \gamma(t,\xi) \\
\dfrac{|\xi|^2}{\gamma(t,\xi)} & i( b(t)+g(t)|\xi|^2)
\end{array} \right)}_{A(t,\xi)}U.
\end{equation}
The fundamental solution $E_{\text{pd}}=E_{\text{pd}}(t,s,\xi)$ to the system \eqref{Eq:Effective-Incr_Zpd_System} is the solution of
\[ D_tE_{\text{pd}}(t,s,\xi)=A(t,\xi)E_{\text{pd}}(t,s,\xi), \quad E_{\text{pd}}(s,s,\xi)=I, \]
for all $0\leq s \leq t$ and $(t,\xi), (s,\xi) \in \Zpd(N)$. Thus, the solution $U=U(t,\xi)$ is represented as
\[ U(t,\xi)=E_{\text{pd}}(t,s,\xi)U(s,\xi). \]
We will use the auxiliary function
\[ \delta=\delta(t,\xi)=\exp\Big( \int_{0}^{t}( b(\tau)+g(\tau)|\xi|^2)d\tau \Big). \]
The entries $E_{\text{pd}}^{(k\ell)}=E_{\text{pd}}^{(k\ell)}(t,s,\xi)$, $k,\ell=1,2,$ of the fundamental solution $E_{\text{pd}}=E_{\text{pd}}(t,s,\xi)$ satisfy the following system for $\ell=1,2$:
\begin{align*}
D_tE_{\text{pd}}^{(1\ell)}(t,s,\xi) &= \frac{D_t\gamma(t,\xi)}{\gamma(t,\xi)}E_{\text{pd}}^{(1\ell)}(t,s,\xi)+\gamma(t,\xi)E_{\text{pd}}^{(2\ell)}(t,s,\xi), \\
D_tE_{\text{pd}}^{(2\ell)}(t,s,\xi) &= \frac{|\xi|^2}{\gamma(t,\xi)}E_{\text{pd}}^{(1\ell)}(t,s,\xi)+i\big( b(t)+g(t)|\xi|^2) \big)E_{\text{pd}}^{(2\ell)}(t,s,\xi).
\end{align*}
Then, by straight-forward calculations (with $\delta_{k \ell}=1$ if $k=\ell$ and $\delta_{k \ell}=0$ otherwise), we get
\begin{align*}
E_{\text{pd}}^{(1\ell)}(t,s,\xi) & = \frac{\gamma(t,\xi)}{\gamma(s,\xi)}\delta_{1\ell}+i\gamma(t,\xi)\int_{s}^{t}E_{\text{pd}}^{(2\ell)}(\tau,s,\xi)d\tau, \\
% E_{\text{pd}}^{(21)}(t,s,\xi) & = \frac{i|\xi|^2}{\delta(t)}\int_{s}^{t}\frac{1}{\gamma(\tau,\xi)}\delta(\tau)E_{\text{pd}}^{(11)}(\tau,s,\xi)d\tau, \\
% E_{\text{pd}}^{(12)}(t,s,\xi) & = i\gamma(t,\xi)\int_{s}^{t}E_{\text{pd}}^{(22)}(\tau,s,\xi)d\tau, \\
E_{\text{pd}}^{(2\ell)}(t,s,\xi) & = \frac{\delta(s,\xi)}{\delta(t,\xi)}\delta_{2\ell}+\frac{i|\xi|^2}{\delta(t,\xi)}\int_{s}^{t}\frac{1}{\gamma(\tau,\xi)}\delta(\tau,\xi)E_{\text{pd}}^{(1\ell)}(\tau,s,\xi)d\tau.
\end{align*}
\begin{proposition} \label{Prop:Effective-Incr_Zpd}
We have the following estimates in the pseudo-differential zone with $(s,\xi),(t,\xi)\in\Zpd(N)$ and $0\leq s\leq t\leq t_\xi$:
if $b'\geq0$ it holds
\begin{equation*}
(|E_{\text{pd}}(t,s,\xi)|) \lesssim \frac{b(t)+g(t)|\xi|^2}{b(s)+g(s)|\xi|^2}
\left( \begin{array}{cc}
1 & 1 \\
1 & 1
\end{array} \right),
\end{equation*}
if $b'<0$ it holds
\begin{equation*}
(|E_{\text{pd}}(t,s,\xi)|) \lesssim \frac{b(t)+g(t)|\xi|^2}{b(s)+g(s)|\xi|^2}
\left( \begin{array}{cc}
1 & 0 \\
1 & 0
\end{array} \right) + \frac{b(t)+g(t)|\xi|^2}{g(s)|\xi|^2}
\left( \begin{array}{cc}
0 & 1 \\
0 & 1
\end{array} \right).
\end{equation*}
\end{proposition}
\begin{proof}
First let us consider the first column. Plugging the representation for $E_{\text{pd}}^{(21)}=E_{\text{pd}}^{(21)}(t,s,\xi)$ into the integral equation for $E_{\text{pd}}^{(11)}=E_{\text{pd}}^{(11)}(t,s,\xi)$ gives
\begin{align*}
E_{\text{pd}}^{(11)}(t,s,\xi) = \frac{\gamma(t,\xi)}{\gamma(s,\xi)}-|\xi|^2\gamma(t,\xi)\int_{s}^{t}\int_{s}^{\tau}\frac{\delta(\theta,\xi)}{\delta(\tau,\xi)}\frac{1}{\gamma(\theta,\xi)}
E_{\text{pd}}^{(11)}(\theta,s,\xi)d\theta d\tau.
\end{align*}
By setting $y(t,s,\xi):=\dfrac{\gamma(s,\xi)}{\gamma(t,\xi)}E_{\text{pd}}^{(11)}(t,s,\xi)$, if $b'(\tau)>0$, then we obtain
\begin{align*}
y(t,s,\xi) &= 1-|\xi|^2\int_{s}^{t}\int_{\theta}^{t}\frac{\delta(\theta,\xi)}{\delta(\tau,\xi)}y(\theta,s,\xi)d\tau d\theta \\
& = 1+|\xi|^2\int_{s}^{t}\bigg( \int_{\theta}^{t}\frac{1}{b(\tau)+g(\tau)|\xi|^2}\partial_\tau \bigg( \frac{\delta(\theta,\xi)}{\delta(\tau,\xi)} \bigg)d\tau \bigg) y(\theta,s,\xi)d\theta \\
& = 1+|\xi|^2\int_{s}^{t}\bigg( \frac{1}{b(\tau)+g(\tau)|\xi|^2}\frac{\delta(\theta,\xi)}{\delta(\tau,\xi)}\Big|_{\theta}^{t}+\int_{\theta}^{t}\frac{b'(\tau)+g'(\tau)|\xi|^2}{( b(\tau)+g(\tau)|\xi|^2)^2} \frac{\delta(\theta,\xi)}{\delta(\tau,\xi)} d\tau \bigg) y(\theta,s,\xi)d\theta.
\end{align*}
Then, it holds
\begin{align*}
|y(t,s,\xi)| &\leq 1+|\xi|^2\int_{s}^{t}\bigg( \frac{1}{b(t)+g(t)|\xi|^2}\underbrace{\frac{\delta(\theta,\xi)}{\delta(t,\xi)}}_{\leq1}+\int_{\theta}^{t}\frac{b'(\tau)+g'(\tau)|\xi|^2}{( b(\tau)+g(\tau)|\xi|^2)^2} \underbrace{\frac{\delta(\theta,\xi)}{\delta(\tau,\xi)}}_{\leq1}d\tau \bigg)|y(\theta,s,\xi)|d\theta \\
& \leq 1+|\xi|^2\int_{s}^{t}\bigg( \frac{1}{b(t)+g(t)|\xi|^2}-\frac{1}{b(\tau)+g(\tau)|\xi|^2}\Big|_{\theta}^{t} \bigg)|y(\theta,s,\xi)|d\theta \\
& = 1+|\xi|^2\int_{s}^{t}\frac{1}{b(\theta)+g(\theta)|\xi|^2}|y(\theta,s,\xi)|d\theta.
\end{align*}
Applying Gronwall's inequality and employing \textbf{(A2)}, we get the estimate
\begin{align*}
|y(t,s,\xi)| \leq \exp \bigg( \int_{s}^{t}\frac{|\xi|^2}{b(\theta)+g(\theta)|\xi|^2}d\theta \bigg) \leq \exp \bigg( \int_{s}^{t}\frac{1}{g(\theta)}d\theta \bigg) \lesssim 1.
\end{align*}
On the other hand, if $b'(\tau)<0$, then we have
\begin{align*}
y(t,s,\xi) &= 1-|\xi|^2\int_{s}^{t}\int_{\theta}^{t}\frac{\delta(\theta,\xi)}{\delta(\tau,\xi)}y(\theta,s,\xi)d\tau d\theta \\
& = 1+|\xi|^2\int_{s}^{t}\bigg( \int_{\theta}^{t}\frac{1}{b(\tau)+g(\tau)|\xi|^2}\partial_\tau \bigg( \frac{\delta(\theta,\xi)}{\delta(\tau,\xi)} \bigg)d\tau \bigg) y(\theta,s,\xi)d\theta \\
& = 1+|\xi|^2\int_{s}^{t}\bigg( \frac{1}{b(\tau)+g(\tau)|\xi|^2}\frac{\delta(\theta,\xi)}{\delta(\tau,\xi)}\Big|_{\theta}^{t}+\int_{\theta}^{t}\frac{b'(\tau)+g'(\tau)|\xi|^2}{( b(\tau)+g(\tau)|\xi|^2)^2} \frac{\delta(\theta,\xi)}{\delta(\tau,\xi)} d\tau \bigg) y(\theta,s,\xi)d\theta.
\end{align*}
Thus, it follows
\begin{align} \label{Eq:Effective-Incr-Zpd-Est-y1}
|y(t,s,\xi)| &\leq 1+|\xi|^2\int_{s}^{t}\bigg( \frac{1}{b(t)+g(t)|\xi|^2}\underbrace{\frac{\delta(\theta,\xi)}{\delta(t,\xi)}}_{\leq1}+\int_{\theta}^{t}\frac{g'(\tau)|\xi|^2-b'(\tau)}{( b(\tau)+g(\tau)|\xi|^2)^2} \underbrace{\frac{\delta(\theta,\xi)}{\delta(\tau,\xi)}}_{\leq1}d\tau \bigg)|y(\theta,s,\xi)|d\theta \nonumber \\
&\leq 1+|\xi|^2\int_{s}^{t}\bigg( \frac{1}{b(t)+g(t)|\xi|^2} - \int_{\theta}^{t}\frac{b'(\tau)+g'(\tau)|\xi|^2}{(b(\tau)+g(\tau)|\xi|^2)^2} + 2\int_{\theta}^{t}\frac{g'(\tau)|\xi|^2}{(b(\tau)+g(\tau)|\xi|^2)^2} d\tau \bigg)|y(\theta,s,\xi)|d\theta \nonumber \\
&\leq 1+|\xi|^2\int_{s}^{t}\bigg( \frac{1}{b(t)+g(t)|\xi|^2} - \int_{\theta}^{t}\frac{b'(\tau)+g'(\tau)|\xi|^2}{(b(\tau)+g(\tau)|\xi|^2)^2} + 2\int_{\theta}^{t}\frac{g'(\tau)|\xi|^2}{(g(\tau)|\xi|^2)^2} d\tau \bigg)|y(\theta,s,\xi)|d\theta \nonumber \\
&= 1+|\xi|^2\int_{s}^{t}\bigg( \frac{1}{b(t)+g(t)|\xi|^2} + \frac{1}{b(\tau)+g(\tau)|\xi|^2}\Big|_\theta^t - \frac{2}{g(\tau)|\xi|^2}\Big|_\theta^t \bigg)|y(\theta,s,\xi)|d\theta \nonumber \\
& = 1+|\xi|^2\int_{s}^{t}\bigg( \frac{2}{b(t)+g(t)|\xi|^2} - \frac{1}{b(\theta)+g(\theta)|\xi|^2} - \frac{2}{g(t)|\xi|^2} + \frac{2}{g(\theta)|\xi|^2}\bigg)|y(\theta,s,\xi)|d\theta \nonumber \\
& \leq 1+|\xi|^2\int_{s}^{t}\bigg( \frac{2}{g(t)|\xi|^2} - \frac{1}{b(\theta)+g(\theta)|\xi|^2} - \frac{2}{g(t)|\xi|^2} + \frac{2}{g(\theta)|\xi|^2}\bigg)|y(\theta,s,\xi)|d\theta \nonumber \\
& \leq 1+|\xi|^2\int_{s}^{t}\bigg( \frac{2}{g(\theta)|\xi|^2} \bigg)|y(\theta,s,\xi)|d\theta.
\end{align}
Applying Gronwall's inequality and employing \textbf{(A2)}, we get the estimate
\begin{align*}
|y(t,s,\xi)| \leq \exp \bigg( \int_{s}^{t}\frac{1}{g(\theta)}d\theta \bigg) \lesssim 1.
\end{align*}
Thus, we may conclude that
\[ |E_{\text{pd}}^{(11)}(t,s,\xi)| \lesssim \frac{\gamma(t,\xi)}{\gamma(s,\xi)}=\frac{b(t)+g(t)|\xi|^2}{b(s)+g(s)|\xi|^2}. \]
Now we consider $E_{\text{pd}}^{(21)}(t,s,\xi)$. By using the estimate for $|E_{\text{pd}}^{(11)}(t,s,\xi)|$ we obtain
\begin{align*}
\frac{\gamma(s,\xi)}{\gamma(t,\xi)}|E_{\text{pd}}^{(21)}(t,s,\xi)| & \lesssim |\xi|^2\int_{s}^{t}\frac{1}{\gamma(\tau,\xi)}\frac{\delta(\tau,\xi)}{\delta(t,\xi)}\frac{\gamma(s,\xi)}{\gamma(t,\xi)}|E_{\text{pd}}^{(11)}(\tau,s,\xi)|d\tau \\
& \lesssim |\xi|^2\int_{s}^{t}\frac{1}{b(\tau)+g(\tau)|\xi|^2}\underbrace{\frac{\delta(\tau,\xi)\gamma(\tau,\xi)}{\delta(t,\xi)\gamma(t,\xi)}}_{\leq1}d\tau \lesssim \int_{s}^{t}\frac{|\xi|^2}{b(\tau)+g(\tau)|\xi|^2}d\tau \lesssim 1,
\end{align*}
where we have used
\begin{align*} \label{Eq:Effective-Incr-Zpd-Inq1}
\frac{\delta(\tau,\xi)\gamma(\tau,\xi)}{\delta(t,\xi)\gamma(t,\xi)} \leq \frac{\delta(t,\xi)\gamma(t,\xi)}{\delta(t,\xi)\gamma(t,\xi)} = 1,
\end{align*}
because $f(t,\xi) := \delta(t,\xi)\gamma(t,\xi)$ is increasing in $t$. Namely, we have
\begin{align*}
f_t(t,\xi) = \exp\Big( \int_{0}^{t}\big( b(\tau)+g(\tau)|\xi|^2 \big)d\tau \Big)\frac{1}{2}\big( (b(t)+g(t)|\xi|^2)^2 + b'(t)+g'(t)|\xi|^2 \big),
\end{align*}
where using the  fact that $|b'(t)| = o(b^2(t))$, we may conclude that $f_t(t,\xi)>0$, namely $f=f(t,\xi)$ is increasing in $t$. Thus, we arrive at
\[ |E_{\text{pd}}^{(21)}(t,s,\xi)| \lesssim \frac{\gamma(t,\xi)}{\gamma(s,\xi)}=\frac{b(t)+g(t)|\xi|^2}{b(s)+g(s)|\xi|^2}. \]
Next, we consider the entries of the second column. Plugging the representation for $E_{\text{pd}}^{(22)}=E_{\text{pd}}^{(22)}(t,s,\xi)$ into the integral equation for $E_{\text{pd}}^{(12)}=E_{\text{pd}}^{(12)}(t,s,\xi)$ gives
\begin{align*}
E_{\text{pd}}^{(12)}(t,s,\xi) = i\gamma(t,\xi)\int_{s}^{t}\frac{\delta(s,\xi)}{\delta(\tau,\xi)}d\tau-
|\xi|^2\gamma(t,\xi)\int_{s}^{t}\int_{s}^{\tau}\frac{\delta(\theta,\xi)}{\delta(\tau,\xi)}\frac{1}{\gamma(\theta,\xi)}E_{\text{pd}}^{(12)}(\theta,s,\xi)d\theta d\tau.
\end{align*}
After setting $y(t,s,\xi):=\dfrac{\gamma(s,\xi)}{\gamma(t,\xi)}E_{\text{pd}}^{(12)}(t,s,\xi)$, it follows
\begin{align*}
y(t,s,\xi) &= -i\gamma(s,\xi)\int_{s}^{t}\frac{1}{b(\tau)+g(\tau)|\xi|^2}\partial_\tau \bigg( \frac{\delta(s,\xi)}{\delta(\tau,\xi)} \bigg)d\tau \\
& \qquad + |\xi|^2\int_{s}^{t}\bigg( \int_{\theta}^{t}\frac{1}{b(\tau)+g(\tau)|\xi|^2}\partial_\tau \bigg( \frac{\delta(\theta,\xi)}{\delta(\tau,\xi)} \bigg)d\tau \bigg)y(\theta,s,\xi)d\theta d\tau \\
&= i\gamma(s,\xi) \bigg( -\frac{1}{b(\tau)+g(\tau)|\xi|^2}\frac{\delta(s,\xi)}{\delta(\tau,\xi)}\Big|_s^t - \int_{s}^{t}\frac{b'(\tau)+g'(\tau)|\xi|^2}{( b(\tau)+g(\tau)|\xi|^2)^2} \frac{\delta(\theta,\xi)}{\delta(\tau,\xi)}d\tau \bigg) \\
& \qquad + |\xi|^2\int_{s}^{t}\bigg( \frac{1}{b(t)+g(t)|\xi|^2}\frac{\delta(\theta,\xi)}{\delta(t,\xi)}+\int_{\theta}^{t}\frac{b'(\tau)+g'(\tau)|\xi|^2}{( b(\tau)+g(\tau)|\xi|^2)^2} \frac{\delta(\theta,\xi)}{\delta(\tau,\xi)}d\tau \bigg)y(\theta,s,\xi)d\theta.
\end{align*}
If $b'(\tau) \geq 0$, then we have
\begin{align*}
|y(t,s,\xi)| &\leq \gamma(s,\xi) \bigg( \frac{1}{b(t)+g(t)|\xi|^2}\underbrace{\frac{\delta(s,\xi)}{\delta(t,\xi)}}_{\leq1} +\frac{1}{b(s)+g(s)|\xi|^2} + \int_{s}^{t}\frac{b'(\tau)+g'(\tau)|\xi|^2}{( b(\tau)+g(\tau)|\xi|^2)^2}\underbrace{\frac{\delta(\theta,\xi)}{\delta(\tau,\xi)}}_{\leq1}d\tau \bigg) \\
& \qquad + |\xi|^2\int_{s}^{t}\bigg( \frac{1}{b(t)+g(t)|\xi|^2}\underbrace{\frac{\delta(\theta,\xi)}{\delta(t,\xi)}}_{\leq1}+\int_{\theta}^{t}\frac{b'(\tau)+g'(\tau)|\xi|^2}{( b(\tau)+g(\tau)|\xi|^2)^2}\underbrace{\frac{\delta(\theta,\xi)}{\delta(\tau,\xi)}}_{\leq1}d\tau \bigg)|y(\theta,s,\xi)|d\theta \\
& \leq  \gamma(s,\xi) \bigg( \frac{2}{b(s)+g(s)|\xi|^2} - \frac{1}{b(\tau)+g(\tau)|\xi|^2}\Big|_s^t \bigg) \\
& \qquad + |\xi|^2\int_{s}^{t}\bigg( \frac{1}{b(t)+g(t)|\xi|^2} - \frac{1}{b(\tau)+g(\tau)|\xi|^2}\Big|_\theta^t \bigg)|y(\theta,s,\xi)|d\theta \\
&\leq \gamma(s,\xi) \bigg( \frac{3}{b(s)+g(s)|\xi|^2} \bigg) + |\xi|^2\int_{s}^{t}\bigg( \frac{1}{b(\theta)+g(\theta)|\xi|^2} \bigg)|y(\theta,s,\xi)|d\theta \\
& = \frac{3}{2} + \int_{s}^{t}\frac{|\xi|^2}{b(\theta)+g(\theta)|\xi|^2}|y(\theta,s,\xi)|d\theta.
\end{align*}
Employing Gronwall's inequality, we have the estimate
\begin{align*}
|y(t,s,\xi)| \lesssim \exp \bigg( \int_{s}^{t}\frac{|\xi|^2}{b(\theta)+g(\theta)|\xi|^2}d\theta \bigg) \lesssim 1.
\end{align*}
Hence, we may conclude that
\[ |E_{\text{pd}}^{(12)}(t,s,\xi)| \lesssim \frac{\gamma(t,\xi)}{\gamma(s,\xi)}=\frac{b(t)+g(t)|\xi|^2}{b(s)+g(s)|\xi|^2}. \]
Finally, let us estimate $|E_{\text{pd}}^{(22)}(t,s,\xi)|$ by using the above estimate for $|E_{\text{pd}}^{(12)}(t,s,\xi)|$. It holds
\begin{align*}
|E_{\text{pd}}^{(22)}(t,s,\xi)| & \lesssim \frac{\delta(s,\xi)}{\delta(t,\xi)}+|\xi|^2\int_{s}^{t}\frac{\delta(\tau,\xi)}{\delta(t,\xi)}\frac{1}{\gamma(\tau,\xi)}|E_{\text{pd}}^{(12)}(\tau,s,\xi)|d\tau \\
& \lesssim \frac{\delta(s,\xi)}{\delta(t,\xi)}+|\xi|^2\int_{s}^{t}\frac{\delta(\tau,\xi)}{\delta(t,\xi)}\frac{1}{\gamma(\tau,\xi)}\frac{\gamma(\tau,\xi)}{\gamma(s,\xi)}d\tau.
\end{align*}
Then, we have
\begin{align*}
\frac{\gamma(s,\xi)}{\gamma(t,\xi)}|E_{\text{pd}}^{(22)}(t,s,\xi)| & \lesssim \underbrace{\frac{\delta(s,\xi)\gamma(s,\xi)}{\delta(t,\xi)\gamma(t,\xi)}}_{\leq1}+|\xi|^2\int_{s}^{t}
\underbrace{\frac{\delta(\tau,\xi)\gamma(\tau,\xi)}{\delta(t,\xi)\gamma(t,\xi)}}_{\leq1}\frac{1}{\gamma(\tau,\xi)}d\tau \\
& \lesssim 1 + |\xi|^2\int_{s}^{t}\frac{1}{b(\tau)+g(\tau)|\xi|^2}d\tau \lesssim 1.
\end{align*}
This shows that
\[ |E_{\text{pd}}^{(22)}(t,s,\xi)| \lesssim \frac{\gamma(t,\xi)}{\gamma(s,\xi)}=\frac{b(t)+g(t)|\xi|^2}{b(s)+g(s)|\xi|^2}. \]
If $b'(\tau)<0$, then we have
\begin{align} \label{Eq:Effective-Incr-Zpd-Est-y2}
|y(t,s,\xi)| &\leq \gamma(s,\xi) \bigg( \frac{1}{b(t)+g(t)|\xi|^2}\underbrace{\frac{\delta(s,\xi)}{\delta(t,\xi)}}_{\leq1} +\frac{1}{b(s)+g(s)|\xi|^2} + \int_{s}^{t}\frac{g'(\tau)|\xi|^2-b'(\tau)}{( b(\tau)+g(\tau)|\xi|^2)^2}\underbrace{\frac{\delta(\theta,\xi)}{\delta(\tau,\xi)}}_{\leq1}d\tau \bigg) \nonumber \\
& \qquad + |\xi|^2\int_{s}^{t}\bigg( \frac{1}{b(t)+g(t)|\xi|^2}\underbrace{\frac{\delta(\theta,\xi)}{\delta(t,\xi)}}_{\leq1}+\int_{\theta}^{t}\frac{g'(\tau)|\xi|^2-b'(\tau)}{( b(\tau)+g(\tau)|\xi|^2)^2}\underbrace{\frac{\delta(\theta,\xi)}{\delta(\tau,\xi)}}_{\leq1}d\tau \bigg)|y(\theta,s,\xi)|d\theta,
\end{align}
where we have
\[ \gamma(s,\xi) \frac{1}{b(t)+g(t)|\xi|^2}\frac{\delta(s,\xi)}{\delta(t,\xi)} \leq \frac{1}{2} \qquad \text{and} \qquad \gamma(s,\xi) \frac{1}{b(s)+g(s)|\xi|^2} =\frac{1}{2}, \]
and
\begin{align*}
& \gamma(s,\xi) \int_{s}^{t}\frac{g'(\tau)|\xi|^2-b'(\tau)}{( b(\tau)+g(\tau)|\xi|^2)^2}\underbrace{\frac{\delta(\theta,\xi)}{\delta(\tau,\xi)}}_{\leq1}d\tau \leq  \gamma(s,\xi) \int_{s}^{t}\frac{g'(\tau)|\xi|^2-b'(\tau)}{( b(\tau)+g(\tau)|\xi|^2)^2} d\tau \\
& \qquad = -\gamma(s,\xi) \int_{s}^{t}\frac{b'(\tau)+g'(\tau)|\xi|^2}{( b(\tau)+g(\tau)|\xi|^2)^2} d\tau + 2\gamma(s,\xi) \int_{s}^{t}\frac{g'(\tau)|\xi|^2}{( b(\tau)+g(\tau)|\xi|^2)^2} d\tau \\
& \qquad = \frac{\gamma(s,\xi)}{2\gamma(t,\xi)}-\frac{1}{2} + 2\gamma(s,\xi) \int_{s}^{t}\frac{g'(\tau)|\xi|^2}{( b(\tau)+g(\tau)|\xi|^2)^2} d\tau \\
& \qquad \leq \frac{\gamma(s,\xi)}{2\gamma(t,\xi)} + 2\gamma(s,\xi) \int_{s}^{t}\frac{g'(\tau)|\xi|^2}{( g(\tau)|\xi|^2)^2} d\tau \\
& \qquad = \frac{\gamma(s,\xi)}{2\gamma(t,\xi)} + \big( b(s)+g(s)|\xi|^2 \big)\Big[ -\frac{1}{g(\tau)||\xi|^2} \Big]_s^t \\
& \qquad = \frac{1}{2}\frac{b(s)+g(s)|\xi|^2}{b(t)+g(t)|\xi|^2} - \frac{b(s)+g(s)|\xi|^2}{g(t)|\xi|^2} + \frac{b(s)+g(s)|\xi|^2}{g(s)|\xi|^2} \\
& \qquad \leq \frac{b(s)+g(s)|\xi|^2}{g(t)|\xi|^2} - \frac{b(s)+g(s)|\xi|^2}{g(t)|\xi|^2} + \frac{b(s)+g(s)|\xi|^2}{g(s)|\xi|^2} \\
& \qquad = \frac{b(s)+g(s)|\xi|^2}{g(s)|\xi|^2} \lesssim \frac{\gamma(s,\xi)}{g(s)|\xi|^2}.
\end{align*}
For the second summand in \eqref{Eq:Effective-Incr-Zpd-Est-y2} following the same approach to \eqref{Eq:Effective-Incr-Zpd-Est-y1}, we find
\[ |y(t,s,\xi)| \leq 1+\frac{\gamma(s,\xi)}{g(s)|\xi|^2}+|\xi|^2\int_{s}^{t}\bigg( \frac{2}{g(\theta)|\xi|^2} \bigg)|y(\theta,s,\xi)|d\theta. \]
Applying Gronwall's inequality and using condition \textbf{(A2)} we have the estimate
\begin{align*}
|y(t,s,\xi)| \lesssim \Big( 1+\frac{\gamma(s,\xi)}{g(s)|\xi|^2} \Big)\exp\bigg( \int_{s}^{t}\frac{2}{g(\theta)}d\theta \bigg) \lesssim 1+\frac{\gamma(s,\xi)}{g(s)|\xi|^2}.
\end{align*}
This implies that
\begin{equation} \label{Eq:Effective-Incr-Zpd-Est-E12}
|E_{\text{pd}}^{(12)}(t,s,\xi)| \lesssim \frac{\gamma(t,\xi)}{\gamma(s,\xi)}|y(t,s,\xi)|\lesssim \frac{b(t)+g(t)|\xi|^2}{g(s)|\xi|^2}.
\end{equation}
Finally, let us estimate $|E_{\text{pd}}^{(22)}(t,s,\xi)|$ by using the estimate of $|E_{\text{pd}}^{(12)}(t,s,\xi)|$ from \eqref{Eq:Effective-Incr-Zpd-Est-E12}. It follows
\begin{align*}
|E_{\text{pd}}^{(22)}(t,s,\xi)| & \lesssim \frac{\delta(s,\xi)}{\delta(t,\xi)}+|\xi|^2\int_{s}^{t}\frac{\delta(\tau,\xi)}{\delta(t,\xi)}\frac{1}{\gamma(\tau,\xi)}|E_{\text{pd}}^{(12)}(\tau,s,\xi)|d\tau \\
& \lesssim \frac{\delta(s,\xi)}{\delta(t,\xi)}+|\xi|^2\int_{s}^{t}\frac{\delta(\tau,\xi)}{\delta(t,\xi)}\frac{1}{\gamma(\tau,\xi)}\frac{\gamma(\tau,\xi)}{g(s)|\xi|^2}d\tau.
\end{align*}
Then, we have
\begin{align*}
\frac{g(s)||\xi|^2}{\gamma(t,\xi)}|E_{\text{pd}}^{(22)}(t,s,\xi)| & \lesssim \underbrace{\frac{\delta(s,\xi)g(s)|\xi|^2}{\delta(t,\xi)\gamma(t,\xi)}}_{\leq1}+|\xi|^2\int_{s}^{t}\underbrace{\frac{\delta(\tau,\xi)\gamma(\tau,\xi)}{\delta(t,\xi)\gamma(t,\xi)}}_{\leq1}\frac{1}{\gamma(\tau,\xi)}d\tau \\
& \lesssim 1 + |\xi|^2\int_{s}^{t}\frac{1}{b(\tau)+g(\tau)|\xi|^2}d\tau \lesssim 1.
\end{align*}
This implies that
\[ |E_{\text{pd}}^{(22)}(t,s,\xi)| \lesssim \frac{\gamma(t,\xi)}{g(s)|\xi|^2} \lesssim \frac{b(t)+g(t)|\xi|^2}{g(s)|\xi|^2}. \]
This completes the proof.
\end{proof}
Now let us come back to
\begin{equation} \label{Eq:Effective-Incr-PseudoZone}
U(t,\xi) = E(t,0,\xi)U(0,\xi) \quad \text{for all} \quad 0\leq t\leq t_{\xi}.
\end{equation}
Because of \eqref{Eq:Effective-Incr-PseudoZone} and Proposition \ref{Prop:Effective-Incr_Zpd}, the following statement can be concluded.
\begin{corollary} \label{Cor:Effective-Incr_IncreasingZpd}
In the pseudo-differential zone $\Zpd(N)$ the following estimates hold for all $0\leq t\leq t_{\xi}$:\\
If $b'\geq0$, then we have
\begin{align*}
|\xi|^{|\beta|}|\hat{u}_t(t,\xi)| &\lesssim ( b(t)+g(t)|\xi|^2)|\xi|^{|\beta|}|\hat{u}_0(\xi)| + ( b(t)+g(t)|\xi|^2) |\xi|^{|\beta|}\langle \xi \rangle^{-2}|\hat{u}_1(\xi)| \quad \mbox{for} \quad |\beta|\geq 0.
\end{align*}
If $b'<0$, then we have
\begin{align*}
|\xi|^{|\beta|}|\hat{u}_t(t,\xi)| &\lesssim ( b(t)+g(t)|\xi|^2)|\xi|^{|\beta|}|\hat{u}_0(\xi)| + ( b(t)+g(t)|\xi|^2) |\xi|^{|\beta|}|\xi|^{-2}|\hat{u}_1(\xi)| \quad \mbox{for} \quad |\beta|\geq 0.
\end{align*}
\end{corollary}
\subsubsection{Conclusion} \label{Concl:Effective-Incr_IncreasingZpd}
From the statements of Corollaries \ref{Cor:Effective_Increasing_Zell} and \ref{Cor:Effective-Incr_IncreasingZpd} we derive the following estimates for $t>0$:\\
If $b'\geq0$, then we have
\begin{align*}
|\xi|^{|\beta|}|\hat{u}_t(t,\xi)| &\lesssim ( b(t)+g(t)|\xi|^2)|\xi|^{|\beta|}|\hat{u}_0(\xi)| + ( b(t)+g(t)|\xi|^2) |\xi|^{|\beta|}\langle \xi \rangle^{-2}|\hat{u}_1(\xi)| \quad \mbox{for} \quad |\beta|\geq 0.
\end{align*}
If $b'<0$, then we have
\begin{align*}
|\xi|^{|\beta|}|\hat{u}_t(t,\xi)| &\lesssim ( b(t)+g(t)|\xi|^2)|\xi|^{|\beta|}|\hat{u}_0(\xi)| + ( b(t)+g(t)|\xi|^2) |\xi|^{|\beta|}|\xi|^{-2}|\hat{u}_1(\xi)| \quad \mbox{for} \quad |\beta|\geq 0.
\end{align*}
This completes the proof of Theorem \ref{Theorem:Effective_Increasing}.
\end{proof}

\subsection{Model with integrable and decaying time-dependent coefficient $g=g(t)$} \label{Subsect_Effective_Integrable}
We assume that the coefficient $g=g(t)$ satisfies the following conditions for all $t \in [0,\infty)$:
\begin{enumerate}
\item[\textbf{(E1)}] $g(t)>0$ and $g'(t)<0$,
\item[\textbf{(E2)}] $g \in L^1([0,\infty))$,
%\item[\textbf{(E3)}] $\dfrac{|g^{(k)}(t)|}{g(t)} \leq C_k\Big( \dfrac{g(t)}{G(t)} \Big)^k$, $k=1,2$,  where $G(t):=\displaystyle\int_t^\infty g(\tau)d\tau$ and $C_1$, $C_2$ are positive constants,
\item[\textbf{(E3)}] $b(t)g(t)\leq \dfrac{1}{2}$ \,\,for all\,\, $t\geq0$,
\item[\textbf{(E4)}] $\dfrac{b'(t)}{b(t)}\leq -\dfrac{g'(t)}{g(t)}$\,\, for all\,\, $t\geq0$,
%\item[\textbf{(E5)}] $\dfrac{1}{b(t)(1+t)^2} \in L^1([0,\infty))$,
\item[\textbf{(E5)}] $\dfrac{g'^2}{g}\in L^1([0,\infty))$.
\end{enumerate}
\begin{theorem} \label{Theorem_Effective-Decreasing_Integrable-Decaying}
Let us consider the Cauchy problem
\begin{equation*}
\begin{cases}
u_{tt} - \Delta u + b(t)u_t -g(t)\Delta u_t=0, &(t,x) \in [0,\infty) \times \mathbb{R}^n, \\
u(0,x)= u_0(x),\quad u_t(0,x)= u_1(x), &x \in \mathbb{R}^n.
\end{cases}
\end{equation*}
We assume that the coefficients $b=b(t)$ and $g=g(t)$ satisfy the conditions \textbf{(B1)} to \textbf{(B3)} and \textbf{(E1)} to \textbf{(E5)} and, \textbf{(EF)}, respectively. Then, we have the following estimates for Sobolev solutions for all $t\geq t_0$:
\begin{align*}
\|\,|D|^{|\beta|}u(t,\cdot)\|_{L^2} &\lesssim \Big( 1+\int_{t_0}^t\frac{1}{b(\tau)}d\tau \Big)^{-\frac{|\beta|}{2}}\|u_0\|_{H^{|\beta|}} + \Big( 1+\int_{t_0}^t\frac{1}{b(\tau)}d\tau \Big)^{-\frac{|\beta|-1}{2}}\|u_1\|_{H^{|\beta|-1}} \quad \mbox{for} \quad |\beta|\geq 1, \\
\|\,|D|^{|\beta|}u_t(t,\cdot)\|_{L^2} &\lesssim \Big( 1+\int_{t_0}^t\frac{1}{b(\tau)}d\tau \Big)^{-\frac{|\beta|+1}{2}} \|u_0\|_{H^{|\beta|+1}} + \Big(1+\int_{t_0}^t\frac{1}{b(\tau)}d\tau \Big)^{-\frac{|\beta|}{2}} \|u_1\|_{H^{|\beta|}} \quad \mbox{for} \quad |\beta|\geq 0.
\end{align*}
\end{theorem}
Let us consider the following transformed equation from \eqref{AuxiliaryEquation3}:
\[ w_{tt} + \Big(|\xi|^2\Big(1-\frac{b(t)g(t)}{2}-\frac{g(t)^2|\xi|^2}{4}-\frac{g'(t)}{2}\Big)-\frac{b(t)^2}{4}-\frac{b'(t)}{2}\Big)w=0.\]
We have $b'(t)=o(b^2(t))$ for $t \to \infty$. For this reason we can restrict ourselves to the equation
\[ w_{tt} + \Big(|\xi|^2\Big( 1-\frac{b(t)g(t)}{2}-\frac{g'(t)}{2} \Big) - \frac{g(t)^2}{4}|\xi|^4-\frac{b(t)^2}{4}\Big)w=0. \]
This equation helps to determine the zones of the extended phase space. The last equation can be written in the form
\[ w_{tt} + |\xi|^2\Big( 1-\frac{g'(t)}{2} \Big)w - \Big(\frac{g(t)}{2}|\xi|^2+\frac{b(t)}{2}\Big)^2 w=0. \]
 In the following we study only the case $-g'(t) \leq 2 -\varepsilon$ for all $t \geq 0$ and with $\varepsilon >0$. It is clear that $g'(t)=o(1)$ for $t \to \infty$. So, the proposed case is reasonable.
For the last equation we develop a WKB analysis. \\
We will consider our equation in the following form:
\begin{equation} \label{Eq:Effective-Integrable-Dfrom}
D_t^2w + \Big( \frac{b(t)}{2}+\frac{g(t)|\xi|^2}{2} \Big)^2w - |\xi|^2w + \frac{g'(t)}{2}|\xi|^2w + \frac{b'(t)}{2}w = 0.
\end{equation}
At first let us turn to the equation
\[ |\xi|^2 - \Big( \frac{b(t)}{2}+\frac{g(t)|\xi|^2}{2} \Big)^2 = 0. \]
The equation
\[ |\xi| = \frac{b(t)}{2}+\frac{g(t)|\xi|^2}{2} \]
implies that we have the following two separating lines:
\[ \Pi_1 := \big\{(t,\xi): |\xi| = \frac{1}{g(t)} + \frac{1}{g(t)}\sqrt{1-b(t)g(t)} \big\} \quad \text{and} \quad \Pi_2 := \big\{(t,\xi): |\xi| = \frac{1}{g(t)} - \frac{1}{g(t)}\sqrt{1-b(t)g(t)} \big\}. \]
Therefore, we can introduce the following functions corresponding to the separating lines:
\[ f_1: t\to \frac{1}{g(t)}\big( 1+\sqrt{1-b(t)g(t)} \big) \qquad\text{and} \qquad  f_2: t\to \frac{1}{g(t)}\big( 1-\sqrt{1-b(t)g(t)} \big). \]
First, we consider $f_1'=f_1'(t)$. It holds
\begin{align*}
f_1'(t) &= -\frac{g'(t)}{g^2(t)}\big( 1+\sqrt{1-b(t)g(t)} \big) - \frac{1}{2g(t)}\frac{b'(t)g(t)+b(t)g'(t)}{\sqrt{1-b(t)g(t)}} \\
&\geq -\frac{g'(t)}{g^2(t)}\big( 1+\frac{1}{\sqrt{2}} \big) \geq0,
\end{align*}
where we have used \textbf{(E3)} and \textbf{(E4)} which imply the followings, respectively:
\[ \sqrt{1-b(t)g(t)} \geq \frac{1}{\sqrt{2}} \qquad \text{and} \qquad -\big( b'(t)g(t)+b(t)g'(t) \big) \geq0. \]
Secondly, we consider $f_2'=f_2'(t)$ for $t\geq t_0$, where $t_0$ is sufficiently large. We have
\begin{align*}
f_2'(t) &= -\frac{g'(t)}{g^2(t)}\big( 1-\sqrt{1-b(t)g(t)} \big) + \frac{1}{2g(t)}\frac{b'(t)g(t)+b(t)g'(t)}{\sqrt{1-b(t)g(t)}} \\
&= \frac{1}{g(t)}\frac{1}{\sqrt{1-b(t)g(t)}} \bigg( \frac{g'(t)}{g(t)}\big( 1-b(t)g(t) \big) -\frac{g'(t)}{g(t)}\sqrt{1-b(t)g(t)} + \frac{b'(t)g(t)}{2} + \frac{b(t)g'(t)}{2} \bigg) \\
&= \frac{1}{g(t)}\frac{1}{\sqrt{1-b(t)g(t)}} \bigg( -\frac{g'(t)}{g(t)}\big( \sqrt{1-b(t)g(t)}-1\big) +\frac{b'(t)g(t)}{2} - \frac{b(t)g'(t)}{2} \bigg) \\
&\approx \frac{1}{g(t)}\frac{1}{\sqrt{1-b(t)g(t)}}\bigg( \frac{b(t)g'(t)}{2} - \frac{b(t)g'(t)}{2} + \frac{b'(t)g(t)}{2} \bigg) \\
&= \frac{b'(t)}{2}\frac{1}{\sqrt{1-b(t)g(t)}},
\end{align*}
where we have used
\[ \sqrt{1-b(t)g(t)}-1 \approx -\frac{b(t)g(t)}{2} \qquad \text{for} \quad t\geq t_0 \,\,\, \text{sufficiently large}. \]
Thus, we may conclude that $f_2'(t)\geq0$ if $b'(t)\geq0$ or $f_2'(t)\leq0$ if $b'(t)\leq0$. That is, the separating line $\Pi_2$ for $t\geq t_0$ is increasing or decreasing, depending on the sign of $b'=b'(t)$. We note that for $t\leq t_0$ and $|\xi|\leq \Pi_2$ the set is compact. For this reason we can restrict ourselves for $\Pi_2$ to $t\geq t_0$.
In the WKB analysis we can restrict ourselves to $t_{\xi,2}$ for $t\geq t_0$ and to $t_{\xi,1}$ for $t\geq 0$.
 Let us define an auxiliary symbol
\begin{equation*}
\langle\xi\rangle_{b(t),g(t)} := \sqrt{\Big| |\xi|^2 - \Big( \frac{b(t)}{2}+\frac{g(t)|\xi|^2}{2} \Big)^2 \Big|}.
\end{equation*}
Thus, we may introduce the following two regions:
\begin{align} \label{Eq:Effective-Integrable-Pi-Hyp-Ell}
\Pi_{\text{hyp}}=\bigg\{(t,\xi): |\xi| > \frac{b(t)}{2}+\frac{g(t)|\xi|^2}{2} \bigg\}, \qquad \Pi_{\text{ell}}=\bigg\{(t,\xi): |\xi| < \frac{b(t)}{2}+\frac{g(t)|\xi|^2}{2} \bigg\}.
\end{align}
\begin{remark}
It holds
\begin{equation} \label{Eq:Effective-Integrable-derv-japan}
\partial_t\langle\xi\rangle_{b(t),g(t)} = \mp\dfrac{\big( \frac{b(t)}{2}+\frac{g(t)|\xi|^2}{2} \big)\big( \frac{b'(t)}{2}+\frac{g'(t)|\xi|^2}{2} \big)}{\langle\xi\rangle_{b(t),g(t)}},
\end{equation}
where the upper sign is taken in the hyperbolic region.
\end{remark}
We introduce
\[ \alpha_1 := \frac{1}{g(0)}\big( 1 + \sqrt{1-b(0)g(0)} \big) \qquad \text{and} \qquad \alpha_2 := \frac{1}{g(0)}\big( 1 - \sqrt{1-b(0)g(0)} \big). \]
We divide the extended phase space $[0,\infty)\times\mathbb{R}^n$ into zones in the following way:
\begin{itemize}
\item elliptic zone:
\[ Z_{\rm ell}(N) = \Big\{ (t,\xi): |\xi|\geq \frac{N}{2}\frac{1}{g(t)}\Big( 1+\sqrt{1-\frac{4}{N^2}b(t)g(t)} \Big)  \Big\}\cap \Pi_{\text{ell}}, \]
\item reduced zone:
\[ Z_{\rm red}(N,\varepsilon) = \Big\{ (t,\xi): \frac{\varepsilon}{2}\frac{1}{g(t)}\Big( 1+\sqrt{1-\frac{4}{\varepsilon^2}b(t)g(t)} \Big)\leq |\xi|\leq \frac{N}{2}\frac{1}{g(t)}\Big( 1+\sqrt{1-\frac{4}{N^2}b(t)g(t)} \Big) \Big\}, \]
\item hyperbolic zone:
\[ \Zhyp(\varepsilon,t_0) = \Big\{ (t,\xi): \frac{\varepsilon}{2}\frac{1}{g(t)}\Big( 1-\sqrt{1-\frac{4}{\varepsilon^2}b(t)g(t)} \Big)\leq|\xi|\leq \frac{\varepsilon}{2}\frac{1}{g(t)}\Big( 1+\sqrt{1-\frac{4}{\varepsilon^2}b(t)g(t)} \Big) \Big\}\cap\Pi_{\text{hyp}}\cap\{t\geq t_0\}, \]
\item dissipative zone:
\[ Z_{\rm diss}(\varepsilon) = \Big\{ (t,\xi): |\xi|\leq \frac{\varepsilon}{2}\frac{1}{g(t)}\big( 1-\sqrt{1-\frac{4}{\varepsilon^2}b(t)g(t)} \big) \Big\} \cap \Pi_{\text{ell}}. \]
\end{itemize}
Here in general, $N$ is a large positive constant and $\varepsilon$ is a small positive constant, which will be chosen later. It is clear that $\Zhyp(\varepsilon,t_0)$ should be between $\Pi_2$ and $\Pi_1$. So, we have to guarantee that
\begin{align*}
& \frac{1}{g(t)}\big(1- \sqrt{1-b(t)g(t)}\big)< \frac{\varepsilon}{2}\frac{1}{g(t)}\big( 1-\sqrt{1-\frac{4}{\varepsilon^2}b(t)g(t)} \big) \\
& \qquad < \frac{\varepsilon}{2}\frac{1}{g(t)}\Big( 1+\sqrt{1-\frac{4}{\varepsilon^2}b(t)g(t)} \Big) < \frac{1}{g(t)}\big(1+ \sqrt{1-b(t)g(t)}\big).
\end{align*}
In the definition of $\Zhyp(\varepsilon,t_0)$ we consider $t\geq t_0(\varepsilon)$ with $t_0=t_0(\varepsilon)$ is sufficiently large. Namely, we fix $\varepsilon = \varepsilon_1$ and choose sufficiently large $t\geq t_0$ such that $b(t)g(t)\leq \varepsilon_1^2/8$, which implies that $\lim_{t\to\infty}b(t)g(t) = 0$. Moreover for $t\leq t_0(\varepsilon)$ the region $\Pi_{\text{hyp}}(\varepsilon,t_0)$ forms a compact set. In the same way the ``left'' elliptic region forms for $t\leq t_0(\varepsilon)$ a compact set.

Let us introduce separating lines between these zones as follows:
\begin{itemize}
\item by $t_{\text{ell}}=t_{\text{ell}}(|\xi|)$, we denote the separating line between the zones $\Zell(N)$ and $\Zred(N,\varepsilon)$;
\item by $t_{\text{red}}=t_{\text{red}}(|\xi|)$, we denote the separating line between the zones $\Zred(N,\varepsilon)$ and $\Zhyp(\varepsilon,t_0)$;
\item by $t_{\text{hyp}}=t_{\text{hyp}}(|\xi|)$, we denote the separating line between the zones $\Zhyp(\varepsilon,t_0)$ and  $\Zdiss(\varepsilon)$.
\end{itemize}
%%%%%%%%%%%%%%%%%%%%%%%%%%%%%%%%%%%%%%%%%%%%%%%%%
{\color{black}
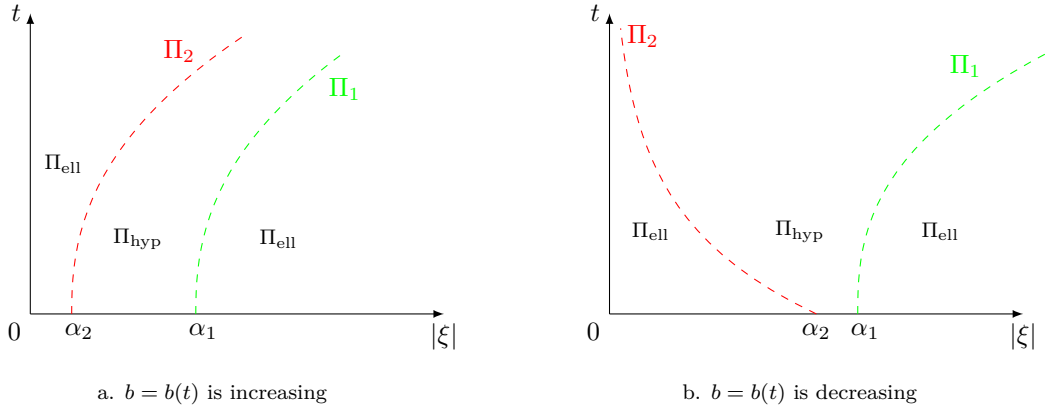
\begin{figure}[h]
\begin{center}
\begin{tikzpicture}[>=latex,xscale=1.1]
\draw[->] (0,0) -- (5,0)node[below]{$|\xi|$};
	\draw[->] (0,0) -- (0,4)node[left]{$t$};
	\node[below left] at(0,0){$0$};
    %%%%%%%%%%%%%%%%%%%%%%%%%%%%%%%%%%%%%%%%%%%%%%%%%%%%%%%%%%%%%%%%%%%%%%%%%
	\node  at (1.8,3.5) {$\textcolor{red}{\Pi_2}$};
	\node  at (3.8,3) {$\textcolor{green}{\Pi_1}$};
    %%%%%%%%%%%%%%%%%%%%%%%%%%%%%%%%%%%%%%%%%%%%%%%%%%%%%%%%%%%%%%%%%%%%%%%%
    \node[below] at(2.09,0){$\alpha_1$};
    \node[below] at(0.59,0){$\alpha_2$};
    %%%%%%%%%%%%%%%%%%%%%%%%%%%%%%%%%%%%%%%%%%%%%%%%%%%%%%%%%%%%%%%%%%%%%%
	\node[color=black] at (3, 1){{\footnotesize $\Pi_{\text{ell}}$}};
	%%%%%%%%%%%%%%%%%%%%  Draw Z_red  %%%%%%%%%%%%%%%%%%%%%%%%%%%%%%%%%%%%%%%%%
	\node[color=black] at (1.3, 1){{\footnotesize $\Pi_{\text{hyp}}$}};
	\draw[dashed, domain=0:3.5,color=green,variable=\t] plot ({2 + 0.09*pow(\t,2.4)},\t);
	%%%%%%%%%%%%%%%%%%%%% Draw Z_ell %%%%%%%%%%%%%%%%%%%%%%%%%%%%%%%%%%%%%%%%%%%
	\draw[dashed, domain=0:3.7,color=red,variable=\t] plot ({0.5 + 0.09*pow(\t,2.4)},\t);
	\node[color=black] at (0.4,2){{\footnotesize $\Pi_{\text{ell}}$}};
	%%%%%%
	\node[below] at (2.2,-.8) {{\footnotesize a. $b=b(t)$ is increasing}};
%%%%%%%%%%%%%%%%%%%%%%%%%%%%%%%%%%%%%%%%%%
\draw[->] (7,0) -- (12,0)node[below]{$|\xi|$};
	\draw[->] (7,0) -- (7,4)node[left]{$t$};
    \node[below left] at(7,0){$0$};
    %%%%%%%%%%%%%%%%%%%%%%%%%%%%%%%%%%%%%%%%%%%%%%%%%%%%%%%%
    \node[below] at(10.09,0){$\alpha_1$};
    \node[below] at(9.5,0){$\alpha_2$};
    %%%%%%%%%%%%%%%%%%%%%%%%%%%%%%%%%%%%%%%%%%%%%%%%%%%%%%%%%%
    \node[right]  at (11,3.3) {$\textcolor{green}{\Pi_1}$};
    \node[right]  at (7.1,3.7) {$\textcolor{red}{\Pi_2}$};
	%%%%%%%%%%%%%%%%%%%%%Draw Z red%%%%%%%%%%%%%%%%%%%%%%%%%%%%%%%%%%%%%%%%%%%%%%%%%%%%%%%%%%%%%%%%%
	\draw[dashed, domain=0:3.8,color=red,variable=\t] plot ({7+2.5*exp(-\t/1.3)},\t);
	\node[color=black] at (7.5,1.1){{\footnotesize $\Pi_{\text{ell}}$}};
	%%%%%%%%%%%%%%%%%%%%%%%%%%%%%%%%%%%%%%%%%%%%%%%%%%%%%%%%%%%%%%%
	\draw[dashed, domain=0:3.5,color=green,variable=\t] plot ({10 + 0.09*pow(\t,2.6)},\t);
	\node[color=black] at (11,1.1){{\footnotesize $\Pi_{\text{ell}}$}};
	%%%%%%%%%%%%%%%%%%%%%%%%%%%%%%%%%%%%%%%%%%%%%%%%%%%%%%%%%%%%%%%%%
	\node[color=black] at (9.3,1.1){{\footnotesize $\Pi_{\text{hyp}}$}};
	%%%%%
	\node[below] at (9.3,-.8) {{\footnotesize b. $b=b(t)$ is decreasing}};	
	%%%%%%%%%%%%%%%%%%%%%%%%%%%%%%%%%%%%%%%%%%%%%%%%%%%%%%%%%%%%%%%%%%%%%%%%%%
\end{tikzpicture}
\caption{Sketch of the regions for the case $b$ is monotonous and $g$ is integrable and decaying}
\label{Fig-Effective-integrable-regions}
\end{center}
\end{figure}
%%%%%%%%%%%%%%%%%%%%%%%%%%%%%%%%%%%%%%   END FIGURE   %%%%%%%%%%%%%%%%%%%%%%%%%%%%%%%%%%%%%%%%%
%%%%%%%%%%%%%%%%%%%%%%%%%%%%%%%%%%%   Draw the Picture   %%%%%%%%%%%%%%%%%%%%%%%%%%%%%%%%%%%%%%%%%%%%
\begin{figure}[h]
\begin{center}
\begin{tikzpicture}[>=latex,xscale=1.1]
    \draw[->] (0,0) -- (5,0)node[below]{$|\xi|$};
    \draw[->] (0,0) -- (0,4)node[left]{$t$};
    \node[below left] at(0,0){$0$};
    %%%%%%%%%%%%%%%%%%%%%%%%%%%%%%%%%%%%%%%%%%%%%%%%%%%%%%%%%%%%%%%%%%%%%%%%%%%
    \node[right] at (1.7,3.7) {$\textcolor{red}{t_{\text{hyp}}}$};
    \node[right] at (2.6,3.5) {$\textcolor{yellow}{t_{\text{red}}}$};
	%%%%%%%%%%%%%%%%%%%%%  Draw Gamma  %%%%%%%%%%%%%%%%%%%%%%%%%%%%%%%%%%%%%%%
    \draw[dashed, domain=0:3.5,color=cyan,variable=\t] plot ({2.1 + 0.09*pow(\t,2.4)},\t);
    %%%%%%%%%%%%%%%%%%%%%%%%%%%%%%%%%%%%%%%%%%%%%%%%%%%%%%%%%%%%%%%%%%%%%%%%
	\node[color=black] at (3.7, 1.5){{\footnotesize $Z_{\text{ell}}$}};
	\node[right] at (4,3.5) {$\textcolor{green}{t_{\text{ell}}}$};
	%%%%%%%%%%%%%%%%%%%%  Draw Z_red  %%%%%%%%%%%%%%%%%%%%%%%%%%%%%%%%%%%%%%%%%
	\node[color=black] at (2.3,1.5){{\footnotesize $Z_{\text{red}}$}};
	\draw[domain=0:3.6,color=green,variable=\t] plot ({3 + 0.09*pow(\t,2.4)},\t);
	%%%%%%%%%%%%%%%%%%%%% Draw Z_ell %%%%%%%%%%%%%%%%%%%%%%%%%%%%%%%%%%%%%%%%%%%
	\draw[domain=0:3.7,color=yellow,variable=\t] plot ({1.5 + 0.09*pow(\t,2.4)},\t);
	\node[color=black] at (1.3,1.5){{\footnotesize $Z_{\text{hyp}}$}};
	%%%%%%%%%%%%%%%%%%%%%%%%%
	\draw[domain=0:3.7,color=red,variable=\t] plot ({0.5 + 0.09*pow(\t,2.4)},\t);
	\node[color=black] at (0.3,1.5){{\footnotesize $Z_{\text{diss}}$}};
	%%%%%%%%%%%%%%%%%%%%%%%%%%%%%%%%%%%%%%%%%%%%%%%%%%%%%%%%%%%%%%%%%%%%%%%%%%%
	\node[below] at (2.2,-.8) {{\footnotesize a. $b=b(t)$ is increasing}};
%%%%%%%%%%%%%%%%%%%%%%%%%%%%%%%%%%%%%%%%%%%%%%%%%%%%%%%%%%%%%%%%%%%%%%%%%%%%%%%%%%
  \draw[->] (7,0) -- (12,0)node[below]{$|\xi|$};
    \draw[->] (7,0) -- (7,4)node[left]{$t$};
    \node[below left] at(7,0){$0$};
    %%%%%%%%%%%%%%%%%%%%%%%%%%%%%%%%%%%%%%%%%%%%%%%%%%%%%%%%%%%%%%%%%%%%%%%%%%%
    \node[right] at (7.1,3.7) {$\textcolor{red}{t_{\text{hyp}}}$};
    \node[right] at (9.6,3.5) {$\textcolor{yellow}{t_{\text{red}}}$};
	%%%%%%%%%%%%%%%%%%%%%  Draw Gamma  %%%%%%%%%%%%%%%%%%%%%%%%%%%%%%%%%%%%%%%
    \draw[dashed, domain=0:3.5,color=cyan,variable=\t] plot ({9.1 + 0.09*pow(\t,2.4)},\t);
    %%%%%%%%%%%%%%%%%%%%%%%%%%%%%%%%%%%%%%%%%%%%%%%%%%%%%%%%%%%%%%%%%%%%%%%%
	\node[color=black] at (10.7, 1.5){{\footnotesize $Z_{\text{ell}}$}};
	\node[right] at (11,3.5) {$\textcolor{green}{t_{\text{ell}}}$};
	%%%%%%%%%%%%%%%%%%%%  Draw Z_red  %%%%%%%%%%%%%%%%%%%%%%%%%%%%%%%%%%%%%%%%%
	\node[color=black] at (9.3,1.5){{\footnotesize $Z_{\text{red}}$}};
	\draw[domain=0:3.6,color=green,variable=\t] plot ({10 + 0.09*pow(\t,2.4)},\t);
	%%%%%%%%%%%%%%%%%%%%% Draw Z_ell %%%%%%%%%%%%%%%%%%%%%%%%%%%%%%%%%%%%%%%%%%%
	\draw[domain=0:3.7,color=yellow,variable=\t] plot ({8.5 + 0.09*pow(\t,2.4)},\t);
	\node[color=black] at (8,1.5){{\footnotesize $Z_{\text{hyp}}$}};
	%%%%%%%%%%%%%%%%%%%%%%%%%
	\draw[domain=0:3.8,color=red,variable=\t] plot ({7+1.3*exp(-\t/1.3)},\t);
	\node[color=black] at (7.3,0.8){{\footnotesize $Z_{\text{diss}}$}};
	%%%%%
	\node[below] at (9.3,-.8) {{\footnotesize b. $b=b(t)$ is decreasing}};
\end{tikzpicture}
\caption{Sketch of the zones for the case $b$ is monotonous and $g$ is integrable and decaying}
\label{Fig-Effective-integrable-zones}
\end{center}
\end{figure}
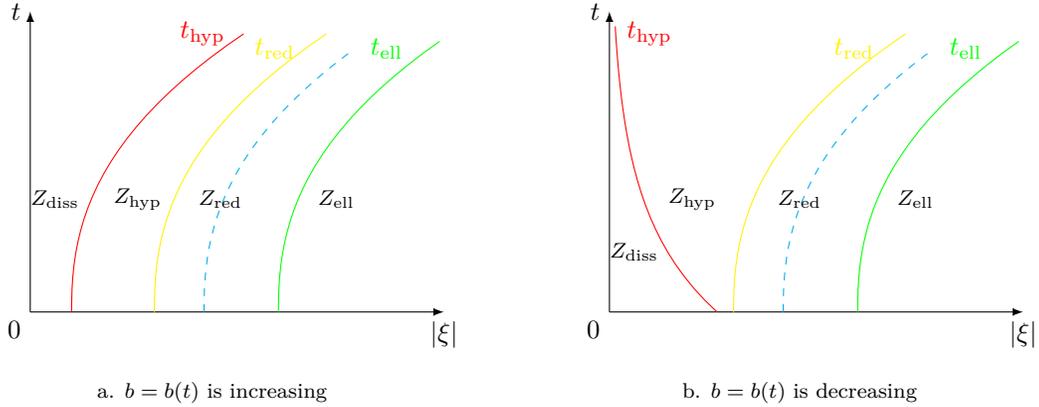
\subsubsection{Considerations in the hyperbolic zone $\Zhyp(\varepsilon,t_0)$} \label{Subsection_Effective_Integrable_Zhyp}
Due to the definition of the hyperbolic zone $\Zhyp(\varepsilon,t_0)$, we have the following estimates for $ \langle\xi\rangle_{b(t),g(t)}$:
\begin{align*}
\langle\xi\rangle_{b(t),g(t)}^2 &= |\xi|^2 - \Big( \frac{b(t)}{2}+\frac{g(t)|\xi|^2}{2} \Big)^2 \leq |\xi|^2, \\
\langle\xi\rangle_{b(t),g(t)}^2 &= |\xi|^2 - \Big( \frac{b(t)}{2}+\frac{g(t)|\xi|^2}{2} \Big)^2 \geq \Big( 1-\frac{\varepsilon^2}{4} \Big)|\xi|^2,
\end{align*}
where in the second inequality from the definition of $Z_{\rm hyp}(\varepsilon,t_0)$, namely, $|\xi|\leq \frac{\varepsilon}{2}\frac{1}{g(t)}\big( 1+\sqrt{1-\frac{4}{\varepsilon^2}b(t)g(t)} \big)$, we used
\begin{align*}
\frac{2g(t)|\xi|}{\varepsilon}-1 \leq \sqrt{1-\frac{4}{\varepsilon^2}b(t)g(t)} \qquad \text{implies} \qquad -\Big( \frac{g(t)|\xi|^2}{2} + \frac{b(t)}{2} \Big)^2 \geq -\frac{\varepsilon^2}{4}|\xi|^2.
\end{align*}
This implies that
\begin{equation*} \label{Eq:Effective-Integrable-Zhyp-japan}
\langle\xi\rangle_{b(t),g(t)} \approx |\xi|.
\end{equation*}
\begin{proposition} \label{Prop_Effective-Integrable-Zhyp}
The following estimates hold in $\Zhyp(\varepsilon,t_0)$ for all $t_{\text{red}}\leq s \leq t\leq t_{\text{hyp}}$ and $t\geq t_0(\varepsilon)$ with sufficiently large $t_0=t_0(\varepsilon)$:
\begin{align*}
|\xi|^{|\beta|}|\hat{u}(t,\xi)| \lesssim \exp\bigg( -\frac{1}{2}\int_s^t\big( b(\tau)+g(\tau)|\xi|^2 \big)d\tau \bigg)\Big( |\xi|^{|\beta|}|\hat{u}(s,\xi)| + |\xi|^{|\beta|-1}|\hat{u}_t(s,\xi)| \Big) \quad \mbox{for} \quad |\beta|\geq 1, \\
|\xi|^{|\beta|}|\hat{u}_t(t,\xi)| \lesssim \exp\bigg( -\frac{1}{2}\int_s^t\big( b(\tau)+g(\tau)|\xi|^2 \big)d\tau \bigg)\Big( |\xi|^{|\beta|+1}|\hat{u}(s,\xi)| + |\xi|^{|\beta|}|\hat{u}_t(s,\xi)| \Big) \quad \mbox{for} \quad |\beta|\geq 0.
\end{align*}
\end{proposition}
\begin{proof}
We define the micro-energy
\[ W(t,\xi) := \big(  \langle\xi\rangle_{b(t),g(t)}w,D_tw \big)^{\text{T}}. \]
Then, from \eqref{Eq:Effective-Integrable-Dfrom} it holds
\begin{equation*}
D_tW=\left( \begin{array}{cc}
0 & \langle\xi\rangle_{b(t),g(t)} \\
\langle\xi\rangle_{b(t),g(t)} & 0
\end{array} \right)W + \left( \begin{array}{cc}
\dfrac{D_t\langle\xi\rangle_{b(t),g(t)}}{\langle\xi\rangle_{b(t),g(t)}} & 0 \\
-\dfrac{g'(t)|\xi|^2+b'(t)}{2\langle\xi\rangle_{b(t),g(t)}} & 0
\end{array} \right)W.
\end{equation*}
Let us carry out the first step of the diagonalization procedure. First, we set
\[ M = \left( \begin{array}{cc}
1 & -1 \\
1 & 1
\end{array} \right), \qquad M^{-1} = \frac{1}{2}\left( \begin{array}{cc}
1 & 1 \\
-1 & 1
\end{array} \right) \qquad\text{and} \qquad W^{(0)}:=M^{-1}W.   \]
Thus, we obtain
\begin{equation*}
D_tW^{(0)}=\big( \mathcal{D}(t,\xi)+\mathcal{R}(t,\xi) \big)W^{(0)},
\end{equation*}
where
\begin{align*}
\mathcal{D}(t,\xi) &= \left( \begin{array}{cc}
\langle\xi\rangle_{b(t),g(t)} & 0 \\
0 & -\langle\xi\rangle_{b(t),g(t)}
\end{array} \right), \\
\mathcal{R}(t,\xi) &= \left( \begin{array}{cc}
\dfrac{D_t\langle\xi\rangle_{b(t),g(t)}}{2\langle\xi\rangle_{b(t),g(t)}}-\dfrac{g'(t)|\xi|^2+b'(t)}{4\langle\xi\rangle_{b(t),g(t)}} &
-\dfrac{D_t\langle\xi\rangle_{b(t),g(t)}}{2\langle\xi\rangle_{b(t),g(t)}}+\dfrac{g'(t)|\xi|^2+b'(t)}{4\langle\xi\rangle_{b(t),g(t)}} \\
-\dfrac{D_t\langle\xi\rangle_{b(t),g(t)}}{2\langle\xi\rangle_{b(t),g(t)}}-\dfrac{g'(t)|\xi|^2+b'(t)}{4\langle\xi\rangle_{b(t),g(t)}} & \dfrac{D_t\langle\xi\rangle_{b(t),g(t)}}{2\langle\xi\rangle_{b(t),g(t)}}+\dfrac{g'(t)|\xi|^2+b'(t)}{4\langle\xi\rangle_{b(t),g(t)}}
\end{array} \right).
\end{align*}
After the first step of diagonalization procedure the entries of the matrix $\mathcal{R}=\mathcal{R}(t,\xi)$ are uniformly integrable over the hyperbolic
zone $\Zhyp(\varepsilon,t_0)$. Namely, we have
\begin{align*}
\Big|\int_s^t \frac{D_\tau\langle\xi\rangle_{b(\tau),g(\tau)}}{2\langle\xi\rangle_{b(\tau),g(\tau)}} d\tau\Big| \leq \int_s^t\Big| \frac{D_\tau\langle\xi\rangle_{b(\tau),g(\tau)}}{2\langle\xi\rangle_{b(\tau),g(\tau)}} \Big|d\tau = \Big|\frac{1}{2}\log\frac{\langle\xi\rangle_{b(t),g(t)}}{\langle\xi\rangle_{b(s),g(s)}} \Big|
\end{align*}
and $\langle\xi\rangle_{b(t),g(t)}\approx|\xi|$ is uniformly in $\Zhyp(\varepsilon,t_0)$. Moreover, if $b'(t)\geq0$ we have
\begin{align*}
\Big| \int_s^t\frac{b'(\tau)+g'(\tau)|\xi|^2}{4\langle\xi\rangle_{b(\tau),g(\tau)}} d\tau\Big| &\leq \int_s^t \Big| \frac{b'(\tau)+g'(\tau)|\xi|^2}{4|\xi|}\Big| d\tau \\
& = -\int_s^t \frac{b'(\tau)+g'(\tau)|\xi|^2}{4|\xi|} d\tau + 2\int_s^t\frac{b'(\tau)}{4|\xi|}d\tau \\
& \leq  \frac{1}{4|\xi|}\big( b(s)+g(s)|\xi|^2 \big) + \frac{1}{2|\xi|}b(t) \lesssim 1.
\end{align*}
On the other hand, if $b'(t)\leq0$ we find
\begin{align*}
\Big| \int_s^t\frac{b'(\tau)+g'(\tau)|\xi|^2}{4\langle\xi\rangle_{b(\tau),g(\tau)}} d\tau\Big| &\leq \int_s^t \Big| \frac{b'(\tau)+g'(\tau)|\xi|^2}{4|\xi|}\Big| d\tau \\
& = -\int_s^t \frac{b'(\tau)+g'(\tau)|\xi|^2}{4|\xi|} d\tau \leq  \frac{1}{4|\xi|}\big( b(s)+g(s)|\xi|^2 \big) \lesssim 1,
\end{align*}
where in both cases we used the definition of $\Zhyp(\varepsilon,t_0)$.

We can write $W^{(1)}(t,\xi) = E_{\text{hyp}}(t,s,\xi)W^{(1)}(s,\xi)$, where $E_{\text{hyp}}=E_{\text{hyp}}(t,s,\xi)$ is the fundamental solution of the system
\begin{align*}
D_tE_{\text{hyp}}(t,s,\xi) = \big( \mathcal{D}(t,\xi) + \mathcal{R}(t,\xi) \big)E_{\text{hyp}}(t,s,\xi), \quad E_{\text{hyp}}(s,s,\xi) = I,
\end{align*}
for all $t\geq s$ and $(s,\xi)\in\Zhyp(\varepsilon,t_0)$. We may immediately obtain that
\[ |E_{\text{hyp}}(t,s,\xi)| \leq C \quad \text{for all} \quad t\geq s \quad \text{and} \quad (t,\xi),\,(s,\xi)\in\Zhyp(\varepsilon,t_0). \]
Finally, we find the following estimate for the micro-energy $W^{(1)}(t,\xi)$ in the hyperbolic zone $\Zhyp(\varepsilon,t_0)$:
\[ |W^{(1)}(t,\xi)| \lesssim |W^{(1)}(s,\xi)|,  \qquad \bigg|\left( \begin{array}{cc}
\langle\xi\rangle_{b(t),g(t)}w(t,\xi) \\
D_tw(t,\xi)
\end{array} \right)\bigg|\lesssim \bigg|\left( \begin{array}{cc}
\langle\xi\rangle_{b(s),g(s)}w(s,\xi) &  \\
D_tw(s,\xi)
\end{array} \right)\bigg| \]
for all $t\geq t_0(\varepsilon)$ and $(t,\xi),(s,\xi)\in \Zhyp(\varepsilon,t_0)$. The backward transformation
\[ \hat{u}(t,\xi) = \exp\bigg( -\frac{1}{2}\int_0^t\big( b(\tau)+g(\tau)|\xi|^2 \big)d\tau \bigg)w(t,\xi), \]
and the equivalence $\langle\xi\rangle_{b(t),g(t)}\approx |\xi|$ gives us the desired estimates.
\end{proof}
\subsubsection{Considerations in the reduced zone $\Zred(N,\varepsilon)$} \label{Subsection_Effective_Integrable_Zred}
We write the equation \eqref{MainEquationFourier} in the following form:
\begin{equation} \label{Eq:Effective-Integrable-Dform-Zred}
D_t^2\hat{u} - |\xi|^2\hat{u} - i\big( b(t) + g(t)|\xi|^2 \big)D_t\hat{u} = 0.
\end{equation}
We define the micro-energy $U=(|\xi|\hat{u},D_t\hat{u})^\text{T}$. Then, Eq. \eqref{Eq:Effective-Integrable-Dform-Zred} leads to the system
\begin{equation*} \label{Eq:Effective-Integrable-System-Zred}
D_tU=\underbrace{\left( \begin{array}{cc}
0 & |\xi| \\
|\xi| & i\big( b(t) + g(t)|\xi|^2 \big)
\end{array} \right)}_{A(t,\xi)}U.
\end{equation*}
We can define for all $(t,\xi) \in \Zred(N,\varepsilon)$ the following energy of the solutions to \eqref{Eq:Effective-Integrable-Dform-Zred}:
\begin{equation*}
\mathcal{E}(t,\xi) := \frac{1}{2}\big( |\xi|^2|\hat{u}(t,\xi)|^2 + |\hat{u}_t(t,\xi)|^2 \big).
\end{equation*}
If we differentiate the energy $\mathcal{E}=\mathcal{E}(t,\xi)$ with respect to $t$ and use our equation \eqref{Eq:Effective-Integrable-Dform-Zred}, it follows
\[ \frac{d}{dt}\mathcal{E}(t,\xi) = -\big( b(t)+g(t)|\xi|^2 \big)|\hat{u}_t(t,\xi)|^2 \leq 0. \]
This means that $\mathcal{E}=\mathcal{E}(t,\xi)$ is monotonically decreasing in $t$. Therefore, we have $\mathcal{E}(t,\xi)\leq \mathcal{E}(s,\xi)$ for $t_{\text{ell}}\leq s\leq t\leq t_{\text{red}}$ and $(s,\xi),(t,\xi)\in \Zred(N,\varepsilon)$. Thus, we may write
\begin{align*}
|\xi||\hat{u}(t,\xi)| &\leq \sqrt{\mathcal{E}(s,\xi)} \leq |\xi||\hat{u}(s,\xi)| + |\hat{u}_t(s,\xi)|, \\
|\hat{u}_t(t,\xi)| &\leq \sqrt{\mathcal{E}(s,\xi)} \leq |\xi||\hat{u}(s,\xi)| + |\hat{u}_t(s,\xi)|.
\end{align*}
\begin{corollary} \label{Cor:Effective-Integrable-Zred}
The following estimates hold with $t_{\text{ell}}\leq s\leq t\leq t_{\text{red}}$ and $(s,\xi), (t,\xi) \in \Zred(N,\varepsilon)$:
\begin{align*}
|\xi|^{|\beta|}|\hat{u}(t,\xi)| &\lesssim |\xi|^{|\beta|}|\hat{u}(s,\xi)| + |\xi|^{|\beta|-1}|\hat{u}_t(s,\xi)| \quad \mbox{for} \quad |\beta| \geq 1, \\
|\xi|^{|\beta|}|\hat{u}_t(t,\xi)| &\lesssim |\xi|^{|\beta|+1}|\hat{u}(s,\xi)| + |\xi|^{|\beta|}|\hat{u}_t(s,\xi)| \quad \mbox{for} \quad |\beta| \geq 0.
\end{align*}
\end{corollary}
\subsubsection{Considerations in the elliptic zone $\Zell(N)$} \label{Subsection_Effective_Integrable_Zell}
In $\Zell(N)$ we have the following estimates for $\langle\xi\rangle_{b(t),g(t)}$ with $N>2$:
\begin{align} \label{Eq:Effective-Integrable-Zell-japan-above}
\langle\xi\rangle_{b(t),g(t)}^2 &= \Big( \frac{b(t)}{2}+\frac{g(t)|\xi|^2}{2} \Big)^2 - |\xi|^2 \leq \Big( \frac{b(t)}{2}+\frac{g(t)|\xi|^2}{2} \Big)^2, \nonumber \\
\langle\xi\rangle_{b(t),g(t)}^2 &= \Big( \frac{b(t)}{2}+\frac{g(t)|\xi|^2}{2} \Big)^2 - |\xi|^2 \geq \Big( 1-\frac{4}{N^2} \Big)\Big( \frac{b(t)}{2}+\frac{g(t)|\xi|^2}{2} \Big)^2,
\end{align}
where in the second inequality from the definition of $\Zell(N)$, that is, from $|\xi|\geq \frac{N}{2}\frac{1}{g(t)}\big( 1+\sqrt{1-\frac{4}{N^2}b(t)g(t)} \big)$, we used
\begin{align*}
\frac{2g(t)}{N}|\xi|-1 \geq \sqrt{1-\frac{4}{N^2}b(t)g(t)} \qquad \text{implies} \qquad -|\xi|^2 \geq -\frac{4}{N^2}\Big( \frac{g(t)|\xi|^2}{2} + \frac{b(t)}{2} \Big)^2.
\end{align*}
Therefore, we may conclude that
\begin{equation} \label{Eq:Effective-Integrable-Zell-japan}
\langle\xi\rangle_{b(t),g(t)} \approx \frac{b(t)}{2}+\frac{g(t)|\xi|^2}{2}.
\end{equation}
We introduce the following family of symbol classes in the elliptic zone $\Zell(N)$.
\begin{definition} \label{Def:Effective-integrable-symbol}
A function $f=f(t,\xi)$ belongs to the elliptic symbol class $S_{\text{ell}}^\ell\{m_1,m_2\}$ if it holds
\begin{equation*}
|D_t^kf(t,\xi)|\leq C_{k}\langle\xi\rangle_{b(t),g(t)}^{m_1}\Big( \frac{g'(t)|\xi|^2+b'(t)}{b(t)+g(t)|\xi|^2} \Big)^{m_2+k}
\end{equation*}
for all $(t,\xi)\in \Zell(N)$ and all $k\leq \ell$.
\end{definition}
Thus, we may conclude the following rules from the definition of the symbol classes.
\begin{proposition} \label{Prop:Effective-integrable-symbol}
The following statements are true:
\begin{itemize}
\item $S_{\text{ell}}^\ell\{m_1,m_2\}$ is a vector space for all nonnegative integers $\ell$;
\item $S_{\text{ell}}^\ell\{m_1,m_2\}\cdot S_{\text{ell}}^{\ell}\{m_1',m_2'\}\hookrightarrow S^{\ell}_{\text{ell}}\{m_1+m_1',m_2+m_2'\}$;
\item $D_t^kS_{\text{ell}}^\ell\{m_1,m_2\}\hookrightarrow S_{\text{ell}}^{\ell-k}\{m_1,m_2+k\}$
for all nonnegative integers $\ell$ with $k\leq \ell$;
\item $S_{\text{ell}}^{0}\{-1,2\}\hookrightarrow L_{\xi}^{\infty}L_t^1\big( \Zell(N) \big)$.
\end{itemize}
\end{proposition}
\begin{proof}
Let us verify the last statement. Indeed, if $f=f(t,\xi)\in S_{\text{ell}}^{0}\{-1,2\}$, then using the estimates
\begin{align*}
&\frac{1}{\langle\xi\rangle_{b(t),g(t)}}\Big( \frac{b'(t)+g'(t)|\xi|^2}{b(t)+g(t)|\xi|^2} \Big)^2 \lesssim \Big| \frac{(b'(t)+g'(t)|\xi|^2)^2}{( b(t)+g(t)|\xi|^2)^3} \Big| \qquad \qquad \,\,\, \small{\text{(we used \eqref{Eq:Effective-Integrable-Zell-japan})}} \\
& \qquad \lesssim \frac{b'(t)^2+|b'(t)|\,|g'(t)|\,|\xi|^2+g'(t)^2|\xi|^4}{( b(t)+g(t)|\xi|^2)^3}  \\
& \qquad \lesssim \frac{b'(t)^2}{b(t)g^2(t)|\xi|^4} + \frac{|b'(t)|\,|g'(t)|\,|\xi|^2}{b(t)g(t)^2|\xi|^4} + \frac{g'(t)^2|\xi|^4}{g^3(t)|\xi|^6} \\
& \qquad \lesssim \frac{b(t)}{(1+t)^2g^2(t)|\xi|^4} + \frac{1}{1+t}\frac{|g'(t)|}{g^2(t)|\xi|^2} + \frac{g'(t)^2}{g(t)}\frac{1}{g^2(t)|\xi|^2}  \qquad \quad \small{\text{(we used condition \textbf{(B3)} for\,\,$b'(t)$)}} \\
& \qquad = \frac{b(t)g^2(t)}{(1+t)^2g^4(t)|\xi|^4} + \frac{|g'(t)|}{1+t}\frac{1}{g^2(t)|\xi|^2} + \frac{g'(t)^2}{g(t)}\frac{1}{g^2(t)|\xi|^2} \\
& \qquad \leq \frac{8}{N^4}\frac{g(t)}{(1+t)^2} + \frac{4}{N^2}\frac{|g'(t)|}{1+t} + \frac{4}{N^2}\frac{g'(t)^2}{g(t)}\quad \small{\text{(we used $b(t)g(t)\leq1/2$ from \textbf{(E3)} and $g(t)|\xi|\geq \frac{N}{2}$ from $\Zell(N)$)}},
\end{align*}
we obtain
\begin{align*}
\int_{0}^{t_{\text{ell}}}|f(\tau,\xi)|d\tau  & \lesssim  \int_{0}^{t_{\text{ell}}}\frac{g(\tau)}{(1+\tau)^2}d\tau + \int_{0}^{t_{\text{ell}}}\frac{-g'(\tau)}{1+\tau}d\tau + \int_{0}^{t_{\text{ell}}}\frac{g'(\tau)^2}{g(\tau)}d\tau \lesssim 1,
\end{align*}
where since $g\in L^1([0,\infty))$, the first and second integrals are also uniformly bounded. For the third integral we use condition \textbf{(E5)}.
\end{proof}
We consider the micro-energy
\[ W(t,\xi) := \big(  \langle\xi\rangle_{b(t),g(t)}w,D_tw \big)^{\text{T}}. \]
Then, from \eqref{Eq:Effective-Integrable-Dfrom} it holds
\begin{equation} \label{Eq:Effective-integrable-system-Zell}
D_tW=\left( \begin{array}{cc}
0 & \langle\xi\rangle_{b(t),g(t)} \\
-\langle\xi\rangle_{b(t),g(t)} & 0
\end{array} \right)W + \left( \begin{array}{cc}
\dfrac{D_t\langle\xi\rangle_{b(t),g(t)}}{\langle\xi\rangle_{b(t),g(t)}} & 0 \\
-\dfrac{b'(t)+g'(t)|\xi|^2}{2\langle\xi\rangle_{b(t),g(t)}} & 0
\end{array} \right)W.
\end{equation}
In the first step we use the diagonalizer of the first matrix as follows:
\[ M = \left( \begin{array}{cc}
i & -i \\
1 & 1
\end{array} \right), \qquad M^{-1}=\frac{1}{2}\left( \begin{array}{cc}
-i & 1 \\
i & 1
\end{array} \right). \]
Then, defining $W^{(0)}:=M^{-1}W$ we get the system
\begin{equation*}
D_tW^{(0)}=\big( \mathcal{D}(t,\xi)+\mathcal{R}(t,\xi) \big)W^{(0)},
\end{equation*}
where
\begin{align*}
\mathcal{D}(t,\xi) &= \left( \begin{array}{cc}
-i\langle\xi\rangle_{b(t),g(t)} & 0 \\
0 & i\langle\xi\rangle_{b(t),g(t)}
\end{array} \right)\in S_{\text{ell}}^{2}\{1,0\}, \\
\mathcal{R}(t,\xi) &= \frac{1}{2} \left( \begin{array}{cc}
\dfrac{D_t\langle\xi\rangle_{b(t),g(t)}}{\langle\xi\rangle_{b(t),g(t)}}-i\dfrac{b'(t)+g'(t)|\xi|^2}{2\langle\xi\rangle_{b(t),g(t)}} & -\dfrac{D_t\langle\xi\rangle_{b(t),g(t)}}{\langle\xi\rangle_{b(t),g(t)}}+i\dfrac{b'(t)+g'(t)|\xi|^2}{2\langle\xi\rangle_{b(t),g(t)}} \\
-\dfrac{D_t\langle\xi\rangle_{b(t),g(t)}}{\langle\xi\rangle_{b(t),g(t)}}-i\dfrac{b'(t)+g'(t)|\xi|^2}{2\langle\xi\rangle_{b(t),g(t)}} & \dfrac{D_t\langle\xi\rangle_{b(t),g(t)}}{\langle\xi\rangle_{b(t),g(t)}}+i\dfrac{b'(t)+g'(t)|\xi|^2}{2\langle\xi\rangle_{b(t),g(t)}}
\end{array} \right)\in S_{\text{ell}}^{1}\{0,1\}.
\end{align*}
We define $F_0(t,\xi):=\diag\mathcal{R}(t,\xi)$. The difference of the diagonal entries of the matrix $\mathcal{D}(t,\xi)+F_0(t,\xi)$ is
\begin{align*}
2\langle\xi\rangle_{b(t),g(t)} + \frac{b'(t)+g'(t)|\xi|^2}{2\langle\xi\rangle_{b(t),g(t)}} &= \frac{4\langle\xi\rangle_{b(t),g(t)}^2+b'(t)+g'(t)|\xi|^2}{2\langle\xi\rangle_{b(t),g(t)}} \\
& = \frac{4\big( \frac{b(t)}{2}+\frac{g(t)|\xi|^2}{2} \big)^2-4|\xi|^2+b'(t)+g'(t)|\xi|^2}{2\langle\xi\rangle_{b(t),g(t)}} \\
& \leq \frac{4\big( \frac{b(t)}{2}+\frac{g(t)|\xi|^2}{2} \big)^2-4|\xi|^2}{2\langle\xi\rangle_{b(t),g(t)}}= \frac{4\langle\xi\rangle_{b(t),g(t)}^2}{2\langle\xi\rangle_{b(t),g(t)}} \leq 2\langle\xi\rangle_{b(t),g(t)} =:i\delta(t,\xi),
\end{align*}
where $|b'(t)|=o(b^2(t))$ and $g'(t)<0$.\\

We introduce a matrix $N^{(1)}=N^{(1)}(t,\xi)$ such that
\begin{align*}
N^{(1)}(t,\xi) &= \left( \begin{array}{cc}
0 & -\dfrac{\mathcal{R}_{12}}{\delta(t,\xi)} \\
\dfrac{\mathcal{R}_{21}}{\delta(t,\xi)} & 0
\end{array} \right) \\
&= \left( \begin{array}{cc}
0 & i\dfrac{D_t\langle\xi\rangle_{b(t),g(t)}}{4\langle\xi\rangle_{b(t),g(t)}^2}-\dfrac{b'(t)+g'(t)|\xi|^2}{8\langle\xi\rangle_{b(t),g(t)}^2} \\
i\dfrac{D_t\langle\xi\rangle_{b(t),g(t)}}{4\langle\xi\rangle_{b(t),g(t)}^2}+\dfrac{b'(t)+g'(t)|\xi|^2}{8\langle\xi\rangle_{b(t),g(t)}^2} & 0
\end{array} \right),
\end{align*}
where $N^{(1)}(t,\xi)\in S_{\text{ell}}^{1}\{1,1\}$ and $N_1(t,\xi)=I+N^{(1)}(t,\xi)= S_{\text{ell}}^{1}\{0,0\}$. For a sufficiently large zone constant $N$ and all $t\geq t_{\text{diss}}$ the matrix $N_1=N_1(t,\xi)$ is invertible with uniformly bounded inverse matrix $N_1^{-1}=N_1^{-1}(t,\xi)$. Indeed, using \eqref{Eq:Effective-Integrable-derv-japan}, it holds
\begin{align*}
& \Big|\dfrac{D_t\langle\xi\rangle_{b(t),g(t)}}{4\langle\xi\rangle_{b(t),g(t)}^2} + \dfrac{b'(t)+g'(t)|\xi|^2}{8\langle\xi\rangle_{b(t),g(t)}^2} \Big| \leq \dfrac{\big( \frac{b(t)}{2}+\frac{g(t)|\xi|^2}{2} \big)\big( \frac{|b'(t)|}{2}+\frac{|g'(t)|\,|\xi|^2}{2} \big)}{4\langle\xi\rangle_{b(t),g(t)}^3} + \dfrac{|b'(t)|+|g'(t)|\,|\xi|^2}{8\langle\xi\rangle_{b(t),g(t)}^2} \\
& \qquad \quad \leq \frac{1}{\sqrt{1-\frac{4}{N^2}}}\dfrac{|b'(t)|+|g'(t)|\,|\xi|^2}{8\langle\xi\rangle_{b(t),g(t)}^2} + \dfrac{|b'(t)|+|g'(t)|\,|\xi|^2}{16\langle\xi\rangle_{b(t),g(t)}^2} \\
& \qquad \quad \leq \frac{1}{2}\frac{1}{\sqrt{1-\frac{4}{N^2}}}\frac{1}{1-\frac{4}{N^2}} \dfrac{|b'(t)|+|g'(t)|\,|\xi|^2}{\big( b(t)+g(t)|\xi|^2 \big)^2} + \frac{1}{4\big( 1-\frac{4}{N^2} \big)} \dfrac{|b'(t)|+|g'(t)|\,|\xi|^2}{\big( b(t)+g(t)|\xi|^2 \big)^2} \\
& \qquad \quad \leq \frac{1}{2}\bigg( \frac{1}{\big( 1-\frac{4}{N^2} \big)^{\frac{3}{2}}} + \frac{1}{1-\frac{4}{N^2}} \bigg)\Big( \frac{|b'(t)|}{b^2(t)} + \frac{|g'(t)|\xi|^2}{g^2(t)||\xi|^4}  \Big) \\
& \qquad \quad \leq \frac{1}{2} \bigg( \frac{1}{\big( 1-\frac{4}{N^2} \big)^{\frac{3}{2}}} + \frac{1}{1-\frac{4}{N^2}} \bigg)\Big( a + \frac{|g'(t)|}{g^2(t)||\xi|^2} \Big) \\
& \qquad \quad \leq \frac{1}{2}\bigg( \frac{1}{\big( 1-\frac{4}{N^2} \big)^{\frac{3}{2}}} + \frac{1}{1-\frac{4}{N^2}} \bigg)\Big( a + \frac{4}{N^2}(-g'(t)) \Big) < 1,
\end{align*}
where for the estimate of $\langle\xi\rangle_{b(t),g(t)}$ we used \eqref{Eq:Effective-Integrable-Zell-japan-above}, $|b'(t)|\leq ab^2(t)$ due to condition \textbf{(EF)} and, $g(t)|\xi|\geq \frac{N}{2}$ due to the definition of $\Zell(N)$.\\
Let
\begin{align*}
B^{(1)}(t,\xi) &= D_tN^{(1)}(t,\xi)-( \mathcal{R}(t,\xi)-F_0(t,\xi))N^{(1)}(t,\xi), \\
\mathcal{R}_1(t,\xi) &= -N_1^{-1}(t,\xi)B^{(1)}(t,\xi)\in S_{\text{ell}}^0\{-1,2\}.
\end{align*}
Then, we have the following operator identity:
\begin{equation*}
( D_t-\mathcal{D}(t,\xi)-\mathcal{R}(t,\xi))N_1(t,\xi)=N_1(t,\xi)( D_t-\mathcal{D}(t,\xi)-F_0(t)-\mathcal{R}_1(t,\xi)).
\end{equation*}
\begin{proposition} \label{Prop_Effective-Integrable_EstZell}
The fundamental solution $E_{\text{ell}}^{W}=E_{\text{ell}}^{W}(t,s,\xi)$ to the transformed operator
\[ D_t-\mathcal{D}(t,\xi)-F_0(t,\xi)-\mathcal{R}_1(t,\xi) \]
can be estimated by
\begin{equation*}
\big( |E_{\text{ell}}^{W}(t,s,\xi)| \big) \lesssim \frac{b(t)+g(t)|\xi|^2}{b(s)+g(s)|\xi|^2}\exp\bigg( \frac{1}{2}\int_{s}^{t}\big( g(\tau)|\xi|^2 +b(\tau) \big)d\tau \bigg)
\left( \begin{array}{cc}
1 & 1 \\
1 & 1
\end{array} \right),
\end{equation*}
with $(t,\xi),(s,\xi)\in \Zell(N)$, $t_0\leq s\leq t\leq t_{\text{ell}}$.
\end{proposition}
\begin{proof}
We transform the system for $E_{\text{ell}}^{W}=E_{\text{ell}}^{W}(t,s,\xi)$ to an integral equation for a new matrix-valued function $\mathcal{Q}_{\text{ell}}=\mathcal{Q}_{\text{ell}}(t,s,\xi)$. If we differentiate the term
\[ \exp \bigg\{ -i\int_{s}^{t}\big( \mathcal{D}(\tau,\xi)+F_0(\tau,\xi) \big)d\tau \bigg\}E_{\text{ell}}^{W}(t,s,\xi) \]
and, then integrate on $[s,t]$, we find that $E_{\text{ell}}^{W}=E_{\text{ell}}^{W}(t,s,\xi)$ satisfies the following integral equation:
\begin{align*}
E_{\text{ell}}^{W}(t,s,\xi) & = \exp\bigg\{ i\int_{s}^{t}\big( \mathcal{D}(\tau,\xi)+F_0(\tau,\xi) \big)d\tau \bigg\}E_{\text{ell}}^{W}(s,s,\xi)\\
& \quad + i\int_{s}^{t} \exp \bigg\{ i\int_{\theta}^{t}\big( \mathcal{D}(\tau,\xi)+F_0(\tau,\xi) \big)d\tau \bigg\}\mathcal{R}_1(\theta,\xi)E_{\text{ell}}^{W}(\theta,s,\xi)\,d\theta.
\end{align*}
We define
\[ \mathcal{Q}_{\text{ell}}(t,s,\xi)=\exp\bigg\{ -\int_{s}^{t}\beta(\tau,\xi)d\tau \bigg\} E_{\text{ell}}^{W}(t,s,\xi), \]
with a suitable $\beta=\beta(t,\xi)$ which will be fixed later. Then, this new term satisfies the new integral equation
\begin{align*}
\mathcal{Q}_{\text{ell}}(t,s,\xi)=&\exp \bigg\{ \int_{s}^{t}\big( i\mathcal{D}(\tau,\xi)+iF_0(\tau,\xi)-\beta(\tau,\xi)I \big)d\tau \bigg\}\\
& \quad + \int_{s}^{t} \exp \bigg\{ \int_{\theta}^{t}\big( i\mathcal{D}(\tau,\xi)+iF_0(\tau,\xi)-\beta(\tau,\xi)I \big)d\tau \bigg\}\mathcal{R}_1(\theta,\xi)\mathcal{Q}_{\text{ell}}(\theta,s,\xi)\,d\theta.
\end{align*}
The function $\mathcal{R}_1=\mathcal{R}_1(\theta,\xi)$ is uniformly integrable over the elliptic zone because of Proposition \ref{Prop:Effective-integrable-symbol}. Hence, if the exponential term is bounded, then the solution $\mathcal{Q}_{\text{ell}}=\mathcal{Q}_{\text{ell}}(t,s,\xi)$ of the integral equation is uniformly bounded over the elliptic zone for a suitable weight $\beta=\beta(t,\xi)$.

The main entries of the diagonal matrix $i\mathcal{D}(t,\xi)+iF_0(t,\xi)$ are given by
\begin{align*}
(I) &= \langle\xi\rangle_{b(t),g(t)} + \dfrac{D_t\langle\xi\rangle_{b(t),g(t)}}{2\langle\xi\rangle_{b(t),g(t)}}+\dfrac{b'(t)+g'(t)|\xi|^2}{4\langle\xi\rangle_{b(t),g(t)}},\\
(II) &= -\langle\xi\rangle_{b(t),g(t)} + \dfrac{D_t\langle\xi\rangle_{b(t),g(t)}}{2\langle\xi\rangle_{b(t),g(t)}}-\dfrac{b'(t)+g'(t)|\xi|^2}{4\langle\xi\rangle_{b(t),g(t)}}.
\end{align*}
We may see that the term $(I)$ is dominant in $\Zell(N)$ with respect to $(II)$ for $t\geq t_{0}$. Therefore, we choose the weight $\beta=\beta(t,\xi)=(I)$. By this choice, we get
\[ i\mathcal{D}(\tau,\xi)+iF_0(\tau,\xi)-\beta(\tau,\xi)I = \left( \begin{array}{cc}
0 & 0 \\
0 & -2\langle\xi\rangle_{b(t),g(t)}-\dfrac{b'(t)+g'(t)|\xi|^2}{2\langle\xi\rangle_{b(t),g(t)}}
\end{array} \right). \]
It follows
\begin{align*}
H(t,s,\xi) & =\exp \bigg\{ \int_{s}^{t}\big( i\mathcal{D}(\tau,\xi)+iF_0(\tau,\xi)-\beta(\tau,\xi)I \big)d\tau \bigg\}\\
& = \diag \bigg( 1, \exp \bigg\{ \int_{s}^{t}\Big( -2\langle\xi\rangle_{b(t),g(t)}-\dfrac{b'(\tau)+g'(\tau)|\xi|^2}{2\langle\xi\rangle_{b(\tau),g(\tau)}} \Big)d\tau \bigg\} \bigg)\rightarrow \left( \begin{array}{cc}
1 & 0 \\
0 & 0
\end{array} \right)
\end{align*}
as $t\rightarrow \infty$ for any fixed $0\leq s\leq t\leq t_{\text{ell}}$. Hence, the matrix $H=H(t,s,\xi)$ is uniformly bounded for $(s,\xi),(t,\xi)\in \Zell(N)$. So, the representation of $\mathcal{Q}_{\text{ell}}=\mathcal{Q}_{\text{ell}}(t,s,\xi)$ by a Neumann series gives
\begin{align*}
\mathcal{Q}_{\text{ell}}(t,s,\xi)=H(t,s,\xi)+\sum_{k=1}^{\infty}i^k\int_{s}^{t}H(t,t_1,\xi)\mathcal{R}_1(t_1,\xi)&\int_{s}^{t_1}H(t_1,t_2,\xi)\mathcal{R}_1(t_2,\xi) \\
& \cdots \int_{s}^{t_{k-1}}H(t_{k-1},t_k,\xi)\mathcal{R}_1(t_k,\xi)dt_k\cdots dt_2dt_1.
\end{align*}
Then, this series is convergent, since $\mathcal{R}_1=\mathcal{R}_1(t,\xi)$ is uniformly integrable over $\Zell(N)$ due to the last item of Proposition \ref{Prop:Effective-integrable-symbol}. Hence, from the last considerations we may conclude
\begin{align*}
E_{\text{ell}}^{W}(t,s,\xi)&=\exp \bigg\{ \int_{s}^{t}\beta(\tau,\xi)d\tau \bigg\}\mathcal{Q}_{\text{ell}}(t,s,\xi)\\
& = \exp \bigg\{ \int_{s}^{t}\bigg( \langle\xi\rangle_{b(\tau),g(\tau)}+\frac{D_\tau\langle\xi\rangle_{b(\tau),g(\tau)}}{2\langle\xi\rangle_{b(\tau),g(\tau)}}+\frac{( b'(\tau)+g'(\tau)|\xi|^2}{4\langle\xi\rangle_{b(\tau),g(\tau)}} \bigg)d\tau \bigg\}\mathcal{Q}_{\text{ell}}(t,s,\xi) \\
& \leq \exp \bigg\{ \int_{s}^{t}\bigg( \langle\xi\rangle_{b(\tau),g(\tau)}+\frac{D_\tau\langle\xi\rangle_{b(\tau),g(\tau)}}{2\langle\xi\rangle_{b(\tau),g(\tau)}}+\frac{( b'(\tau)+g'(\tau)|\xi|^2}{2\big( b(\tau)+g(\tau)|\xi|^2 \big)} \bigg)d\tau \bigg\}\mathcal{Q}_{\text{ell}}(t,s,\xi) \\
& = \Big( \frac{\langle\xi\rangle_{b(t),g(t)}}{\langle\xi\rangle_{b(s),g(s)}} \Big)^{\frac{1}{2}}\Big( \frac{b(t)+g(t)|\xi|^2}{b(s)+g(s)|\xi|^2} \Big)^{\frac{1}{2}} \exp \bigg\{ \int_{s}^{t}\langle\xi\rangle_{b(\tau),g(\tau)}d\tau \bigg\}\mathcal{Q}_{\text{ell}}(t,s,\xi),
\end{align*}
where $\mathcal{Q}_{\text{ell}}=\mathcal{Q}_{\text{ell}}(t,s,\xi)$ is a uniformly bounded matrix. Then, it follows
\begin{align*}
(|E_{\text{ell}}^{W}(t,s,\xi)|) & \lesssim \frac{b(t)+g(t)|\xi|^2}{b(s)+g(s)|\xi|^2}\exp \bigg\{ \int_{s}^{t}\langle\xi\rangle_{b(\tau),g(\tau)}d\tau \bigg\}|\mathcal{Q}_{\text{ell}}(t,s,\xi)| \left( \begin{array}{cc}
1 & 1 \\
1 & 1
\end{array} \right) \\
& \lesssim \frac{b(t)+g(t)|\xi|^2}{b(s)+g(s)|\xi|^2}\exp \bigg\{ \int_{s}^{t}\langle\xi\rangle_{b(\tau),g(\tau)}d\tau \bigg\}\left( \begin{array}{cc}
1 & 1 \\
1 & 1
\end{array} \right).
\end{align*}
This completes the proof.
\end{proof}
%%%%%%%%%%%%%%%%%%%%%%%%%%%%%%%%%%%%
\textbf{Transforming back to the original Cauchy problem.} Now we want to obtain an estimate for the energy of the solution to our original Cauchy problem. For this reason we need to transform back to get an estimate of the fundamental solution $E_{\text{ell}}=E_{\text{ell}}(t,s,\xi)$ which is related to a system of first order for the micro-energy $(|\xi|\hat{u},D_t\hat{u}).$
\begin{lemma} \label{Lemma_Effective-Integrable_Transf-back}
Under the conditions \textbf{(B1)} to \textbf{(B3)} the following holds:
\begin{itemize}
\item [1.] in the elliptic zone it holds $\langle\xi\rangle_{b(t),g(t)}-\dfrac{b(t)+g(t)|\xi|^2}{2} \leq -\dfrac{|\xi|^2}{b(t)+g(t)|\xi|^2}$.
\item [2.] $\dfrac{\lambda(s,\xi)}{\lambda(t,\xi)}\exp \Big( \displaystyle \int_{s}^{t}\langle\xi\rangle_{b(\tau),g(\tau)}d\tau \Big)\leq \exp
\Big( -|\xi|^2\displaystyle \int_{s}^{t}\dfrac{1}{b(\tau)+g(\tau)|\xi|^2}d\tau \Big)$,\\
where $\lambda=\lambda(t,\xi)=\exp\Big( \dfrac{1}{2}\displaystyle\int_{0}^{t}\big( b(\tau)+g(\tau)|\xi|^2 \big)d\tau \Big)$.
%\item [3.] $\exp\Big( -C|\xi|^2\displaystyle\int_{t_{\text{ell}}}^t\frac{1}{b(\tau)+g(\tau)|\xi|^2}d\tau \Big) \lesssim 1$,
\end{itemize}	
\end{lemma}
\begin{proof} Using the elementary inequality
\[ \sqrt{x+y}\leq \sqrt{x}+\frac{y}{2\sqrt{x}} \]
for any $x\geq 0$ and $y\geq -x$, the first statement is equivalent to the following inequality:
\[ \sqrt{\Big( \dfrac{b(t)+g(t)|\xi|^2}{2} \Big)^2-|\xi|^2} - \dfrac{b(t)+g(t)|\xi|^2}{2} \leq -\frac{|\xi|^2}{b(t)+g(t)|\xi|^2}. \]
The second statement follows directly from the first one together with the definition of $\lambda=\lambda(t,\xi)$.
\end{proof}
From Proposition \ref{Prop_Effective-Integrable_EstZell}, for $(t,\xi), (s,\xi)\in \Zell(N)$ we have the estimate
\[ (|E_{\text{ell}}^W(t,s,\xi)|) \lesssim \frac{b(t)+g(t)|\xi|^2}{b(s)+g(s)|\xi|^2}\exp \Big( \int_{s}^{t}\langle\xi\rangle_{b(\tau),g(\tau)}d\tau \Big)
\left( \begin{array}{cc}
1 & 1 \\
1 & 1
\end{array} \right). \]
This yields the estimate
\begin{align} \label{Eq:Effective-Integrable_Transf-back}
\nonumber
(|E_{\text{ell}}(t,s,\xi)|) & \lesssim \left( \begin{array}{cc}
|\xi| & 0 \\  [5pt] \nonumber
b(t)+g(t)|\xi|^2 & b(t)+g(t)|\xi|^2
\end{array} \right) \exp \bigg( \int_{s}^{t}\Big( \langle\xi\rangle_{b(\tau),g(\tau)}-\frac{b(\tau)+g(\tau)|\xi|^2}{2} \Big)d\tau \bigg) \\
& \qquad \times
\left( \begin{array}{cc}
1 & 1 \\
1 & 1
\end{array} \right)
\left( \begin{array}{cc}
\frac{1}{|\xi|} & 0 \\ [5pt]
\frac{1}{|\xi|} & \frac{1}{b(s)+g(s)|\xi|^2}
\end{array} \right) \nonumber \\
& \,\,\, \lesssim \exp \Big( -|\xi|^2\int_{s}^{t}\frac{1}{b(\tau)+g(\tau)|\xi|^2}d\tau \Big) \left( \begin{array}{cc}
1 & \frac{|\xi|}{b(s)+g(s)|\xi|^2} \\ [5pt]
\frac{b(t)+g(t)|\xi|^2}{|\xi|} & \frac{b(t)+g(t)|\xi|^2}{b(s)+g(s)|\xi|^2}
\end{array} \right),
\end{align}
where we used Lemma \ref{Lemma_Effective-Integrable_Transf-back}. \\

\textbf{A refined estimate for the fundamental solution in the elliptic zone $\Zell(N)$}
\begin{proposition} \label{Prop:Effective-Integrable_IncreasingZell}
The fundamental solution $E_{\text{ell}}=E_{\text{ell}}(t,s,\xi)$ satisfies the following estimate:
\begin{align*}
(|E_{\text{ell}}(t,s,\xi)|) &\lesssim \exp \bigg( -|\xi|^2\int_{s}^{t}\frac{1}{b(\tau)+g(\tau)|\xi|^2}d\tau \bigg)
\left( \begin{array}{cc}
1 & \frac{|\xi|}{b(s)+g(s)|\xi|^2} \\ [5pt]
\frac{|\xi|}{b(t)+g(t)|\xi|^2} & \frac{|\xi|^2}{(b(s)+g(s)|\xi|^2)(b(t)+g(t)|\xi|^2)}
\end{array} \right) \\
& \qquad + \exp \bigg( -\int_{s}^{t} \big( b(\tau)+g(\tau)|\xi|^2 \big)d\tau \bigg)
\left( \begin{array}{cc}
0 & 0 \\
0 & 1
\end{array} \right)
\end{align*}
for all $t\geq s$ and $(t,\xi), (s,\xi)\in\Zell(N)$.
\end{proposition}
\begin{proof} Let us assume that $\Phi_k=\Phi_k(t,s,\xi)$, $k=1,2,$ are solutions to the equation
\[ \Phi_{tt}+|\xi|^2\Phi+\big( b(t) + g(t)|\xi|^2 \big)\Phi_t=0 \]
with initial values $\Phi_k(s,s,\xi)=\delta_{1k}$ and $\partial_t\Phi_k(s,s,\xi)=\delta_{2k}$. Then, we have
\[ \left( \begin{array}{cc}
    |\xi|w(t,\xi) \\
	D_tw(t,\xi)
	\end{array} \right)=\left( \begin{array}{cc}
	\Phi_1(t,s,\xi) & i|\xi|\Phi_2(t,s,\xi) \\ [5pt]
	\dfrac{D_t\Phi_1(t,s,\xi)}{|\xi|} & iD_t\Phi_2(t,s,\xi)
	\end{array} \right) \left( \begin{array}{cc}
	|\xi|w(s,\xi) \\
	D_tw(s,\xi)
	\end{array} \right). \]
Our idea is to relate the entries of the above given estimates to the multipliers $\Phi_k=\Phi_k(t,s,\xi)$ and use Duhamel's formula to improve the estimates for the second row using estimates from the first one. Hence, if we compare with the estimates \eqref{Eq:Effective-Integrable_Transf-back}, then we obtain
\begin{align*}
|\Phi_1(t,s,\xi)| &\lesssim \exp \Big( -|\xi|^2\int_{s}^{t}\frac{1}{b(\tau)+g(\tau)|\xi|^2}d\tau \Big), \\
|\Phi_2(t,s,\xi)| &\lesssim \frac{1}{b(s)+g(s)|\xi|^2}\exp \Big( -|\xi|^2\int_{s}^{t}\frac{1}{b(\tau)+g(\tau)|\xi|^2}d\tau \Big), \\
|\partial_t\Phi_1(t,s,\xi)| &\lesssim \big( b(t)+g(t)|\xi|^2 \big)\exp \Big( -|\xi|^2\int_{s}^{t}\frac{1}{b(\tau)+g(\tau)|\xi|^2}d\tau \Big), \\
\big| \partial_t\Phi_2(t,s,\xi) \big| &\lesssim \frac{b(t)+g(t)|\xi|^2}{b(s)+g(s)|\xi|^2}\exp \Big( -|\xi|^2\int_{s}^{t}\frac{1}{b(\tau)+g(\tau)|\xi|^2}d\tau \Big).
\end{align*}
Let $\Psi_k=\Psi_k(t,s,\xi)=\partial_t\Phi_k(t,s,\xi), \,k=1,2$. Then, we obtain the equations of first order
\[ \partial_t\Psi_k+\big( b(t) + g(t)|\xi|^2 \big)\Psi_k=-|\xi|^2\Phi_k, \quad \Psi_k(s,s,\xi)=\delta_{2k}. \]
Standard calculations lead to
\begin{align*}
\Psi_1(t,s,\xi) &= -|\xi|^2\int_{s}^{t}\frac{\lambda^2(\tau,\xi)}{\lambda^2(t,\xi)}\Phi_1(\tau,s,\xi)d\tau, \\
\Psi_2(t,s,\xi) &= \frac{\lambda^2(s,\xi)}{\lambda^2(t,\xi)}-|\xi|^2\int_{s}^{t}\frac{\lambda^2(\tau,\xi)}{\lambda^2(t,\xi)}\Phi_2(\tau,s,\xi)d\tau.
\end{align*}
If we are able to derive the desired estimate for $|\Psi_1(t,s,\xi)|$, then we conclude immediately the desired estimate for $|\Psi_2(t,s,\xi)|$. Using the estimates for $|\Phi_1(t,s,\xi)|$ and applying partial integration we get
\begin{align*}
|\Psi_1(t,s,\xi)| &\lesssim \frac{|\xi|^2}{\lambda^2(t,\xi)}\int_{s}^{t}\lambda^2(\tau,\xi)\exp \Big( -|\xi|^2\int_{s}^{\tau}\frac{1}{b(\theta)+g(\theta)|\xi|^2}d\theta \Big)d\tau \\
& \lesssim \frac{|\xi|^2}{\lambda^2(t,\xi)}\int_{s}^{t}\frac{1}{b(\tau)+g(\tau)|\xi|^2}\exp \Big( -|\xi|^2\int_{s}^{\tau}\frac{1}{b(\theta)+g(\theta)|\xi|^2}d\theta \Big)\partial_\tau\lambda^2(\tau,\xi)d\tau \\
& \lesssim \frac{|\xi|^2}{\lambda^2(t,\xi)}\bigg( \lambda^2(\tau,\xi)\frac{1}{b(\tau)+g(\tau)|\xi|^2}\exp \Big( -|\xi|^2\int_{s}^{\tau}\frac{1}{b(\theta)+g(\theta)|\xi|^2}d\theta \Big) \bigg)\Big|_s^t \\
& \,\,\,\, + \frac{|\xi|^2}{\lambda^2(t,\xi)}\int_{s}^{t}\lambda^2(\tau,\xi)\bigg( \underbrace{\frac{b'(\tau)-g'(\tau)|\xi|^2}{\big( b(\tau)+g(\tau)|\xi|^2 \big)^2}+\frac{|\xi|^2}{\big( b(\tau)+g(\tau)|\xi|^2 \big)^2}}_{\lesssim C(\tau)\leq C_0<1} \bigg) \\
& \qquad \times \exp\Big( -|\xi|^2\int_{s}^{\tau}\frac{1}{b(\theta)+g(\theta)|\xi|^2}d\theta \Big)d\tau \\
& \lesssim \frac{|\xi|^2}{b(t)+g(t)|\xi|^2} \exp\Big( -|\xi|^2\int_{s}^{t}\frac{1}{b(\tau)+g(\tau)|\xi|^2}d\tau \Big)-\frac{|\xi|^2}{b(s)+g(s)|\xi|^2}\frac{\lambda^2(s,\xi)}{\lambda^2(t,\xi)},
\end{align*}
where we have used
\begin{align*}
\frac{b'(\tau)-g'(\tau)|\xi|^2}{\big( b(\tau)+g(\tau)|\xi|^2 \big)^2}+\frac{|\xi|^2}{\big( b(\tau)+g(\tau)|\xi|^2 \big)^2} &\leq \frac{b'(\tau)}{b^2(\tau)} + \frac{-g'(\tau)|\xi|^2}{\big( b(\tau)+g(\tau)|\xi|^2 \big)^2} + \frac{|\xi|^2}{\big( b(\tau)+g(\tau)|\xi|^2 \big)^2} \\
& \leq a + \frac{-g'(t)}{g^2(t)|\xi|^2} + \frac{1}{g^2(t)|\xi|^2} \leq a + \frac{4}{N^2}\big( -g'(t) + 1 \big) < 1,
\end{align*}
for sufficiently large $N$. Here we have employed $|b'(t)|\leq ab^2(t)$ with $a<1$ due to condition \textbf{(EF)} and $g(t)|\xi|\geq \frac{N}{2}$ due to the definition of $\Zell(N)$.
\end{proof}
Taking account of the representation of solutions from the previous corollary with $s=0$ and the refined estimates from Proposition \ref{Prop:Effective-Integrable_IncreasingZell}, it follows
\begin{align*}
|\xi||\hat{u}(t,\xi)| &\lesssim \exp\Big( -C|\xi|^2\int_0^t\frac{1}{b(\tau)+g(\tau)|\xi|^2}d\tau \Big)\big( |\xi||\hat{u}_0(\xi)| + |\hat{u}_1(\xi)| \big), \\
|\hat{u}_t(t,\xi)| &\lesssim  \exp\bigg( -C|\xi|^2\int_0^t\frac{1}{b(\tau)+g(\tau)|\xi|^2}d\tau \bigg)\big( |\xi||\hat{u}_0(\xi)| + |\hat{u}_1(\xi)| \big) \\
& \qquad + \exp\bigg( -\int_0^t \big( b(\tau)+g(\tau)|\xi|^2 \big)d\tau \bigg)|\hat{u}_1(\xi)|.
\end{align*}
\begin{corollary} \label{Cor:Effective-Integrable_IncreasingZell}
The following estimates hold for the solutions to \eqref{Eq:Effective-integrable-system-Zell} for all $t\in[0,t_{\text{ell}}]$:
\begin{align*}
|\xi|^{|\beta|}|\hat{u}(t,\xi)| &\lesssim \exp\bigg( -C|\xi|^2\int_0^t\frac{1}{b(\tau)+g(\tau)|\xi|^2}d\tau \bigg)\Big( |\xi|^{|\beta|}|\hat{u}_0(\xi)| + |\xi|^{|\beta|-1|}|\hat{u}_1(\xi)| \Big) \quad \mbox{for} \quad |\beta|\geq1, \\
|\xi|^{|\beta|}|\hat{u}_t(t,\xi)| &\lesssim \exp\bigg( -C|\xi|^2\int_{0}^t\frac{1}{b(\tau)+g(\tau)|\xi|^2}d\tau \bigg)\Big( |\xi|^{|\beta|+1}|\hat{u}_0(\xi)| + |\xi|^{|\beta|}|\hat{u}_1(\xi)| \Big) \\
& \qquad + \exp\bigg( -\int_{0}^t \big( b(\tau)+g(\tau)|\xi|^2 \big)d\tau \bigg)|\xi|^{|\beta|}|\hat{u}_1(\xi)| \quad \mbox{for} \quad |\beta|\geq0.
\end{align*}
\end{corollary}
\subsubsection{Considerations in the dissipative zone $\Zdiss(\varepsilon)$} \label{Subsection_Effective_Integrable_Zdiss}
The definition of the dissipative zone is
\[ Z_{\rm diss}(\varepsilon) = \bigg\{ (t,\xi): |\xi|\leq \frac{\varepsilon}{2}\frac{1}{g(t)}\big( 1-\sqrt{1-\frac{4}{\varepsilon^2}b(t)g(t)} \big) \bigg\}\cap \Pi_{\text{ell}}\cap\{t\geq t_0\}. \]
We fix $\varepsilon>0$. Then, we can choose with a sufficiently large $M$ the parameter $t_0=t_0(\varepsilon,M)$ such that $b(t)g(t)\leq \frac{1}{M}\varepsilon^2$ for $t \geq t_0$. This gives $\frac{4}{\varepsilon^2}b(t)g(t)\leq \frac{4}{M}<1$ for large $M$. So, $\sqrt{1-\frac{4}{\varepsilon^2}b(t)g(t)}$ is well-defined for $t \geq t_0(\varepsilon,M)$. For $t \geq t_0$ we have
\[ \frac{\varepsilon}{2}\frac{1}{g(t)}\big( 1-\sqrt{1-\frac{4}{\varepsilon^2}b(t)g(t)} \big)\sim \frac{b(t)}{\varepsilon}.\]
So, this part of the definition of the dissipative zone is consistent with the definitions on pages 15 and 36. On pages 15 and 36 we use $\{|\xi|\leq \frac{\varepsilon}{g(t)}\}$.

It holds $|\xi| \leq \frac{b(t)}{2}$ if $t_0$ is chosen sufficiently large. So, $|\xi|$ is dominated by $b(t)$ in this zone. Hence,
$g(t)|\xi|^2 \leq \frac{g(t)b(t)}{4} b(t)$. Now we can use $t_0=t_0(\kappa)$ so large, that $\frac{g(t)b(t)}{4} \leq \kappa$ for $t \geq t_0(\kappa)$
and a sufficiently small $\kappa$. Hence $b(t) \leq b(t)+g(t)|\xi|^2 \leq (1+\kappa)b(t)$.
\medskip

Let us consider
\begin{equation} \label{Eq:Effective-integrable-EM1-Zdiss}
\hat{u}_{tt} + |\xi|^2 \hat{u} + \tilde{b}(t,\xi)\hat{u}_t = 0, \,\,\,\,\,\, \mbox{where}\,\,\,\,\,\, \tilde b(t,\xi) = b(t)+g(t)|\xi|^2.
\end{equation}
%Due to our condition \textbf{(E5)} on page 61, we do not have a dissipative zone.
%Namely, using $b(t) \leq b(t)+g(t)|\xi|^2 \leq (1+\kappa)b(t)$, then the above dissipative zone is divided into elliptic zone and reduced zone as follows:
Using $b(t) \leq b(t)+g(t)|\xi|^2 \leq (1+\kappa)b(t)$, then the above dissipative zone may be divided into elliptic zone and reduced zone as follows:
\begin{align}
\label{Eq:Effective-integrable-EM01-Zdiss} \Zell(0,\varepsilon) &= \Big\{ (t,\xi): 0<|\xi|\leq \frac{b(t)}{2}\sqrt{1-\varepsilon^2} \Big\}, \\
\label{Eq:Effective-integrable-EM02-Zdiss} \Zred(0,\varepsilon) &= \Big\{ (t,\xi): \frac{b(t)}{2}\sqrt{1-\varepsilon^2} \leq  |\xi| \leq \frac{b(t)}{2} \Big\}.
\end{align}
The separating line $t_1(|\xi|)$ solves $|\xi| = \dfrac{b(t)}{2}\sqrt{1-\varepsilon^2}$.
\medskip

We have in these zones $\tilde{b}(t,\xi) \sim b(t)$. But for the elliptic approach we need derivatives of $\tilde{b} = \tilde{b}(t,\xi)$ with respect to $t$ which do not satisfy, in general, the condition $\textbf{(B3)}$ on page 49. Therefore, we are going to use the low regularity in the coefficient $\tilde{b}$, that is, without further assumptions on derivatives of $\tilde{b}$, employing the energy method in Fourier space and zones with energy multipliers we may derive the energy estimates. This method was developed in \cite{JuniordaLuzNoneffective, JuniordaLuzEffective}.
\medskip

We define the energy to the solution of \eqref{Eq:Effective-integrable-EM1-Zdiss} as follows:
\begin{equation} \label{Eq:Effective-integrable-EM3-Zdiss}
E(t,\xi) := \frac{1}{2}\big( |\hat{u}_t|^2 + |\xi|^2|\hat{u}|^2 \big) \quad \text{for all} \quad t\geq 0.
\end{equation}
Let us multiply both sides of \eqref{Eq:Effective-integrable-EM1-Zdiss} by $\hat{u}_t$. Then, we get
\begin{equation} \label{Eq:Effective-integrable-EM2-Zdiss}
\frac{1}{2}\frac{d}{dt}\big( |\hat{u}_t|^2 + |\xi|^2|\hat{u}|^2 \big) + \tilde{b}(t,\xi)|\hat{u}_t|^2 = 0.
\end{equation}
Now we integrate the previous equality on $[s,t]$ and find
\begin{equation*}
E(t,\xi) + \int_s^t \tilde{b}(\tau,\xi)|u_t(\tau,\xi)|^2d\tau = E(s,\xi).
\end{equation*}
Let $Z\subset[0,\infty)\times \mathbb{R}^n$ and define $K: Z\to [0,\infty)$. Now we multiply both sides of \eqref{Eq:Effective-integrable-EM2-Zdiss} by $K(t,\xi)\hat{u}$. Then, for $(t,\xi)\in Z$ we have
\begin{align*} \label{Eq:Effective-integrable-EM4-Zdiss}
\frac{d}{dt}\big( K(t,\xi)\operatorname{Re}&(\hat{u}_t\hat{u}) \big) + \tilde{b}(t,\xi)K(t,\xi)\operatorname{Re}(\hat{u}_t\hat{u}) + K(t,\xi)|\xi|^2|\hat{u}|^2 \nonumber \\
& = K(t,\xi)|\hat{u}_t|^2 + \Big( \frac{d}{dt}K(t,\xi) \Big)\operatorname{Re}(\hat{u}_t\hat{u}).
\end{align*}
\begin{proposition}[Proposition 2.1, \cite{JuniordaLuzNoneffective}] \label{Prop:Effective-integrable-EM5-Zdiss}
Let $\hat{u}=\hat{u}(t,\xi)$ be the solution of \eqref{Eq:Effective-integrable-EM1-Zdiss}, $Z\subset[0,\infty)\times \mathbb{R}^n$ and $K:Z\to [0,\infty)$, where $K=K(t,\cdot)$ is a measurable and $K(\cdot,\xi)$ is a piecewise $\mathcal{C}^1$ function for each $(t,\xi)\in Z$. Suppose that there exists $\lambda_1$, $\lambda_2$ and $\lambda_3\geq 0$ such that $K$ and $Z$ satisfy the following properties for all $(t,\xi)\in Z$:
\begin{enumerate}
\item if $(\tau_1,\xi)$, $(\tau_2,\xi)\in Z$, then $(\tau,\xi)\in Z$ for all $\tau_1\leq \tau\leq \tau_2$;
\item $K(t,\xi) \leq \lambda_1\tilde{b}(t,\xi)$;
\item $\tilde{b}(t,\xi)K(t,\xi) \leq \lambda_2|\xi|^2$;
\item $\big( \dfrac{d}{dt}K(t,\xi) \big)^2 \leq \lambda_3\tilde{b}(t,\xi)K(t,\xi)|\xi|^2$.
\end{enumerate}
Then, there exists $C>0$ and for all $(s,\xi), (t,\xi)\in Z$ with $t\geq s\geq0$ such that
\begin{equation*}
\int_s^t K(\theta,\xi)|\xi|^2|\hat{u}(\theta,\xi)|^2d\theta \leq CE(s,\xi).
\end{equation*}
\end{proposition}
\begin{definition}
We say that $K: Z \to [0,\infty)$ is a multiplier of energy in $Z\subset[0,\infty)\times \mathbb{R}^n$ if $K$ and $Z$ satisfy the properties 1 to 4 of Proposition \ref{Prop:Effective-integrable-EM5-Zdiss}.
\end{definition}
\begin{proposition} \label{Prop:Effective-integrable-EM6-Zdiss}
The multiplier of the energy in $\Zell(0,\varepsilon)$ is $K(t,\xi) := \dfrac{|\xi|^2}{b(t)}$.
\end{proposition}
\begin{proof}
We notice that the first property of Proposition \ref{Prop:Effective-integrable-EM5-Zdiss} is satisfied since $b=b(t)$ is monotone. Now, using $b(t) \leq \tilde{b}(t,\xi) \leq (1+\kappa)b(t)$, definition of $\Zell(0,\varepsilon)$ \eqref{Eq:Effective-integrable-EM01-Zdiss} and condition \textbf{(EF)} on page 49, we get
\begin{align*}
& 2. \,\,\, K(t,\xi) = \frac{1}{b(t)}|\xi|^2 \leq \frac{1}{b(t)}\frac{b^2(t)}{4}(1-\varepsilon^2) = \frac{(1-\varepsilon^2)}{4}b(t) \leq \frac{(1-\varepsilon^2)}{4}\tilde{b}(t,\xi), \\
& 3. \,\,\, \tilde{b}(t,\xi)K(t,\xi) = \tilde{b}(t,\xi)\frac{1}{b(t)}|\xi|^2 \leq (\kappa+1)|\xi|^2, \\
& 4. \,\,\, \Big( \frac{d}{dt}K(t,\xi) \Big)^2 = \Big( \frac{b'(t)}{b^2(t)} \Big)^2|\xi|^4 \leq a^2|\xi|^4 = a^2\underbrace{\frac{|\xi|^2}{b(t)}}_{=K(t,\xi)}b(t)|\xi|^2 = a^2K(t,\xi)b(t)|\xi|^2\\
& \qquad \qquad \leq a^2 \tilde{b}(t,\xi)K(t,\xi)|\xi|^2.
\end{align*}
\end{proof}
\begin{lemma} [Lemma 2.1, \cite{JuniordaLuzNoneffective}]\label{Lemma:Effective-integrable-EM7-Zdiss}
Let $E: [t_0,\infty)\to [0,\infty)$ differentiable and non-increasing, $f: [t_0,T_0) \to [0,\infty)$ continuous, where $t_0\in \mathbb{R}$ and $t_0<T_0\in\mathbb{R}\cup\{\infty\}$. Suppose that there exists $C>0$ such that
\[ \int_s^{T_0}f(\tau)E(\tau)d\tau \leq CE(s) \quad \mbox{for all} \quad s\in[t_0,T_0), \]
then, for every $0<\epsilon<1$ holds
\[ E(t)\leq \frac{E(t_0)}{1-\epsilon}\exp\Big( -\frac{\epsilon}{C}\int_{t_0}^tf(\tau)d\tau \Big) \quad \mbox{for all} \quad t\in[t_0,T_0). \]
\end{lemma}
\begin{proposition} \label{Prop:Effective-integrable-EM8-Zdiss}
For a fixed $\xi\in \mathbb{R}^n$ and we define
\[ t_0(\xi) = \inf\{t\in[t_0,t_1]: (t,\xi)\in \Zell(0,\varepsilon) \} \quad \text{and} \quad t_1(\xi) = \sup\{t\in[t_0,t_1]: (t,\xi)\in \Zell(0,\varepsilon) \}. \]
Then, there exists $C>0$ independent of $\xi$ such that
\[ E(t,\xi) \lesssim \exp\Big( -\frac{1}{C}|\xi|^2\int_{t_0}^t \frac{1}{b(\tau)}d\tau \Big)E(t_0,\xi) \quad \text{for all} \quad t_0 \leq t \leq t_1. \]
\end{proposition}
\begin{proof}
Let $K: Z\to [0,\infty)$ a multiplier of energy in $Z\subset[0,\infty)\times \mathbb{R}^n$. By Proposition \ref{Prop:Effective-integrable-EM5-Zdiss}, we know that
\[ \int_s^t K(\tau,\xi)|\xi|^2|\hat{u}(\tau,\xi)|^2d\tau \lesssim E(s,\xi) \]
for all $(s,\xi)$, $(t,\xi)\in\Zell(0,\varepsilon)$ with $t\geq s\geq t_0$. On the other hand, by the definition of $\Zell(0,\varepsilon)$ and \eqref{Eq:Effective-integrable-EM3-Zdiss}, we obtain
\begin{align*}
\int_s^t K(\tau,\xi)|\hat{u}_t(\tau,\xi)|^2d\tau = \int_s^t \frac{|\xi|^2}{b(\tau)}|\hat{u}_t(\tau,\xi)|^2d\tau \lesssim \int_s^t b(\tau)|\hat{u}_t(\tau,\xi)|^2d\tau \lesssim E(s,\xi)
\end{align*}
for all $(s,\xi)$, $(t,\xi)\in\Zell(0,\varepsilon)$. Therefore, we have
\[ \int_s^t K(\tau,\xi)E(\tau,\xi)d\tau \lesssim E(s,\xi). \]
Fixing $\xi\in\mathbb{R}^n$ such that $\{t\in[t_0,t_1]\in\Zell(0,\varepsilon)\}$ has nonzero measure, we may apply Lemma \ref{Lemma:Effective-integrable-EM7-Zdiss} and conclude
\begin{equation*} \label{Eq:Effective-integrable-EM9-Zdiss}
E(t,\xi) \lesssim \exp\Big( -\frac{1}{C}\int_{t_0}^t K(\tau,\xi)d\tau \Big)E(t_0,\xi)
\end{equation*}
for all $t_0\leq t\leq t_1$. This completes the proof.
\end{proof}
Finally, from Proposition \ref{Prop:Effective-integrable-EM8-Zdiss} and the definition of $E=E(t,\xi)$ in \eqref{Eq:Effective-integrable-EM3-Zdiss} we may conclude the following estimates.
\begin{corollary} \label{Cor:Effective-Integrable-EM10-Zdiss}
The following estimates hold for all $t\in[t_0,t_1]$ in $\Zell(0,\varepsilon)$:
\begin{align*}
|\xi|^{|\beta|}|\hat{u}(t,\xi)| &\lesssim \exp\Big( -C_1|\xi|^2\int_{t_0}^t\frac{1}{b(\tau)}d\tau \Big)\Big( |\xi|^{|\beta|}|\hat{u}(t_0,\xi)| + |\xi|^{|\beta|-1|}|\hat{u}_t(t_0,\xi)| \Big) \quad \mbox{for} \quad |\beta|\geq1, \\
|\xi|^{|\beta|}|\hat{u}_t(t,\xi)| &\lesssim \exp\Big( -C_1|\xi|^2\int_{t_0}^t\frac{1}{b(\tau)}d\tau \Big)\Big( |\xi|^{|\beta|+1}|\hat{u}(t_0,\xi)| + |\xi|^{|\beta|}|\hat{u}_t(t_0,\xi)| \Big) \quad \mbox{for} \quad |\beta|\geq0.
\end{align*}
\end{corollary}
Now we consider $\Zred(0,\varepsilon)$ in \eqref{Eq:Effective-integrable-EM02-Zdiss} in the same manner as we treated $\Zell(0,\varepsilon)$.
\begin{proposition} \label{Prop:Effective-integrable-EM11-Zdiss}
The multiplier of the energy in $\Zred(0,\varepsilon)$ is $K(t) := b(t)$.
\end{proposition}
\begin{proof}
Analogously to Proposition \ref{Prop:Effective-integrable-EM6-Zdiss}, we use $b(t) \leq \tilde{b}(t,\xi) \leq (1+\kappa)b(t)$, definition of $\Zred(0,\varepsilon)$ from \eqref{Eq:Effective-integrable-EM02-Zdiss} and $(1+t)b(t)>1$ from condition \textbf{(B2)} on page 49. Then, we find
\begin{align*}
& 2. \,\,\, K(t) = b(t) \leq \tilde{b}(t,\xi), \\
& 3. \,\,\, \tilde{b}(t,\xi)K(t) = \tilde{b}(t,\xi)b(t) \leq (\kappa+1)b^2(t) \leq \frac{4}{1-\varepsilon^2}(1+\kappa)|\xi|^2, \\
& 4. \,\,\, \Big( \frac{d}{dt}K(t) \Big)^2 = ( b'(t) \big)^2 \leq C_1b^2(t)\frac{1}{(1+t)^2} \leq C_1b^2(t)b^2(t) \leq C_1b^2(t)\frac{4}{1-\varepsilon^2}|\xi|^2 \\
& \qquad \quad = \frac{4C_1}{1-\varepsilon^2}K(t)b(t)|\xi|^2 \leq \frac{4C_1}{1-\varepsilon^2}K(t)\tilde{b}(t,\xi)|\xi|^2.
\end{align*}
This completes the proof.
\end{proof}
\begin{proposition} \label{Prop:Effective-integrable-EM12-Zdiss}
For a fixed $\xi\in \mathbb{R}^n$ and we define
\[ t_1(\xi) = \inf\{t\in[t_1,t_{\text{hyp}}]: (t,\xi)\in \Zred(0,\varepsilon) \} \quad \text{and} \quad t_{\text{hyp}}(\xi) = \sup\{t\in[t_1,t_{\text{hyp}}]: (t,\xi)\in \Zred(0,\varepsilon) \}. \]
Then, there exists $C>0$ independent of $\xi$ such that
\[ E(t,\xi) \lesssim \exp\Big( -\frac{1}{C}\int_{s}^t b(\tau)d\tau \Big)E(s,\xi) \quad \text{for all} \quad t_1 \leq s\leq t  \leq t_{\text{hyp}}. \]
\end{proposition}
\begin{proof}
Let $K: Z\to [0,\infty)$ a multiplier of energy in $Z\subset[0,\infty)\times \mathbb{R}^n$. By Proposition \ref{Prop:Effective-integrable-EM5-Zdiss}, we know that
\[ \int_s^t K(\tau)|\xi|^2|\hat{u}(\tau,\xi)|^2d\tau \lesssim E(s,\xi) \]
for all $(s,\xi)$, $(t,\xi)\in\Zred(0,\varepsilon)$ with $t_1 \leq s\leq t  \leq t_{\text{hyp}}$. On the other hand, by the definition of $\Zred(0,\varepsilon)$ in \eqref{Eq:Effective-integrable-EM02-Zdiss} and \eqref{Eq:Effective-integrable-EM3-Zdiss}, we obtain
\begin{align*}
\int_s^t K(\tau)|\hat{u}_t(\tau,\xi)|^2d\tau = \int_s^t b(\tau)|\hat{u}_t(\tau,\xi)|^2d\tau \lesssim E(s,\xi)
\end{align*}
for all $(s,\xi)$, $(t,\xi)\in\Zred(0,\varepsilon)$. Therefore, we have
\[ \int_s^t K(\tau,\xi)E(\tau,\xi)d\tau \lesssim E(s,\xi). \]
Fixing $\xi\in\mathbb{R}^n$ such that $\{t\in[t_1,t_{\text{hyp}}]\in\Zred(0,\varepsilon)\}$ has nonzero measure, we may apply Lemma \ref{Lemma:Effective-integrable-EM7-Zdiss} and conclude
\begin{equation*} \label{Eq:Effective-integrable-EM13-Zdiss}
E(t,\xi) \lesssim \exp\Big( -\frac{1}{C}\int_{s}^t K(\tau)d\tau \Big)E(s,\xi)
\end{equation*}
for all $t_1 \leq s\leq t  \leq t_{\text{hyp}}$. This completes the proof.
\end{proof}
Thus, from Proposition \ref{Prop:Effective-integrable-EM12-Zdiss} and the definition of $E=E(t,\xi)$ in \eqref{Eq:Effective-integrable-EM3-Zdiss} we may arrive at the following estimates.
\begin{corollary} \label{Cor:Effective-Integrable-EM14-Zdiss}
The following estimates hold for all $t\in[t_1,t_{\text{hyp}}]$ in $\Zred(0,\varepsilon)$:
\begin{align*}
|\xi|^{|\beta|}|\hat{u}(t,\xi)| &\lesssim \exp\Big( -C_2\int_{s}^tb(\tau)d\tau \Big)\Big( |\xi|^{|\beta|}|\hat{u}(s,\xi)| + |\xi|^{|\beta|-1|}|\hat{u}_t(s,\xi)| \Big) \quad \mbox{for} \quad |\beta|\geq1, \\
|\xi|^{|\beta|}|\hat{u}_t(t,\xi)| &\lesssim \exp\Big( -C_2\int_{s}^t b(\tau)d\tau \Big)\Big( |\xi|^{|\beta|+1}|\hat{u}(s,\xi)| + |\xi|^{|\beta|}|\hat{u}_t(s,\xi)| \Big) \quad \mbox{for} \quad |\beta|\geq0.
\end{align*}
\end{corollary}
\subsubsection{Conclusion} \label{Sec:Effective-Decaying-Conclusions}
From the statements of Proposition \ref{Prop_Effective-Integrable-Zhyp} and Corollaries \ref{Cor:Effective-Integrable-Zred}, \ref{Cor:Effective-Integrable_IncreasingZell}, \ref{Cor:Effective-Integrable-EM10-Zdiss}, and \ref{Cor:Effective-Integrable-EM14-Zdiss}, we derive our desired statements. However, due to the localization of the zones where $b=b(t)$ is increasing or decreasing, we will divide our considerations into two cases in order to apply the \textit{gluing procedure}.
\subsubsection*{a. $b=b(t)$ is increasing}
For this case see Figure \ref{Fig-Effective-integrable-zones}.a.
\medskip

\noindent \underline{Small frequencies}:
\medskip

\noindent \textit{Case 1:} $t_0\leq t\leq t_1(|\xi|)$. Due to Corollary \ref{Cor:Effective-Integrable-EM10-Zdiss} we have the following lemma.
\begin{lemma} \label{Lemma:Effective-Decaying-Large0}
We have the following estimates for large frequencies and $t_0\leq t\leq t_1(|\xi|)$:
\begin{align*}
|\xi|^{|\beta|}|\hat{u}(t,\xi)| &\lesssim \exp\Big( -C_1|\xi|^2\int_{t_0}^t\frac{1}{b(\tau)}d\tau \Big)\big( |\xi|^{|\beta|}|\hat{u}(t_0,\xi)| + |\xi|^{|\beta|-1|}|\hat{u}_t(t_0,\xi)| \big) \quad \mbox{for} \quad |\beta|\geq1, \\
|\xi|^{|\beta|}|\hat{u}_t(t,\xi)| &\lesssim \exp\Big( -C_1|\xi|^2\int_{t_0}^t\frac{1}{b(\tau)}d\tau \Big)\big( |\xi|^{|\beta|+1}|\hat{u}(t_0,\xi)| + |\xi|^{|\beta|}|\hat{u}_t(t_0,\xi)| \big) \quad \mbox{for} \quad |\beta|\geq0.
\end{align*}
\end{lemma}

\noindent \underline{Large frequencies}:
\medskip

\noindent \textit{Case 1:} $t\leq t_{\text{ell}}$. In this case, we use Corollary \ref{Cor:Effective-Integrable_IncreasingZell}.
\begin{lemma} \label{Lemma:Effective-Decaying-Large1}
We have the following estimates for large frequencies and $t\leq t_{\text{ell}}$:
\begin{align*}
|\xi|^{|\beta|}|\hat{u}(t,\xi)| &\lesssim \exp\Big( -C|\xi|^2\int_0^t\frac{1}{b(\tau)+g(\tau)|\xi|^2}d\tau \Big)\big( |\xi|^{|\beta|}|\hat{u}_0(\xi)| + |\xi|^{|\beta|-1}|\hat{u}_1(\xi)| \big) \quad \mbox{for} \quad |\beta| \geq 1, \\
|\xi|^{|\beta|}|\hat{u}_t(t,\xi)|& \lesssim \exp\Big( -C|\xi|^2\int_0^t\frac{1}{b(\tau)+g(\tau)|\xi|^2}d\tau \Big)\big( |\xi|^{|\beta|+1}|\hat{u}_0(\xi)| + |\xi|^{|\beta|}|\hat{u}_1(\xi)| \big) \quad \mbox{for} \quad |\beta| \geq 0.
\end{align*}
\end{lemma}
\begin{proof}
We have
\begin{align*}
|\xi|^{|\beta|}|\hat{u}(t,\xi)| &\lesssim \exp\Big( -C|\xi|^2\int_0^t\frac{1}{b(\tau)+g(\tau)|\xi|^2}d\tau \Big)\big( |\xi|^{|\beta|}|\hat{u}_0(\xi)| + |\xi|^{|\beta|-1}|\hat{u}_1(\xi)| \big), \\
|\xi|^{|\beta|}|\hat{u}_t(t,\xi)| &\lesssim  \exp\Big( -C|\xi|^2\int_0^t\frac{1}{b(\tau)+g(\tau)|\xi|^2}d\tau \Big)\big( |\xi|^{|\beta|+1}|\hat{u}_0(\xi)| + |\xi|^{|\beta|}|\hat{u}_1(\xi)| \big) \\
& \qquad + \exp\Big( -\int_0^t \big( b(\tau)+g(\tau)|\xi|^2 \big)d\tau \Big)|\xi|^{|\beta|}|\hat{u}_1(\xi)| \\
& \lesssim  \exp\Big( -C|\xi|^2\int_0^t\frac{1}{b(\tau)+g(\tau)|\xi|^2}d\tau \Big)\big( |\xi|^{|\beta|+1}|\hat{u}_0(\xi)| + |\xi|^{|\beta|}|\hat{u}_1(\xi)| \big),
\end{align*}
where due to the definition of $\Pi_{\text{ell}}$ in \eqref{Eq:Effective-Integrable-Pi-Hyp-Ell}, we used
\[ 2|\xi| < b(t) + g(t)|\xi|^2, \qquad \text{which implies that} \qquad -\frac{4|\xi|^2}{b(t) + g(t)|\xi|^2} > - \big( b(t) + g(t)|\xi|^2 \big), \]
which shows the desired estimates.
\end{proof}
\noindent \textit{Case 2:} $t_{\text{ell}} \leq t \leq t_{\text{red}}$. In this case, we use first Corollary \ref{Cor:Effective-Integrable-Zred} with $s=t_{\text{ell}}$ and, then the estimates from Lemma \ref{Lemma:Effective-Decaying-Large1}.
\begin{lemma} \label{Lemma:Effective-Decaying-Large2}
We have the following estimates for large frequencies and $t_{\text{ell}} \leq t \leq t_{\text{red}}$:
\begin{align*}
|\xi|^{|\beta|}|\hat{u}(t,\xi)| &\lesssim \exp\Big( -C|\xi|^2\int_0^{t}\frac{1}{b(\tau)+g(\tau)|\xi|^2}d\tau \Big)\big( |\xi|^{|\beta|}|\hat{u}_0(\xi)| + |\xi|^{|\beta|-1}|\hat{u}_1(\xi)| \big) \quad \mbox{for} \quad |\beta| \geq 1, \\
|\xi|^{|\beta|}|\hat{u}_t(t,\xi)|& \lesssim  \exp\Big( -C|\xi|^2\int_0^{t}\frac{1}{b(\tau)+g(\tau)|\xi|^2}d\tau \Big)\big( |\xi|^{|\beta|+1}|\hat{u}_0(\xi)| + |\xi|^{|\beta|}|\hat{u}_1(\xi)| \big) \quad \mbox{for} \quad |\beta| \geq 0.
\end{align*}
\end{lemma}
\begin{proof}
We have
\begin{align*}
|\xi|^{|\beta|}|\hat{u}(t,\xi)| &\lesssim |\xi|^{|\beta|}|\hat{u}(t_{\text{ell}},\xi)| + |\xi|^{|\beta|-1}|\hat{u}_t(t_{\text{ell}},\xi)| \\
& \leq \exp\Big( -C|\xi|^2\int_0^{t_{\text{ell}}}\frac{1}{b(\tau)+g(\tau)|\xi|^2}d\tau \Big)\big( |\xi|^{|\beta|}|\hat{u}_0(\xi)| + |\xi|^{|\beta|-1}|\hat{u}_1(\xi)| \big) \\
& \leq \exp\Big( -C|\xi|^2\int_0^{t}\frac{1}{b(\tau)+g(\tau)|\xi|^2}d\tau \Big)\big( |\xi|^{|\beta|}|\hat{u}_0(\xi)| + |\xi|^{|\beta|-1}|\hat{u}_1(\xi)| \big), \\
|\xi|^{|\beta|}|\hat{u}_t(t,\xi)| &\lesssim |\xi|^{|\beta|+1}|\hat{u}(t_{\text{ell}},\xi)| + |\xi|^{|\beta|}|\hat{u}_t(t_{\text{ell}},\xi)| \\
& \leq \exp\Big( -C|\xi|^2\int_0^{t_{\text{ell}}}\frac{1}{b(\tau)+g(\tau)|\xi|^2}d\tau \Big)\big( |\xi|^{|\beta|+1}|\hat{u}_0(\xi)| + |\xi|^{|\beta|}|\hat{u}_1(\xi)| \big) \\
& \leq \exp\Big( -C|\xi|^2\int_0^{t}\frac{1}{b(\tau)+g(\tau)|\xi|^2}d\tau \Big)\big( |\xi|^{|\beta|+1}|\hat{u}_0(\xi)| + |\xi|^{|\beta|}|\hat{u}_1(\xi)| \big).
\end{align*}
Here to extend the integral, we need to show that
\[ \exp\Big( -C|\xi|^2\int_{t_{\text{ell}}}^t\frac{1}{b(\tau)+g(\tau)|\xi|^2}d\tau \Big) \sim 1. \]
For this, it is enough to show that
\[ |\xi|^2\int_{t_{\text{ell}}}^t\frac{1}{b(\tau)+g(\tau)|\xi|^2}d\tau \lesssim 1. \]
Indeed, we have
\[ \frac{b(\tau)}{|\xi|^2} + g(\tau) \leq \frac{4}{\varepsilon^2} b(\tau)g(\tau)^2+g(\tau)\leq \Big( \frac{2}{\varepsilon^2}+1 \Big)g(\tau), \]
where we used $|\xi|\geq \dfrac{\varepsilon}{2g(t)}$ due to the definition of $\Zred(N,\varepsilon)$ and $b(t)g(t)\leq \dfrac{1}{2}$ due to condition \textbf{(E3)}. This implies
\[ |\xi|^2\int_{t_{\text{ell}}}^t\frac{1}{b(\tau)+g(\tau)|\xi|^2}d\tau \leq \Big( \frac{2}{\varepsilon^2}+1 \Big)\int_{t_{\text{ell}}}^tg(\tau)d\tau \lesssim 1. \]
Hence, we may conclude the proof.
\end{proof}
\noindent \textit{Case 3:} $t_{\text{red}}\leq t \leq t_{\text{hyp}}$. In this case, we use Proposition \ref{Prop_Effective-Integrable-Zhyp} for $s=t_{\text{red}}$ and, then the estimates from Lemma \ref{Lemma:Effective-Decaying-Large2}.
\begin{lemma} \label{Lemma:Effective-Decaying-Large3}
We have the following estimates for large frequencies and $t_{\text{red}}\leq t \leq t_{\text{hyp}}$:
\begin{align*}
|\xi|^{|\beta|}|\hat{u}(t,\xi)| & \lesssim \exp\Big( -\widetilde{C}\int_{0}^t b(\tau)d\tau \Big)\big( |\xi|^{|\beta|}|\hat{u}_0(\xi)| + |\xi|^{|\beta|-1}|\hat{u}_1(\xi)| \big), \\
|\xi|^{|\beta|}|\hat{u}_t(t,\xi)| & \lesssim \exp\Big( -\widetilde{C}\int_{0}^t b(\tau)d\tau \Big)\big( |\xi|^{|\beta|+1}|\hat{u}_0(\xi)| + |\xi|^{|\beta|}|\hat{u}_1(\xi)| \big).
\end{align*}
\end{lemma}
\begin{proof}
We have
\begin{align*}
|\xi|^{|\beta|}|\hat{u}(t,\xi)| &\lesssim \exp\Big( -\frac{1}{2}\int_{t_{\text{red}}}^t\big( b(\tau)+g(\tau)|\xi|^2 \big)d\tau \Big)\big( |\xi|^{|\beta|}|\hat{u}(t_{\text{red}},\xi)| + |\xi|^{|\beta|-1}|\hat{u}_t(t_{\text{red}},\xi)| \big) \\
& \lesssim  \exp\Big( -\frac{1}{2}\int_{t_{\text{red}}}^t\big( b(\tau)+g(\tau)|\xi|^2 \big)d\tau \Big)\exp\Big( -C|\xi|^2\int_0^{t_{\text{red}}}\frac{1}{b(\tau)+g(\tau)|\xi|^2}d\tau \Big)\\
& \quad \times \big( |\xi|^{|\beta|}|\hat{u}_0(\xi)| + |\xi|^{|\beta|-1}|\hat{u}_1(\xi)| \big) \\
& \lesssim \exp\Big( -\frac{1}{2}\int_{t_{\text{red}}}^t\big( b(\tau)+g(\tau)|\xi|^2 \big)d\tau \Big)\exp\Big( -\frac{C}{4}\int_0^{t_{\text{red}}}\big( b(\tau)+g(\tau)|\xi|^2 \big)d\tau \Big)\\
& \quad \times \big( |\xi|^{|\beta|}|\hat{u}_0(\xi)| + |\xi|^{|\beta|-1}|\hat{u}_1(\xi)| \big) \\
& \lesssim \exp\Big( -\widetilde{C}\int_{0}^t\big( b(\tau)+g(\tau)|\xi|^2 \big)d\tau \Big) \big( |\xi|^{|\beta|}|\hat{u}_0(\xi)| + |\xi|^{|\beta|-1}|\hat{u}_1(\xi)| \big), \\
|\xi|^{|\beta|}|\hat{u}_t(t,\xi)| &\lesssim \exp\Big( -\frac{1}{2}\int_{t_{\text{red}}}^t\big( b(\tau)+g(\tau)|\xi|^2 \big)d\tau \Big)\big( |\xi|^{|\beta|+1}|\hat{u}(t_{\text{red}},\xi)| + |\xi|^{|\beta|}|\hat{u}_t(t_{\text{red}},\xi)| \big) \\
& \lesssim  \exp\Big( -\frac{1}{2}\int_{t_{\text{red}}}^t\big( b(\tau)+g(\tau)|\xi|^2 \big)d\tau \Big)\exp\Big( -C|\xi|^2\int_0^{t_{\text{ell}}}\frac{1}{b(\tau)+g(\tau)|\xi|^2}d\tau \Big)\\
& \quad \times \big( |\xi|^{|\beta|+1}|\hat{u}_0(\xi)| + |\xi|^{|\beta|}|\hat{u}_1(\xi)| \big) \\
& \lesssim \exp\Big( -\widetilde{C}\int_{0}^t\big( b(\tau)+g(\tau)|\xi|^2 \big)d\tau \Big) \big( |\xi|^{|\beta|+1}|\hat{u}_0(\xi)| + |\xi|^{|\beta|}|\hat{u}_1(\xi)| \big),
\end{align*}
where due to the definition of $\Pi_{\text{hyp}}$ in \eqref{Eq:Effective-Integrable-Pi-Hyp-Ell}, we have used
\[ 2|\xi| > b(t) + g(t)|\xi|^2, \qquad \text{which implies that} \qquad -\frac{4|\xi|^2}{b(t) + g(t)|\xi|^2} < - \big( b(t) + g(t)|\xi|^2 \big). \]
Since $g\in L^1([0,\infty))$ due to condition \textbf{(E2)}, we have
\begin{align*}
\exp\Big( -\widetilde{C}\int_{0}^t\big( b(\tau)+g(\tau)|\xi|^2 \big)d\tau \Big) &=  \exp\Big( -\widetilde{C}\int_{0}^t b(\tau)d\tau \Big) \exp\Big( -\widetilde{C}|\xi|^2\int_{0}^t g(\tau)d\tau \Big) \\
& \leq \exp\Big( -\widetilde{C}\int_{0}^t b(\tau)d\tau \Big).
\end{align*}
Then, the proof is completed.
\end{proof}
\noindent \textit{Case 4:} $t_{\text{hyp}}\leq t \leq t_1(|\xi|)$. In this case, we use Corollary \ref{Cor:Effective-Integrable-EM14-Zdiss} for $s=t_{\text{hyp}}$ and, then the estimates from Lemma \ref{Lemma:Effective-Decaying-Large3}.
\begin{lemma} \label{Lemma:Effective-Decaying-Large4}
We have the following estimates for large frequencies and $t_{\text{hyp}}\leq t \leq t_1(|\xi|)$:
\begin{align*}
|\xi|^{|\beta|}|\hat{u}(t,\xi)| & \lesssim \exp\Big( -C'\int_{0}^t b(\tau)d\tau \Big)\big( |\xi|^{|\beta|}|\hat{u}_0(\xi)| + |\xi|^{|\beta|-1}|\hat{u}_1(\xi)| \big), \\
|\xi|^{|\beta|}|\hat{u}_t(t,\xi)| & \lesssim \exp\Big( -C'\int_{0}^t b(\tau)d\tau \Big)\big( |\xi|^{|\beta|+1}|\hat{u}_0(\xi)| + |\xi|^{|\beta|}|\hat{u}_1(\xi)| \big).
\end{align*}
\end{lemma}
\begin{proof}
We have
\begin{align*}
|\xi|^{|\beta|}|\hat{u}(t,\xi)| &\lesssim\exp\Big( -C_2\int_{t_{\text{hyp}}}^t b(\tau)d\tau \Big) \big( |\xi|^{|\beta|}|\hat{u}(t_{\text{hyp}},\xi)| + |\xi|^{|\beta|-1}|\hat{u}_t(t_{\text{hyp}},\xi)| \big)  \\
& \lesssim  \exp\Big( -C_2\int_{t_{\text{hyp}}}^t b(\tau)d\tau \Big) \exp\Big( -\widetilde{C}\int_{0}^{t_{\text{hyp}}} b(\tau)d\tau \Big)\big( |\xi|^{|\beta|}|\hat{u}_0(\xi)| + |\xi|^{|\beta|-1}|\hat{u}_1(\xi)| \big) \\
& \lesssim \exp\Big( -C'\int_{0}^t b(\tau)d\tau \Big)\big( |\xi|^{|\beta|}|\hat{u}_0(\xi)| + |\xi|^{|\beta|-1}|\hat{u}_1(\xi)| \big), \\
|\xi|^{|\beta|}|\hat{u}_t(t,\xi)| &\lesssim \exp\Big( -C_2\int_{t_{\text{hyp}}}^t b(\tau)d\tau \Big) \big( |\xi|^{|\beta|+1}|\hat{u}(t_{\text{hyp}},\xi)| + |\xi|^{|\beta|}|\hat{u}_t(t_{\text{hyp}},\xi)| \big) \\
& \lesssim \Big( -C_2\int_{t_{\text{hyp}}}^t b(\tau)d\tau \Big)\exp\Big( -\widetilde{C}\int_{0}^{t_{\text{hyp}}} b(\tau)d\tau \Big)\big( |\xi|^{|\beta|+1}|\hat{u}_0(\xi)| + |\xi|^{|\beta|}|\hat{u}_1(\xi)| \big) \\
& \lesssim \exp\Big( -C'\int_{0}^{t} b(\tau)d\tau \Big)\big( |\xi|^{|\beta|+1}|\hat{u}_0(\xi)| + |\xi|^{|\beta|}|\hat{u}_1(\xi)| \big).
\end{align*}
\end{proof}
\noindent \textit{Case 5:} $t \geq t_1(|\xi|)$. In this case, we use Corollary \ref{Cor:Effective-Integrable-EM10-Zdiss} and the estimates from Lemma \ref{Lemma:Effective-Decaying-Large4}.
\begin{lemma}  \label{Lemma:Effective-Decaying-Large5}
We have the following estimates for large frequencies and $t\geq t_1(|\xi|)$:
\begin{align*}
|\xi|^{|\beta|}|\hat{u}(t,\xi)| & \lesssim \exp\Big( -C_1'|\xi|^2\int_{0}^t \frac{1}{b(\tau)}d\tau \Big)\big( |\xi|^{|\beta|}|\hat{u}_0(\xi)| + |\xi|^{|\beta|-1}|\hat{u}_1(\xi)| \big), \\
|\xi|^{|\beta|}|\hat{u}_t(t,\xi)| & \lesssim \exp\Big( -C_1'|\xi|^2\int_{0}^t \frac{1}{b(\tau)}d\tau \Big)\big( |\xi|^{|\beta|+1}|\hat{u}_0(\xi)| + |\xi|^{|\beta|}|\hat{u}_1(\xi)| \big).
\end{align*}
\end{lemma}
\begin{proof}
We have
\begin{align*}
|\xi|^{|\beta|}|\hat{u}(t,\xi)| &\lesssim \exp\Big( -C_1|\xi|^2\int_{t_0}^t\frac{1}{b(\tau)}d\tau \Big)\big( |\xi|^{|\beta|}|\hat{u}(t_0,\xi)| + |\xi|^{|\beta|-1|}|\hat{u}_t(t_0,\xi)| \big) \\
&\lesssim \exp\Big( -C_1|\xi|^2\int_{t_0}^t\frac{1}{b(\tau)}d\tau \Big)\exp\Big( -C'\int_{0}^{t_0} b(\tau)d\tau \Big)\big( |\xi|^{|\beta|}|\hat{u}_0(\xi)| + |\xi|^{|\beta|-1}|\hat{u}_1(\xi)| \big) \\
&\lesssim \exp\Big( -C_1'|\xi|^2\int_{t_0}^t\frac{1}{b(\tau)}d\tau \Big)\big( |\xi|^{|\beta|}|\hat{u}_0(\xi)| + |\xi|^{|\beta|-1}|\hat{u}_1(\xi)| \big), \\
|\xi|^{|\beta|}|\hat{u}_t(t,\xi)| &\lesssim \exp\Big( -C_1|\xi|^2\int_{t_0}^t\frac{1}{b(\tau)}d\tau \Big)\big( |\xi|^{|\beta|+1}|\hat{u}(t_0,\xi)| + |\xi|^{|\beta|}|\hat{u}_t(t_0,\xi)| \big) \\
& \lesssim \exp\Big( -C_1|\xi|^2\int_{t_0}^t\frac{1}{b(\tau)}d\tau \Big)\exp\Big( -C'\int_{0}^{t_0} b(\tau)d\tau \Big)\big( |\xi|^{|\beta|+1}|\hat{u}_0(\xi)| + |\xi|^{|\beta|}|\hat{u}_1(\xi)| \big) \\
&\lesssim \exp\Big( -C_1'|\xi|^2\int_{t_0}^t\frac{1}{b(\tau)}d\tau \Big)\big( |\xi|^{|\beta|+1}|\hat{u}_0(\xi)| + |\xi|^{|\beta|}|\hat{u}_1(\xi)| \big),
\end{align*}
where from the definition of $\Zell(0,\varepsilon)$ in \eqref{Eq:Effective-integrable-EM01-Zdiss}, we used
\[ -b(\tau)\leq -\frac{4}{1-\varepsilon^2}|\xi|^2\frac{1}{b(\tau)}, \]
which completes the proof.
\end{proof}

\subsubsection*{b. $b=b(t)$ is decreasing}
For this case see Figure \ref{Fig-Effective-integrable-zones}.b.
\medskip

\noindent \underline{Small frequencies}:
\medskip

\noindent \textit{Case 1:} $t_0\leq t\leq t_1(|\xi|)$. Due to Corollary \ref{Cor:Effective-Integrable-EM10-Zdiss}, we have the following estimates hold with $t_0\leq t\leq t_1(|\xi|)$ and $(t,\xi) \in \Zell(0,\varepsilon)$:
\begin{align*}
|\xi|^{|\beta|}|\hat{u}(t,\xi)| &\lesssim \exp\Big( -C_1|\xi|^2\int_{t_0}^t\frac{1}{b(\tau)}d\tau \Big)\big( |\xi|^{|\beta|}|\hat{u}(t_0,\xi)| + |\xi|^{|\beta|-1|}|\hat{u}_t(t_0,\xi)| \big) \quad \mbox{for} \quad |\beta|\geq1, \\
|\xi|^{|\beta|}|\hat{u}_t(t,\xi)| &\lesssim \exp\Big( -C_1|\xi|^2\int_{t_0}^t\frac{1}{b(\tau)}d\tau \Big)\big( |\xi|^{|\beta|+1}|\hat{u}(t_0,\xi)| + |\xi|^{|\beta|}|\hat{u}_t(t_0,\xi)| \big) \quad \mbox{for} \quad |\beta|\geq0.
\end{align*}
\medskip

\noindent \textit{Case 2:} $t_1(|\xi|)\leq t \leq t_{\text{hyp}}$. In this case, we use Corollary \ref{Cor:Effective-Integrable-EM14-Zdiss} for $s=t_1(|\xi|)$ and, then the estimates from \textit{Case 1}.
\begin{lemma} \label{Lemma:Effective-Decaying-Small-2}
We have the following estimates for small frequencies and $t_1(|\xi|)\leq t \leq t_{\text{hyp}}$:
\begin{align*}
|\xi|^{|\beta|}|\hat{u}(t,\xi)| & \lesssim \exp\Big( -C_2'|\xi|^2\int_{t_0}^{t}\frac{1}{b(\tau)}d\tau \Big)\big( |\xi|^{|\beta|}|\hat{u}(t_0,\xi)| + |\xi|^{|\beta|-1}|\hat{u}_t(t_0,\xi)| \big), \\
|\xi|^{|\beta|}|\hat{u}_t(t,\xi)| & \lesssim \exp\Big( -C_2'|\xi|^2\int_{t_0}^{t}\frac{1}{b(\tau)}d\tau \Big)\big( |\xi|^{|\beta|+1}|\hat{u}(t_0,\xi)| + |\xi|^{|\beta|}|\hat{u}_t(t_0,\xi)| \big).
\end{align*}
\end{lemma}
\begin{proof}
We have
\begin{align*}
|\xi|^{|\beta|}&|\hat{u}(t,\xi)| \lesssim \exp\Big( -C_2\int_{t_1(|\xi|)}^tb(\tau)d\tau \Big)\big( |\xi|^{|\beta|}|\hat{u}(t_1(|\xi|),\xi)| + |\xi|^{|\beta|-1}|\hat{u}_t(t_1(|\xi|),\xi)| \big) \\
& \lesssim \exp\Big( -C_2\int_{t_1(|\xi|)}^tb(\tau)d\tau \Big)\exp\Big( -C_1|\xi|^2\int_{t_0}^{t_1(|\xi|)}\frac{1}{b(\tau)}d\tau \Big)\big( |\xi|^{|\beta|}|\hat{u}(t_0,\xi)| + |\xi|^{|\beta|-1}|\hat{u}_t(t_0,\xi)| \big)\\
& \lesssim \exp\Big( -C_2'\int_{t_0}^{t}b(\tau)d\tau \Big)\big( |\xi|^{|\beta|}|\hat{u}(t_0,\xi)| + |\xi|^{|\beta|-1}|\hat{u}_t(t_0,\xi)| \big), \\
|\xi|^{|\beta|}&|\hat{u}_t(t,\xi)| \lesssim \exp\Big( -C_2\int_{t_1(|\xi|)}^tb(\tau)d\tau \Big)\big( |\xi|^{|\beta|+1}|\hat{u}(t_1(|\xi|),\xi)| + |\xi|^{|\beta|}|\hat{u}_t(t_1(|\xi|),\xi)| \big) \\
& \lesssim \exp\Big( -C_2\int_{t_1(|\xi|)}^tb(\tau)d\tau \Big)\exp\Big( -C_1|\xi|^2\int_{t_0}^{t_1(|\xi|)}\frac{1}{b(\tau)}d\tau \Big)\big( |\xi|^{|\beta|+1}|\hat{u}(t_0,\xi)| + |\xi|^{|\beta|}|\hat{u}_t(t_0,\xi)| \big) \\
& \lesssim \exp\Big( -C_2'\int_{t_0}^{t}b(\tau)d\tau \Big)\big( |\xi|^{|\beta|+1}|\hat{u}(t_0,\xi)| + |\xi|^{|\beta|}|\hat{u}_t(t_0,\xi)| \big),
\end{align*}
where from the definition of $\Zred(0,\varepsilon)$ in \eqref{Eq:Effective-integrable-EM02-Zdiss}, we used
\[ -\frac{1}{b(\tau)}|\xi|^2 \leq -\frac{1-\varepsilon^2}{4}b(\tau), \]
which completes the proof.
\end{proof}
\medskip

\noindent \textit{Case 3:} $t \geq t_{\text{hyp}}$ and $t \geq t_{\text{red}}$. In this case, we use Proposition \ref{Prop_Effective-Integrable-Zhyp} for $s=t_{\text{hyp}}$ and, then the estimates from Lemma \ref{Lemma:Effective-Decaying-Small-2}.
\begin{lemma} \label{Lemma:Effective-Decaying-Small-3}
\begin{align*}
|\xi|^{|\beta|}|\hat{u}(t,\xi)| & \lesssim \exp\Big( -C_3\int_{t_0}^{t}b(\tau)d\tau \Big)\big( |\xi|^{|\beta|}|\hat{u}(t_0,\xi)| + |\xi|^{|\beta|-1}|\hat{u}_t(t_0,\xi)| \big), \\
|\xi|^{|\beta|}|\hat{u}_t(t,\xi)| & \lesssim \exp\Big( -C_3\int_{t_0}^{t}b(\tau)d\tau \Big)\big( |\xi|^{|\beta|+1}|\hat{u}(t_0,\xi)| + |\xi|^{|\beta|}|\hat{u}_t(t_0,\xi)| \big).
\end{align*}
\end{lemma}
\begin{proof}
We have
\begin{align*}
|\xi|^{|\beta|}|\hat{u}(t,\xi)| &\lesssim \exp\Big( -\frac{1}{2}\int_{t_{\text{hyp}}}^{t}\big( b(\tau)+g(\tau)|\xi|^2 \big)d\tau \Big) \big( |\xi|^{|\beta|}|\hat{u}(t_{\text{hyp}},\xi)| + |\xi|^{|\beta|-1}|\hat{u}_t(t_{\text{hyp}},\xi)| \big)  \\
& \lesssim \exp\Big( -\frac{1}{2}\int_{t_{\text{hyp}}}^{t} b(\tau) d\tau \Big)\exp\Big( -C_2'|\xi|^2\int_{t_0}^{t_{\text{hyp}}}\frac{1}{b(\tau)}d\tau \Big) \big( |\xi|^{|\beta|}|\hat{u}(t_0,\xi)| + |\xi|^{|\beta|-1}|\hat{u}_t(t_0,\xi)| \big) \\
& \lesssim \exp\Big( -C_3\int_{t_0}^{t} b(\tau)d\tau \Big) \big( |\xi|^{|\beta|}|\hat{u}_0(\xi)| + |\xi|^{|\beta|-1}|\hat{u}_1(\xi)| \big), \\
|\xi|^{|\beta|}|\hat{u}_t(t,\xi)| &\lesssim \exp\Big( -\frac{1}{2}\int_{t_{\text{hyp}}}^{t}\big( b(\tau)+g(\tau)|\xi|^2 \big)d\tau \Big) \big( |\xi|^{|\beta|+1}|\hat{u}(t_{\text{hyp}},\xi)| + |\xi|^{|\beta|}|\hat{u}_t(t_{\text{hyp}},\xi)| \big)  \\
& \lesssim \exp\Big( -\frac{1}{2}\int_{t_{\text{hyp}}}^{t} b(\tau) d\tau \Big)\exp\Big( -C_2'|\xi|^2\int_{t_0}^{t_{\text{hyp}}}\frac{1}{b(\tau)}d\tau \Big) \big( |\xi|^{|\beta|+1}|\hat{u}(t_0,\xi)| + |\xi|^{|\beta|}|\hat{u}_t(t_0,\xi)| \big) \\
& \lesssim \exp\Big( -C_3\int_{t_0}^{t} b(\tau)d\tau \Big) \big( |\xi|^{|\beta|+1}|\hat{u}_0(\xi)| + |\xi|^{|\beta|}|\hat{u}_1(\xi)| \big),
\end{align*}
where from definition of $\Pi_{\text{hyp}}$ in \eqref{Eq:Effective-Integrable-Pi-Hyp-Ell}, we used
\[ |\xi|> \frac{b(t)}{2} +\frac{g(t)|\xi|^2}{2} > \frac{b(t)}{2} \qquad \text{implies} \qquad -|\xi|^2\frac{1}{b(t)} < -\frac{1}{4}b(t). \]
This concludes the proof.
\end{proof}

\noindent \underline{Large frequencies}:
\medskip

In the case where $b=b(t)$ is decreasing, the estimates for large frequencies remain exactly the same as in \textit{Case 1}, \textit{Case 2} and \textit{Case 3} when $b=b(t)$ is increasing. This is because the localization of the zones in these cases does not change. Note that \textit{Case 4} and \textit{Case 5} do not occur here.
\medskip

\noindent \textit{Case 1:} $t\leq t_{\text{ell}}$. In this case, we use Lemma \ref{Lemma:Effective-Decaying-Large1}. Then, we have
\begin{align*}
|\xi|^{|\beta|}|\hat{u}(t,\xi)| &\lesssim \exp\Big( -C|\xi|^2\int_0^t\frac{1}{b(\tau)+g(\tau)|\xi|^2}d\tau \Big)\big( |\xi|^{|\beta|}|\hat{u}_0(\xi)| + |\xi|^{|\beta|-1}|\hat{u}_1(\xi)| \big) \quad \mbox{for} \quad |\beta| \geq 1, \\
|\xi|^{|\beta|}|\hat{u}_t(t,\xi)|& \lesssim \exp\Big( -C|\xi|^2\int_0^t\frac{1}{b(\tau)+g(\tau)|\xi|^2}d\tau \Big)\big( |\xi|^{|\beta|+1}|\hat{u}_0(\xi)| + |\xi|^{|\beta|}|\hat{u}_1(\xi)| \big) \quad \mbox{for} \quad |\beta| \geq 0.
\end{align*}
\medskip

\noindent \textit{Case 2:} $t_{\text{ell}} \leq t \leq t_{\text{red}}$. In this case, we use Lemma \ref{Lemma:Effective-Decaying-Large2}. Then, it holds
\begin{align*}
|\xi|^{|\beta|}|\hat{u}(t,\xi)| &\lesssim \exp\Big( -C|\xi|^2\int_0^t\frac{1}{b(\tau)+g(\tau)|\xi|^2}d\tau \Big)\big( |\xi|^{|\beta|}|\hat{u}_0(\xi)| + |\xi|^{|\beta|-1}|\hat{u}_1(\xi)| \big) \quad \mbox{for} \quad |\beta| \geq 1, \\
|\xi|^{|\beta|}|\hat{u}_t(t,\xi)|& \lesssim \exp\Big( -C|\xi|^2\int_0^t\frac{1}{b(\tau)+g(\tau)|\xi|^2}d\tau \Big)\big( |\xi|^{|\beta|+1}|\hat{u}_0(\xi)| + |\xi|^{|\beta|}|\hat{u}_1(\xi)| \big) \quad \mbox{for} \quad |\beta| \geq 0.
\end{align*}
\medskip

\noindent \textit{Case 3:} $t_{\text{red}}\leq t \leq t_{\text{hyp}}$. In this case, we use Lemma \ref{Lemma:Effective-Decaying-Large3}. Then, we find
\begin{align*}
|\xi|^{|\beta|}|\hat{u}(t,\xi)| & \lesssim \exp\Big( -\widetilde{C}\int_{0}^t b(\tau)d\tau \Big)\big( |\xi|^{|\beta|}|\hat{u}_0(\xi)| + |\xi|^{|\beta|-1}|\hat{u}_1(\xi)| \big), \\
|\xi|^{|\beta|}|\hat{u}_t(t,\xi)| & \lesssim \exp\Big( -\widetilde{C}\int_{0}^t b(\tau)d\tau \Big)\big( |\xi|^{|\beta|+1}|\hat{u}_0(\xi)| + |\xi|^{|\beta|}|\hat{u}_1(\xi)| \big).
\end{align*}
%%%
\begin{proof}[Proof of Theorem \ref{Theorem_Effective-Decreasing_Integrable-Decaying}]
In the elliptic zone $\Zell(N)$, we have
\[ \frac{b(\tau)}{|\xi|^2} + g(\tau) \leq C_N^2 b(\tau)g(\tau)^2+g(\tau)\leq \Big(\frac{C_N^2}{2}+1\Big)g(\tau), \]
where we used $|\xi|\geq \dfrac{N}{2g(t)}$ due to the definition of $\Zell(N)$ and $b(t)g(t)\leq \dfrac{1}{2}$ due to condition \textbf{(E3)}. Hence, we find
\[ -\dfrac{1}{\frac{b(\tau)}{|\xi|^2} + g(\tau)} \leq -\frac{\widetilde{C}_N}{g(\tau)}. \]
Therefore, from Lemma \ref{Lemma:Effective-Decaying-Large1}, in the case where $b=b(t)$ increasing (and similarly for the decreasing case), we obtain
\begin{equation*}
\exp\Big( -C|\xi|^2\int_0^t\frac{1}{b(\tau)+g(\tau)|\xi|^2}d\tau \Big) \leq \exp\Big( -\widetilde{C}_N\int_0^t\frac{1}{g(\tau)}d\tau \Big).
\end{equation*}
Moreover, from Lemma \ref{Lemma:Effective-Decaying-Large2}, in the reduced zone $\Zred(N,\varepsilon)$, using  $|\xi|\geq \dfrac{\varepsilon}{2g(t)}$ due to the definition of $\Zred(N,\varepsilon)$ we have
\begin{equation*}
\exp\Big( -C|\xi|^2\int_0^t\frac{1}{b(\tau)+g(\tau)|\xi|^2}d\tau \Big) \leq \exp\Big( -\widetilde{C}_\varepsilon\int_0^t\frac{1}{g(\tau)}d\tau \Big).
\end{equation*}
Thus, from Lemma \ref{Lemma:Effective-Decaying-Large1} and Lemma \ref{Lemma:Effective-Decaying-Large2}, in the case where $b=b(t)$ increasing (similarly for the decreasing case), we have
\begin{align*}
|\xi|^{|\beta|}|\hat{u}(t,\xi)| &\lesssim \exp\Big( -\widetilde{C}_{N,\varepsilon}\int_0^t\frac{1}{g(\tau)}d\tau \Big)\big( |\xi|^{|\beta|}|\hat{u}_0(\xi)| + |\xi|^{|\beta|-1}|\hat{u}_1(\xi)| \big) \quad \mbox{for} \quad |\beta| \geq 1, \\
|\xi|^{|\beta|}|\hat{u}_t(t,\xi)|& \lesssim  \exp\Big( -\widetilde{C}_{N,\varepsilon}\int_0^t\frac{1}{g(\tau)}d\tau \Big)\big( |\xi|^{|\beta|+1}|\hat{u}_0(\xi)| + |\xi|^{|\beta|}|\hat{u}_1(\xi)| \big) \quad \mbox{for} \quad |\beta| \geq 0.
\end{align*}
On the other hand, from Lemma \ref{Lemma:Effective-Decaying-Large3} and Lemma \ref{Lemma:Effective-Decaying-Large4}, in the case where $b=b(t)$ increasing (similarly for the decreasing case), we have
\begin{align*}
|\xi|^{|\beta|}|\hat{u}(t,\xi)| & \lesssim \exp\Big( -\widetilde{C}\int_{0}^t b(\tau)d\tau \Big)\big( |\xi|^{|\beta|}|\hat{u}_0(\xi)| + |\xi|^{|\beta|-1}|\hat{u}_1(\xi)| \big), \\
|\xi|^{|\beta|}|\hat{u}_t(t,\xi)| & \lesssim \exp\Big( -\widetilde{C}\int_{0}^t b(\tau)d\tau \Big)\big( |\xi|^{|\beta|+1}|\hat{u}_0(\xi)| + |\xi|^{|\beta|}|\hat{u}_1(\xi)| \big).
\end{align*}
Hence, in order to compare these estimates, using condition \textbf{(E3)}, that is,
\[ b(t)g(t)\leq \frac{1}{2} \qquad \text{implies} \qquad -\frac{1}{2g(t)} \leq -b(t), \]
we may obtain
\[ \exp\Big( -C_1\int_0^t\frac{1}{g(\tau)}d\tau \Big) \leq \exp\bigg( -C_2\int_{0}^t b(\tau)d\tau \bigg). \]
Finally, from Lemma \ref{Lemma:Effective-Decaying-Large0} and Lemma \ref{Lemma:Effective-Decaying-Large5}, in the case where $b=b(t)$ increasing (similarly for the decreasing case), using
\[ |\xi|^r\exp\Big( -C_N|\xi|^2\int_{t_0}^{t}\frac{1}{b(\tau)}d\tau \Big) \lesssim \Big( 1+\int_{t_0}^{t}\frac{1}{b(\tau)}d\tau \Big)^{-\frac{r}{2}} \quad \mbox{for} \quad r\geq0, \]
may complete the proof.
\end{proof}

\section{The friction term is over-damping producing} \label{Section_Overdamping}
In this section, we consider $b(t)u_t$ as an over-damping producing dissipation. We assume that the following properties of $b=b(t)$ hold as in the paper \cite{Wirth-Effective=2007}:
\begin{enumerate}
\item[\textbf{(B1)}] $b(t)>0$ and $b'(t)>0$,
\item[\textbf{(B2)}] $\dfrac{|b^{(k)}(t)|}{b(t)} \lesssim \dfrac{1}{(1+t)^{k}}$\,\, for\,\, $k=1,2$.
\end{enumerate}
Moreover, we assume the conditions
\begin{enumerate}
\item[\textbf{(OD1)}] $\displaystyle\int_0^\infty\dfrac{1}{b(\tau)}d\tau < \infty$,
\item[\textbf{(OD2)}] $b'(t) = o(b(t)^2)$ for $t\to\infty$.
\end{enumerate}
We are going to use the the equation \eqref{AuxiliaryEquation3}. This means, we consider the following transformed equation:
\begin{equation} \label{Eq:Overdamping1}
D_t^2w + \Big(-|\xi|^2 + \dfrac{b(t)g(t)}{2}|\xi|^2 + \dfrac{g(t)^2}{4}|\xi|^4 + \frac{g'(t)}{2}|\xi|^2 +  \dfrac{b(t)^2}{4} + \frac{b'(t)}{2} \Big)w=0.
\end{equation}
\subsection{Model with increasing time-dependent coefficient $g=g(t)$} \label{Subsect:Overdamping-Increasing}
As we studied in Section 2 of the paper \cite{AslanReissig2023}, we assume the following properties of the function $g=g(t)$:
\begin{enumerate}
\item[\textbf{(A1)}] $g(t)>0$ and $g'(t)>0$ for all $t \in [0,\infty)$,
\item[\textbf{(A2)}] $\displaystyle\int_0^\infty\dfrac{1}{g(\tau)}d\tau < \infty$,
\item[\textbf{(A3)}] $|d_t^kg(t)|\leq C_kg(t)\Big( \dfrac{g(t)}{G(t)} \Big)^k$ for all $t \in [0,\infty)$, $k=1,2$, where $G(t):=\dfrac{1}{2}\displaystyle\int_0^t g(\tau)d\tau$ and $C_1$, $C_2$ are positive constants.
\end{enumerate}
\begin{theorem} \label{Theorem:Overdamping1}
Let us consider the Cauchy problem
\begin{equation*}
\begin{cases}
u_{tt}- \Delta u + b(t)u_t -g(t)\Delta u_t=0, &(t,x) \in (0,\infty) \times \mathbb{R}^n, \\
u(0,x)= u_0(x),\quad u_t(0,x)= u_1(x), &x \in \mathbb{R}^n.
\end{cases}
\end{equation*}
We assume that the coefficient $g=g(t)$ satisfies the conditions \textbf{(A1)} to \textbf{(A3)} and $b=b(t)$ holds the conditions \textbf{(B1)} and \textbf{(B2)}, \textbf{(OD1)} and \textbf{(OD2)}. Moreover, for the data we suppose $(u_0,\langle D \rangle^{-2}u_1)\in \dot{H}^{|\beta|} \times \dot{H}^{|\beta|+2}$ with $|\beta|\geq 0$. Then, we have the following estimates for Sobolev solutions:
\begin{align*}
\|\,|D|^{|\beta|} u(t,\cdot)\|_{L^2} & \lesssim  \|u_0\|_{\dot{H}^{|\beta|}} + \|\langle D \rangle^{-2} u_1\|_{\dot{H}^{|\beta|}},\\
\|\,|D|^{|\beta|} u_t(t,\cdot)\|_{L^2} & \lesssim g(t)\big( \|u_0\|_{\dot{H}^{|\beta|+2}} + \|\langle D \rangle^{-2}u_1\|_{\dot{H}^{|\beta|+2}} \big) + b(t)\big( \|u_0\|_{\dot{H}^{|\beta|}} + \|\langle D \rangle^{-2}u_1\|_{\dot{H}^{|\beta|}} \big).
\end{align*}
\end{theorem}
\begin{proof}
We divide the extended phase space $[0,\infty)\times \mathbb{R}^n$ into zones as follows:
\begin{itemize}
\item elliptic zone:
\begin{align*}
\Zell(N)=\left\{ (t,\xi)\in [0,\infty)\times \mathbb{R}^n: G(t)|\xi|^2 \geq N \right\},
\end{align*}
\item pseudo-differential zone:
\begin{align*} \label{zonesellipticcase}
\Zpd(N)=\left\{ (t,\xi)\in [0,\infty)\times\mathbb{R}^n: G(t)|\xi|^2 \leq N \right\},
\end{align*}
\end{itemize}
where $N>0$ is sufficiently large zone constant. Moreover, the separating line $t_\xi=t(|\xi|)$ is defined by
\[ t_\xi=\left\{ (t,\xi) \in [0,\infty) \times \mathbb{R}^n: G(t)|\xi|^2 = N \right\} .\]
\subsubsection{Considerations in the elliptic zone $\Zell(N)$} \label{Subsection_Overdamping_Zell}
Let us introduce the following family of symbol classes in the elliptic zone $\Zell(N)$.
\begin{definition} \label{Def:Overdamping-increasing-symbol}
A function $f=f(t,\xi)$ belongs to the elliptic symbol class $S_{\text{ell}}^\ell\{m_1,m_2\}$ if it holds
\begin{equation*}
|D_t^kf(t,\xi)|\leq C_{k}\big(b(t)+g(t)|\xi|^2\big)^{m_1}\Big( \frac{g'(t)|\xi|^2+b'(t)}{b(t)+g(t)|\xi|^2} \Big)^{m_2+k}
\end{equation*}
for all $(t,\xi)\in \Zell(N)$ and all $k\leq \ell$.
\end{definition}
Some useful rules of the symbolic calculus are collected in the following proposition.
\begin{proposition} \label{Prop:Overdamping-symbol}  The following statements are true:
\begin{itemize}
\item $S_{\text{ell}}^\ell\{m_1,m_2\}$ is a vector space for all nonnegative integers $\ell$;
\item $S_{\text{ell}}^\ell\{m_1,m_2\}\cdot S_{\text{ell}}^{\ell}\{m_1',m_2'\}\hookrightarrow S^{\ell}_{\text{ell}}\{m_1+m_1',m_2+m_2'\}$;
\item $D_t^kS_{\text{ell}}^\ell\{m_1,m_2\}\hookrightarrow S_{\text{ell}}^{\ell-k}\{m_1,m_2+k\}$
for all nonnegative integers $\ell$ with $k\leq \ell$;
\item $S_{\text{ell}}^{0}\{-1,2\}\hookrightarrow L_{\xi}^{\infty}L_t^1\big( \Zell(N) \big)$.
\end{itemize}
\end{proposition}
\begin{proof}
Let us verify the last statement. Indeed, if $f=f(t,\xi)\in S_{\text{ell}}^{0}\{-1,2\}$, then using the estimate:
\begin{align*}
 \frac{( g'(t)|\xi|^2+b'(t))^2}{( b(t)+g(t)|\xi|^2)^3}  &\lesssim \frac{g'(t)^2|\xi|^4+|b'(t)|\,g'(t)|\xi|^2+b'(t)^2}{( b(t)+g(t)|\xi|^2)^3}  \\
& \lesssim \frac{g(t)^4|\xi|^4+G(t)|b'(t)|\, g(t)^2|\xi|^2+G(t)^2 \,b'(t)^2}{G(t)^2( b(t)+g(t)|\xi|^2)^3} \qquad \small{(\text{we used condition \textbf{(A3)}})} \\
& \lesssim \frac{g(t)^4|\xi|^4}{G(t)^2(g(t)|\xi|^2)^3} + \frac{|b'(t)|\, g(t)^2|\xi|^2}{G(t)(g(t)|\xi|^2)^2\,b(t)} + \frac{b'(t)^2}{b(t)^3} \\
& \lesssim \frac{g(t)}{G(t)^2|\xi|^2} + \frac{b(t)}{1+t}\frac{1}{G(t)|\xi|^2b(t)} + \frac{b(t)^2}{(1+t)^2}\frac{1}{b(t)^3} \qquad \small{(\text{we used condition \textbf{(B3)}})} \\
& = \frac{g(t)}{G(t)^2|\xi|^2} + \frac{1}{(1+t)G(t)|\xi|^2} +\frac{1}{b(t)(1+t)^2} \\
& \lesssim \frac{g(t)}{G(t)^2|\xi|^2}  + \frac{1}{b(t)} \qquad \small{(\text{we used the definition of}\,\, \Zell(N))},
\end{align*}
we get
\begin{align*}
\int_{t_\xi}^{\infty}|f(\tau,\xi)|d\tau  & \lesssim \frac{1}{|\xi|^2}\int_{t_\xi}^{\infty}\frac{g(\tau)}{G(\tau)^2} d\tau + \int_{0}^{\infty}\frac{1}{b(\tau)}d\tau \leq \frac{C}{G(t_{\xi})|\xi|^2} + \widetilde{C} = \frac{C}{N} + \widetilde{C},
\end{align*}
where we used the definition of the separating line $t_\xi$ and condition \textbf{(OD1)}.
\end{proof}
In \eqref{Eq:Overdamping1}, due to the identity
\[  \frac{g(t)^2}{4}|\xi|^4 + \dfrac{b(t)g(t)}{2}|\xi|^2 + \dfrac{b(t)^2}{4} = \Big( \frac{g(t)}{2}|\xi|^2 + \frac{b(t)}{2} \Big)^2, \]
we consider the following micro-energy:
\[ W=W(t,\xi) := \Big[ \big( \frac{g(t)}{2}|\xi|^2 + \frac{b(t)}{2} \big)w,D_tw \Big]^{\text{T}}. \]
Then, by \eqref{Eq:Overdamping1} we obtain that $W=W(t,\xi)$ satisfies the following system of first order:
\begin{equation} \label{SystemOverDamping}
D_tW=\underbrace{\left[ \left( \begin{array}{cc}
0 & \dfrac{g(t)}{2}|\xi|^2+\dfrac{b(t)}{2} \\
-\dfrac{g(t)}{2}|\xi|^2-\dfrac{b(t)}{2} & 0
\end{array} \right) + \left( \begin{array}{cc}
\dfrac{D_t( b(t)+g(t)|\xi|^2)}{b(t)+g(t)|\xi|^2} & 0 \\
-\dfrac{(g'(t)-2)|\xi|^2}{b(t)+g(t)|\xi|^2}-\dfrac{b'(t)}{b(t)+g(t)|\xi|^2} & 0
\end{array} \right)\right]}_{A_W}W.
\end{equation}
We want to estimate the fundamental solution $E_W=E_W(t,s,\xi)$ to the system \eqref{SystemOverDamping}, namely, the solution to
\begin{equation*}
D_tE_W(t,s,\xi)=A_W(t,\xi)E_W(t,s,\xi), \quad E_W(s,s,\xi)=I \quad \mbox{for any} \quad t\geq s\geq t_\xi.
\end{equation*}
We denote by $M$ the matrix consisting of eigenvectors of the first matrix on the right-hand side and its inverse matrix
\[ M = \left( \begin{array}{cc}
i & -i \\
1 & 1
\end{array} \right), \qquad M^{-1}=\frac{1}{2}\left( \begin{array}{cc}
-i & 1 \\
i & 1
\end{array} \right). \]
Then, defining $W^{(0)}:=M^{-1}W$ we get the system
\begin{equation*}
D_tW^{(0)}=\big( \mathcal{D}(t,\xi)+\mathcal{R}(t,\xi) \big)W^{(0)},
\end{equation*}
where
\begin{align*}
\mathcal{D}(t,\xi) &= \left( \begin{array}{cc}
-i\Big( \dfrac{g(t)}{2}|\xi|^2 + \dfrac{b(t)}{2} \Big) & 0 \\
0 & i\Big( \dfrac{g(t)}{2}|\xi|^2 + \dfrac{b(t)}{2} \Big)
\end{array} \right), \\
\mathcal{R}_1(t,\xi) &= \frac{1}{2} \left( \begin{array}{cc}
\dfrac{D_t( b(t)+g(t)|\xi|^2)}{b(t)+g(t)|\xi|^2}-i\dfrac{(g'(t)-2)|\xi|^2}{b(t)+g(t)|\xi|^2} & -\dfrac{D_t( b(t)+g(t)|\xi|^2)}{b(t)+g(t)|\xi|^2}+i\dfrac{(g'(t)-2)|\xi|^2}{b(t)+g(t)|\xi|^2} \\
-\dfrac{D_t( b(t)+g(t)|\xi|^2)}{b(t)+g(t)|\xi|^2}-i\dfrac{(g'(t)-2)|\xi|^2}{b(t)+g(t)|\xi|^2} & \dfrac{D_t( b(t)+g(t)|\xi|^2)}{b(t)+g(t)|\xi|^2}+i\dfrac{(g'(t)-2)|\xi|^2}{b(t)+g(t)|\xi|^2}
\end{array} \right), \\
\mathcal{R}_2(t,\xi) &= \frac{1}{2} \left( \begin{array}{cc}
-i\dfrac{b'(t)}{b(t)+g(t)|\xi|^2} & i\dfrac{b'(t)}{b(t)+g(t)|\xi|^2} \\
-i\dfrac{b'(t)}{b(t)+g(t)|\xi|^2} & i\dfrac{b'(t)}{b(t)+g(t)|\xi|^2}
\end{array} \right),
\end{align*}
where $\mathcal{D}\in S_{\text{ell}}^2\{1,0\}$ and with $\mathcal{R}:=\mathcal{R}_1+\mathcal{R}_2$, $\mathcal{R}\in S_{\text{ell}}^1\{0,1\}$.

We perform one more step of diagonalization procedure. We define $F_0(t,\xi):=\diag\mathcal{R}(t,\xi)$. The difference of the diagonal entries of the matrix $\mathcal{D}(t,\xi)+F_0(t,\xi)$ is
\begin{align*}
g(t)|\xi|^2&+b(t) + \frac{(g'(t)-2)|\xi|^2+b'(t)}{b(t)+g(t)|\xi|^2} \\
& \leq \frac{g(t)^2|\xi|^4+2b(t)g(t)|\xi|^2+b(t)^2+g'(t)|\xi|^2+b'(t)}{b(t)+g(t)|\xi|^2} \\
& \leq \frac{( b(t)+g(t)|\xi|^2)^2+\frac{g(t)^2|\xi|^4}{G(t)|\xi|^2}+b(t)^2}{b(t)+g(t)|\xi|^2} \qquad \small{(\text{we used condition \textbf{(A3)}})} \\
& \lesssim \frac{( b(t)+g(t)|\xi|^2)^2+g(t)^2|\xi|^4+b(t)^2}{b(t)+g(t)|\xi|^2} \leq b(t)+g(t)|\xi|^2=:i\delta(t,\xi),
\end{align*}
where we use the fact that $b'(t)=o(b(t)^2)$ for $t \to \infty$.

Now we choose a matrix $N^{(1)}=N^{(1)}(t,\xi)$ such that
\begin{align*}
&N^{(1)}(t,\xi) := \left( \begin{array}{cc}
0 & -\dfrac{\mathcal{R}_{12}}{\delta(t,\xi)} \\
\dfrac{\mathcal{R}_{21}}{\delta(t,\xi)} & 0
\end{array} \right) \\
& = \left( \begin{array}{cc}
0 & i\dfrac{D_t( b(t)+g(t)|\xi|^2)}{2\big( b(t)+g(t)|\xi|^2 \big)^2}+\dfrac{(g'(t)-2)|\xi|^2+b'(t)}{2( b(t)+g(t)|\xi|^2)^2} \\
-i\dfrac{D_t( b(t)+g(t)|\xi|^2)}{2( b(t)+g(t)|\xi|^2)^2}+\dfrac{(g'(t)-2)|\xi|^2+b'(t)}{2( b(t)+g(t)|\xi|^2)^2} & 0
\end{array} \right)
\end{align*}
with $N^{(1)}\in S_{\text{ell}}^1\{-1,1\}$. For a sufficiently large zone constant $N$ and all $t\geq t_\xi$ the matrix $N_1=N_1(t,\xi) := I + N^{(1)}(t,\xi)$ belongs to $S_{\text{ell}}^1\{0,0\}$ and is invertible with uniformly bounded inverse matrix $N_1^{-1}=N_1^{-1}(t,\xi)$. Namely, by using the condition \textbf{(A3)} of $g=g(t)$ and condition \textbf{(OD2)} of $b=b(t)$, we find
\begin{align*} \label{Eq:OverdampingR1}
& \Big|\dfrac{D_t( b(t)+g(t)|\xi|^2)}{( b(t)+g(t)|\xi|^2)^2}\Big| \leq \dfrac{C_1\frac{g(t)^2}{G(t)}|\xi|^2+|b'(t)|}{( b(t)+g(t)|\xi|^2)^2} \leq \dfrac{C_1g(t)^2|\xi|^2+G(t)|b'(t)|}{G(t)( g(t)^2|\xi|^4+b(t)^2)} \nonumber \\
& \qquad \leq \frac{C_1}{G(t)|\xi|^2} + \frac{|b'(t)|}{b(t)^2} \leq \frac{C}{N}<1,
\end{align*}
where we used the definition of $\Zell(N)$ with sufficiently large zone constant $N$ and $b'(t)=o(b(t)^2)$ for $t \to \infty$.

Let
\begin{align*}
B^{(1)}(t,\xi) &= D_tN^{(1)}(t,\xi)-( \mathcal{R}(t,\xi)-F_0(t,\xi))N^{(1)}(t,\xi), \\
\mathcal{R}_3(t,\xi) &= -N_1^{-1}(t,\xi)B^{(1)}(t,\xi)\in S_{\text{ell}}^0\{-1,2\}.
\end{align*}
Then, we have the following operator identity:
\begin{equation*}
( D_t-\mathcal{D}(t,\xi)-\mathcal{R}(t,\xi))N_1(t,\xi)=N_1(t,\xi)( D_t-\mathcal{D}(t,\xi)-F_0(t)-\mathcal{R}_3(t,\xi)).
\end{equation*}
\begin{proposition} \label{Prop_Overdapming_EstZell}
The fundamental solution $E_{\text{ell}}^{W}=E_{\text{ell}}^{W}(t,s,\xi)$ to the transformed operator
\[ D_t-\mathcal{D}(t,\xi)-F_0(t,\xi)-\mathcal{R}_3(t,\xi) \]
can be estimated by
\begin{equation*}
\big( |E_{\text{ell}}^{W}(t,s,\xi)| \big) \lesssim \frac{b(t)+g(t)|\xi|^2}{b(s)+g(s)|\xi|^2}\exp\bigg( \frac{1}{2}\int_{s}^{t}\big( g(\tau)|\xi|^2 +b(\tau) \big)d\tau \bigg)
\left( \begin{array}{cc}
1 & 1 \\
1 & 1
\end{array} \right),
\end{equation*}
with $(t,\xi),(s,\xi)\in \Zell(N)$, $t_\xi\leq s\leq t$.
\end{proposition}
\begin{proof}
We will follow the approach from the proof of Proposition \ref{Prop_Scattering_EllZone}. Namely, we transform the system for $E_{\text{ell}}^{W}=E_{\text{ell}}^{W}(t,s,\xi)$ to an integral equation for a new matrix-valued function $\mathcal{Q}_{\text{ell}}=\mathcal{Q}_{\text{ell}}(t,s,\xi)$. If we differentiate the term
\[ \exp \bigg\{ -i\int_{s}^{t}\big( \mathcal{D}(\tau,\xi)+F_0(\tau,\xi) \big)d\tau \bigg\}E_{\text{ell}}^{W}(t,s,\xi) \]
and, then integrate on $[s,t]$, we find that $E_{\text{ell}}^{W}=E_{\text{ell}}^{W}(t,s,\xi)$ satisfies the following integral equation:
\begin{align*}
E_{\text{ell}}^{W}(t,s,\xi) & = \exp\bigg\{ i\int_{s}^{t}\big( \mathcal{D}(\tau,\xi)+F_0(\tau,\xi) \big)d\tau \bigg\}E_{\text{ell}}^{W}(s,s,\xi)\\
& \quad + i\int_{s}^{t} \exp \bigg\{ i\int_{\theta}^{t}\big( \mathcal{D}(\tau,\xi)+F_0(\tau,\xi) \big)d\tau \bigg\}\mathcal{R}_1(\theta,\xi)E_{\text{ell}}^{W}(\theta,s,\xi)\,d\theta.
\end{align*}
We define
\[ \mathcal{Q}_{\text{ell}}(t,s,\xi)=\exp\bigg\{ -\int_{s}^{t}\beta(\tau,\xi)d\tau \bigg\} E_{\text{ell}}^{W}(t,s,\xi), \]
with a suitable $\beta=\beta(t,\xi)$ which will be fixed later. Then, this new term satisfies the new integral equation
\begin{align*}
\mathcal{Q}_{\text{ell}}(t,s,\xi)=&\exp \bigg\{ \int_{s}^{t}\big( i\mathcal{D}(\tau,\xi)+iF_0(\tau,\xi)-\beta(\tau,\xi)I \big)d\tau \bigg\}\\
& \quad + \int_{s}^{t} \exp \bigg\{ \int_{\theta}^{t}\big( i\mathcal{D}(\tau,\xi)+iF_0(\tau,\xi)-\beta(\tau,\xi)I \big)d\tau \bigg\}\mathcal{R}_3(\theta,\xi)\mathcal{Q}_{\text{ell}}(\theta,s,\xi)\,d\theta.
\end{align*}
The function $\mathcal{R}_3=\mathcal{R}_3(\theta,\xi)$ is uniformly integrable over the elliptic zone because of Proposition \ref{Prop:Overdamping-symbol}.
Hence, if the exponential term is bounded, then the solution $\mathcal{Q}_{\text{ell}}=\mathcal{Q}_{\text{ell}}(t,s,\xi)$ of the integral equation is uniformly bounded over the elliptic zone for a suitable weight $\beta=\beta(t,\xi)$.

The main entries of the diagonal matrix $i\mathcal{D}(t,\xi)+iF_0(t,\xi)$ are given by
\begin{align*}
(I) &= \dfrac{g(t)|\xi|^2}{2} + \dfrac{b(t)}{2}+\dfrac{g'(t)|\xi|^2+b'(t)}{2\big( b(t)+g(t)|\xi|^2 \big)}+\dfrac{(g'(t)-2)|\xi|^2+b'(t)}{2\big( b(t)+g(t)|\xi|^2 \big)},\\
(II) &= -\dfrac{g(t)|\xi|^2}{2} - \dfrac{b(t)}{2}+\dfrac{g'(t)|\xi|^2+b'(t)}{2\big( b(t)+g(t)|\xi|^2 \big)}-\dfrac{(g'(t)-2)|\xi|^2+b'(t)}{2\big( b(t)+g(t)|\xi|^2 \big)}.
\end{align*}
We may see that the term $(I)$ is dominant in $\Zell(N)$ with respect to $(II)$ for $t\geq t_\xi$. Therefore, we choose the weight $\beta=\beta(t,\xi)=(I)$. By this choice, we get
\[ i\mathcal{D}(\tau,\xi)+iF_0(\tau,\xi)-\beta(\tau,\xi)I = \left( \begin{array}{cc}
0 & 0 \\
0 & -\big( b(t)+g(t)|\xi|^2 \big)-\dfrac{(g'(t)-2)|\xi|^2+b'(t)}{b(t)+g(t)|\xi|^2}
\end{array} \right). \]
It follows
\begin{align*}
H(t,s,\xi) & =\exp \bigg\{ \int_{s}^{t}\big( i\mathcal{D}(\tau,\xi)+iF_0(\tau,\xi)-\beta(\tau,\xi)I \big)d\tau \bigg\}\\
& = \diag \bigg( 1, \exp \bigg\{ \int_{s}^{t}\Big( -\big( b(t)+g(t)|\xi|^2 \big)-\dfrac{(g'(t)-2)|\xi|^2+b'(t)}{b(t)+g(t)|\xi|^2} \Big)d\tau \bigg\} \bigg)\rightarrow \left( \begin{array}{cc}
1 & 0 \\
0 & 0
\end{array} \right)
\end{align*}
as $t\rightarrow \infty$ for any fixed $s\geq t_\xi$. Hence, the matrix $H=H(t,s,\xi)$ is uniformly bounded for $(s,\xi),(t,\xi)\in \Zell(N)$. So, the representation of $\mathcal{Q}_{\text{ell}}=\mathcal{Q}_{\text{ell}}(t,s,\xi)$ by a Neumann series gives
\begin{align*}
\mathcal{Q}_{\text{ell}}(t,s,\xi)=H(t,s,\xi)+\sum_{k=1}^{\infty}i^k\int_{s}^{t}H(t,t_1,\xi)\mathcal{R}_3(t_1,\xi)&\int_{s}^{t_1}H(t_1,t_2,\xi)\mathcal{R}_3(t_2,\xi) \\
& \cdots \int_{s}^{t_{k-1}}H(t_{k-1},t_k,\xi)\mathcal{R}_3(t_k,\xi)dt_k\cdots dt_2dt_1.
\end{align*}
Then, this series is convergent, since $\mathcal{R}_3=\mathcal{R}_3(t,\xi)$ is uniformly integrable over $\Zell(N)$ due to the last item of Proposition \ref{Prop:Overdamping-symbol}. Hence, from the last considerations we may conclude
\begin{align*}
E_{\text{ell}}^{W}(t,s,\xi)&=\exp \bigg\{ \int_{s}^{t}\beta(\tau,\xi)d\tau \bigg\}\mathcal{Q}_{\text{ell}}(t,s,\xi)\\
& = \exp \bigg\{ \int_{s}^{t}\bigg( \frac{g(\tau)|\xi|^2}{2}+\frac{b(\tau)}{2}+\frac{b'(\tau)+g'(\tau)|\xi|^2}{2( b(\tau)+g(\tau)|\xi|^2)}+\frac{( g'(\tau)-2)|\xi|^2+b'(\tau)}{2( b(\tau)+g(\tau)|\xi|^2)} \bigg)d\tau \bigg\}\mathcal{Q}_{\text{ell}}(t,s,\xi),
\end{align*}
where $\mathcal{Q}_{\text{ell}}=\mathcal{Q}_{\text{ell}}(t,s,\xi)$ is a uniformly bounded matrix. Then, it follows
\begin{align*}
(|E_{\text{ell}}^{V}(t,s,\xi)|) & \lesssim \exp \bigg\{ \int_{s}^{t}\bigg( \frac{g(\tau)|\xi|^2}{2}+\frac{b(\tau)}{2}+\frac{b'(\tau)+g'(\tau)|\xi|^2}{ b(\tau)+g(\tau)|\xi|^2}-\frac{|\xi|^2}{b(\tau)+g(\tau)|\xi|^2} \bigg\} \left( \begin{array}{cc}
1 & 1 \\
1 & 1
\end{array} \right) \\
& \lesssim \frac{b(t)+g(t)|\xi|^2}{b(s)+g(s)|\xi|^2} \exp \bigg( \frac{1}{2}\int_{s}^{t} \big( b(\tau)+g(\tau)|\xi|^2 \big)d\tau \bigg)\left( \begin{array}{cc}
1 & 1 \\
1 & 1
\end{array} \right).
\end{align*}
This completes the proof.
\end{proof}
Now let us come back to
\begin{equation} \label{Eq:Overdamping_ZellBack}
W(t,\xi) = E_W(t,s,\xi)W(s,\xi) \qquad \text{for all} \qquad t_\xi\leq s\leq t,
\end{equation}
that is,
\begin{align*}
\left( \begin{array}{cc}
\gamma(t,\xi)w(t,\xi) \\
D_t w(t,\xi)
\end{array} \right) = E_W(t,s,\xi)\left( \begin{array}{cc}
\gamma(s,\xi)w(s,\xi) \\
D_tw(s,\xi)
\end{array} \right) \qquad \text{for} \qquad t_\xi\leq s\leq t,
\end{align*}
where $\gamma=\gamma(t,\xi):=\frac{g(t)}{2}|\xi|^2+\frac{b(t)}{2}$. Therefore, from Proposition \ref{Prop_Overdapming_EstZell} and \eqref{Eq:Overdamping_ZellBack} we may conclude the following estimates for $t_\xi \leq s \leq t$:
\begin{align*}
\gamma(t,\xi)|w(t,\xi)| & \lesssim \frac{b(t)+g(t)|\xi|^2}{b(s)+g(s)|\xi|^2}\exp\bigg( \frac{1}{2}\int_{s}^{t}\big( b(\tau)+g(\tau)|\xi|^2 \big)d\tau \bigg)\Big( \gamma(s,\xi)|w(s,\xi)| + |w_t(s,\xi)| \Big), \\
|w_t(t,\xi)| & \lesssim \frac{b(t)+g(t)|\xi|^2}{b(s)+g(s)|\xi|^2}\exp\bigg( \frac{1}{2}\int_{s}^{t}\big( b(\tau)+g(\tau)|\xi|^2 \big)d\tau \bigg)\Big( \gamma(s,\xi)|w(s,\xi)| + |w_t(s,\xi)| \Big).
\end{align*}
Using the backward transformation
\[ w(t,\xi)=\exp\bigg( \frac{1}{2} \int_0^t \big( b(\tau)+g(\tau)|\xi|^2 \big)d\tau \bigg)\hat{u}(t,\xi),  \]
we arrive immediately at the following result.
\begin{corollary} \label{Cor:Overdamping_IncreasingZell}
We have the following estimates in the elliptic zone $\Zell(N)$ for $t_\xi \leq s \leq t$:
\begin{align*}
|\xi|^{|\beta|}|\hat{u}(t,\xi)| & \lesssim |\xi|^{|\beta|}|\hat{u}(s,\xi)| + \frac{|\xi|^{|\beta|}}{b(s)+g(s)|\xi|^2}|\hat{u}_t(s,\xi)| \quad \mbox{for} \quad |\beta|\geq 0, \\
|\xi|^{|\beta|}|\hat{u}_t(t,\xi)| & \lesssim \big( b(t)+g(t)|\xi|^2 \big)|\xi|^{|\beta|}|\hat{u}(s,\xi)| + \frac{b(t)+g(t)|\xi|^2}{b(s)+g(s)|\xi|^2}|\xi|^{|\beta|}|\hat{u}_t(s,\xi)| \quad \mbox{for} \quad |\beta|\geq 0.
\end{align*}
\end{corollary}
\subsubsection{Considerations in the pseudo-differential zone $\Zpd(N)$} \label{Subsection_Overdamping_Zpd}
Let us introduce the micro-energy $U=\big( \gamma(t,\xi)\hat{u},D_t\hat{u} \big)^\text{T}$ with $\gamma(t,\xi):=\dfrac{b(t)+g(t)|\xi|^2}{2}$. Then, by \eqref{MainEquationFourier}, we find the system
\begin{equation} \label{Eq:Overdamping_Zpd_System}
D_tU=\underbrace{\left( \begin{array}{cc}
\dfrac{D_t\gamma(t,\xi)}{\gamma(t,\xi)} & \gamma(t,\xi) \\
\dfrac{|\xi|^2}{\gamma(t,\xi)} & i\big( b(t)+g(t)|\xi|^2 \big)
\end{array} \right)}_{A(t,\xi)}U.
\end{equation}
The fundamental solution $E_{\text{pd}}=E_{\text{pd}}(t,s,\xi)$ to the system \eqref{Eq:Overdamping_Zpd_System} is the solution of
\[ D_tE_{\text{pd}}(t,s,\xi)=A(t,\xi)E_{\text{pd}}(t,s,\xi), \quad E_{\text{pd}}(s,s,\xi)=I, \]
for all $0\leq s \leq t$ and $(t,\xi), (s,\xi) \in \Zpd(N)$. Thus, the solution $U=U(t,\xi)$ is represented as
\[ U(t,\xi)=E_{\text{pd}}(t,s,\xi)U(s,\xi). \]
We will use the auxiliary function
\[ \delta=\delta(t,\xi)=\exp\bigg( \int_{0}^{t}\big( b(\tau)+g(\tau)|\xi|^2 \big)d\tau \bigg). \]
The entries $E_{\text{pd}}^{(k\ell)}(t,s,\xi)$, $k,\ell=1,2,$ of the fundamental solution $E_{\text{pd}}(t,s,\xi)$ satisfy the following system for $\ell=1,2$:
\begin{align*}
D_tE_{\text{pd}}^{(1\ell)}(t,s,\xi) &= \frac{D_t\gamma(t,\xi)}{\gamma(t,\xi)}E_{\text{pd}}^{(1\ell)}(t,s,\xi)+\gamma(t,\xi)E_{\text{pd}}^{(2\ell)}(t,s,\xi), \\
D_tE_{\text{pd}}^{(2\ell)}(t,s,\xi) &= \frac{|\xi|^2}{\gamma(t,\xi)}E_{\text{pd}}^{(1\ell)}(t,s,\xi)+i\big( b(t)+g(t)|\xi|^2 \big)E_{\text{pd}}^{(2\ell)}(t,s,\xi).
\end{align*}
Then, by straight-forward calculations (with $\delta_{k \ell}=1$ if $k=\ell$ and $\delta_{k \ell}=0$ otherwise), we get
\begin{align*}
E_{\text{pd}}^{(1\ell)}(t,s,\xi) & = \frac{\gamma(t,\xi)}{\gamma(s,\xi)}\delta_{1\ell}+i\gamma(t,\xi)\int_{s}^{t}E_{\text{pd}}^{(2\ell)}(\tau,s,\xi)d\tau, \\
% E_{\text{pd}}^{(21)}(t,s,\xi) & = \frac{i|\xi|^2}{\delta(t)}\int_{s}^{t}\frac{1}{\gamma(\tau,\xi)}\delta(\tau)E_{\text{pd}}^{(11)}(\tau,s,\xi)d\tau, \\
% E_{\text{pd}}^{(12)}(t,s,\xi) & = i\gamma(t,\xi)\int_{s}^{t}E_{\text{pd}}^{(22)}(\tau,s,\xi)d\tau, \\
E_{\text{pd}}^{(2\ell)}(t,s,\xi) & = \frac{\delta(s,\xi)}{\delta(t,\xi)}\delta_{2\ell}+\frac{i|\xi|^2}{\delta(t,\xi)}\int_{s}^{t}\frac{1}{\gamma(\tau,\xi)}\delta(\tau,\xi)E_{\text{pd}}^{(1\ell)}(\tau,s,\xi)d\tau.
\end{align*}
\begin{proposition} \label{Prop:Overdamping_Zpd}
We have the following estimates in the pseudo-differential zone:
\begin{equation*}
(|E_{\text{pd}}(t,s,\xi)|) \lesssim \frac{b(t)+g(t)|\xi|^2}{b(s)+g(s)|\xi|^2}
\left( \begin{array}{cc}
1 & 1 \\
1 & 1
\end{array} \right)
\end{equation*}
with $(s,\xi),(t,\xi)\in\Zpd(N)$ and $0\leq s\leq t\leq t_\xi$.
\end{proposition}
\begin{proof}
First let us consider the first column. Plugging the representation for $E_{\text{pd}}^{(21)}=E_{\text{pd}}^{(21)}(t,s,\xi)$ into the integral equation for $E_{\text{pd}}^{(11)}=E_{\text{pd}}^{(11)}(t,s,\xi)$ gives
\begin{align*}
E_{\text{pd}}^{(11)}(t,s,\xi) = \frac{\gamma(t,\xi)}{\gamma(s,\xi)}-|\xi|^2\gamma(t,\xi)\int_{s}^{t}\int_{s}^{\tau}\frac{\delta(\theta,\xi)}{\delta(\tau,\xi)}\frac{1}{\gamma(\theta,\xi)}
E_{\text{pd}}^{(11)}(\theta,s,\xi)d\theta d\tau.
\end{align*}
By setting $y(t,s,\xi):=\dfrac{\gamma(s,\xi)}{\gamma(t,\xi)}E_{\text{pd}}^{(11)}(t,s,\xi)$ we obtain
\begin{align*}
y(t,s,\xi) &= 1-|\xi|^2\int_{s}^{t}\int_{\theta}^{t}\frac{\delta(\theta,\xi)}{\delta(\tau,\xi)}y(\theta,s,\xi)d\tau d\theta \\
& = 1+|\xi|^2\int_{s}^{t}\bigg( \int_{\theta}^{t}\frac{1}{b(\tau)+g(\tau)|\xi|^2}\partial_\tau \bigg( \frac{\delta(\theta,\xi)}{\delta(\tau,\xi)} \bigg)d\tau \bigg) y(\theta,s,\xi)d\theta \\
& = 1+|\xi|^2\int_{s}^{t}\bigg( \frac{1}{b(\tau)+g(\tau)|\xi|^2}\frac{\delta(\theta,\xi)}{\delta(\tau,\xi)}\Big|_{\theta}^{t}+\int_{\theta}^{t}\frac{b'(\tau)+g'(\tau)|\xi|^2}{( b(\tau)+g(\tau)|\xi|^2)^2} \frac{\delta(\theta,\xi)}{\delta(\tau,\xi)} d\tau \bigg) y(\theta,s,\xi)d\theta.
\end{align*}
Then, we get
\begin{align*}
|y(t,s,\xi)| &\lesssim 1+|\xi|^2\int_{s}^{t}\bigg( \frac{1}{b(t)+g(t)|\xi|^2}\underbrace{\frac{\delta(\theta,\xi)}{\delta(t,\xi)}}_{\leq1}+\int_{\theta}^{t}\frac{b'(\tau)+g'(\tau)|\xi|^2}{( b(\tau)+g(\tau)|\xi|^2)^2} \underbrace{\frac{\delta(\theta,\xi)}{\delta(\tau,\xi)}}_{\leq1}d\tau \bigg) |y(\theta,s,\xi)|d\theta \\
& \lesssim 1+|\xi|^2\int_{s}^{t}\bigg( \frac{1}{b(t)+g(t)|\xi|^2}-\frac{1}{b(\tau)+g(\tau)|\xi|^2}\Big|_{\theta}^{t} \bigg)|y(\theta,s,\xi)|d\theta \\
& \lesssim 1+|\xi|^2\int_{s}^{t}\frac{1}{b(\theta)+g(\theta)|\xi|^2}|y(\theta,s,\xi)|d\theta.
\end{align*}
Applying Gronwall's inequality and employing \textbf{(A2)}, we get the estimate
\begin{align*}
|y(t,s,\xi)| \lesssim \exp \bigg( \int_{s}^{t}\frac{|\xi|^2}{b(\theta)+g(\theta)|\xi|^2}d\theta \bigg) \lesssim \exp \bigg( \int_{s}^{t}\frac{1}{g(\theta)}d\theta \bigg) \lesssim 1.
\end{align*}
This implies
\[ |E_{\text{pd}}^{(11)}(t,s,\xi)| \lesssim \frac{\gamma(t,\xi)}{\gamma(s,\xi)}=\frac{b(t)+g(t)|\xi|^2}{b(s)+g(s)|\xi|^2}. \]
Now we consider $E_{\text{pd}}^{(21)}(t,s,\xi)$. By using the estimate for $|E_{\text{pd}}^{(11)}(t,s,\xi)|$ we obtain
\begin{align*}
\frac{\gamma(s,\xi)}{\gamma(t,\xi)}|E_{\text{pd}}^{(21)}(t,s,\xi)| & \lesssim |\xi|^2\int_{s}^{t}\frac{1}{\gamma(\tau,\xi)}\frac{\delta(\tau,\xi)}{\delta(t,\xi)}\frac{\gamma(s,\xi)}{\gamma(t,\xi)}|E_{\text{pd}}^{(11)}(\tau,s,\xi)|d\tau \\
& \lesssim |\xi|^2\int_{s}^{t}\frac{1}{b(\tau)+g(\tau)|\xi|^2}\underbrace{\frac{\delta(\tau,\xi)\gamma(\tau,\xi)}{\delta(t,\xi)\gamma(t,\xi)}}_{\leq1}d\tau \lesssim \int_{s}^{t}\frac{|\xi|^2}{b(\tau)+g(\tau)|\xi|^2}d\tau \lesssim 1.
\end{align*}
This gives
\[ |E_{\text{pd}}^{(21)}(t,s,\xi)| \lesssim \frac{\gamma(t,\xi)}{\gamma(s,\xi)}=\frac{b(t)+g(t)|\xi|^2}{b(s)+g(s)|\xi|^2}. \]
Next, we consider the entries of the second column. Plugging the representation for $E_{\text{pd}}^{(22)}=E_{\text{pd}}^{(22)}(t,s,\xi)$ into the integral equation for $E_{\text{pd}}^{(12)}=E_{\text{pd}}^{(12)}(t,s,\xi)$ gives
\begin{align*}
E_{\text{pd}}^{(12)}(t,s,\xi) = i\gamma(t,\xi)\int_{s}^{t}\frac{\delta(s,\xi)}{\delta(\tau,\xi)}d\tau-
|\xi|^2\gamma(t,\xi)\int_{s}^{t}\int_{s}^{\tau}\frac{\delta(\theta,\xi)}{\delta(\tau,\xi)}\frac{1}{\gamma(\theta,\xi)}E_{\text{pd}}^{(12)}(\theta,s,\xi)d\theta d\tau.
\end{align*}
After setting $y(t,s,\xi):=\dfrac{\gamma(s,\xi)}{\gamma(t,\xi)}E_{\text{pd}}^{(12)}(t,s,\xi)$, it follows
\begin{align*}
y(t,s,\xi) &= -i\gamma(s,\xi)\int_{s}^{t}\frac{1}{b(\tau)+g(\tau)|\xi|^2}\partial_\tau \bigg( \frac{\delta(s,\xi)}{\delta(\tau,\xi)} \bigg)d\tau \\
& \qquad + |\xi|^2\int_{s}^{t}\bigg( \int_{\theta}^{t}\frac{1}{b(\tau)+g(\tau)|\xi|^2}\partial_\tau \bigg( \frac{\delta(\theta,\xi)}{\delta(\tau,\xi)} \bigg)d\tau \bigg)y(\theta,s,\xi)d\theta d\tau \\
&= i\gamma(s,\xi) \bigg( -\frac{1}{b(\tau)+g(\tau)|\xi|^2}\frac{\delta(s,\xi)}{\delta(\tau,\xi)}\Big|_s^t - \int_{s}^{t}\frac{b'(\tau)+g'(\tau)|\xi|^2}{( b(\tau)+g(\tau)|\xi|^2)^2} \underbrace{\frac{\delta(\theta,\xi)}{\delta(\tau,\xi)}}_{\leq1}d\tau \bigg) \\
& \qquad + |\xi|^2\int_{s}^{t}\bigg( \frac{1}{b(t)+g(t)|\xi|^2}\underbrace{\frac{\delta(\theta,\xi)}{\delta(t,\xi)}}_{\leq1}+\int_{\theta}^{t}\frac{b'(\tau)+g'(\tau)|\xi|^2}{( b(\tau)+g(\tau)|\xi|^2)^2} \underbrace{\frac{\delta(\theta,\xi)}{\delta(\tau,\xi)}}_{\leq1}d\tau \bigg)y(\theta,s,\xi)d\theta.
\end{align*}
Therefore, we have
\begin{align*}
|y(t,s,\xi)| &\lesssim \gamma(s,\xi) \bigg( \frac{1}{b(s)+g(s)|\xi|^2} + \frac{1}{b(t)+g(t)|\xi|^2} \bigg) \\
& \qquad + |\xi|^2\int_{s}^{t}\bigg( \frac{1}{b(t)+g(t)|\xi|^2}+\int_{\theta}^{t}\frac{b'(\tau)+g'(\tau)|\xi|^2}{( b(\tau)+g(\tau)|\xi|^2)^2}d\tau \bigg)|y(\theta,s,\xi)|d\theta \\
&\lesssim \gamma(s,\xi) \bigg( \frac{2}{b(s)+g(s)|\xi|^2} \bigg) + |\xi|^2\int_{s}^{t}\bigg( \frac{1}{b(t)+g(t)|\xi|^2} -\frac{1}{b(\tau)+g(\tau)|\xi|^2}\Big|_{\theta}^{t} \bigg)|y(\theta,s,\xi)|d\theta \\
& = 1 + \int_{s}^{t}\frac{|\xi|^2}{b(\theta)+g(\theta)|\xi|^2}|y(\theta,s,\xi)|d\theta.
\end{align*}
Employing Gronwall's inequality, we have the estimate
\begin{align*}
|y(t,s,\xi)| \lesssim \exp \bigg( \int_{s}^{t}\frac{|\xi|^2}{b(\theta)+g(\theta)|\xi|^2}d\theta \bigg) \lesssim 1.
\end{align*}
Hence, we may conclude that
\[ |E_{\text{pd}}^{(12)}(t,s,\xi)| \lesssim \frac{\gamma(t,\xi)}{\gamma(s,\xi)}=\frac{b(t)+g(t)|\xi|^2}{b(s)+g(s)|\xi|^2}. \]
Finally, let us estimate $|E_{\text{pd}}^{(22)}(t,s,\xi)|$ by using the above estimate for $|E_{\text{pd}}^{(12)}(t,s,\xi)|$. It holds
\begin{align*}
|E_{\text{pd}}^{(22)}(t,s,\xi)| & \lesssim \frac{\delta(s,\xi)}{\delta(t,\xi)}+|\xi|^2\int_{s}^{t}\frac{\delta(\tau,\xi)}{\delta(t,\xi)}\frac{1}{\gamma(\tau,\xi)}|E_{\text{pd}}^{(12)}(\tau,s,\xi)|d\tau \\
& \lesssim \frac{\delta(s,\xi)}{\delta(t,\xi)}+|\xi|^2\int_{s}^{t}\frac{\delta(\tau,\xi)}{\delta(t,\xi)}\frac{1}{\gamma(\tau,\xi)}\frac{\gamma(\tau,\xi)}{\gamma(s,\xi)}d\tau.
\end{align*}
Then, we have
\begin{align*}
\frac{\gamma(s,\xi)}{\gamma(t,\xi)}|E_{\text{pd}}^{(22)}(t,s,\xi)| & \lesssim \underbrace{\frac{\delta(s,\xi)\gamma(s,\xi)}{\delta(t,\xi)\gamma(t,\xi)}}_{\leq1}+|\xi|^2\int_{s}^{t}
\underbrace{\frac{\delta(\tau,\xi)\gamma(\tau,\xi)}{\delta(t,\xi)\gamma(t,\xi)}}_{\leq1}\frac{1}{\gamma(\tau,\xi)}d\tau \\
& \lesssim 1 + |\xi|^2\int_{s}^{t}\frac{1}{b(\tau)+g(\tau)|\xi|^2}d\tau \lesssim 1.
\end{align*}
This shows that
\[ |E_{\text{pd}}^{(22)}(t,s,\xi)| \lesssim \frac{\gamma(t,\xi)}{\gamma(s,\xi)}=\frac{b(t)+g(t)|\xi|^2}{b(s)+g(s)|\xi|^2}. \]
This completes the proof.
\end{proof}
Now let us come back to
\begin{equation} \label{Eq:OverdampingPseudoZone}
U(t,\xi) = E(t,0,\xi)U(0,\xi) \quad \text{for all} \quad 0\leq t\leq t_{\xi}.
\end{equation}
Because of \eqref{Eq:OverdampingPseudoZone} and Proposition \ref{Prop:Overdamping_Zpd}, the following statement can be concluded.
\begin{corollary} \label{Cor:Overdamping_IncreasingZpd}
In the pseudo-differential zone $\Zpd(N)$ the following estimates hold for all $0\leq t\leq t_{\xi}$:
\begin{align*}
|\xi|^{|\beta|}|\hat{u}(t,\xi)| &\lesssim |\xi|^{|\beta|}|\hat{u}_0(\xi)| +|\xi|^{|\beta|}\langle \xi \rangle^{-2}|\hat{u}_1(\xi)| \quad \mbox{for} \quad |\beta|\geq 0, \\
|\xi|^{|\beta|}|\hat{u}_t(t,\xi)| &\lesssim ( b(t)+g(t)|\xi|^2)|\xi|^{|\beta|}|\hat{u}_0(\xi)| + ( b(t)+g(t)|\xi|^2) |\xi|^{|\beta|}\langle \xi \rangle^{-2}|\hat{u}_1(\xi)| \quad \mbox{for} \quad |\beta|\geq 0.
\end{align*}
\end{corollary}
\subsubsection{Conclusion} \label{Section2.3}
From the statements of Corollaries \ref{Cor:Overdamping_IncreasingZell} and \ref{Cor:Overdamping_IncreasingZpd} we derive the following estimates for $t>0$:
\begin{align*}
|\xi|^{|\beta|}|\hat{u}(t,\xi)| &\lesssim |\xi|^{|\beta|}|\hat{u}_0(\xi)| +|\xi|^{|\beta|}\langle \xi \rangle^{-2}|\hat{u}_1(\xi)| \quad \mbox{for} \quad |\beta|\geq 0, \\
|\xi|^{|\beta|}|\hat{u}_t(t,\xi)| &\lesssim ( b(t)+g(t)|\xi|^2)|\xi|^{|\beta|}|\hat{u}_0(\xi)| + ( b(t)+g(t)|\xi|^2) |\xi|^{|\beta|}\langle \xi \rangle^{-2}|\hat{u}_1(\xi)| \quad \mbox{for} \quad |\beta|\geq 0.
\end{align*}
This completes the proof of Theorem \ref{Theorem:Overdamping1}.
\end{proof}

\subsection{Model with integrable and decaying time-dependent coefficient $g=g(t)$} \label{Subsect_Overdamping_Integrable}
We assume that the coefficient $g=g(t)$ satisfies the following conditions for all $t \in [0,\infty)$:
\begin{enumerate}
\item[\textbf{(E1)}] $g(t)>0$ and $g'(t)<0$,
\item[\textbf{(E2)}] $g \in L^1([0,\infty))$.
\end{enumerate}
Let us consider the following transformed equation from \eqref{AuxiliaryEquation3}:
\[ w_{tt} + \Big(|\xi|^2\Big(1-\frac{b(t)g(t)}{2}-\frac{g(t)^2|\xi|^2}{4}-\frac{g'(t)}{2}\Big)-\frac{b(t)^2}{4}-\frac{b'(t)}{2}\Big)w=0.\]
We have $b'(t)=o(b^2(t))$ for $t \to \infty$. For this reason we can restrict ourselves to the equation
\[ w_{tt} + \Big(|\xi|^2\Big( 1-\frac{b(t)g(t)}{2}-\frac{g'(t)}{2} \Big) - \frac{g(t)^2}{4}|\xi|^4-\frac{b(t)^2}{4}\Big)w=0. \]
This equation helps to determine the zones of the extended phase space. The last equation can be written in the form
\[ w_{tt} + |\xi|^2\Big( 1-\frac{g'(t)}{2} \Big)w - \Big(\frac{g(t)}{2}|\xi|^2+\frac{b(t)}{2}\Big)^2 w=0. \]
In the following we study only the case $-g'(t) \leq 2 -\varepsilon$ for all $t \geq 0$ and with $\varepsilon >0$. It is clear that $g'(t)=o(1)$ for $t \to \infty$. So, the proposed case is reasonable.
For the last equation we develop a WKB analysis. \\
We will consider our equation in the following form:
\begin{equation} \label{Eq:Effective-Overdamping-Dfrom}
D_t^2w + \Big( \frac{b(t)}{2}+\frac{g(t)|\xi|^2}{2} \Big)^2w - |\xi|^2w + \Big( \frac{b'(t)}{2}+\frac{g'(t)}{2}|\xi|^2 \Big)w = 0.
\end{equation}
At first let us turn to the equation
\[ |\xi|^2 - \Big( \dfrac{b(t)}{2}+\dfrac{g(t)|\xi|^2}{2} \Big)^2 = 0. \]
The equation
\[ |\xi| = \frac{b(t)}{2}+\frac{g(t)|\xi|^2}{2} \]
implies that we have the following two separating lines:
\[ \Pi_1 := \big\{(t,\xi): |\xi| = \frac{1}{g(t)} + \frac{1}{g(t)}\sqrt{1-b(t)g(t)} \big\} \quad \text{and} \quad \Pi_2 := \big\{(t,\xi): |\xi| = \frac{1}{g(t)} - \frac{1}{g(t)}\sqrt{1-b(t)g(t)} \big\}. \]
Therefore, we can introduce the following functions corresponding to the separating lines:
\begin{equation} \label{Eq:Effective-Overdamping-Sep-lines}
 f_1: t\to \frac{1}{g(t)}\big( 1+\sqrt{1-b(t)g(t)} \big) \qquad\text{and} \qquad  f_2: t\to \frac{1}{g(t)}\big( 1-\sqrt{1-b(t)g(t)} \big).
\end{equation}
The interaction between these two damping mechanisms is governed by the behavior of the product $b(t)g(t)$. Therefore, we distinguish the following examples:
\begin{itemize}
\item[\textbf{a)}] $b(t)=e^t$ and $g(t)=\frac{1}{2}e^{-e^t}$,
\item[\textbf{b)}] $b(t)=e^t$ and $g(t)=\frac{1}{2}e^{-t}$,
\item[\textbf{c)}] $b(t)=e^{e^t}$ and $g(t)=2e^{-t}$.
%\item[\textbf{d)}] $g(t)=\dfrac{1}{(1+t)^\alpha}$ with $\alpha>1$.
\end{itemize}
All these examples satisfy the over-damping conditions
\begin{enumerate}
\item[\textbf{(OD1)}] $\displaystyle\int_0^\infty\dfrac{1}{b(\tau)}d\tau < \infty$,
\item[\textbf{(OD2)}] $b'(t) = o(b(t)^2)$ for $t\to\infty$.
\end{enumerate}

\subsubsection{The model with $b(t)=e^t$ and $g(t)=\frac{1}{2}e^{-e^t}$} \label{Subsect_Overdamping_Integrable-SubSec1}
In this case, we have
\[ b(t)g(t) = \frac{1}{2}e^te^{-e^t}\leq \frac{1}{2}. \]
The definition of functions in \eqref{Eq:Effective-Overdamping-Sep-lines} imply the following division of extended phase space:
\begin{itemize}
\item hyperbolic region:
\[ \Pi_{\text{hyp}}:= \Big\{(t,\xi): \frac{1}{g(t)} - \frac{1}{g(t)}\sqrt{1-b(t)g(t)} < |\xi| < \frac{1}{g(t)} + \frac{1}{g(t)}\sqrt{1-b(t)g(t)} \Big\} \]
\item elliptic regions:
\begin{align*}
\Pi_{\text{ell}}^{\text{low}} := \Big\{ (t,\xi): |\xi| < \dfrac{1}{g(t)} - \dfrac{1}{g(t)}\sqrt{1-b(t)g(t)} \Big\} \quad \text{and} \quad \Pi_{\text{ell}}^{\text{high}} := \Big\{ (t,\xi): |\xi| > \dfrac{1}{g(t)} + \dfrac{1}{g(t)}\sqrt{1-b(t)g(t)} \Big\}.
\end{align*}
\end{itemize}
Under the explicit expression $b'(t)=b(t)$ and the conditions \textbf{(E1)} to \textbf{(E4)} from Section \ref{Subsect_Effective_Integrable}, formal application of Theorem \ref{Theorem_Effective-Decreasing_Integrable-Decaying} implies the following result.
\begin{theorem} \label{Theorem:Overdamping-Decreasing-Case1}
Let us consider the Cauchy problem
\begin{equation*}
\begin{cases}
u_{tt} - \Delta u + b(t)u_t -g(t)\Delta u_t=0, &(t,x) \in [0,\infty) \times \mathbb{R}^n, \\
u(0,x)= u_0(x),\quad u_t(0,x)= u_1(x), &x \in \mathbb{R}^n.
\end{cases}
\end{equation*}
We assume that $b(t)=e^t$ and $g(t)=\frac{1}{2}e^{-e^t}$. Then, we have the following estimates for Sobolev solutions for all $t\geq t_0$:
\begin{align*}
\|\,|D|^{|\beta|}u(t,\cdot)\|_{L^2} &\lesssim \|u_0\|_{H^{|\beta|}} + \|u_1\|_{H^{|\beta|-1}} \quad \mbox{for} \quad |\beta|\geq 1, \\
\|\,|D|^{|\beta|}u_t(t,\cdot)\big\|_{L^2} &\lesssim \|u_0\|_{H^{|\beta|+1}} + \|u_1\|_{H^{|\beta|}} \quad \mbox{for} \quad |\beta|\geq 0.
\end{align*}
\end{theorem}
\begin{proof}
Let us observe that in the proof of Theorem \ref{Theorem_Effective-Decreasing_Integrable-Decaying}, we previously relied on condition \textbf{(B3)} (page 49), namely, $|b'(t)|\leq \frac{b(t)}{1+t}$, and condition \textbf{(E5)} (page 62), which states $\frac{{g'}^2}{g}\in L^1([0,\infty))$ during the analysis of the elliptic zone $\Zell(N)$ in Section \ref{Subsection_Effective_Integrable_Zell}. However, in the current setting, we have the explicit expressions
\[ b'(t)=b(t) \qquad \text{and} \qquad g'(t)=-e^tg(t), \]
from which it follows directly that the conditions \textbf{(E1)}-\textbf{(E4)} are satisfied, and that the additional assumptions \textbf{(B3)} and \textbf{(E5)} are no longer required for the proof of Proposition \ref{Prop:Effective-integrable-symbol}. For the sake of completeness and clarity, we now present the proof of Proposition \ref{Prop:Effective-integrable-symbol} without invoking these assumptions. Under these considerations, and following the structure of the proof of Theorem \ref{Theorem_Effective-Decreasing_Integrable-Decaying}, we may then proceed to complete the proof of Theorem \ref{Theorem:Overdamping-Decreasing-Case1}.
\medskip

If $f=f(t,\xi)\in S_{\text{ell}}^{0}\{-1,2\}$, then using the estimates
\begin{align*}
\frac{1}{\langle\xi\rangle_{b(t),g(t)}}\Big( \frac{b'(t)+g'(t)|\xi|^2}{b(t)+g(t)|\xi|^2} \Big)^2 &\lesssim \Big| \frac{(b'(t)+g'(t)|\xi|^2)^2}{( b(t)+g(t)|\xi|^2)^3} \Big| \qquad \qquad \,\,\, \small{\text{(we used \eqref{Eq:Effective-Integrable-Zell-japan})}} \\
& \lesssim \frac{b'(t)^2+|b'(t)|\,|g'(t)|\,|\xi|^2+g'(t)^2|\xi|^4}{( b(t)+g(t)|\xi|^2)^3}  \\
& \lesssim \frac{b'(t)^2}{b^2(t)g(t)|\xi|^2} + \frac{|b'(t)|\,|g'(t)|\,|\xi|^2}{b(t)g(t)^2|\xi|^4} + \frac{g'(t)^2|\xi|^4}{g^3(t)|\xi|^6} \\
& \lesssim \frac{1}{g(t)|\xi|^2} + \frac{|g'(t)|}{g^2(t)|\xi|^2} + \frac{e^{2t}g^2(t)}{g^3(t)|\xi|^2}  \qquad \small{\text{(we used \,\,$b'(t)=b(t)$ and $g'(t)=-e^tg(t)$)}} \\
& = \frac{g(t)}{g^2(t)|\xi|^2} + \frac{|g'(t)|}{g^2(t)|\xi|^2} + \frac{e^{2t}}{g(t)|\xi|^2} \\
& = \frac{g(t)}{g^2(t)|\xi|^2} + \frac{|g'(t)|}{g^2(t)|\xi|^2} + \frac{e^{2t}g(t)}{g^2(t)|\xi|^2} \\
& \leq \frac{4}{N^2} \big( g(t) + |g'(t)| + e^{2t}g(t) \big) \qquad \small{\text{(we used $g(t)|\xi|\geq \frac{N}{2}$ from $\Zell(N)$)}},
\end{align*}
we obtain
\begin{align*}
\int_{0}^{t_{\text{ell}}}|f(\tau,\xi)|d\tau  & \lesssim  \int_{0}^{t_{\text{ell}}}g(\tau)d\tau + \int_{0}^{t_{\text{ell}}}e^{\tau}g(\tau)d\tau + \int_{0}^{t_{\text{ell}}}e^{2\tau}g(\tau)d\tau \lesssim 1,
\end{align*}
where since $g=\frac{1}{2}e^{-e^t}$, the integrals are also uniformly bounded. On the other hand, instead of relying on condition \textbf{(EF)} from page 49, namely $|b'(t)|\leq ab^2(t)$, we will now use the explicit expression $b'(t)=b(t)$ in the treatment of Theorem \ref{Theorem_Effective-Decreasing_Integrable-Decaying}, particularly in the elliptic zone $\Zell(N)$.
\end{proof}
\subsubsection{The model with $b(t)=e^t$ and $g(t)=\frac{1}{2}e^{-t}$} \label{Subsec:Overdamping-Integrable-SubSec2}
In this case, we have
\[ b(t)g(t) = \frac{1}{2}e^te^{-t}=\frac{1}{2}. \]
Then, we find
\[ |\xi|_{1/2} = \frac{1}{g(t)}\big( 1\pm \sqrt{1/2} \big). \]
So, the hyperbolic region is
\[ \Pi_{\text{hyp}}:= \Big\{(t,\xi): \frac{1}{g(t)}\Big( 1-\frac{1}{\sqrt{2}} \Big) < |\xi| <  \frac{1}{g(t)}\Big( 1+\frac{1}{\sqrt{2}} \Big) \Big\}, \]
and the elliptic regions are
\begin{align*}
\Pi_{\text{ell}}^{\text{low}} := \Big\{ (t,\xi): |\xi| < \frac{1}{g(t)}\Big( 1-\frac{1}{\sqrt{2}} \Big) \Big\} \qquad \text{and} \qquad \Pi_{\text{ell}}^{\text{high}} := \Big\{ (t,\xi): |\xi| > \frac{1}{g(t)}\Big( 1+\frac{1}{\sqrt{2}} \Big) \Big\}.
\end{align*}
Thus, we divide the extended phase space $[0,\infty)\times \mathbb{R}^n$ into the following zones:
\begin{itemize}
\item elliptic zone:
\[ \Zell(N) = \Big\{ (t,\xi)\in[0,\infty)\times\mathbb{R}^n : |\xi|\geq \frac{N}{g(t)} \Big\} \cap\Pi_{\text{ell}}^{\text{high}}, \]
\item first reduced zone:
\[ \Zred^1(N,\varepsilon_1) = \Big\{ (t,\xi)\in[0,\infty)\times\mathbb{R}^n : \frac{\varepsilon_1}{g(t)}\leq |\xi|\leq \frac{N}{g(t)} \Big\}, \]
with $\varepsilon_1<1+\frac{1}{\sqrt{2}}$,
\item hyperbolic zone:
\[ \Zhyp(\varepsilon_1,\varepsilon_2) = \Big\{ (t,\xi)\in[0,\infty)\times\mathbb{R}^n : \frac{\varepsilon_2}{g(t)}\leq |\xi|\leq \frac{\varepsilon_1}{g(t)} \Big\}\cap \Pi_{\text{hyp}}, \]
with $1-\frac{1}{\sqrt{2}}<\varepsilon_2<\varepsilon_1<1+\frac{1}{\sqrt{2}}$,
\item second reduced zone:
\[ \Zred^2(\varepsilon_2,\varepsilon_3) = \Big\{ (t,\xi)\in[0,\infty)\times\mathbb{R}^n : \frac{\varepsilon_3}{g(t)}\leq |\xi|\leq \frac{\varepsilon_2}{g(t)} \Big\}, \]
\item dissipative zone:
\[ \Zdiss(\varepsilon_3) = \Big\{ (t,\xi)\in[0,\infty)\times\mathbb{R}^n : |\xi|\leq \frac{\varepsilon_3}{g(t)} \Big\}\cap\Pi_{\text{ell}}^{\text{low}}, \]
with $\varepsilon_3<1-\frac{1}{\sqrt{2}}$,
\end{itemize}
where $N>0$ is sufficiently large and $\varepsilon_1, \varepsilon_2, \varepsilon_3>0$ are sufficiently small constants. Moreover, the separating lines between these zones may be defined by using the following functions:
\begin{align*}
t_{\xi,1} &= \big\{ (t,\xi) \in [0,\infty) \times \mathbb{R}^n: g(t)|\xi| = \varepsilon_3 \big\},\quad\,\,\, (\text{between} \,\,\, \Zdiss(\varepsilon_3) \,\,\, \text{and} \,\,\, \Zred(\varepsilon_2,\varepsilon_3) \,), \\
t_{\xi,2} &= \big\{ (t,\xi) \in [0,\infty) \times \mathbb{R}^n: g(t)|\xi| = \varepsilon_2 \big\}, \quad\,\,\, (\text{between} \,\,\, \Zred(\varepsilon_2,\varepsilon_3) \,\,\, \text{and} \,\,\,\Zhyp(\varepsilon_1,\varepsilon_2) \,),\\
t_{\xi,3} &= \big\{ (t,\xi) \in [0,\infty) \times \mathbb{R}^n: g(t)|\xi| = \varepsilon_1 \big\}, \quad\,\,\,  (\text{between} \,\,\, \Zhyp(\varepsilon_1,\varepsilon_2) \,\,\, \text{and} \,\,\, \Zred(N,\varepsilon_1) \,), \\
t_{\xi,4} &= \big\{ (t,\xi) \in [0,\infty) \times \mathbb{R}^n: g(t)|\xi| = N \big\}, \quad\,\,\, (\text{between} \,\,\, \Zred(N,\varepsilon_1) \,\,\, \text{and} \,\,\, \Zell(N)\,).
\end{align*}

{\color{black}
\begin{figure}[h]
\begin{center}
\begin{tikzpicture}[>=latex,xscale=1.1]
    \draw[->] (0,0) -- (5,0)node[below]{$|\xi|$};
    \draw[->] (0,0) -- (0,4)node[left]{$t$};
    \node[below left] at(0,0){$0$};
    %%%%%%%%%%%%%%%%%%%%%%%%%%%%%%%%%%%%%%%%%%%%%%%%%%%%%%%%%%%%%%%%%%%%%%%%%%%
    \node[right] at (1.6,3.7) {$\textcolor{red}{t_{\xi,1}}$};
    \node[right] at (2.4,3.5) {$\textcolor{yellow}{t_{\xi,2}}$};
    \node[right] at (3.4,3.5) {$\textcolor{cyan}{t_{\xi,3}}$};
    \node[right] at (4.5,3.5) {$\textcolor{green}{t_{\xi,4}}$};
    %%%%%%%%%%%%%%%%%%%%%%%%%%%%%%%%%%%%%%%%%%%%%%%%%%%%%%%%%
    \node[below] at(3.59,0){$N$};
    \node[below] at(2.59,0){$\varepsilon_1$};
    \node[below] at(1.49,0){$\varepsilon_2$};
    \node[below] at(0.49,0){$\varepsilon_3$};
	%%%%%%%%%%%%%%%%%%%%%%%%%%%%%%%%%%%%%%%%%%%%%%%%%%%%%%%%%%%%
	\node[color=black] at (4.1, 1.5){{\footnotesize $Z_{\text{ell}}$}};
	\node[color=black] at (3.2,1.5){{\footnotesize $Z_{\text{red}}^1$}};
	\node[color=black] at (2.2,1.5){{\footnotesize $Z_{\text{hyp}}$}};
	\node[color=black] at (1.1,1.5){{\footnotesize $Z_{\text{red}}^2$}};
	\node[color=black] at (0.3,1.5){{\footnotesize $Z_{\text{diss}}$}};
	%%%%%%%%%%%%%%%%%%%%%%%%%%%%%%%%%%%%%%%%%%%%%%%%%%%%%%%%%%%%%%
	\draw[domain=0:3.5,color=green,variable=\t] plot ({3.5 + 0.09*pow(\t,2.4)},\t);
    \draw[domain=0:3.5,color=cyan,variable=\t] plot ({2.5 + 0.09*pow(\t,2.4)},\t);
    \draw[domain=0:3.7,color=yellow,variable=\t] plot ({1.4 + 0.09*pow(\t,2.4)},\t);
	\draw[domain=0:3.7,color=red,variable=\t] plot ({0.4 + 0.09*pow(\t,2.4)},\t);
%%%%%%%%%%%%%%%%%%%%%%%%%%%%%%%%%%%%%%%%%%%%%%%%%%%%%%%%%%%%%%%%%%%%%%%%%%%%%%%%%%
\end{tikzpicture}
\caption{Sketch of the zones for the case $b(t)=e^t$ and $g(t)=\frac{1}{2}e^{-t}$}
\label{Fig-Overdamping-integrable-zones}
\end{center}
\end{figure}
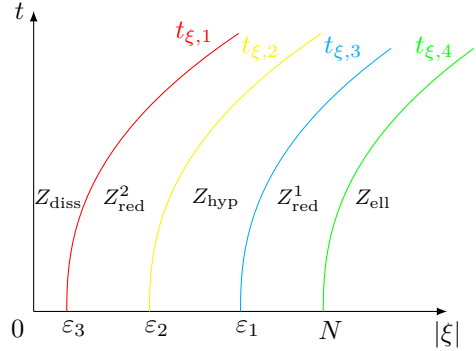}
\begin{theorem} \label{Theorem_Overdamping-Decreasing_Integrable}
Let us consider the Cauchy problem
\begin{equation*}
\begin{cases}
u_{tt} - \Delta u + b(t)u_t - g(t)\Delta u_t=0, &(t,x) \in [0,\infty) \times \mathbb{R}^n, \\
u(0,x)= u_0(x),\quad u_t(0,x)= u_1(x), &x \in \mathbb{R}^n.
\end{cases}
\end{equation*}
We assume that $b(t)=e^t$ and $g(t)=\frac{1}{2}e^{-t}$. Then, we have the following estimates for Sobolev solutions:
\begin{align*}
\|\,|D|^{|\beta|}u(t,\cdot)\|_{L^2} &\lesssim \|u_0\|_{\dot{H}^{|\beta|}\cap \dot{H}^{|\beta|+1}} + \|u_1\|_{\dot{H}^{|\beta|-1}\cap \dot{H}^{|\beta|}} \quad \mbox{for} \quad |\beta|\geq 1, \\
\|\,|D|^{|\beta|}u_t(t,\cdot)\big\|_{L^2} &\lesssim \|u_0\|_{\dot{H}^{|\beta|+1}\cap \dot{H}^{|\beta|+2}} + \|u_1\|_{\dot{H}^{|\beta|}\cap \dot{H}^{|\beta|+1}} \quad \mbox{for} \quad |\beta|\geq 0.
\end{align*}
\end{theorem}

\subsubsection*{Considerations in the hyperbolic zone $\Zhyp(\varepsilon_1,\varepsilon_2)$} \label{Sect-Overdamping-Integrable-Zhyp}
We will consider our equation in \eqref{AuxiliaryEquation3} as follows:
\begin{equation} \label{Eq:Overdamping-Integrable-Dfrom-Zhyp}
D_t^2w + \Big( \frac{b(t)}{2}+\frac{g(t)|\xi|^2}{2} \Big)^2w - |\xi|^2w + \frac{g'(t)}{2}|\xi|^2w + \frac{b'(t)}{2}w = 0.
\end{equation}
Since we have $b(t)g(t)=1/2$, the previous equation reduces to the following equation:
\begin{equation*} \label{Eq:Overdamping-Int-Super-Dform-Zhyp}
D_t^2 w - \Big( \underbrace{\frac{3}{4}|\xi|^2 - \dfrac{b^2(t)}{4} - \dfrac{g^2(t)|\xi|^4}{4}}_{=:p^2(t,\xi)} \Big)w + \Big( \frac{b'(t)}{2} + \frac{g'(t)|\xi|^2}{2} \Big)w = 0.
\end{equation*}
Due to the definition of the hyperbolic zone $\Zhyp(\varepsilon_1,\varepsilon_2)$, we have the following estimates for $p=p(t,\xi)$:
\begin{align}
p^2(t,\xi) &= \frac{3}{4}|\xi|^2 - \dfrac{b^2(t)}{4} - \dfrac{g^2(t)|\xi|^4}{4} \leq \frac{3}{4}|\xi|^2 \nonumber \\
\label{Eq:Overdamping-Estimates-p2-Zhyp} p^2(t,\xi) &= \frac{3}{4}|\xi|^2 - \dfrac{b^2(t)}{4} - \dfrac{g^2(t)|\xi|^4}{4} \geq \Big( \frac{3}{4}-\frac{1}{16\varepsilon_2^2} - \frac{\varepsilon_1^2}{4} \Big)|\xi|^2,
\end{align}
where from the definition of $\Zhyp(\varepsilon_1,\varepsilon_2)$ we used $\varepsilon_2\leq g(t)|\xi| \leq \varepsilon_1$, which also implies that $2\varepsilon_2b(t)\leq |\xi| \leq 2\varepsilon_1b(t)$. To show the quantity in the bracket in \eqref{Eq:Overdamping-Estimates-p2-Zhyp} is positive, we suppose that $\varepsilon_1=1+\kappa$ and $\varepsilon_2=1-\kappa$ with sufficiently small $\kappa>0$. Then, we want to show that
\[ f(\kappa) = \frac{1}{16(1-\kappa)^2} + \frac{(1+\kappa)^2}{4} < \frac{3}{4}. \]
If we choose $\kappa=0.1$, then we see that
\[ f(\kappa) = \frac{1}{12.96} + \frac{1.21}{4} \approx 03797<0.75, \]
which holds the inequality. Thus, we may conclude that $p(t,\xi) \approx|\xi|$.
\begin{proposition} \label{Prop:Overdamping-Integrable-Zhyp}
The following estimates hold in $\Zhyp(\varepsilon_1,\varepsilon_2)$ for all $t_{\xi,3}\leq s \leq t\leq t_{\xi,2}$:
\begin{align*}
|\xi|^{|\beta|}|\hat{u}(t,\xi)| \lesssim \exp\Big( -\frac{1}{2}\int_s^t\big( b(\tau)+g(\tau)|\xi|^2 \big)d\tau \Big)\big( |\xi|^{|\beta|}|\hat{u}(s,\xi)| + |\xi|^{|\beta|-1}|\hat{u}_t(s,\xi)| \big) \quad \mbox{for} \quad |\beta|\geq 1, \\
|\xi|^{|\beta|}|\hat{u}_t(t,\xi)| \lesssim \exp\Big( -\frac{1}{2}\int_s^t\big( b(\tau)+g(\tau)|\xi|^2 \big)d\tau \Big)\big( |\xi|^{|\beta|+1}|\hat{u}(s,\xi)| + |\xi|^{|\beta|}|\hat{u}_t(s,\xi)| \big) \quad \mbox{for} \quad |\beta|\geq 0.
\end{align*}
\end{proposition}
\begin{proof}
We define the micro-energy $W(t,\xi) := (p(t,\xi)w,D_tw)^{\text{T}}$. Then, from \eqref{Eq:Overdamping-Integrable-Dfrom-Zhyp} it holds
\begin{equation*}
D_tW=\left( \begin{array}{cc}
0 & p(t,\xi) \\
p(t,\xi) & 0
\end{array} \right)W + \left( \begin{array}{cc}
\dfrac{D_tp(t,\xi)}{p(t,\xi)} & 0 \\
-\dfrac{b'(t)+g'(t)|\xi|^2}{2p(t,\xi)} & 0
\end{array} \right)W.
\end{equation*}
Let us carry out the first step of the diagonalization procedure. First, we set
\[ M = \left( \begin{array}{cc}
1 & -1 \\
1 & 1
\end{array} \right), \qquad M^{-1} = \frac{1}{2}\left( \begin{array}{cc}
1 & 1 \\
-1 & 1
\end{array} \right) \qquad\text{and} \qquad W^{(0)}:=M^{-1}W.   \]
Then, we obtain
\begin{equation*}
D_tW^{(0)}=\big( \mathcal{D}(t,\xi)+\mathcal{R}(t,\xi) \big)W^{(0)},
\end{equation*}
where
\begin{align*}
\mathcal{D}(t,\xi) &= \left( \begin{array}{cc}
p(t,\xi) & 0 \\
0 & -p(t,\xi)
\end{array} \right), \\
\mathcal{R}(t,\xi) &= \left( \begin{array}{cc}
\dfrac{D_tp(t,\xi)}{2p(t,\xi)}-\dfrac{b'(t)+g'(t)|\xi|^2}{4p(t,\xi)} &
-\dfrac{D_tp(t,\xi)}{2p(t,\xi)}+\dfrac{b'(t)+g'(t)|\xi|^2}{4p(t,\xi)} \\
-\dfrac{D_tp(t,\xi)}{2p(t,\xi)}-\dfrac{b'(t)+g'(t)|\xi|^2}{4p(t,\xi)} & \dfrac{D_tp(t,\xi)}{2p(t,\xi)}+\dfrac{b'(t)+g'(t)|\xi|^2}{4p(t,\xi)}
\end{array} \right).
\end{align*}
After the first step of diagonalization procedure, the entries of the matrix $\mathcal{R}=\mathcal{R}(t,\xi)$ are uniformly integrable over the hyperbolic
zone $\Zhyp(\varepsilon_1,\varepsilon_2)$. Namely, we have
\begin{align*}
\Big|\int_s^t \frac{D_\tau p(\tau,\xi)}{2p(\tau,\xi)} d\tau\Big| \leq \int_s^t\Big| \frac{D_\tau p(\tau,\xi)}{2p(\tau,\xi)} \Big|d\tau = \Big|\frac{1}{2}\log\frac{p(t,\xi)}{p(s,\xi)} \Big|
\end{align*}
and $p(t,\xi)\approx|\xi|$ is uniformly in $\Zhyp(\varepsilon_1,\varepsilon_2)$. Moreover, since $b'(t)\geq0$ and $g'(t)\leq0$ we have
\begin{align*}
\Big| \int_s^t\frac{b'(\tau)+g'(\tau)|\xi|^2}{4p(\tau,\xi)} d\tau\Big| &\leq \int_s^t \Big| \frac{b'(\tau)+g'(\tau)|\xi|^2}{4|\xi|}\Big| d\tau \\
& = -\int_s^t \frac{b'(\tau)+g'(\tau)|\xi|^2}{4|\xi|} d\tau + 2\int_s^t\frac{b'(\tau)}{4|\xi|}d\tau \\
& \leq  \frac{1}{4|\xi|}\big( b(s)+g(s)|\xi|^2 \big) + \frac{1}{2|\xi|}b(t) \lesssim 1,
\end{align*}
where from the definition of $\Zhyp(\varepsilon_1,\varepsilon_2)$, we used $\varepsilon_2\leq g(t)|\xi| \leq \varepsilon_1$ and $2\varepsilon_2b(t)\leq |\xi| \leq 2\varepsilon_1b(t)$.

We can write $W^{(1)}(t,\xi) = E_{\text{hyp}}(t,s,\xi)W^{(1)}(s,\xi)$, where $E_{\text{hyp}}=E_{\text{hyp}}(t,s,\xi)$ is the fundamental solution of the system
\begin{align*}
D_tE_{\text{hyp}}(t,s,\xi) = \big( \mathcal{D}(t,\xi) + \mathcal{R}(t,\xi) \big)E_{\text{hyp}}(t,s,\xi), \quad E_{\text{hyp}}(s,s,\xi) = I,
\end{align*}
for all $t\geq s$ and $(s,\xi)\in\Zhyp(\varepsilon,t_0)$. We may immediately obtain that
\[ |E_{\text{hyp}}(t,s,\xi)| \leq C \quad \text{for all} \quad t\geq s \quad \text{and} \quad (t,\xi),\,(s,\xi)\in\Zhyp(\varepsilon_1,\varepsilon_2). \]
Finally, we find the following estimate for the micro-energy $W^{(1)}(t,\xi)$ in the hyperbolic zone $\Zhyp(\varepsilon,t_0)$:
\[ |W^{(1)}(t,\xi)| \lesssim |W^{(1)}(s,\xi)|,  \qquad \bigg|\left( \begin{array}{cc}
\langle\xi\rangle_{b(t),g(t)}w(t,\xi) \\
D_tw(t,\xi)
\end{array} \right)\bigg|\lesssim \bigg|\left( \begin{array}{cc}
\langle\xi\rangle_{b(s),g(s)}w(s,\xi) &  \\
D_tw(s,\xi)
\end{array} \right)\bigg| \]
for all $(t,\xi),(s,\xi)\in \Zhyp(\varepsilon_1,\varepsilon_2)$. The backward transformation
\[ \hat{u}(t,\xi) = \exp\Big( -\frac{1}{2}\int_0^t\big( b(\tau)+g(\tau)|\xi|^2 \big)d\tau \Big)w(t,\xi), \]
and the equivalence $p(t,\xi)\approx |\xi|$ gives us the desired estimates.
\end{proof}
\subsubsection*{Considerations in the elliptic zone $\Zell(N)$} \label{Sect-Overdamping-Integrable-Zell}
We will consider our equation in \eqref{AuxiliaryEquation3} as follows:
\begin{equation*} \label{Eq:Overdamping-Integrable-Dfrom}
D_t^2w + \Big( \frac{b(t)}{2}+\frac{g(t)|\xi|^2}{2} \Big)^2w - |\xi|^2w + \frac{g'(t)}{2}|\xi|^2w + \frac{b'(t)}{2}w = 0.
\end{equation*}
Since we have $b(t)g(t)=1/2$, $b'(t)=b(t)$ and $g'(t)=-g(t)$, the previous equation reduces to the following equation:
\begin{equation} \label{Eq:Overdamping-Int-Super-Dform-Zell}
D_t^2 w + \Big( \underbrace{\dfrac{g^2(t)}{4}|\xi|^4-\frac{3}{4}|\xi|^2}_{=:h^2(t,\xi)} \Big)w + \Big( \underbrace{ \frac{b^2(t)}{4} + \frac{b(t)}{2} - \frac{g(t)|\xi|^2}{2}}_{=:m(t,\xi)} \Big)w = 0.
\end{equation}
\begin{remark} \label{Rem:Overdamping-Integrable-Zell}
We have the following inequalities with sufficiently large $N$:
\begin{align} \label{Eq:Overdamping-Estimates-h-Zell}
h^2(t,\xi) \leq \frac{1}{4}g^2(t)|\xi|^4 \qquad \text{and} \qquad h^2(t,\xi)\geq \Big( \frac{1}{4}-\frac{3}{4N^2} \Big)g^2(t)|\xi|^4.
\end{align}
Therefore, we get $h(t,\xi)\approx g(t)|\xi|^2$. Furthermore, it holds
\begin{align} \label{Eq:Overdamping-Estimates-ht-Zell}
|h_t(t,\xi)| =  \bigg| \frac{1}{4}\frac{g'(t)g(t)|\xi|^4}{\sqrt{\frac{g^2(t)}{4}|\xi|^4-\frac{3}{4}|\xi|^2}} \bigg| \leq  \frac{1}{2\sqrt{1-\frac{3}{N^2}}}g(t)|\xi|^2.
\end{align}
On the other hand, we have
\begin{align*}
h_t^2(t,\xi) &=  \frac{1}{4}\frac{\Big( g''(t)g(t)|\xi|^4+(g'(t))^2|\xi|^4 \Big)\sqrt{\frac{g^2(t)}{4}|\xi|^4-\frac{3}{4}|\xi|^2}}{\frac{g^2(t)}{4}|\xi|^4-\frac{3}{4}|\xi|^2} - \frac{1}{16}\frac{\big( g'(t)g(t)|\xi|^4 \big)^2}{\Big( \frac{g^2(t)}{4}|\xi|^4-\frac{3}{4}|\xi|^2 \Big)\sqrt{\frac{g^2(t)}{4}|\xi|^4-\frac{3}{4}|\xi|^2}}.
\end{align*}
Using the estimates in \eqref{Eq:Overdamping-Estimates-h-Zell} and, $g'(t)=-g(t)$ and $g''(t)=g(t)$ we get
\begin{align*}
|h_t^2(t,\xi)| &\leq  \frac{1}{2}\frac{g^2(t)|\xi|^4}{d(t,\xi)} + \frac{1}{16}\frac{\big( g^2(t)|\xi|^4 \big)^2}{d^3(t,\xi)} \leq \bigg( \frac{1}{\big( 1-\frac{3}{N^2} \big)^{\frac{1}{2}}} + \frac{1}{2\big( 1-\frac{3}{N^2} \big)^{\frac{3}{2}}} \bigg)g(t)|\xi|^2.
\end{align*}
Moreover, let us note that $m(t,\xi)>0$ for $t\geq t_0$.
\end{remark}
We introduce the following family of symbol classes in the elliptic zone $\Zell(N)$.
\begin{definition} \label{Def:Overdamping-integrable-symbol-Zell}
A function $f=f(t,\xi)$ belongs to the elliptic symbol class $S_{\text{ell}}^\ell\{m_1,m_2\}$ if it holds
\begin{equation*}
|D_t^kf(t,\xi)|\leq C_{k}\big( g(t)|\xi|^2 \big)^{m_1}b(t)^{m_2+k}
\end{equation*}
for all $(t,\xi)\in \Zell(N)$ and all $k\leq \ell$.
\end{definition}
We introduce the micro-energy
\[ V=V(t,\xi):=\big( h(t,\xi)v,D_t v \big)^{\text{T}} \qquad \mbox{with} \qquad h(t,\xi) := \sqrt{\frac{g^2(t)}{4}|\xi|^4-\frac{3}{4}|\xi|^2}. \]
Transforming \eqref{Eq:Overdamping-Int-Super-Dform-Zell} to a system of first order for $V=V(t,\xi)$ gives
\begin{equation*}
D_tV=\left( \begin{array}{cc}
0 & h(t,\xi) \\
-h(t,\xi) & 0
\end{array} \right)V + \left( \begin{array}{cc}
\dfrac{D_th(t,\xi)}{h(t,\xi)} & 0 \\
-\dfrac{m(t,\xi)}{h(t,\xi)} & 0
\end{array} \right)V.
\end{equation*}
Using $V=MV^{(0)}$ with $M=\begin{pmatrix} i & -i \\ 1 & 1\end{pmatrix}$, then after the first step of diagonalization we obtain
\[ D_tV^{(0)} = \big( \mathcal{D}(t,\xi) + \mathcal{R}(t,\xi) \big)V^{(0)}, \]
where
\begin{align*}
\mathcal{D}(t,\xi) &= \left( \begin{array}{cc}
-ih(t,\xi) & 0 \\
0 & ih(t,\xi)
\end{array} \right)\in S_{\text{ell}}^1\{1,0\}, \\
\mathcal{R}(t,\xi) &= \frac{1}{2} \left( \begin{array}{cc}
\dfrac{D_th(t,\xi)}{h(t,\xi)}-i\dfrac{m(t,\xi)}{h(t,\xi)} & -\dfrac{D_th(t,\xi)}{h(t,\xi)}+i\dfrac{m(t,\xi)}{h(t,\xi)} \\
-\dfrac{D_th(t,\xi)}{h(t,\xi)}-i\dfrac{m(t,\xi)}{h(t,\xi)} & \dfrac{D_th(t,\xi)}{h(t,\xi)}+i\dfrac{m(t,\xi)}{h(t,\xi)}
\end{array} \right)\in S_{\text{ell}}^1\{0,1\}.
\end{align*}
Indeed we have
\begin{align*}
\Big|\frac{D_th(t,\xi)}{h(t,\xi)}\Big| &\leq \frac{g(t)|\xi|^2}{2\sqrt{1-\frac{3}{N^2}}}\frac{2}{\sqrt{1-\frac{3}{N^2}}g(t)|\xi|^2} = \frac{1}{1-\frac{3}{N^2}}, \\
\Big|\frac{m(t,\xi)}{h(t,\xi)}\Big| & \leq \frac{b^2(t)}{4h(t,\xi)} + \frac{b(t)}{2h(t,\xi)} + \frac{g(t)|\xi|^2}{2h(t,\xi)} \\
& \leq \frac{b^2(t)}{2\sqrt{1-\frac{3}{N^2}}g(t)|\xi|^2} + \frac{b(t)}{\sqrt{1-\frac{3}{N^2}}g(t)|\xi|^2} + \frac{g(t)|\xi|^2}{\sqrt{1-\frac{3}{N^2}}g(t)|\xi|^2} \\
& = \frac{b^2(t)g(t)}{2\sqrt{1-\frac{3}{N^2}}g^2(t)|\xi|^2} + \frac{b(t)g(t)}{\sqrt{1-\frac{3}{N^2}}g^2(t)|\xi|^2} + \frac{1}{\sqrt{1-\frac{3}{N^2}}} \\
& = \frac{b(t)}{4N^2\sqrt{1-\frac{3}{N^2}}} + \frac{1}{2N^2\sqrt{1-\frac{3}{N^2}}} + \frac{1}{\sqrt{1-\frac{3}{N^2}}}.
\end{align*}
Let us introduce $F_0(t,\xi):=\diag\mathcal{R}(t,\xi)$ and $\mathcal{R}_1(t,\xi):=\antidiag\mathcal{R}(t,\xi)$. To apply the second step of diagonalization procedure, we define $\delta(t,\xi)$ as the difference of the diagonal entries of the matrix $\mathcal{D}(t,\xi)+F_0(t,\xi)$ as follows:
\begin{align*}
i\delta(t,\xi) := 2h(t,\xi) +  \frac{m(t,\xi)}{h(t,\xi)} &= 2h(t,\xi) + \frac{b^2(t)}{4h(t,\xi)} + \frac{b(t)}{2h(t,\xi)} - \frac{g(t)|\xi|^2}{2h(t,\xi)} \\
&= 2h(t,\xi) + \frac{e^{2t}}{4h(t,\xi)} + \frac{e^t}{2h(t,\xi)} - \frac{e^{-t}|\xi|^2}{4h(t,\xi)} \\
& \approx e^{-t}|\xi|^2 + \frac{e^{3t}}{|\xi|^2} + \frac{e^{3t}}{|\xi|^2} - 1,
\end{align*}
where the first term $e^{-t}|\xi|^2$ grows as $|\xi|^2$, but since $|\xi|\geq 2Ne^t$, we have $e^{-t}|\xi|^2\geq 4N^2e^t$, which grows exponentially for large $t$, whereas the second term, since $|\xi|\geq 2Ne^t$, then $\frac{e^{3t}}{|\xi|^2}\leq \frac{e^{3t}}{4N^2e^{2t}}=\frac{e^t}{4N^2}$, which also grows exponentially for large $t$, however slower than the first term. Thus, we choose the first term, that is, $i\delta(t,\xi)\sim 2h(t,\xi)$.

Now we choose a matrix $N^{(1)}=N^{(1)}(t,\xi)$ such that
\begin{align*}
N^{(1)}(t,\xi) &= \left( \begin{array}{cc}
0 & -\dfrac{\mathcal{R}_{12}}{\delta(t,\xi)} \\
\dfrac{\mathcal{R}_{21}}{\delta(t,\xi)} & 0
\end{array} \right) = \left( \begin{array}{cc}
0 & \dfrac{h_t(t,\xi)}{4h^2(t,\xi)}+\dfrac{m(t,\xi)}{4h^2(t,\xi)} \\
-\dfrac{h_t(t,\xi)}{4h^2(t,\xi)}+\dfrac{m(t,\xi)}{4h^2(t,\xi)} & 0
\end{array} \right)\in S_{\text{ell}}^1\{-1,1\}.
\end{align*}
We put $N_1=N_1(t,\xi):=I+N^{(1)}(t,\xi)$. For a sufficiently large zone constant $N$ and all $t\leq t_{\xi,3}$ the matrix $N_1=N_1(t,\xi)$ is invertible with uniformly bounded inverse $N_1^{-1}=N_1^{-1}(t,\xi)$. Indeed, in the elliptic zone $\Zell(N)$, for large $N$ it holds
\begin{align*}
|N^{(1)}(t,\xi)| &\leq \frac{|h_t(t,\xi)|}{4h^2(t,\xi)}+\frac{|m(t,\xi)|}{4h^2(t,\xi)} \\
&\leq \frac{1}{2\big( 1-\frac{3}{N^2}\big)^{\frac{3}{2}}}\frac{g(t)|\xi|^2}{g^2(t)|\xi|^4} + \frac{1}{\big( 1-\frac{3}{N^2}\big)g^2(t)|\xi|^4}\Big( \frac{b^2(t)}{4} + \frac{b(t)}{2} + \frac{g(t)|\xi|^2}{2} \Big) \\
& = \frac{1}{2\big( 1-\frac{3}{N^2}\big)^{\frac{3}{2}}}\frac{g(t)}{g^2(t)|\xi|^2} + \frac{1}{\big( 1-\frac{3}{N^2}\big)}\Big( \frac{b^2(t)g^2(t)}{4g^4(t)|\xi|^4} + \frac{b(t)g^2(t)}{2g^4(t)|\xi|^4} + \frac{g(t)}{2g^2(t)|\xi|^2} \Big) \\
& \leq \frac{g(t)}{2N^2\big( 1-\frac{3}{N^2}\big)^{\frac{3}{2}}} + \frac{1}{\big( 1-\frac{3}{N^2}\big)}\Big( \frac{1}{16N^4} + \frac{g(t)}{4N^4} + \frac{g(t)}{2N^2} \Big)<1,
\end{align*}
where we used estimates \eqref{Eq:Overdamping-Estimates-h-Zell} and \eqref{Eq:Overdamping-Estimates-ht-Zell} from Remark \ref{Rem:Overdamping-Integrable-Zell}. Moreover, in the last line we employed definition of $\Zell(N)$ and the fact that $b(t)g(t)=1/2$.

Let us define
\begin{align*}
B^{(1)}(t,\xi) &:= D_tN^{(1)}(t,\xi)-( \mathcal{R}(t,\xi)-F_0(t,\xi))N^{(1)}(t,\xi), \\
\mathcal{R}_2(t,\xi) &:= -N_1^{-1}(t,\xi)B^{(1)}(t,\xi)\in S_{\text{ell}}^0\{-1,2\}.
\end{align*}
Consequently, we have the following operator identity:
\begin{equation*}
\big( D_t-\mathcal{D}(t,\xi)-\mathcal{R}(t,\xi) \big)N_1(t,\xi)=N_1(t,\xi)\big( D_t-\mathcal{D}(t,\xi)-F_0(t,\xi)-\mathcal{R}_2(t,\xi) \big).
\end{equation*}
\begin{proposition} \label{Prop:Overdamping-Integrable-Estimates-Zell}
The fundamental solution $E_{\text{ell}}^{W}=E_{\text{ell}}^{V}(t,s,\xi)$ to the transformed operator
\[ D_t-\mathcal{D}(t,\xi)-F_0(t,\xi)-\mathcal{R}_2(t,\xi) \]
can be estimated by
\begin{equation*}
(|E_{\text{ell}}^{W}(t,s,\xi)|) \lesssim \frac{g(t)}{g(s)}\exp\Big( \frac{1}{2}\int_{s}^{t} \big( b(\tau)+g(\tau)|\xi|^2 \big)d\tau \Big)\left( \begin{array}{cc}
1 & 1 \\
1 & 1
\end{array} \right),
\end{equation*}
with $(t,\xi),(s,\xi)\in \Zell(N)$, $0\leq s\leq t\leq t_{\xi,3}$.
\end{proposition}
\begin{proof}
We transform the system for $E_{\text{ell}}^{W}=E_{\text{ell}}^{W}(t,s,\xi)$ to an integral equation for a new matrix-valued function $\mathcal{Q}_{\text{ell}}=\mathcal{Q}_{\text{ell}}(t,s,\xi)$ as in the proof of Proposition \ref{Prop_Scattering_EllZone}. We define
\[ \mathcal{Q}_{\text{ell}}(t,s,\xi):=\exp\bigg\{ -\int_{s}^{t}\beta(\tau,\xi)d\tau \bigg\} E_{\text{ell}}^{V}(t,s,\xi), \]
where $\beta=\beta(t,\xi)$ is chosen from the main entries of the diagonal matrix $i\mathcal{D}(t,\xi)+iF_0(t,\xi)$ as follows:
\[ \beta(t,\xi)=h(t,\xi)+\frac{h_t(t,\xi)}{2h(t,\xi)}+\frac{m(t,\xi)}{2h(t,\xi)}. \]
The following new integral equation holds:
\begin{align*}
\mathcal{Q}_{\text{ell}}(t,s,\xi)=&\exp \bigg\{ \int_{s}^{t}\big( i\mathcal{D}(\tau,\xi)+iF_0(\tau,\xi)-\beta(\tau,\xi)I \big)d\tau \bigg\}\\
& \quad + \int_{s}^{t} \exp \bigg\{ \int_{\theta}^{t}\big( i\mathcal{D}(\tau,\xi)+iF_0(\tau,\xi)-\beta(\tau,\xi)I \big)d\tau \bigg\}\mathcal{R}_2(\theta,\xi)\mathcal{Q}_{\text{ell}}(\theta,s,\xi)\,d\theta.
\end{align*}
If we define
\begin{align*}
H(t,s,\xi) & =\exp \bigg\{ \int_{s}^{t}\big( i\mathcal{D}(\tau,\xi)+iF_0(\tau,\xi)-\beta(\tau,\xi)I \big)d\tau \bigg\}\\
& = \diag \bigg( 1, \exp \bigg\{ \int_{s}^{t}\bigg( -2h(\tau,\xi)-\frac{m(\tau,\xi)}{h(\tau,\xi)} \bigg)d\tau \bigg\} \bigg)\rightarrow \left( \begin{array}{cc}
1 & 0 \\
0 & 0
\end{array} \right)
\end{align*}
as $t\rightarrow \infty$, then we may conclude that the matrix $H=H(t,s,\xi)$ is uniformly bounded for $(s,\xi),(t,\xi)\in \Zell(N)$. So, the representation of $\mathcal{Q}_{\text{ell}}=\mathcal{Q}_{\text{ell}}(t,s,\xi)$ by a Neumann series gives
\begin{align*}
\mathcal{Q}_{\text{ell}}(t,s,\xi)=H(t,s,\xi)+\sum_{k=1}^{\infty}i^k\int_{s}^{t}H(t,t_1,\xi)\mathcal{R}_2(t_1,\xi)&\int_{s}^{t_1}H(t_1,t_2,\xi)\mathcal{R}_2(t_2,\xi)\\
& \cdots \int_{s}^{t_{k-1}}H(t_{k-1},t_k,\xi)\mathcal{R}_2(t_k,\xi)dt_k\cdots dt_2dt_1.
\end{align*}
This series can be estimated by
\[ |\mathcal{Q}_{\text{ell}}(t,s,\xi)| \lesssim \exp\Big( \int_s^t|\mathcal{R}_2(\tau,\xi)|d\tau \Big). \]
Since $\mathcal{R}_2(t,\xi)\in S_{\text{ell}}^0\{-1,2\}$, we have
\begin{align} \label{Eq:Overdamping-Estimate-Qell-Zell}
|\mathcal{Q}_{\text{ell}}&(t,s,\xi)| \leq \int_0^{t_{\xi,4}}\frac{b^2(\tau)}{g(\tau)|\xi|^2}d\tau = 2\int_0^{t_{\xi,4}}\frac{e^{2\tau}}{e^{-\tau}|\xi|^2}d\tau = 2\int_0^{t_{\xi,4}}\frac{e^{3\tau}}{|\xi|^2}d\tau \nonumber \\
& \leq \frac{1}{2N^2}\int_0^{t_{\xi,4}}e^{\tau}d\tau \leq \frac{1}{2N^2}e^{t_{\xi,4}} = \frac{1}{4N^3}|\xi|,
\end{align}
where due to $g(t)|\xi|\geq N$ in $\Zell(N)$, we used $|\xi|^2\geq 4N^2e^{2t}$ and $2Nb(t_{\xi,4})=|\xi|$. Thus, we may conclude
\begin{align*} \label{Eq:Overdamping-Estimate-EellV-Zell}
&E_{\text{ell}}^{W}(t,s,\xi)=\exp \bigg\{ \int_{s}^{t}\beta(\tau,\xi)d\tau \bigg\}\mathcal{Q}_{\text{ell}}(t,s,\xi) \nonumber \\
& = \exp \bigg\{ \int_{s}^{t}\bigg( h(\tau,\xi)+\frac{\partial_\tau h(\tau,\xi)}{2h(\tau,\xi)}+\frac{m(\tau,\xi)}{2h(\tau,\xi)} \bigg)d\tau \bigg\}\mathcal{Q}_{\text{ell}}(t,s,\xi) \nonumber \\
& = \exp \bigg\{ \int_{s}^{t}\bigg( h(\tau,\xi)+\frac{\partial_\tau h(\tau,\xi)}{2h(\tau,\xi)}+\frac{b^2(\tau)}{8h(\tau,\xi)} + \frac{b(\tau)}{4h(\tau,\xi)} - \frac{g(\tau)|\xi|^2}{4h(\tau,\xi)} \bigg)d\tau \bigg\}\mathcal{Q}_{\text{ell}}(t,s,\xi) \nonumber \\
& \leq \sqrt{\frac{h(t,\xi)}{h(s,\xi)}} \exp \bigg\{ \int_{s}^{t}\bigg( h(\tau,\xi)+\frac{b^2(\tau)}{4\sqrt{1-\frac{3}{N^2}}g(\tau)|\xi|^2} + \frac{b(\tau)}{2\sqrt{1-\frac{3}{N^2}}g(\tau)|\xi|^2} + \frac{g'(\tau)|\xi|^2}{2g(\tau)|\xi|^2} \bigg)d\tau \bigg\}\mathcal{Q}_{\text{ell}}(t,s,\xi) \nonumber \\
& = \sqrt{\frac{h(t,\xi)}{h(s,\xi)}} \exp \bigg\{ \int_{s}^{t}\bigg( h(\tau,\xi)+\frac{b^2(\tau)g(\tau)}{4\sqrt{1-\frac{3}{N^2}}g^2(\tau)|\xi|^2} + \frac{b(\tau)g(\tau)}{2\sqrt{1-\frac{3}{N^2}}g^2(\tau)|\xi|^2} + \frac{g'(\tau)}{2g(\tau)} \bigg)d\tau \bigg\}\mathcal{Q}_{\text{ell}}(t,s,\xi) \nonumber \\
& \leq \sqrt{\frac{h(t,\xi)}{h(s,\xi)}}\sqrt{\frac{g(t)}{g(s)}} \exp \bigg\{ \int_{s}^{t}\bigg( h(\tau,\xi)+\frac{b(\tau)}{8N^2\sqrt{1-\frac{3}{N^2}}} + \frac{1}{4N^2\sqrt{1-\frac{3}{N^2}}} \bigg)d\tau \bigg\}\mathcal{Q}_{\text{ell}}(t,s,\xi) \nonumber \\
& \leq \sqrt{\frac{h(t,\xi)}{h(s,\xi)}}\sqrt{\frac{g(t)}{g(s)}}\exp \bigg\{ \int_{s}^{t}h(\tau,\xi)d\tau + \frac{1}{{8N^2\sqrt{1-\frac{3}{N^2}}}}\int_s^t b(\tau)d\tau \bigg\}\mathcal{Q}_{\text{ell}}(t,s,\xi),
\end{align*}
where we used again estimates \eqref{Eq:Overdamping-Estimates-h-Zell} and \eqref{Eq:Overdamping-Estimates-ht-Zell} from Remark \ref{Rem:Overdamping-Integrable-Zell}, definition of $\Zell(N)$ and, the fact that $b(t)g(t)=1/2$. Then, for $N>\sqrt{3}$, it follows
\begin{align*}
(|E_{\text{ell}}^{W}(t,s,\xi)|) & \lesssim \frac{g(t)}{g(s)} \exp\bigg\{ \int_{s}^{t} h(\tau,\xi)d\tau + \frac{1}{{8N^2\sqrt{1-\frac{3}{N^2}}}}\int_s^tb(\tau)d\tau \bigg\} \left( \begin{array}{cc}
1 & 1 \\
1 & 1 \end{array} \right)|\mathcal{Q}_{\text{ell}}(t,s,\xi)| \\
& \lesssim \frac{g(t)|\xi|}{g(s)} \bigg\{ \int_{s}^{t} h(\tau,\xi)d\tau + \frac{1}{{8N^2\sqrt{1-\frac{3}{N^2}}}}\int_s^tb(\tau)d\tau \bigg\}\left( \begin{array}{cc}
1 & 1 \\
1 & 1
\end{array} \right),
\end{align*}
where we used $|\mathcal{Q}_{\text{ell}}(t,s,\xi)|\lesssim |\xi|$ from \eqref{Eq:Overdamping-Estimate-Qell-Zell}. This completes the proof.
\end{proof}
\textbf{Transforming back to the original Cauchy problem.} As we did in Section \ref{Subsection_Effective_Integrable_Zell}, in order to obtain an estimate for the energy of the solution to our original Cauchy problem, we need to derive an estimate for the fundamental solution $E_{\text{ell}}=E_{\text{ell}}(t,s,\xi)$, which is associated with a first order system for the micro-energy $(|\xi|\hat{u},D_t\hat{u})$.

\begin{lemma} \label{Lemma_Overdamping-Integrable_Transf-back}
The following inequalities hold:
\begin{itemize}
\item [1.] in the elliptic zone it holds $h(t,\xi)-\dfrac{g(t)|\xi|^2}{2} \leq -\dfrac{3}{4}\dfrac{1}{g(t)} = -\dfrac{3}{2}b(t)$.
\item [2.] $\dfrac{\beta(s,\xi)}{\beta(t,\xi)}\exp \Big( \displaystyle \int_{s}^{t}h(\tau,\xi)d\tau \Big)\leq \exp \Big( -\frac{3}{4}\displaystyle \int_{s}^{t}\dfrac{1}{g(\tau)}d\tau \Big)=\exp \Big( -\frac{3}{2}\displaystyle \int_{s}^{t}b(\tau)d\tau \Big)$,\\
where $\beta=\beta(t,\xi)=\exp\Big( \dfrac{1}{2}\displaystyle\int_{0}^{t}g(\tau)|\xi|^2d\tau \Big)$.
\end{itemize}	
\end{lemma}
\begin{proof} Using the elementary inequality
\[ \sqrt{x+y}\leq \sqrt{x}+\frac{y}{2\sqrt{x}} \]
for any $x\geq 0$ and $y\geq -x$, the first statement is equivalent to the following inequality:
\[ \sqrt{\dfrac{g^2(t)|\xi|^4}{4}-\frac{3}{4}|\xi|^2} - \dfrac{g(t)|\xi|^2}{2} \leq -\frac{3}{4}\frac{1}{g(t)} = -\frac{3}{2}b(t). \]
The second statement follows directly from the first one together with the definition of $\beta=\beta(t,\xi)$.
\end{proof}
From Proposition \ref{Prop:Overdamping-Integrable-Estimates-Zell}, for $(t,\xi), (s,\xi)\in \Zell(N)$ and denoting $8N^2\sqrt{1-\frac{3}{N^2}}:=C(N)$, we have the estimate
\[ (|E_{\text{ell}}^W(t,s,\xi)|) \lesssim \frac{g(t)|\xi|}{g(s)}\exp \Big( \int_{s}^{t} h(\tau,\xi)d\tau \Big)\exp\Big( \frac{1}{C(N)}\int_s^tb(\tau)d\tau \Big)
\left( \begin{array}{cc}
1 & 1 \\
1 & 1
\end{array} \right). \]
Analogously as we have done in Section \ref{Subsection_Effective_Integrable_Zell} in \eqref{Eq:Effective-Integrable_Transf-back}, we arrive at the estimate
\begin{align} \label{Eq:Overdamping-Integrable_Transf-back}
\nonumber
(|E_{\text{ell}}(t,s,\xi)|) & \lesssim \exp\bigg( \int_{s}^{t} \Big( h(\tau,\xi)-\frac{g(\tau)|\xi|^2}{2} \Big) d\tau \bigg)\exp \bigg( \Big( \frac{1}{C(N)}-\frac{1}{2} \Big)\int_{s}^{t} b(\tau)d\tau \bigg) \\
& \qquad \times \left( \begin{array}{cc}
|\xi|^2 & 0 \\  [5pt] \nonumber
|\xi|\big( b(t)+g(t)|\xi|^2 \big) & g(t)|\xi|^3
\end{array} \right)\left( \begin{array}{cc}
1 & 1 \\
1 & 1
\end{array} \right)
\left( \begin{array}{cc}
\dfrac{1}{|\xi|} & 0 \\ [5pt]
\dfrac{b(s)+g(s)|\xi|^2}{g(s)|\xi|^3} & \dfrac{1}{g(s)|\xi|^2}
\end{array} \right) \nonumber \\
& \,\,\, \lesssim \exp \Big( -\frac{3}{2}\int_s^tb(\tau)d\tau \Big)\exp\bigg( \Big( \frac{1}{C(N)}-\frac{1}{2} \Big)\int_{s}^{t} b(\tau)d\tau \bigg)\left( \begin{array}{cc}
|\xi| & |\xi| \\ [5pt]
b(t)+g(t)|\xi|^2 & b(t)+g(t)|\xi|^2
\end{array} \right) \nonumber \\
& \,\,\, = \exp\bigg( \Big( \frac{1}{C(N)}-2 \Big)\int_{s}^{t} b(\tau)d\tau \bigg)
\left( \begin{array}{cc}
|\xi| & |\xi| \\ [5pt]
b(t)+g(t)|\xi|^2 & b(t)+g(t)|\xi|^2
\end{array} \right),
\end{align}
where due to the definition of $\Zell(N)$ we used $g(t)|\xi|\geq N$. Let us verify that for large $N>0$
\[ \frac{1}{C(N)}-2= \frac{1}{{8N^2\sqrt{1-\frac{3}{N^2}}}}-2 <0. \]
Since $\sqrt{1-r}\sim 1-\frac{r}{2}$ for small $r$, we can write for large $N$
\[ \sqrt{1-\frac{3}{N^2}}=1-\frac{3}{2N^2}+o\Big( \frac{1}{N^2} \Big). \]
Now plug into the original expression
\begin{align*}
 \frac{1}{{8N^2\sqrt{1-\frac{3}{N^2}}}}-2 \sim \frac{1}{8N^2\Big( 1-\frac{3}{2N^2}+o\big( \frac{1}{N^2} \big) \Big)}-2 \sim \frac{1}{8N^2-12+o(1)}-2.
\end{align*}
For large $N$, the first term is smaller than $2$, which guarantees that the expression is negative.
\medskip

\textbf{A refined estimate for the fundamental solution in the elliptic zone $\Zell(N)$}
\begin{proposition} \label{Prop:Overdamping-Integrable-Refined-Zell}
The fundamental solution $E_{\text{ell}}=E_{\text{ell}}(t,s,\xi)$ satisfies the following estimate:
\begin{align*}
(|E_{\text{ell}}(t,s,\xi)|) &\lesssim \exp \bigg( -C_N\int_{s}^{t}b(\tau)d\tau \bigg)
\left( \begin{array}{cc}
|\xi| & |\xi| \\ [5pt]
\dfrac{|\xi|^2}{b(t)+g(t)|\xi|^2} & \dfrac{|\xi|^2}{b(t)+g(t)|\xi|^2}
\end{array} \right) \\
& \qquad + \exp \bigg( -\int_{s}^{t} \big( b(\tau)+g(\tau)|\xi|^2 \big)d\tau \bigg)
\left( \begin{array}{cc}
0 & 0 \\
0 & 1
\end{array} \right)
\end{align*}
for all $t\geq s$ and $(t,\xi), (s,\xi)\in\Zell(N)$.
\end{proposition}
\begin{proof} Let us assume that $\Phi_k=\Phi_k(t,s,\xi)$, $k=1,2,$ are solutions to the equation
\[ \Phi_{tt}+|\xi|^2\Phi+\big( b(t) + g(t)|\xi|^2 \big)\Phi_t=0 \]
with initial values $\Phi_k(s,s,\xi)=\delta_{1k}$ and $\partial_t\Phi_k(s,s,\xi)=\delta_{2k}$. Then, we have
\[ \left( \begin{array}{cc}
    |\xi|w(t,\xi) \\
	D_tw(t,\xi)
	\end{array} \right)=\left( \begin{array}{cc}
	\Phi_1(t,s,\xi) & i|\xi|\Phi_2(t,s,\xi) \\ [5pt]
	\dfrac{D_t\Phi_1(t,s,\xi)}{|\xi|} & iD_t\Phi_2(t,s,\xi)
	\end{array} \right) \left( \begin{array}{cc}
	|\xi|w(s,\xi) \\
	D_tw(s,\xi)
	\end{array} \right). \]
Our idea is to relate the entries of the above given estimates to the multipliers $\Phi_k=\Phi_k(t,s,\xi)$ and use Duhamel's formula to improve the estimates for the second row using estimates from the first one. Hence, if we compare with the estimates \eqref{Eq:Overdamping-Integrable_Transf-back}, then we obtain
\begin{align*}
|\Phi_1(t,s,\xi)| &\lesssim |\xi|\exp \Big( -C_N\int_{s}^{t}b(\tau)d\tau \Big), \\
|\Phi_2(t,s,\xi)| &\lesssim \exp\Big( -C_N\int_{s}^{t}b(\tau)d\tau \Big), \\
|\partial_t\Phi_1(t,s,\xi)| &\lesssim \frac{b(t)+g(t)|\xi|^2}{|\xi|}\exp \Big( -C_N\int_{s}^{t}b(\tau)d\tau \Big), \\
\big| \partial_t\Phi_2(t,s,\xi) \big| &\lesssim \big( b(t)+g(t)|\xi|^2 \big)\exp \Big( -C_N\int_{s}^{t}b(\tau)d\tau \Big).
\end{align*}
Let $\Psi_k=\Psi_k(t,s,\xi)=\partial_t\Phi_k(t,s,\xi), \,k=1,2$. Then, we obtain the equations of first order
\[ \partial_t\Psi_k+\big( b(t) + g(t)|\xi|^2 \big)\Psi_k=-|\xi|^2\Phi_k, \quad \Psi_k(s,s,\xi)=\delta_{2k}. \]
Standard calculations lead to
\begin{align*}
\Psi_1(t,s,\xi) &= -|\xi|^2\int_{s}^{t}\frac{\lambda^2(\tau,\xi)}{\lambda^2(t,\xi)}\Phi_1(\tau,s,\xi)d\tau, \\
\Psi_2(t,s,\xi) &= \frac{\lambda^2(s,\xi)}{\lambda^2(t,\xi)}-|\xi|^2\int_{s}^{t}\frac{\lambda^2(\tau,\xi)}{\lambda^2(t,\xi)}\Phi_2(\tau,s,\xi)d\tau,
\end{align*}
where $\lambda=\lambda(t,\xi)=\exp\Big( \dfrac{1}{2}\displaystyle\int_{0}^{t}\big( b(\tau)+g(\tau)|\xi|^2 \big)d\tau \Big)$.

If we are able to derive the desired estimate for $|\Psi_1(t,s,\xi)|$, then we conclude immediately the desired estimate for $|\Psi_2(t,s,\xi)|$. Using the estimates for $|\Phi_1(t,s,\xi)|$ and applying partial integration we get
\begin{align*}
|\Psi_1(t,s,\xi)| &\lesssim \frac{|\xi|^3}{\lambda^2(t,\xi)}\int_{s}^{t}\lambda^2(\tau,\xi)\exp \Big( -C_N\int_{s}^{\tau} b(\theta)d\theta \Big)d\tau \\
& \lesssim \frac{|\xi|^3}{\lambda^2(t,\xi)}\int_{s}^{t}\frac{1}{b(\tau)+g(\tau)|\xi|^2}\exp \Big( -C_N\int_{s}^{\tau}b(\theta)d\theta \Big)\partial_\tau\lambda^2(\tau,\xi)d\tau \\
& \lesssim \frac{|\xi|^3}{\lambda^2(t,\xi)}\bigg( \lambda^2(\tau,\xi)\frac{1}{b(\tau)+g(\tau)|\xi|^2}\exp \Big( -C_N\int_{s}^{\tau}b(\theta)d\theta \Big) \bigg)\Big|_s^t \\
& \,\,\,\, + \frac{|\xi|^3}{\lambda^2(t,\xi)}\int_{s}^{t}\lambda^2(\tau,\xi)\bigg( \underbrace{\frac{b'(\tau)-g'(\tau)|\xi|^2}{\big( b(\tau)+g(\tau)|\xi|^2 \big)^2}+\frac{C_Nb(\tau)}{b(\tau)+g(\tau)|\xi|^2}}_{\lesssim C(\tau)\leq C_0<1} \bigg)\exp\Big( -C_N\int_{s}^{\tau}b(\theta)d\theta \Big)d\tau \\
& \lesssim \frac{|\xi|^3}{b(t)+g(t)|\xi|^2} \exp\Big( -C_N\int_{s}^{t}b(\tau)d\tau \Big)-\frac{|\xi|^3}{b(s)+g(s)|\xi|^2}\frac{\lambda^2(s,\xi)}{\lambda^2(t,\xi)},
\end{align*}
where we have used
\begin{align*}
\frac{b'(\tau)-g'(\tau)|\xi|^2}{\big( b(\tau)+g(\tau)|\xi|^2 \big)^2}+\frac{C_Nb(\tau)}{b(\tau)+g(\tau)|\xi|^2} &\leq \frac{b'(\tau)}{b^2(\tau)} + \frac{-g'(\tau)|\xi|^2}{g^2(\tau)|\xi|^4} + \frac{C_Nb(\tau)}{b(\tau)} \\
& = o(1) + \frac{-g'(\tau)}{g^2(\tau)|\xi|^2} + C_N \leq o(1) + \frac{1}{N^2}\big( -g'(\tau) \big) < 1,
\end{align*}
where we have employed $|b'(t)|=o(b^2(t))$ and $g(t)|\xi|\geq N$ due to the definition of $\Zell(N)$. So, we find
\[ |\Phi_t^1(t,s,\xi)| \lesssim \frac{|\xi|^3}{b(t)+g(t)|\xi|^2} \exp\Big( -C_N\int_{s}^{t}b(\tau)d\tau \Big). \]
Similarly, we can estimate $|\Psi_2(t,s,\xi)|$ in the following way:
\begin{align*}
|\Psi_2(t,s,\xi)| & \lesssim  \frac{\lambda^2(s,\xi)}{\lambda^2(t,\xi)}+\frac{|\xi|^2}{\lambda^2(t,\xi)}\int_{s}^{t}\lambda^2(\tau,\xi)\exp \Big( -C_N\int_{s}^{\tau}b(\theta)d\theta \Big)d\tau \\
& \lesssim \frac{\lambda^2(s,\xi)}{\lambda^2(t,\xi)}+\frac{|\xi|^2}{b(t)+g(t)|\xi|^2}\exp \Big( -C_N\int_{s}^{t}b(\tau)d\tau \Big).
\end{align*}
Summarizing all desired estimates are proved.
\end{proof}
Taking account of the representation of solutions from the previous corollary with $s=0$ and the refined estimates from Proposition \ref{Prop:Overdamping-Integrable-Refined-Zell}, it follows
\begin{align*}
|\xi||\hat{u}(t,\xi)| &\lesssim \exp\Big( -C_N\int_0^t b(\tau)d\tau \Big)\big( |\xi|^2|\hat{u}_0(\xi)| + |\xi||\hat{u}_1(\xi)| \big), \\
|\hat{u}_t(t,\xi)| &\lesssim \exp\Big( -C_N\int_0^t b(\tau)d\tau \Big)\big( |\xi|^2|\hat{u}_0(\xi)| + |\xi||\hat{u}_1(\xi)| \big) \\
& \qquad + \exp\Big( -\int_0^t \big( b(\tau)+g(\tau)|\xi|^2 \big)d\tau \Big)|\xi||\hat{u}_1(\xi)| \\
& \lesssim \exp\Big( -C_N\int_0^t b(\tau)d\tau \Big)\big( |\xi|^2|\hat{u}_0(\xi)| + |\xi||\hat{u}_1(\xi)| \big).
\end{align*}
\begin{corollary} \label{Cor:Overdamping-Integrable_IncreasingZell}
The following estimates hold in $\Zell(N)$ for all $t\in[0,t_{\xi,3}]$:
\begin{align*}
|\xi|^{|\beta|}|\hat{u}(t,\xi)| &\lesssim \exp\Big( -C_N\int_0^t b(\tau)d\tau \Big)\big( |\xi|^{|\beta|+1}|\hat{u}_0(\xi)| + |\xi|^{|\beta|}|\hat{u}_1(\xi)| \big)\quad \mbox{for} \quad |\beta|\geq0, \\
|\xi|^{|\beta|}|\hat{u}_t(t,\xi)| &\lesssim \exp\Big( -C_N\int_0^t b(\tau)d\tau \Big)\big( |\xi|^{|\beta|+2}|\hat{u}_0(\xi)| + |\xi|^{|\beta|+1}|\hat{u}_1(\xi)| \big) \quad \mbox{for} \quad |\beta|\geq0.
\end{align*}
\end{corollary}
\subsubsection*{Considerations in the first reduced zone $\Zred^1(N,\varepsilon_1)$ and the second reduced zone $\Zred^2(\varepsilon_2,\varepsilon_3)$} \label{Subsection_Overdamping_Integrable_Zred}
Analogously to Section \ref{Subsection_Effective_Integrable_Zred}, we may conclude the following estimates in $\Zred^1(N,\varepsilon_1)$ and $\Zred^2(\varepsilon_2,\varepsilon_3)$.
\begin{corollary} \label{Cor:Overdamping-Integrable-Zred}
The following estimates hold with $t_{\xi,4}\leq s\leq t\leq t_{\xi,3}$ and $(s,\xi), (t,\xi) \in \Zred^1(N,\varepsilon_1)$, and with $t_{\xi,2}\leq s\leq t\leq t_{\xi,1}$ and $(s,\xi), (t,\xi) \in \Zred^2(\varepsilon_2,\varepsilon_3)$:
\begin{align*}
|\xi|^{|\beta|}|\hat{u}(t,\xi)| &\lesssim |\xi|^{|\beta|}|\hat{u}(s,\xi)| + |\xi|^{|\beta|-1}|\hat{u}_t(s,\xi)| \quad \mbox{for} \quad |\beta| \geq 1, \\
|\xi|^{|\beta|}|\hat{u}_t(t,\xi)| &\lesssim |\xi|^{|\beta|+1}|\hat{u}(s,\xi)| + |\xi|^{|\beta|}|\hat{u}_t(s,\xi)| \quad \mbox{for} \quad |\beta| \geq 0.
\end{align*}
\end{corollary}
\subsubsection*{Considerations in the dissipative zone $\Zdiss(\varepsilon_3)$} \label{Sect-Overdamping-Integrable-Zdiss}
From the definition of $\Zdiss(\varepsilon_3)$, we have $g(t)|\xi|\leq \varepsilon_3$. So, due to the fact that $b(t)g(t)=1/2$
\[  g(t)|\xi|^2 \leq \frac{\varepsilon_3^2}{g(t)} = \frac{\varepsilon_3^2b(t)}{b(t)g(t)} = 2\varepsilon_3^2b(t). \]
Therefore, as we have done in Section \ref{Subsection_Effective_Integrable_Zdiss} we obtain that $b(t)\leq b(t)+g(t)|\xi|^2 \leq (1+2\varepsilon_3^2)b(t)$ and, this allows us to follow the same approach. We have
\begin{equation*} \label{Eq:Overdamping-integrable-EM1-Zdiss}
\hat{u}_{tt} + |\xi|^2 \hat{u} + \tilde{b}(t,\xi)\hat{u}_t = 0, \,\,\,\,\,\, \mbox{where}\,\,\,\,\,\, \tilde b(t,\xi) = b(t)+g(t)|\xi|^2.
\end{equation*}
Using $b(t) \leq b(t)+g(t)|\xi|^2 \leq (1+\kappa)b(t)$, then the above dissipative zone is divided into elliptic zone and reduced zone as follows:
\begin{align*}
\Zell(0,\varepsilon) &= \Big\{ (t,\xi): 0<|\xi|\leq \frac{b(t)}{2}\sqrt{1-\varepsilon^2} \Big\}, \\
\Zred(0,\varepsilon) &= \Big\{ (t,\xi): \frac{b(t)}{2}\sqrt{1-\varepsilon^2} \leq  |\xi| \leq \frac{b(t)}{2} \Big\}.
\end{align*}
The separating line $t_1(|\xi|)$ solves $|\xi| = \dfrac{b(t)}{2}\sqrt{1-\varepsilon^2}$.

From Corollary \ref{Cor:Effective-Integrable-EM10-Zdiss} and due to $\int_0^\infty\frac{1}{b(\tau)}d\tau\in L^1([0,\infty)$, we have the  following estimates in $\Zell(0,\varepsilon)$.
\begin{corollary} \label{Cor:Overdamping-Integrable-EM10-Zdiss}
The following estimates hold for all $t\in[0,t_1]$ in $\Zell(0,\varepsilon)$:
\begin{align*}
|\xi|^{|\beta|}|\hat{u}(t,\xi)| &\lesssim |\xi|^{|\beta|}|\hat{u}_0(\xi)| + |\xi|^{|\beta|-1|}|\hat{u}_1(\xi)| \quad \mbox{for} \quad |\beta|\geq1, \\
|\xi|^{|\beta|}|\hat{u}_t(t,\xi)| &\lesssim |\xi|^{|\beta|+1}|\hat{u}_0(\xi)| + |\xi|^{|\beta|}|\hat{u}_1(\xi)| \quad \mbox{for} \quad |\beta|\geq0.
\end{align*}
\end{corollary}
Moreover, we have the following estimates in $\Zred(0,\varepsilon)$.
\begin{corollary} \label{Cor:Overdamping-Integrable-EM14-Zdiss}
The following estimates hold for all $t\in[t_{\xi,1},t_1]$ in $\Zred(0,\varepsilon)$:
\begin{align*}
|\xi|^{|\beta|}|\hat{u}(t,\xi)| &\lesssim \exp\Big( -C\int_{s}^tb(\tau)d\tau \Big)\Big( |\xi|^{|\beta|}|\hat{u}(s,\xi)| + |\xi|^{|\beta|-1|}|\hat{u}_t(s,\xi)| \Big) \quad \mbox{for} \quad |\beta|\geq1, \\
|\xi|^{|\beta|}|\hat{u}_t(t,\xi)| &\lesssim \exp\Big( -C\int_{s}^t b(\tau)d\tau \Big)\Big( |\xi|^{|\beta|+1}|\hat{u}(s,\xi)| + |\xi|^{|\beta|}|\hat{u}_t(s,\xi)| \Big) \quad \mbox{for} \quad |\beta|\geq0.
\end{align*}
\end{corollary}

\subsubsection*{Conclusion} \label{Sec:Overdamping-Decaying-Conclusions}
From the statements of Proposition \ref{Prop:Overdamping-Integrable-Zhyp} and Corollaries \ref{Cor:Overdamping-Integrable_IncreasingZell}, \ref{Cor:Overdamping-Integrable-Zred}, \ref{Cor:Overdamping-Integrable-EM10-Zdiss}, and \ref{Cor:Overdamping-Integrable-EM14-Zdiss}, we derive our desired statements.\\

\noindent \underline{Small frequencies}:
\medskip

\noindent \textit{Case 1:} $t\leq t_1(|\xi|)$. Due to Corollary \ref{Cor:Overdamping-Integrable-EM10-Zdiss} we have the following estimates:
\begin{align*}
|\xi|^{|\beta|}|\hat{u}(t,\xi)| &\lesssim |\xi|^{|\beta|}|\hat{u}_0(\xi)| + |\xi|^{|\beta|-1|}|\hat{u}_1(\xi)|  \quad \mbox{for} \quad |\beta|\geq1, \\
|\xi|^{|\beta|}|\hat{u}_t(t,\xi)| &\lesssim |\xi|^{|\beta|+1}|\hat{u}_0(\xi)| + |\xi|^{|\beta|}|\hat{u}_0(\xi)| \quad \mbox{for} \quad |\beta|\geq0.
\end{align*}

\noindent \underline{Large frequencies}:
\medskip

\noindent \textit{Case 1:} $t\leq t_{\xi,4}$. In this case, we use Corollary \ref{Cor:Overdamping-Integrable_IncreasingZell}. We have
\begin{align*}
|\xi|^{|\beta|}|\hat{u}(t,\xi)| &\lesssim \exp\Big( -C_N\int_0^t b(\tau)d\tau \Big)\big( |\xi|^{|\beta|+1}|\hat{u}_0(\xi)| + |\xi|^{|\beta|}|\hat{u}_1(\xi)| \big) \quad \mbox{for} \quad |\beta| \geq 0, \\
|\xi|^{|\beta|}|\hat{u}_t(t,\xi)|& \lesssim \exp\Big( -C_N\int_0^t b(\tau)d\tau \Big) \big( |\xi|^{|\beta|+2}|\hat{u}_0(\xi)| + |\xi|^{|\beta|+1}|\hat{u}_1(\xi)| \big) \quad \mbox{for} \quad |\beta| \geq 0.
\end{align*}
\noindent \textit{Case 2:} $t_{\xi,4} \leq t \leq t_{\xi,3}$. In this case, we use first Corollary \ref{Cor:Overdamping-Integrable-Zred} with $s=t_{\xi,4}$ and, then the estimates from \textit{Case 1}. Then, we have the following estimates:
\begin{align*}
|\xi|^{|\beta|}|\hat{u}(t,\xi)| &\lesssim |\xi|^{|\beta|}|\hat{u}(t_{\xi,4},\xi)| + |\xi|^{|\beta|-1}|\hat{u}_t(t_{\xi,4},\xi)| \\
& \lesssim \exp\Big( -C_N\int_0^{t_{\xi,4}} b(\tau)d\tau \Big)\big( |\xi|^{|\beta|+1}|\hat{u}_0(\xi)| + |\xi|^{|\beta|}|\hat{u}_1(\xi)| \big), \\
|\xi|^{|\beta|}|\hat{u}_t(t,\xi)|& \lesssim  |\xi|^{|\beta|+1}|\hat{u}(t_{\xi,4},\xi)| + |\xi|^{|\beta|}|\hat{u}_t(t_{\xi,4},\xi)| \\
& \lesssim \exp\Big( -C_N\int_0^{t_{\xi,4}} b(\tau)d\tau \Big)\big( |\xi|^{|\beta|+2}|\hat{u}_0(\xi)| + |\xi|^{|\beta|+1}|\hat{u}_1(\xi)| \big).
\end{align*}
\noindent \textit{Case 3:} $t_{\xi,3}\leq t \leq t_{\xi,2}$. In this case, we use Proposition \ref{Prop:Overdamping-Integrable-Zhyp} for $s=t_{\xi,3}$ and, then the estimates from \textit{Case 2}. So, we have the following estimates:
\begin{align*}
|\xi|^{|\beta|}|\hat{u}(t,\xi)| &\lesssim \exp\Big( -\frac{1}{2}\int_{t_{\xi,3}}^t\big( b(\tau)+g(\tau)|\xi|^2 \big)d\tau \Big)\big( |\xi|^{|\beta|}|\hat{u}(t_{\xi,3},\xi)| + |\xi|^{|\beta|-1}|\hat{u}_t(t_{\xi,3},\xi)| \big) \\
& \lesssim \exp\Big( -\big( \frac{1}{2}+\varepsilon_2^2 \big)\int_{t_{\xi,3}}^t b(\tau)d\tau \Big)\exp\Big( -C_N\int_0^{t_{\xi,4}} b(\tau)d\tau \Big)\big( |\xi|^{|\beta|+1}|\hat{u}_0(\xi)| + |\xi|^{|\beta|}|\hat{u}_1(\xi)| \big) \\
& \lesssim \exp\Big( -\big( \frac{1}{2}+\varepsilon_2^2 \big)\int_{t_{\xi,4}}^t b(\tau)d\tau \Big)\exp\Big( -C_N\int_0^{t_{\xi,4}} b(\tau)d\tau \Big)\big( |\xi|^{|\beta|+1}|\hat{u}_0(\xi)| + |\xi|^{|\beta|}|\hat{u}_1(\xi)| \big) \\
& \lesssim \exp\Big( -C_{N,\varepsilon_2}\int_0^t b(\tau)d\tau \Big)\big( |\xi|^{|\beta|+1}|\hat{u}_0(\xi)| + |\xi|^{|\beta|}|\hat{u}_1(\xi)| \big), \\
|\xi|^{|\beta|}|\hat{u}_t(t,\xi)| &\lesssim \exp\Big( -\frac{1}{2}\int_{t_{\xi,3}}^t\big( b(\tau)+g(\tau)|\xi|^2 \big)d\tau \Big)\big( |\xi|^{|\beta|+2}|\hat{u}(t_{\xi,3},\xi)| + |\xi|^{|\beta|+1}|\hat{u}_t(t_{\xi,2},\xi)| \big) \\
& \lesssim \exp\Big( -\big( \frac{1}{2}+\varepsilon_2^2 \big)\int_{t_{\xi,3}}^t b(\tau)d\tau \Big)\exp\Big( -C_N\int_0^{t_{\xi,4}} b(\tau)d\tau \Big)\big( |\xi|^{|\beta|+2}|\hat{u}_0(\xi)| + |\xi|^{|\beta|+1}|\hat{u}_1(\xi)| \big) \\
& \lesssim \exp\Big( -\big( \frac{1}{2}+\varepsilon_2^2 \big)\int_{t_{\xi,4}}^t b(\tau)d\tau \Big)\exp\Big( -C_N\int_0^{t_{\xi,4}} b(\tau)d\tau \Big)\big( |\xi|^{|\beta|+2}|\hat{u}_0(\xi)| + |\xi|^{|\beta|+1}|\hat{u}_1(\xi)| \big) \\
& \lesssim \exp\Big( -C_{N,\varepsilon_2}\int_0^t b(\tau)d\tau \Big)\big( |\xi|^{|\beta|+2}|\hat{u}_0(\xi)| + |\xi|^{|\beta|+1}|\hat{u}_1(\xi)|  \big),
\end{align*}
where due to $\varepsilon_2\leq g(t)|\xi|\leq \varepsilon_1$ from the definition of $\Zhyp(\varepsilon_1,\varepsilon_2)$, we used
\[ -g(\tau)|\xi|^2\leq -\varepsilon_2|\xi| \leq -2\varepsilon_2^2b(\tau), \]
moreover, for $t_{\xi,4}\leq t_{\xi,3}\leq t \leq t_{\xi,2}$ we used
\begin{equation} \label{Eq:Overdamping-integrable-Gluing1}
\exp\Big( -\big( \frac{1}{2}+\varepsilon_2^2 \big)\int_{t_{\xi,3}}^t b(\tau)d\tau \Big) \leq \exp\Big( -\big( \frac{1}{2}+\varepsilon_2^2 \big)\int_{t_{\xi,4}}^t b(\tau)d\tau \Big).
\end{equation}
\noindent \textit{Case 4:} $t_{\xi,2} \leq t \leq t_{\xi,1}$. In this case, we use first Corollary \ref{Cor:Overdamping-Integrable-Zred} with $s=t_{\xi,2}$ and, then the estimates from \textit{Case 3:}. Then, we have the following estimates:
\begin{align*}
|\xi|^{|\beta|}|\hat{u}(t,\xi)| &\lesssim |\xi|^{|\beta|}|\hat{u}(t_{\xi,2},\xi)| + |\xi|^{|\beta|-1}|\hat{u}_t(t_{\xi,2},\xi)| \\
& \lesssim \exp\Big( -C_{N,\varepsilon_2}\int_0^{t_{\xi,2}} b(\tau)d\tau \Big)\big( |\xi|^{|\beta|+1}|\hat{u}_0(\xi)| + |\xi|^{|\beta|}|\hat{u}_1(\xi)| \big), \\
|\xi|^{|\beta|}|\hat{u}_t(t,\xi)|& \lesssim  |\xi|^{|\beta|+1}|\hat{u}(t_{\xi,2},\xi)| + |\xi|^{|\beta|}|\hat{u}_t(t_{\xi,2},\xi)| \\
& \lesssim \exp\Big( -C_{N,\varepsilon_2}\int_0^{t_{\xi,2}} b(\tau)d\tau \Big)\big( |\xi|^{|\beta|+2}|\hat{u}_0(\xi)| + |\xi|^{|\beta|+1}|\hat{u}_1(\xi)| \big).
\end{align*}
\noindent \textit{Case 5:} $t_{\xi,1}\leq t \leq t_1(|\xi|)$. In this case, we use Corollary \ref{Cor:Overdamping-Integrable-EM14-Zdiss} for $s=t_{\xi,1}$ and, then the estimates from \textit{Case 4}. Thus, we have
\begin{align*}
|\xi|^{|\beta|}|\hat{u}(t,\xi)| &\lesssim\exp\Big( -C\int_{t_{\xi,1}}^t b(\tau)d\tau \Big) \big( |\xi|^{|\beta|}|\hat{u}(t_{\xi,1},\xi)| + |\xi|^{|\beta|-1}|\hat{u}_t(t_{\xi,1},\xi)| \big) \\
& \lesssim \exp\Big( -C\int_{t_{\xi,1}}^t b(\tau)d\tau \Big)\exp\Big( -C_{N,\varepsilon_2}\int_0^{t_{\xi,2}} b(\tau)d\tau \Big)\big( |\xi|^{|\beta|+1}|\hat{u}_0(\xi)| + |\xi|^{|\beta|}|\hat{u}_1(\xi)| \big) \\
& \lesssim \exp\Big( -\tilde{C}_{N,\varepsilon_2}\int_{0}^t b(\tau)d\tau \Big)\big( |\xi|^{|\beta|+1}|\hat{u}_0(\xi)| + |\xi|^{|\beta|}|\hat{u}_1(\xi)| \big), \\
|\xi|^{|\beta|}|\hat{u}_t(t,\xi)| &\lesssim \exp\Big( -C\int_{t_{\xi,1}}^t b(\tau)d\tau \Big) \big( |\xi|^{|\beta|+1}|\hat{u}(t_{\xi,1},\xi)| + |\xi|^{|\beta|}|\hat{u}_t(t_{\xi,1},\xi)| \big) \\
& \lesssim \exp\Big( -C\int_{t_{\xi,1}}^t b(\tau)d\tau \Big)\exp\Big( -C_{N,\varepsilon_2}\int_0^{t_{\xi,2}} b(\tau)d\tau \Big)\big( |\xi|^{|\beta|+2}|\hat{u}_0(\xi)| + |\xi|^{|\beta|+1}|\hat{u}_1(\xi)| \big) \\
& \lesssim \exp\Big( -\tilde{C}_{N,\varepsilon_2}\int_{0}^t b(\tau)d\tau \Big)\big( |\xi|^{|\beta|+2}|\hat{u}_0(\xi)| + |\xi|^{|\beta|+1}|\hat{u}_1(\xi)| \big),
\end{align*}
where similarly as in \eqref{Eq:Overdamping-integrable-Gluing1}, we used
\[ \exp\Big( -C\int_{t_{\xi,1}}^t b(\tau)d\tau \Big) \leq \exp\Big( -C\int_{t_{\xi,2}}^t b(\tau)d\tau \Big). \]
\noindent \textit{Case 6:} $t \geq t_1(|\xi|)$. In this case, we use Corollary \ref{Cor:Overdamping-Integrable-EM10-Zdiss} with $s=t_1(|\xi|)$ and the estimates from \textit{Case 5}. We get
\begin{align*}
|\xi|^{|\beta|}|\hat{u}(t,\xi)| &\lesssim |\xi|^{|\beta|}|\hat{u}(t_1(|\xi|),\xi)| + |\xi|^{|\beta|-1|}|\hat{u}_t(t_1(|\xi|),\xi)| \\
& \lesssim \exp\Big( -\tilde{C}_{N,\varepsilon_2}\int_{0}^{t_1(|\xi|)} b(\tau)d\tau \Big)\big( |\xi|^{|\beta|+1}|\hat{u}_0(\xi)| + |\xi|^{|\beta|}|\hat{u}_1(\xi)| \big), \\
|\xi|^{|\beta|}|\hat{u}_t(t,\xi)| &\lesssim |\xi|^{|\beta|+1}|\hat{u}(t_1(|\xi|),\xi)| + |\xi|^{|\beta|}|\hat{u}_t(t_1(|\xi|),\xi)| \\
& \lesssim \exp\Big( -\tilde{C}_{N,\varepsilon_2}\int_{0}^{t_1(|\xi|)} b(\tau)d\tau \Big)\big( |\xi|^{|\beta|+2}|\hat{u}_0(\xi)| + |\xi|^{|\beta|+1}|\hat{u}_1(\xi)| \big).
\end{align*}
\subsubsection*{Proof of Theorem \ref{Theorem_Overdamping-Decreasing_Integrable}}
From the preceding estimates in the small and large frequencies, it is evident that the dominant contributions to the final estimates arise from the dissipative zone $\Zdiss(\varepsilon_3)$ and the elliptic zone $\Zell(N)$. These results together complete the proof of Theorem \ref{Theorem_Overdamping-Decreasing_Integrable}.

\subsubsection{The model with $b(t)=e^{e^t}$ and $g(t)=2e^{-t}$} \label{Subsec:Overdamping-Integrable-SubSec3}
In this case, we consider $b(t)=e^{e^t}$ and $g(t)=2e^{-t}$, which yield
\[ b(t)g(t) = 2e^{-t}e^{e^t}\geq2. \]
From \eqref{Eq:Effective-Overdamping-Sep-lines}, this implies that the separating lines are given by
\[ |\xi|_{1/2} = \frac{1}{g(t)} \pm i\frac{1}{g(t)}\sqrt{b(t)g(t)-1}. \]
However, since $|\xi|_{1/2}$ becomes complex-valued, it is no longer possible to define real-valued separating lines. Consequently, we consider the transformed equation
\begin{align*}
D_t^2w + \Big(-|\xi|^2 + \dfrac{b(t)g(t)}{2}|\xi|^2 + \dfrac{g(t)^2}{4}|\xi|^4 + \frac{g'(t)}{2}|\xi|^2 +  \dfrac{b(t)^2}{4} + \frac{b'(t)}{2} \Big)w=0.
\end{align*}
For all $(t,\xi)$ in the extended phase space, we have
\begin{equation*} \label{Eq:Overdamping-Decreasing_Integrable-Elliptic-WKB}
-|\xi|^2 + \dfrac{b(t)g(t)}{2}|\xi|^2 + \dfrac{g(t)^2}{4}|\xi|^4 + \frac{g'(t)}{2}|\xi|^2 +  \dfrac{b(t)^2}{4} + \frac{b'(t)}{2} \geq C_0>0.
\end{equation*}
This shows that the equation is uniformly elliptic, and thus the analysis must rely entirely on techniques from elliptic WKB-analysis.
\begin{theorem} \label{Theorem_Overdamping-Decreasing_Integrable-3}
Let us consider the Cauchy problem
\begin{equation*}
\begin{cases}
u_{tt} - \Delta u + b(t)u_t - g(t)\Delta u_t=0, &(t,x) \in [0,\infty) \times \mathbb{R}^n, \\
u(0,x)= u_0(x),\quad u_t(0,x)= u_1(x), &x \in \mathbb{R}^n.
\end{cases}
\end{equation*}
We assume that $b(t)=e^{e^t}$ and $g(t)=2e^{-t}$. Then, we have the following estimates for Sobolev solutions:
\begin{align*}
\|\,|D|^{|\beta|}u(t,\cdot)\|_{L^2} &\lesssim \|u_0\|_{\dot{H}^{|\beta|}} + \|u_1\|_{\dot{H}^{|\beta|-2}} \quad \mbox{for} \quad |\beta|\geq 2, \\
\|\,|D|^{|\beta|}u_t(t,\cdot)\big\|_{L^2} &\lesssim b(t)\big( \|u_0\|_{\dot{H}^{|\beta|}} + \|u_1\|_{\dot{H}^{|\beta|-2}} \big) + g(t)\big( \|u_0\|_{\dot{H}^{|\beta|+2}} + \|u_1\|_{\dot{H}^{|\beta|}} \big) \quad \mbox{for} \quad |\beta|\geq 2.
\end{align*}
\end{theorem}
\begin{proof}
We introduce the following family of symbol classes in the extended phase space.
\begin{definition} \label{Def:Overdamping-Integrable-symbol-Zod}
A function $f=f(t,\xi)$ belongs to the symbol class $S_{\text{ell}}^\ell\{m_1,m_2\}$ if it holds
\begin{equation*}
|D_t^kf(t,\xi)|\leq C_{k}\big( b(t)+g(t)|\xi|^2 \big)^{m_1}\Big( \frac{b'(t)+g'(t)|\xi|^2}{b(t)+g(t)|\xi|^2} \Big)^{m_2+k}
\end{equation*}
for all $(t,\xi)\in [0,\infty)\times \mathbb{R}^n$ and all $k\leq \ell$.
\end{definition}
Thus, we may conclude the following rules from the definition of the symbol classes.
\begin{proposition} \label{Prop:Overdamping-Integrable-symbol-Zod}
The following statements are true:
\begin{itemize}
\item $S_{\text{ell}}^\ell\{m_1,m_2\}$ is a vector space for all nonnegative integers $\ell$;
\item $S_{\text{ell}}^\ell\{m_1,m_2\}\cdot S_{\text{ell}}^{\ell}\{m_1',m_2'\}\hookrightarrow S_{\text{ell}}^{\ell}\{m_1+m_1',m_2+m_2'\}$;
\item $D_t^kS_{\text{ell}}^\ell\{m_1,m_2\}\hookrightarrow S_{\text{ell}}^{\ell-k}\{m_1,m_2+k\}$
for all nonnegative integers $\ell$ with $k\leq \ell$;
\item $S_{\text{ell}}^{0}\{-1,2\}\hookrightarrow L_{\xi}^{\infty}L_t^1([0,\infty))$.
\end{itemize}
\end{proposition}
\begin{proof}
Let us verify the last statement. Indeed, if $f=f(t,\xi)\in S_{\text{ell}}^{0}\{-1,2\}$, then we have
\begin{align*}
\Big| \frac{( b'(t)+g'(t)|\xi|^2)^2}{(b(t)+g(t)|\xi|^2)^3} \Big| &\leq \frac{b'(t)^2+2b'(t)|g'(t)|\xi|^2+g'(t)^2|\xi|^4}{(b(t)+g(t)|\xi|^2)^3}  \\
& = \frac{e^{2t}b^2(t)+4b(t)|\xi|^2+g^2(t)|\xi|^4}{(b(t)+g(t)|\xi|^2)^3} \\
& \leq \frac{e^{2t}b^2(t)}{b^3(t)} + \frac{4b(t)|\xi|^2}{b^2(t)g(t)|\xi|^2} + \frac{g^2(t)|\xi|^4}{b(t)g^2(t)|\xi|^4} \\
& = \frac{e^{2t}}{b(t)} + \frac{4}{b(t)g(t)} + \frac{1}{b(t)} = e^{2t}e^{-e^t} + 2e^te^{-e^t} + e^{-e^t} \leq e^{2t}e^{-e^t},
\end{align*}
where since $b(t)=e^{e^t}$ and $g(t)=2e^{-t}$, we used the fact that $b'(t)=e^tb(t)$, $g'(t)=-g(t)$ and $b'(t)g'(t)=-2e^{e^t}=-2b(t)$. Hence, we get
\begin{align*}
\int_{0}^{\infty}|f(\tau,\xi)|d\tau  & \leq \int_{0}^{\infty}e^{2\tau}e^{-e^\tau}d\tau \leq 1.
\end{align*}
\end{proof}
Taking account of \eqref{Eq:Effective-Overdamping-Dfrom}, we choose the following micro-energy:
\[ W=W(t,\xi) := \Big[ \big( \frac{b(t)}{2}+\frac{g(t)|\xi|^2}{2} \big)w,D_tw \Big]^{\text{T}}. \]
Then, we obtain that $W=W(t,\xi)$ satisfies the following system of first order:
\begin{equation*} \label{Eq:Overdamping-Integrable-System-Zod}
D_tW=\underbrace{\left[ \left( \begin{array}{cc}
0 & \dfrac{g(t)|\xi|^2}{2}+\dfrac{b(t)}{2} \\
-\dfrac{g(t)|\xi|^2}{2}-\dfrac{b(t)}{2} & 0
\end{array} \right) + \left( \begin{array}{cc}
\dfrac{D_t(b(t)+g(t)|\xi|^2)}{b(t)+g(t)|\xi|^2} & 0 \\
-\dfrac{(g'(t)-2)|\xi|^2}{b(t)+g(t)|\xi|^2}-\dfrac{b'(t)}{b(t)+g(t)|\xi|^2} & 0
\end{array} \right)\right]}_{A_W}W.
\end{equation*}
We want to estimate the fundamental solution $E_W=E_W(t,s,\xi)$ to the above system, namely, the solution to
\begin{equation*}
D_tE_W(t,s,\xi)=A_W(t,\xi)E_W(t,s,\xi), \quad E_W(s,s,\xi)=I \quad \mbox{for any} \quad t\geq s\geq 0.
\end{equation*}
We denote by $M$ the matrix consisting of eigenvectors of the first matrix on the right-hand side and its inverse matrix
\[ M = \left( \begin{array}{cc}
i & -i \\
1 & 1
\end{array} \right), \qquad M^{-1}=\frac{1}{2}\left( \begin{array}{cc}
-i & 1 \\
i & 1
\end{array} \right). \]
Then, defining $W^{(0)}:=M^{-1}W$ we get the system
\begin{equation*}
D_tW^{(0)}=\big( \mathcal{D}(t,\xi)+\mathcal{R}(t,\xi) \big)W^{(0)},
\end{equation*}
where
\begin{align*}
\mathcal{D}(t,\xi) &= \left( \begin{array}{cc}
-i\Big( \dfrac{g(t)}{2}|\xi|^2 + \dfrac{b(t)}{2} \Big) & 0 \\
0 & i\Big( \dfrac{g(t)}{2}|\xi|^2 + \dfrac{b(t)}{2} \Big)
\end{array} \right), \\
\mathcal{R}_1(t,\xi) &= \frac{1}{2} \left( \begin{array}{cc}
\dfrac{D_t(b(t)+g(t)|\xi|^2)}{b(t)+g(t)|\xi|^2}-i\dfrac{(g'(t)-2)|\xi|^2}{b(t)+g(t)|\xi|^2} & -\dfrac{D_t(b(t)+g(t)|\xi|^2)}{b(t)+g(t)|\xi|^2}+i\dfrac{(g'(t)-2)|\xi|^2}{b(t)+g(t)|\xi|^2} \\
-\dfrac{D_t(b(t)+g(t)|\xi|^2)}{b(t)+g(t)|\xi|^2}-i\dfrac{(g'(t)-2)|\xi|^2}{b(t)+g(t)|\xi|^2} & \dfrac{D_t(b(t)+g(t)|\xi|^2)}{b(t)+g(t)|\xi|^2}+i\dfrac{(g'(t)-2)|\xi|^2}{b(t)+g(t)|\xi|^2}
\end{array} \right), \\
\mathcal{R}_2(t,\xi) &= \frac{1}{2} \left( \begin{array}{cc}
-i\dfrac{b'(t)}{b(t)+g(t)|\xi|^2} & i\dfrac{b'(t)}{b(t)+g(t)|\xi|^2} \\
-i\dfrac{b'(t)}{b(t)+g(t)|\xi|^2} & i\dfrac{b'(t)}{b(t)+g(t)|\xi|^2}
\end{array} \right),
\end{align*}
where $\mathcal{D}(t,\xi)\in S_{\text{ell}}^2\{1,0\}$ and with $\mathcal{R}:=\mathcal{R}_1+\mathcal{R}_2\in S_{\text{ell}}^1\{0,1\}$.
\medskip

We perform one more step of diagonalization procedure. We define $F_0(t,\xi):=\diag\mathcal{R}(t,\xi)$. The difference of the diagonal entries of the matrix $\mathcal{D}(t,\xi)+F_0(t,\xi)$ is
\begin{align*}
b(t) + g(t)|\xi|^2 + \frac{b'(t)+(g'(t)-2)|\xi|^2}{b(t)+g(t)|\xi|^2} &= e^{e^t} + 2e^{-t}|\xi|^2 + \frac{e^te^{e^t}+(-2e^{-t}-2)|\xi|^2}{e^{e^t}+2e^{-t}|\xi|^2} \\
& \leq  e^{e^t} + 2e^{-t}|\xi|^2 + \frac{e^te^{e^t}+2e^{-t}|\xi|^2}{e^{e^t}+2e^{-t}|\xi|^2} \\
& \leq e^{e^t} + 2e^{-t}|\xi|^2 = b(t) + g(t)|\xi|^2 =:i\eta(t,\xi).
\end{align*}
Now we introduce the matrix $N^{(1)}=N^{(1)}(t,\xi)$ defined by
\begin{align*}
&N^{(1)}(t,\xi) := \left( \begin{array}{cc}
0 & -\dfrac{\mathcal{R}_{12}}{\eta(t,\xi)} \\
\dfrac{\mathcal{R}_{21}}{\eta(t,\xi)} & 0
\end{array} \right) \\
& = \left( \begin{array}{cc}
0 & i\dfrac{D_t(b(t)+g(t)|\xi|^2)}{2(b(t)+g(t)|\xi|^2)^2}-\dfrac{(g'(t)-2)|\xi|^2+b'(t)}{2(b(t)+g(t)|\xi|^2)^2} \\
i\dfrac{D_t(b(t)+g(t)|\xi|^2)}{2(b(t)+g(t)|\xi|^2)^2}+\dfrac{(g'(t)-2)|\xi|^2+b'(t)}{2(b(t)+g(t)|\xi|^2)^2} & 0
\end{array} \right),
\end{align*}
where $N^{(1)}(t,\xi)\in S_{\text{ell}}^1\{-1,1\}$. The matrix $N_1(t,\xi) := I + N^{(1)}(t,\xi)$ belongs to $S_{\text{ell}}^1\{0,0\}$ and is invertible with uniformly bounded inverse matrix $N_1^{-1}=N_1^{-1}(t,\xi)$. So, we have the following estimates for large time $t$:
\begin{align*}
\Big|\dfrac{D_t(b(t)+g(t)|\xi|^2)}{(b(t)+g(t)|\xi|^2)^2}\Big| &\leq \dfrac{b'(t)+|g'(t)|\,|\xi|^2}{(b(t)+g(t)|\xi|^2)^2} = \dfrac{e^tb(t)+g(t)|\xi|^2}{(b(t)+g(t)|\xi|^2)^2} \leq \frac{e^tb(t)}{b^2(t)} + \frac{g(t)|\xi|^2}{b(t)g(t)|\xi|^2} = \frac{e^t}{b(t)} + \frac{1}{b(t)}<1.
\end{align*}
Let
\begin{align*}
B^{(1)}(t,\xi) &= D_tN^{(1)}(t,\xi)-( \mathcal{R}(t,\xi)-F_0(t,\xi))N^{(1)}(t,\xi), \\
\mathcal{R}_3(t,\xi) &= -N_1^{-1}(t,\xi)B^{(1)}(t,\xi)\in S_{\text{ell}}^0\{-1,2\}.
\end{align*}
Then, we have the following operator identity:
\begin{equation*}
( D_t-\mathcal{D}(t,\xi)-\mathcal{R}(t,\xi))N_1(t,\xi)=N_1(t,\xi)( D_t-\mathcal{D}(t,\xi)-F_0(t,\xi)-\mathcal{R}_3(t,\xi)).
\end{equation*}
\begin{proposition} \label{Prop:Overdamping-Decreasing-Est-Zod}
The fundamental solution $E^{W}=E^{W}(t,s,\xi)$ to the transformed operator
\[ D_t-\mathcal{D}(t,\xi)-F_0(t,\xi)-\mathcal{R}_3(t,\xi) \]
can be estimated by
\begin{equation*}
(|E^{W}(t,s,\xi)|) \lesssim \frac{b(t)+g(t)|\xi|^2}{b(s)+g(s)|\xi|^2}\exp\Big( \frac{1}{2}\int_{s}^{t}\big( b(\tau)+g(\tau)|\xi|^2 \big)d\tau \Big)
\left( \begin{array}{cc}
1 & 1 \\
1 & 1
\end{array} \right),
\end{equation*}
with $(t,\xi),(s,\xi)\in [0,\infty)\times \mathbb{R}^n$.
\end{proposition}
\begin{proof}
We transform the system for $E^{W}=E^{W}(t,s,\xi)$ to an integral equation for a new matrix-valued function $\mathcal{Q}=\mathcal{Q}(t,s,\xi)$. If we differentiate the term
\[ \exp \bigg\{ -i\int_{s}^{t}\big( \mathcal{D}(\tau,\xi)+F_0(\tau,\xi) \big)d\tau \bigg\}E^{W}(t,s,\xi), \]
and then, integrate on $[s,t]$, we obtain that $E^{W}=E^{W}(t,s,\xi)$ satisfies the following integral equation:
\begin{align*}
E^{W}(t,s,\xi) & = \exp\bigg\{ i\int_{s}^{t}\big( \mathcal{D}(\tau,\xi)+F_0(\tau,\xi) \big)d\tau \bigg\}E^{W}(s,s,\xi)\\
& \quad + i\int_{s}^{t} \exp \bigg\{ i\int_{\theta}^{t}\big( \mathcal{D}(\tau,\xi)+F_0(\tau,\xi) \big)d\tau \bigg\}\mathcal{R}_3(\theta,\xi)E^{W}(\theta,s,\xi)\,d\theta.
\end{align*}
Let us define
\[ \mathcal{Q}(t,s,\xi)=\exp\bigg\{ -\int_{s}^{t}\beta(\tau,\xi)d\tau \bigg\} E^{W}(t,s,\xi), \]
with a suitable $\beta=\beta(t,\xi)$ which will be fixed later. It satisfies the new integral equation
\begin{align*}
\mathcal{Q}(t,s,\xi)=&\exp \bigg\{ \int_{s}^{t}\big( i\mathcal{D}(\tau,\xi)+iF_0(\tau,\xi)-\beta(\tau,\xi)I \big)d\tau \bigg\}\\
& \quad + \int_{s}^{t} \exp \bigg\{ \int_{\theta}^{t}\big( i\mathcal{D}(\tau,\xi)+iF_0(\tau,\xi)-\beta(\tau,\xi)I \big)d\tau \bigg\}\mathcal{R}_3(\theta,\xi)\mathcal{Q}(\theta,s,\xi)\,d\theta.
\end{align*}
The function $\mathcal{R}_3=\mathcal{R}_3(\theta,\xi)\in S_{\text{ell}}^0\{-1,2\}$ is uniformly integrable because of the last property of Proposition \ref{Prop:Overdamping-Integrable-symbol-Zod}. Hence, if the exponential term is bounded, then the solution $\mathcal{Q}=\mathcal{Q}(t,s,\xi)$ of the integral equation is uniformly bounded for a suitable weight $\beta=\beta(t,\xi)$.\\
The main entries of the diagonal matrix $i\mathcal{D}(t,\xi)+iF_0(t,\xi)$ are given by
\begin{align*}
(I) &= \dfrac{g(t)|\xi|^2}{2} + \dfrac{b(t)}{2}+\dfrac{g'(t)|\xi|^2+b'(t)}{2(b(t)+g(t)|\xi|^2)}+\dfrac{(g'(t)-2)|\xi|^2+b'(t)}{2(b(t)+g(t)|\xi|^2)},\\
(II) &= -\dfrac{g(t)|\xi|^2}{2} - \dfrac{b(t)}{2}+\dfrac{g'(t)|\xi|^2+b'(t)}{2(b(t)+g(t)|\xi|^2)}-\dfrac{(g'(t)-2)|\xi|^2+b'(t)}{2(b(t)+g(t)|\xi|^2)}.
\end{align*}
From the difference of $(II)-(I)$, we may see that the term $(I)$ is dominant with respect to (II). Therefore, we choose the weight $\beta=\beta(t,\xi)=(I)$. By this choice, we get
\[ i\mathcal{D}(\tau,\xi)+iF_0(\tau,\xi)-\beta(\tau,\xi)I = \left( \begin{array}{cc}
0 & 0 \\
0 & -\big( b(t)+g(t)|\xi|^2 \big)-\dfrac{(g'(t)-2)|\xi|^2+b'(t)}{b(t)+g(t)|\xi|^2}
\end{array} \right). \]
It follows
\begin{align*}
H(t,s,\xi) & =\exp \bigg\{ \int_{s}^{t}\big( i\mathcal{D}(\tau,\xi)+iF_0(\tau,\xi)-\beta(\tau,\xi)I \big)d\tau \bigg\}\\
& = \diag \bigg( 1, \exp \bigg\{ \int_{s}^{t}\Big( -\big( b(t)+g(t)|\xi|^2 \big)-\dfrac{(g'(t)-2)|\xi|^2+b'(t)}{b(t)+g(t)|\xi|^2} \Big)d\tau \bigg\} \bigg)\rightarrow \left( \begin{array}{cc}
1 & 0 \\
0 & 0
\end{array} \right)
\end{align*}
as $t\rightarrow \infty$. Hence, the matrix $H=H(t,s,\xi)$ is uniformly bounded for $(s,\xi),(t,\xi)\in [0,\infty)\times \mathbb{R}^n$. So, the representation of $\mathcal{Q}=\mathcal{Q}(t,s,\xi)$ by a Neumann series gives
\begin{align*}
\mathcal{Q}(t,s,\xi)=H(t,s,\xi)+\sum_{k=1}^{\infty}i^k\int_{s}^{t}H(t,t_1,\xi)\mathcal{R}_3(t_1,\xi)&\int_{s}^{t_1}H(t_1,t_2,\xi)\mathcal{R}_3(t_2,\xi) \\
& \cdots \int_{s}^{t_{k-1}}H(t_{k-1},t_k,\xi)\mathcal{R}_3(t_k,\xi)dt_k\cdots dt_2dt_1.
\end{align*}
Then, this series is convergent, since $\mathcal{R}_3=\mathcal{R}_3(t,\xi)$ is uniformly integrable. Hence, from the last considerations we may conclude
\begin{align*}
E^{W}(t,s,\xi)&=\exp \bigg\{ \int_{s}^{t}\beta(\tau,\xi)d\tau \bigg\}\mathcal{Q}(t,s,\xi)\\
& = \exp \bigg\{ \int_{s}^{t}\bigg( \frac{g(\tau)|\xi|^2}{2}+\frac{b(\tau)}{2}+\frac{b'(\tau)+g'(\tau)|\xi|^2}{2( b(\tau)+g(\tau)|\xi|^2)}+\frac{( g'(\tau)-2 )|\xi|^2+b'(\tau)}{2( b(\tau)+g(\tau)|\xi|^2)} \bigg)d\tau \bigg\}\mathcal{Q}(t,s,\xi),
\end{align*}
where $\mathcal{Q}=\mathcal{Q}(t,s,\xi)$ is a uniformly bounded matrix. Then, it follows
\begin{align*}
(|E^{W}(t,s,\xi)|) & \lesssim \exp \bigg\{ \int_{s}^{t}\bigg( \frac{g(\tau)|\xi|^2}{2}+\frac{b(\tau)}{2}+\frac{b'(\tau)+g'(\tau)|\xi|^2}{ b(\tau)+g(\tau)|\xi|^2}-\frac{|\xi|^2}{b(\tau)+g(\tau)|\xi|^2} \bigg\} \left( \begin{array}{cc}
1 & 1 \\
1 & 1
\end{array} \right) \\
& \lesssim \frac{b(t)+g(t)|\xi|^2}{b(s)+g(s)|\xi|^2} \exp \bigg( \frac{1}{2}\int_{s}^{t} ( b(\tau)+g(\tau)|\xi|^2)d\tau \bigg)\left( \begin{array}{cc}
1 & 1 \\
1 & 1
\end{array} \right).
\end{align*}
This completes the proof.
\end{proof}
Now let us come back to
\begin{equation} \label{Eq:Overdamping-Decreasing-Back-Zod}
W(t,\xi) = E_W(t,s,\xi)W(s,\xi) \qquad \text{for all} \qquad 0\leq s\leq t,
\end{equation}
that is,
\begin{align*}
\left( \begin{array}{cc}
\gamma(t,\xi)w(t,\xi) \\
D_t w(t,\xi)
\end{array} \right) = E_W(t,s,\xi)\left( \begin{array}{cc}
\gamma(s,\xi)w(s,\xi) \\
D_tw(s,\xi)
\end{array} \right),
\end{align*}
where $\gamma=\gamma(t,\xi):=\frac{b(t)}{2}+\frac{g(t)|\xi|^2}{2}$. Therefore, from Proposition \ref{Prop:Overdamping-Decreasing-Est-Zod} and \eqref{Eq:Overdamping-Decreasing-Back-Zod} we may conclude the following estimates for $0\leq s \leq t$:
\begin{align*}
\gamma(t,\xi)|w(t,\xi)| & \lesssim \frac{b(t)+g(t)|\xi|^2}{b(s)+g(s)|\xi|^2}\exp\Big( \frac{1}{2}\int_{s}^{t}\big( b(\tau)+g(\tau)|\xi|^2 \big)d\tau \Big)\big( \gamma(s,\xi)|w(s,\xi)| + |w_t(s,\xi)| \big), \\
|w_t(t,\xi)| & \lesssim \frac{b(t)+g(t)|\xi|^2}{b(s)+g(s)|\xi|^2}\exp\bigg( \frac{1}{2}\int_{s}^{t}\big( b(\tau)+g(\tau)|\xi|^2 \big)d\tau \bigg)\big( \gamma(s,\xi)|w(s,\xi)| + |w_t(s,\xi)| \big).
\end{align*}
Using the backward transformation
\[ w(t,\xi)=\exp\Big( \frac{1}{2} \int_0^t \big( b(\tau)+g(\tau)|\xi|^2 \big)d\tau \Big)\hat{u}(t,\xi),  \]
we arrive immediately at the following estimates for $(t,\xi)\in[0,\infty)\times \mathbb{R}^n$:
\begin{align*}
|\xi|^{|\beta|}|\hat{u}(t,\xi)| & \lesssim |\xi|^{|\beta|}|\hat{u}_0(\xi)| + |\xi|^{|\beta|-2}|\hat{u}_1(\xi)| \quad \mbox{for} \quad |\beta|\geq 2, \\
|\xi|^{|\beta|}|\hat{u}_t(t,\xi)| & \lesssim \big( b(t)+g(t)|\xi|^2 \big)|\xi|^{|\beta|}|\hat{u}_0(\xi)| + \big( b(t)+g(t)|\xi|^2 \big)|\xi|^{|\beta|-2}|\hat{u}_1(\xi)| \quad \mbox{for} \quad |\beta|\geq 2,
\end{align*}
which are our desired estimates in the theorem and the proof is completed.
\end{proof}

\section{Final remarks} \label{Section5}
\begin{remark}[Evaluation of results] \label{Rem:FinalRemarks-1}
Up to now we have results for the following general cases:
\begin{itemize}
\item $b(t)u_t$ scattering producing and $1/g \in L^1(\mathbb{R}^+)$ or $g \in L^1(\mathbb{R}^+)$,
\item $b(t)u_t$ non-effective and $1/g \in L^1(\mathbb{R}^+)$ or $g \in L^1(\mathbb{R}^+)$,
\item $b(t)u_t$ effective and $1/g \in L^1(\mathbb{R}^+)$ or $g \in L^1(\mathbb{R}^+)$,
\item $b(t)u_t$ overdamping producing and $1/g \in L^1(\mathbb{R}^+)$ and for special cases $g \in L^1(\mathbb{R}^+)$.
\end{itemize}
\end{remark}
\begin{remark}[Open Problem 1] \label{openproblem1}
As presented in Section \ref{Sec_Intro}, we classified 16 cases based on the structure of the friction term $b(t)u_t$ and of the viscoelastic damping term $-g(t)\Delta u_t$. In this paper, we have explored 8 of these cases (see Remark \ref{Rem:FinalRemarks-1}), leaving the remaining 8 for future work. Specifically, we will investigate the following cases of the time-dependent coefficient $g=g(t)$:
\begin{enumerate}
\item models with non-integrable and decreasing time-dependent coefficient $g=g(t)$,
\item models with non-integrable and slowly increasing time-dependent coefficient $g=g(t)$,
\end{enumerate}
with the following classification of the term $b(t)u_t$:
\begin{enumerate}
\item scattering producing to free wave equation,
\item non-effective dissipation,
\item effective dissipation,
\item over-damping producing.
\end{enumerate}
\end{remark}
\begin{remark} [Open Problem 2] \label{openproblem2}
Other interesting problems to study are the following Cauchy problems for the semilinear equations:
\begin{equation} \label{Eq:OpenProblem}
\begin{cases}
u_{tt}- \Delta u -g(t)\Delta u_t=|u|^p, &(t,x) \in (0,\infty) \times \mathbb{R}^n, \\
u(0,x) = u_0(x),\quad u_t(0,x) = u_1(x), &x \in \mathbb{R}^n,
\end{cases}
\end{equation}
and
\begin{equation} \label{Eq:OpenProblem1}
\begin{cases}
u_{tt}- \Delta u + b(t)u_t -g(t)\Delta u_t=|u|^p, &(t,x) \in (0,\infty) \times \mathbb{R}^n, \\
u(0,x) = u_0(x),\quad u_t(0,x) = u_1(x), &x \in \mathbb{R}^n,
\end{cases}
\end{equation}
with $p>1$. Very recently, decay estimates for solutions to the corresponding linear Cauchy problems \eqref{Eq:OpenProblem} and \eqref{Eq:OpenProblem1} have been derived by the authors of this paper in \cite{AslanReissig2023} and in the present work, specifically in the following cases:
\begin{enumerate}
\item models with non-integrable and decreasing time-dependent coefficient $g=g(t)$,
\item models with non-integrable and slowly increasing time-dependent coefficient $g=g(t)$.
\end{enumerate}
These results naturally motivate the investigation of the associated semilinear problems \eqref{Eq:OpenProblem} and \eqref{Eq:OpenProblem1} in regimes where decay estimates are available.
\end{remark}
\begin{remark} [Open Problem 3] \label{openproblem3}
Let us consider the semilinear model \eqref{Eq:OpenProblem1},
where the term $b(t)u_t$ is supposed to be effective in the sense of the paper \cite{Wirth-Effective=2007}.
If $g=g(t)\equiv 0$, then we have a critical exponent larger than $1$, this is the so-called Fujita exponent.
The question of the open problem 3 is as follows:
{\it Do we have any influence of the visco-elastic damping term $- g(t)\Delta u_t$ such that we have no longer a critical exponent larger than $1$?}
\end{remark}

%==========================================================
\section*{Acknowledgments}
Halit S. Aslan was supported by Funda\c c\~ao de Amparo \`a Pesquisa do Estado de S\~ao Paulo (FAPESP) (Grant No. 2021/01743-3 and Grant No. 2023/07827-0).
%==========================================================


\begin{thebibliography}{ABC}
\bibitem{AslanReissig2023}
H.S. Aslan, M. Reissig, Visco-elastic damped wave models with time-dependent coefficient, Math. Nachr. \textbf{297} (2024), 4535--4581.
%%%
%\bibitem{Dabbicco=2015}
%M. D'Abbicco, The threshold of effective damping for semilinear wave equations, Math. Methods Appl. Sci. \textbf{38}(6) (2015), 1032--1045.
%%%
\bibitem{DAbbiccoEbert2016}
M. D'Abbicco, M.R. Ebert,
\textit{A classification of structural dissipations for evolution operators}, Math. Meth. Appl. Sci. \textbf{39} (2016), 2558--2582.
%%%
%\bibitem{EbertReissigJHDE} M. R. Ebert, M. Reissig, \textit{Theory of damped wave models with integrable and decaying in time speed of propagation,} \emph{J. Hyperbolic Differ. Equ.} \textbf{13} (2016), 2, 417--439.
%%%
\bibitem{EbertReissigBook}
M. R. Ebert, M. Reissig,
Methods for partial differential equations. Qualitative properties of solutions, phase space analysis, semilinear models. Birkh\"auser/Springer, Cham, 2018.
%%%
%\bibitem{FangLuReissig2010}
%D. Fang, X. Lu, M. Reissig, High-order energy decay for structural damped systems in the electromagnetical field, \emph{Chin. Ann. Math. Ser. B}, \textbf{31} (2010), 237--246.
%%%
%\bibitem{Kainane2014}
%M. Kainane, Structural damped $\sigma$-evolution operators, PhD thesis, TU Bergakademie Freiberg, Germany, 2014.
%%%
\bibitem{AbdelatifReissig=2023}
A.K. Mezadek, M. Reissig, Semilinear evolution models with scale-invariant friction and visco-elastic damping, \textit{Commun. Pure Appl. Anal.} \textbf{22} (2023), no. 8, 2501--2532.
%%%
\bibitem{KainaneReissig2015-1}
M.K. Mezadek, M. Reissig, Qualitative properties of solutions to structurally damped $\sigma$-evolution models with time decreasing coefficient in the dissipation, \emph{Complex analysis and dynamical systems VI. Part 1}, Contemp. Math., Amer. Math. Soc. \textbf{653} (2015), 191--217.
%%%
\bibitem{KainaneReissig2015-2}
M.K. Mezadek, M. Reissig, Qualitative properties of solutions to structurally damped $\sigma$-evolution models with time increasing coefficient in the dissipation, \emph{Adv. Differential Equations}, \textbf{20} (2015), 433--462.
%\bibitem{KainaneReissig2015}
%M. Kainane, M. Riessig, Qualitative properties of solutions to structurally damped $\sigma$-evolution models with time decreasing coefficient in the dissipation. Complex analysis and dynamical systems VI, Contemporary Mathematics, vol. 653, Amer. Math. Soc., Providence, RI, (2015), 191–-217
%%%
%\bibitem{MatthesReissig2013}
%S. Matthes, M. Reissig,
%Qualitative properties of structurally damped wave models, \textit{Eurasian Math. J.} \textbf{4} (2013), 3, 84--106.
%%%
%\bibitem{Nishihara=2003}
%K. Nishihara K. $L^p-L^q$ estimates for solutions to the damped wave equations in 3-dimensional space and their applications, \textit{Math. Z.} \textbf{244} (2003), 631--649.
%%%
\bibitem{ReissigLu2009}
X. Lu, M. Reissig, Rates of decay for structural damped models with decreasing in time coefficients, \emph{Int. J. Dyn. Syst. Differ. Equ.} \textbf{2} (2009), 21--55.
%%%
%\bibitem{Ponce1985}
%G. Ponce, Global existence of small solutions to a class of nonlinear evolution equations, \emph{Nonlinear Anal.} \textbf{9} (1985), 5, 399--418.
%%%
\bibitem{Reissig=2011}
M. Reissig, Rates of decay for structural damped models with strictly increasing in time coefficients,
Complex Analysis and Dynamical Systems IV, Contemporary Mathematics, vol.554, Amer. Math. Soc. Providence, RI, (2011), 187--206.
%%%%
%\bibitem{Shibata2000}
%Y. Shibata, On the rate of decay of solutions to linear viscoelastic equation, \emph{Math. Methods Appl. Sci.} \textbf{23} (2000), 3, 203--226.
%%%
%\bibitem{Wirth-weak-diss=2004}
%J. Wirth, Solution representations for a wave equation with weak dissipation. \emph{Math. Methods Appl. Sci.} \textbf{27} (2004), 101--124.
%%%
\bibitem{JuniordaLuzNoneffective}
E.C. Vargas J\'{u}nior, C.R. da Luz, $\sigma$-evolution models with low regular time-dependent non-effective structural damping. \textit{Asymptotic Analysis}, \textbf{119} (2020), 61--86.
%%%
\bibitem{JuniordaLuzEffective}
E.C. Vargas J\'{u}nior, C.R. da Luz, $\sigma$-evolution models with low regular time-dependent effective structural damping. \textit{J. Math. Anal. Appl.}, \textbf{499} (2021) no. 2, 125030.
%%%
\bibitem{WirthThesis}
J. Wirth, \textit{Asymptotic properties of solutions to wave equations with time-dependent dissipation}, PhD thesis, TU Bergakademie Freiberg, Germany, 2004.
%%%
\bibitem{Wirth-Noneffective=2006}
J. Wirth, Wave equations with time-dependent dissipation I. Non-Effective dissipation, \emph{J. Differential Equations}, \textbf{222} (2006) 487-514.
%%%
\bibitem{Wirth-Effective=2007}
J. Wirth, Wave equations with time-dependent dissipation II. Effective dissipation, \emph{J. Differential Equations}, \textbf{232} (2007), 1, 74--103.
%%%
%\bibitem{Wirth-Scattering=2007}
%Jens Wirth, Scattering and modified scattering for abstract wave equations with time-dependent dissipation, \emph{Adv. Differential Equations}, \textbf{12} (2007)10, 1115--1133.
\end{thebibliography}
\end{document}